\documentclass[11pt]{amsart}
\usepackage{amsmath}
\usepackage{amsfonts}
\usepackage{amssymb}
\usepackage{graphicx}
\usepackage{hyperref}
\usepackage{mathrsfs}
\usepackage{pb-diagram}
\usepackage{epstopdf}
\usepackage{amsmath,amsfonts,amsthm,enumerate,amscd,latexsym,curves}
\usepackage{bbm}
\usepackage[mathscr]{eucal}
\usepackage{epsfig,epsf,epic}
\usepackage{caption}
\usepackage{subcaption}
\usepackage{multirow}
\usepackage{color}
\usepackage[all]{xy}
\usepackage{fourier}
\usepackage{comment}

\newtheorem{thm}{Theorem}[section]
\newtheorem{lemma}[thm]{Lemma}

\newtheorem{cor}[thm]{Corollary}
\newtheorem{prop}[thm]{Proposition}
\newtheorem{defn}[thm]{Definition}
\newtheorem{example}[thm]{Example}
\newtheorem{remark}[thm]{Remark}

\newtheorem{assumption}[thm]{Assumption}

\numberwithin{equation}{section}

\newcommand{\conn}{\nabla}

\newcommand{\CY}{\mathcal{Y}}

\newcommand{\Z}{\mathbb{Z}}

\newcommand{\R}{\mathbb{R}}
\newcommand{\C}{\mathbb{C}}

\newcommand{\bP}{\mathbb{P}}

\newcommand{\be}{\mathbf{1}}

\newcommand{\bS}{\mathbb{S}}

\newcommand{\BL}{\mathbf{L}}
\newcommand{\AI}{A_\infty}

\newcommand{\WT}[1]{\widetilde{#1}}

\newcommand{\pt}{\mathrm{pt}}

\newcommand{\CA}{\mathcal{A}}
\newcommand{\cF}{\mathcal{F}}
\newcommand{\cL}{\mathcal{L}}

\newcommand{\Hom}{\mathrm{Hom}}

\newcommand{\one}{\mathbf{1}}

\newcommand{\CF}{\mathrm{CF}}
\newcommand{\Fuk}{\mathrm{Fuk}}

\newcommand{\MF}{\mathrm{MF}}
\newcommand{\bL}{\mathbb{L}}

\newcommand{\cO}{\mathcal{O}}
\newcommand{\val}{\mathrm{val}}
\newcommand{\ex}{\mathrm{ex}}
\newcommand{\ee}{\mathfrak{e}}

\begin{document}

\title{Gluing localized mirror functors}
\author[Cho]{Cheol-Hyun Cho}
\address{Department of Mathematical Sciences, Research Institute in Mathematics\\ Seoul National University\\ Gwanak-gu\\Seoul \\ South Korea}
\email{chocheol@snu.ac.kr}
\author[Hong]{Hansol Hong}
\address{Center of Mathematical Sciences and Applications\\Harvard University \\ Cambridge \\ MA \\ USA}
\email{hhong@cmsa.fas.harvard.edu, hansol84@gmail.com}
\author[Lau]{Siu-Cheong Lau}
\address{Department of Mathematics and Statistics\\ Boston University\\ Boston\\ MA\\ USA}
\email{lau@math.bu.edu}

\begin{abstract}
	We develop a method of gluing the local mirrors and functors constructed from immersed Lagrangians in the same deformation class.  As a result, we obtain a global mirror geometry and a canonical mirror functor.  We apply the method to construct the mirrors of punctured Riemann surfaces and show that our functor derives homological mirror symmetry.
\end{abstract}

\maketitle
\tableofcontents

\section{Introduction}
In \cite{CHL, CHL2, CHLtoric}, we introduced a localized mirror formalism to construct Landau-Ginzburg (LG) mirrors $W:Y \to \Lambda$ of a symplectic manifold $X$, and to understand homological mirror symmetry (HMS in short) between them. ($\Lambda$ is the Novikov field.)
We used a single Lagrangian $L$ as a reference to construct a LG model $W$, as well as an $\AI$-functor 
from the Fukaya category $\Fuk(X)$ to the matrix factorization (MF in short) category of $W$.

In this case the mirror space $Y$ is affine, which is defined by the Maurer-Cartan equation for formal deformations of $L$.
We referred to this approach as a localized mirror formalism since the constructed `mirror' as well as the functor reflects the symplecto-geometric information probed by $L$.  In good cases, for instance when $L$ is the Seidel Lagrangian in an elliptic or a hyperbolic orbifold $X=\bP^1_{a,b,c}$ studied in \cite{CHL}, we can show that $L$ generates $\Fuk(X)$ and its image under the functor generates $\MF(W)$.  Hence the constructed LG model is the true mirror in the sense of HMS.

In general, a single Lagrangian is not enough to probe all the symplecto-geometric information of $X$.  In this paper, we consider a collection of Lagrangians lying in the same deformation class, and develop a method to glue the localized mirrors and functors constructed from these Lagrangians to obtain a global mirror LG model together with a global $\AI$-functor for the study of HMS.  We apply this to punctured Riemann surfaces equipped with pair-of-pants decompositions and show that our functor is an equivalence on the derived level.

There have been several works for a local-to-global approach to HMS: Fukaya-Oh \cite{fukaya-oh}, Kontsevich \cite{Kont-SG},  Seidel \cite{Sei-spec}, Nadler-Zaslow \cite{NZ09}, Abouzaid-Seidel \cite{AS10}, Dyckerhoff \cite{Dyc17}, Lee \cite{HLee}, Pascaleff-Sibilla \cite{PS16}, Gammage-Shende \cite{GS17} and so on.  The local-to-global approach for MF categories has been studied by Orlov \cite{Orlov12} and Lin-Pomerleano \cite{LP}.  The idea in these studies is to decompose $X$ and its mirror into local pieces, prove HMS for each local piece, glue the local derived categories on the two sides in the same way, and finally show that they coincide with the original global categories.  This is a grand program which has been only partially understood and requires much further effort.

The approach we take in this paper is different from the above in the following aspects.  First, we always deal with the global Fukaya category and never decompose it into local categories.  Our initial data is a collection of (immersed) Lagrangians $\{L_i:i=1,\ldots,N\}$ in the same deformation class.  From the localized mirror formalism, each $L_i$ gives an LG model $W_i : U_i \to \Lambda$ as
well as an $\AI$-functor for the \emph{global} Fukaya category of $X$.  $L_i$ is always treated as an object in $\Fuk(X)$ rather than the Fukaya category of some local pieces.  

The main difficulty of decomposing the Fukaya category into local pieces is that, there are quantum corrections to the glued category from global pseudo-holomorphic curves not contained in any of the local pieces.  Our formulation automatically incorporates global pseudo-holomorphic curves and hence bypasses this difficulty.  Note that even when $L_i$ is contained in a local part of $X$, as an object in $\Fuk(X)$, its endomorphism space (and also the $A_\infty$ operations with other Lagrangians) already receives contributions from global curves.  

%For example, for a punctured sphere, consider generating non-compact Lagrangians from real loci. There is a non-trivial holomorphic polygon for these Lagrangians, and if we consider Seidel Lagrangian in a single pair of pants, then one can find that this polygon contributes to the higher part of $\AI$-functor from this reference Lagrangian.

Second, the mirror geometry $(Y,W)$ and the mirror functor is geometrically constructed in a systematic way rather than predicted from physics or some other reasonings.  This is analogous to the Gross-Siebert reconstruction program \cite{GS07}, while we use symplecto-geometric and categorical techniques rather than algebraic techniques in toric degenerations.
We find a way to glue the local mirror spaces $U_i$ to obtain a generally non-affine mirror space $Y$, with
a well-defined potential function $W:Y \to \Lambda$.
More importantly, we can glue localized mirror functors $\mathcal{F}^{L_i}$ as well to form a global $\AI$-functor 
$$\mathcal{F}^{\textrm{global}} : \Fuk(X) \to \MF(W).$$
Here, $\MF(W)$ is defined as a \emph{homotopy fiber product} of each $\MF(W_i)$.

In fact, a single criterion for gluing mirror charts can be used to glue the potential functions as well as localized mirror functors, which is our main theorem \ref{thm:maingluealgintro} explained below.
The construction is purely algebraic and works in full generality.

It is illustrative to compare our approach with the SYZ program \cite{SYZ} and the family Floer theory \cite{Fukaya-famFl,Tu-FM,Tu-reconstruction,Ab-famFl1,Ab-famFl2}.  The collection of immersed Lagrangians $\{L_i:i=1,\ldots,N\}$ in the same deformation class plays the role of a Lagrangian fibration in the SYZ program.  One crucial difference is that, the immersed Lagrangians that we take intersect in a proper way so that there are isomorphisms between their formal deformations over $\Lambda_+$.  On the other hand, fibers of the SYZ fibration are disjoint from each other and so they can never be isomorphic.  In other words, in our formulation the charts of formal deformations of $L_i$ overlap and glue to a global space, while in the SYZ scenario the charts of formal deformations of the torus fibers (which are formed by flat $\Lambda_0^\times$-connections) are disjoint from each other.  Our approach involves only finitely many Lagrangians, while the family Floer theory in the SYZ setup needs all the SYZ fibers and involves taking limit.  Thus our functor is more directly constructed and can be computed more explicitly.

We carry out this construction for punctured Riemann surfaces and prove that our functor derives an equivalence.  Mikhalkin \cite{Mik04} constructed pair-of-pants decompositions for Calabi-Yau hypersurfaces.  For a single one-dimensional pair-of-pants, Seidel \cite{Seidel-g2} constructed an immersed Lagrangian to study HMS.  Based on this, we find a suitable collection $\{L_i:i=1,\ldots,N\}$ of immersed Lagrangians with the help of tropical geometry (see Figure \ref{fig:decomp}).
We make certain consistent choices and glue up the localized mirrors to obtain a global mirror space, which is an open subset of a $\Lambda$-valued toric CY.  Using the exactness of $L_i$, we find that the image of $W\Fuk(X)$ under our functor lies in the $\C$-valued part $\MF(W_\C)$.  Then we show that the functor sends generators of $W\Fuk(X)$ to generators of $\MF(W_\C)$, and it induces isomorphisms of morphism spaces in the cohomology level. 

\begin{figure}[htb!]
	\includegraphics[scale=0.45]{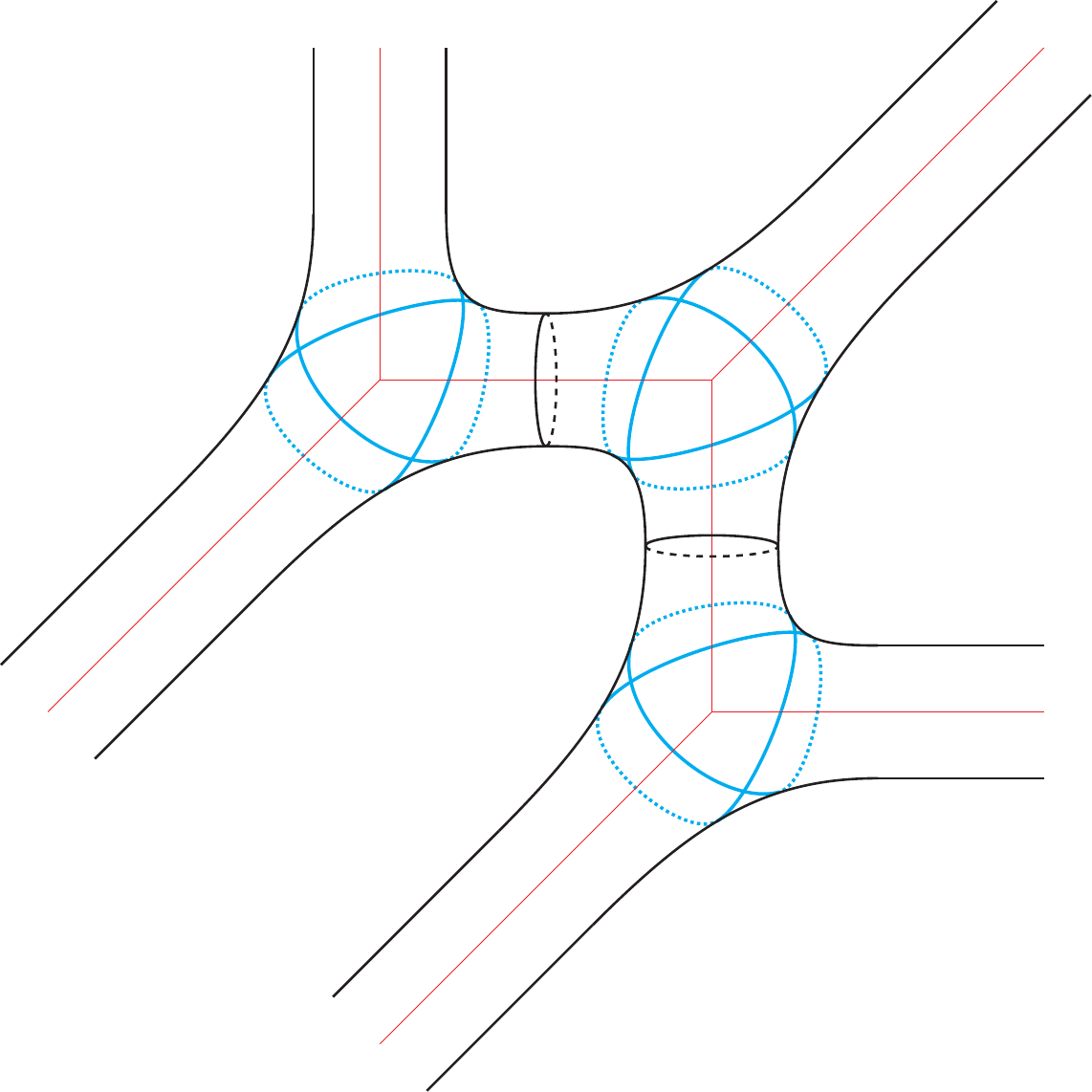}
	\caption{A pair-of-pants decomposition gives a collection of immersed Lagrangians.  The collection we use is constructed by deforming these Lagrangians such that they intersect in a suitable way.}
	\label{fig:decomp}
\end{figure}

In the beautiful work \cite{Lee},  Lee started with a punctured Riemann surface with a pair-of-pants decomposition, and a corresponding ($\C$-valued) toric CY with a superpotential serving as a mirror candidate.  She proved HMS for such a mirror pair by using certain neck-stretching Hamiltonians to reduce computations for the global category to that for the wrapped Fukaya category of each pair-of-pants.  

We have a different goal in this paper.  As explained above, an important goal of us is to construct the mirror toric CY geometrically from pair-of-pants decompositions, which is achieved by using our gluing technique for a collection of immersed Lagrangians.  Moreover, while the Novikov field $\Lambda$ plays a crucial role in our gluing construction, it does not appear in \cite{Lee} since the objects she considered stays within the wrapped Fukaya category.  Here we use formal deformations of compact immersed Lagrangians, and these formal deformations go out of the wrapped Fukaya category.  Furthermore, in this paper HMS is derived by the mirror functor canonically constructed from the immersed Lagrangians, rather than assigning maps between objects and morphisms by hand.

Below is a more detailed summary of our construction.
Let us start with an immersed Lagrangian $\bL_0$ in a symplectic manifold $X$.  Take odd-degree immersed generators
$X_1,\cdots, X_n$ and set $b_0=\sum x_j X_j$ where $x_j \in \Lambda_+$.  Consider the $\AI$-algebra of $\bL_0$ \cite{AJ}, \cite{FOOO}.
We solve the weak Maurer-Cartan equation $m(e^{b_0}) = W_0(b_0) \cdot \be_{\bL_0}$ to get a formal deformation space $U_0$.  For simplicity let's suppose weak Maurer-Cartan equation holds for all $b_0$, in other words, $U_0 \cong \Lambda_+^n$.  We get a localized mirror functor $\mathcal{F}^{\bL_0}: \Fuk(X) \to MF(W_0)$. The functor is canonically obtained as a curved Yoneda embedding (\cite{CHL}).

To illustrate the gluing intuitively, let's consider the scenario that there is another immersed Lagrangian $\bL_1$ with odd-degree immersed generators $X_1',\cdots, X_{n}'$, such that a certain smoothing of $\bL_0$ at $X_1,\cdots, X_j$  and that of $\bL_1$ at $X_1',\cdots, X_{j}'$ are related by Lagrangian isotopies.  For the moment we assume that $X_i$ and $X_j$ are located at different immersed points for every $i\not=j$, and similar for $X_i'$.\footnote{When $X_j$ is the complement of $X_i$ at the same immersed point, wall-crossing phenomenon for the smoothings will occur which is understood in \cite{HKL,ERT}.}  

We need to glue the formal deformation space of $\bL_0$ with that of its smoothing at $X_1,\cdots, X_j$.  We have vanishing cycles $T_i$ in the smoothing corresponding to $X_i$ for $i=1,\ldots,j$.  Let $\nabla^{(t_1,\ldots,t_j)}$ be the flat connection on $\bL_0$ that has holonomy $t_i$ across $T_i$ for all $i=1,\ldots,j$ and trivial elsewhere.  Intuitively discs bounded by $\bL_0$ with corners $X_i$ correspond to discs bounded by its smoothing that passes through the gauge cycle $T_i$.  Hence the leading order term of the gluing should be 
$$x_i=t_i$$
for $i=1,\ldots,j$.  In good situations this is precisely the formula.

Our key observation is that there is another process involved in the gluing of $\bL_0$ and $\bL_1$, namely gauge change between the smoothings of $\bL_0$ and $\bL_1$.  The vanishing cycles $T_1,\ldots,T_j$ obtained from smoothing of $\bL_0$ are not necessarily identified with those $T_1',\ldots,T_{j}'$ obtained from smoothing of $\bL_1$ under the Lagrangian isotopy.  We need to take a cobordism between the gauges $\{T_1,\ldots,T_j\}$ and $\{T_1',\ldots,T_{j}'\}$ (assuming such an cobordism exists).  This leads to non-trivial gluing between $(x_{j+1},\ldots,x_n)$ and $(x'_{j+1},\ldots,x'_{n})$ which takes the form
$$ x_l = \left(\prod_{i=1}^j t_i^{a_i}\right) x'_l$$
for some $a_i \in \Z$.  It also shows an interesting feature that isotopic choices of gauge cycles produce equivalent Floer theories, which explains why there can be infinitely many Landau-Ginzburg mirrors of $X$ (whose derived categories are all equivalent).
For example, when $X$ is a 4-punctured sphere, we show that $\cO_{\bP^1}(-k) \oplus \cO_{\bP^1}(k-2)$ for any $k \in \Z$ can be obtained as a mirror space in Section \ref{sec:4sp}.

The above ideas give the gluing formula of mirror spaces in simple situations.  In order to derive the gluing formula in general situations and to construct the global mirror functor, we need to take a Floer theoretical and algebraic approach.

The algebraic setting is that $\bL_0$ and $\bL_1$ belong to the same deformation class, which means that they can be connected by isomorphisms of formal deformations of a chain of Lagrangians (see Definition \ref{def:class}).  The formal deformation spaces of $\bL_0$ and $\bL_1$ are related from each other by gluing of formal deformation spaces of the chain of Lagrangians according to the quasi-isomorphisms between their deformed Floer theories.  Thus the disc potentials of $\bL_0$ and $\bL_1$ are related by analytic continuation.

Recall that the notion of an isomorphism in an ordinary category can be extended to an $\AI$-category $\mathcal{C}$:
Two objects $\bL_0,\bL_1$ are isomorphic if  there are $\alpha \in \Hom_{\mathcal{C}} (\bL_0,\bL_1)$ and $\beta \in \Hom_{\mathcal{C}} (\bL_1,\bL_0)$ such that 
$$m_1(\alpha) = m_1 (\beta)=0,  \quad m_2(\alpha,\beta)=\one_{L_0} + m_1(\gamma_1), \quad m_2(\beta, \alpha)=\one_{L_1}+ m_1(\gamma_2)$$ 
for some $\gamma_1, \gamma_2$. The isomorphisms $\alpha,\beta$ can be used to show that
their Yoneda functors $\mathcal{Y}^{0}$ and $\mathcal{Y}^{1}$ are quasi-isomorphic, which is
proved in Theorem \ref{thm:ye}. In particular, the two objects $\bL_0$ and $\bL_1$ are quasi-isomorphic in $\mathcal{C}$ 
since Yoneda embedding is fully faithful (\cite{fukaya02}).

Given two reference Lagrangians $\bL_0, \bL_1$ with formal deformation spaces $U_0,U_1$,
suppose that we have open subsets   $V_i \subset U_i$ for $i=0,1$ and a homeomorphism $\phi:V_0 \to V_1$ (under the Novikov topology) such that
for each $b_0 \in V_0$, and $b_1 = \phi(b_0) \in V_1$, we have isomorphisms 
$$\alpha \in \Hom_{\mathcal{C}} \big((\bL_0,b_0),(\bL_1,b_1)\big), \beta \in \Hom_{\mathcal{C}} \big((\bL_1,b_1),(\bL_0,b_0) \big).$$
We identify $V_0$ and $V_1$ via $\phi$ and this is $U_0 \cap U_1$. 
In practice, we find the homeomorphism $\phi$ by solving the cocycle condition of $\alpha$, $m_1^{b_0,b_1}(\alpha) = 0$.  This gives the coordinate change between $(x_1,\cdots,x_n)$ and $(x_1',\cdots, x_n')$.

Given the intersection $U_0 \cap U_1$ of  deformation spaces as above, we have the following main theorem which
glues the mirror data $(U_1,W_1,\mathcal{F}^{\bL_1})$ and $(U_2,W_2,\mathcal{F}^{\bL_2})$.

\begin{thm}[Theorem \ref{thm:maingluealg}]\label{thm:maingluealgintro}
Suppose $(\bL_i,b_i)$ for $i=0,1$ are isomorphic on $U_0 \cap U_1$.
Then  the following holds.
\begin{enumerate}
\item Their potential functions agree (and defines the potential $W_{01}$ on $U_0 \cap U_1$). i.e. we have $$W_0(b_0) = W_1(b_1), \;\; \textrm{for} \;\; b_0 \in V_0, b_1 = \phi(b_0) \in V_1.$$
\item  The $\AI$-functors $\mathcal{F}^{\bL_i}: \Fuk(X) \to MF(W_i)$ composed with restrictions $r_i: MF(W_i) \to MF(W_i|_{V_i})$ for $i=0,1$ are quasi-isomorphic to each other. i.e. we have
$$ r_0 \circ \mathcal{F}^{\bL_0} \cong r_1 \circ \mathcal{F}^{\bL_1} $$
Moreover, the required natural transformations as well as homotopies for this quasi-isomorphisms are explicitly given using $\alpha,\beta$.
\item There exists 
$$\mathcal{F}^{\mathrm{global}}: \Fuk(X) \to MF(W_0) \times^h_{MF(W_{01})} MF(W_1) $$
an $\AI$-functor from $\Fuk(X)$ to the homotopy fiber product of the two
dg-categories $MF(W_0)$ and $MF(W_1)$.

\end{enumerate}
\end{thm}
The above theorem easily generalizes to the case of many charts with no non-trivial triple intersections.

This gluing method has the following interesting features.  First, in case $\bL_0$ and $\bL_1$ are disjoint from each other, $\Hom_{\mathcal{C}} \big((\bL_0,b_0),(\bL_1,b_1)\big)$ is zero and hence there are no isomorphisms between $U_0$ and $U_1$. Therefore, their formal neighborhoods $U_0$ and $U_1$ are disjoint.  Second, since immersed Lagrangians have formal deformations being $\Lambda_+$-valued (which simultaneously cover all positive energy levels), even if $\bL_1$ is not Hamiltonian isotopic to $\bL_0$,
there is a chance that $(\bL_0,b_0)$ and $(\bL_1,b_1)$ are isomorphic to each other.

In Section \ref{sec:localisot}, we give such an example in which $\bL_1$ is a Lagrangian isotopy (but not a Hamiltonian isotopy) of $\bL_0$, and their formal deformations are isomorphic under the coordinate change $$ x' = T^{2 \delta} x,  y' = T^{- \delta} y,  z' =  T^{-\delta} z.  $$
There is a large overlap  $U_0 \cap U_1$ of formal deformation spaces, namely $\val (x) > 0, \val(y) > \delta, \val(z) > \delta$.  In this case Lagrangian isotopy can be understood as a translation of (the valuation of) the formal deformation space. 

Even when $\bL_0, \bL_2$ are disjoint (and hence $U_0 \cap U_2 = \emptyset$)
we can try to isotope $\bL_0$ to $\bL_1$ so that the formal deformation spaces of $\bL_1$ and $\bL_2$ has a non-trivial intersection: \begin{equation}\label{eq:012}
U_0 \cap U_1 \neq \emptyset, U_1 \cap U_2 \neq \emptyset.
\end{equation}
We give an example of such a construction in Section \ref{sec:glue2s}. In this case, we consider 4-punctured sphere
which is a union of two pair of pants. We take a Seidel Lagrangian $\bL_0,\bL_2$ in each pair of pants which are
disjoint, and isotope $\bL_0$ to the other pair of pants to obtain $\bL_1$ where \eqref{eq:012} holds. 
Working with Novikov coefficients is essential in this regard. 

In the exact setting where $X$ is an exact symplectic manifold and $\bL_i$ are exact Lagrangians, we can absorb all the area terms (namely the Novikov elements) into the variables $x_i$ for the gluing and also for the disc potential $W$.  Hence $W$ reduces to a $\C$-valued function $W_\C$ over a complex manifold.\footnote{More precisely, we assume that the domain of $W$ can be extended to the $\Lambda_0$-part of the mirror space in the sense of Novikov convergence, so that when restricted to $\C$-valued charts (which has valuation either $0$ or $+\infty$) $W_\C$ is just a polynomial.}
$W_\C$ can be trivially extended as $W_{\Lambda_0}$ over the manifold over $\Lambda_0$ (defined by the same transition maps).

We have the categories $\MF(W_\C)$, $\MF(W_{\Lambda_0})$ and $\MF(W)$.  Note that the inclusion $\C \to \Lambda$ is NOT continuous (with respect to the usual topology of $\C$ and the Novikov valuation of $\Lambda$).  Thus \emph{the restriction from $\MF(W)$ to $\MF(W_\C)$ is ill defined} in general.  On the other hand,
both $W_{\Lambda_0}$ and $W$ are in the Novikov topology and do not have this issue.
Assuming that in our construction $W_{\Lambda_0}^{-1}\{0\} \subset W^{-1}\{0\}$.  Then we have the restriction from $\MF(W)$ to $\MF(W_{\Lambda_0})$.  $\MF(W_\C)$ can be treated as a subcategory of $\MF(W_{\Lambda_0})$.  In this situation the image of the wrapped Fukaya category under our functor lies in $\MF(W_\C)$.

%we can take $\C$-reduction of the above construction. As our construction works only in $\Lambda$-coefficients,  we first embed the exact  (wrapped) Fukaya category $\mathcal{F}^{\C}(M)$ to   $\Lambda$-valued Fukaya caegory  $\mathcal{F}^\Lambda(M)$.  In $\mathcal{F}^\Lambda(M)$, we obtain mirror charts, gluing data and homological mirror functors in $\Lambda$.  In fact, these data come from Lagrangian Floer theory between exact Lagrangian submanifolds.  Therefore we can absorb all the area terms and obtain $\C$-valued theory using the embedding of exact Fukaya category one more time.      This $\C$-reduction produces elements in power series ring $\C[[x_1,\cdots,x_n]]$ and we need to make convergence assumption in general to make it into polynomial ring $\C[x_1,\cdots,x_n]$.

%..   More general type of mirror symmetry has been discovered (\cite{KKOY09}, \cite{Seidel-g2}, ...)

Another approach to construct the mirror space using several reference Lagrangians was given in \cite{CHL2}.  Note that the setting in \cite{CHL2} is different from this paper.  The collection of reference Lagrangians in this paper are in the same deformation class.  On the other hand, in \cite{CHL2}, the union of the reference Lagrangians is treated as a single immersed Lagrangian, whose formal deformations are used to construct a noncommutative LG model. 

It is an interesting problem to find a relation between these two constructions.  We illustrate this by a 4-punctured sphere in Section \ref{sec:flop}.  We show how to use the non-commutative mirror of \cite{CHL2} to construct commutative mirror charts. %Turning on an additional formal parameter corresponds to a commutative mirror chart and different choices of parameters are related by desired coordinate changes. 
 There are two different choices which lead to two commutative mirror spaces, and they are related to each other by an Atiyah flop. Indeed, these choices are related to two different ways to take pair-of-pants decompositions of the 4-punctured sphere (Figure \ref{fig:flop}). Conjecturally, different pair-of-pants decompositions for a general punctured Riemann surface lead to (underlying spaces of) the mirror LG models related by (a sequence of) operations of a similar nature.

%In \cite{CHL2}, we took several Lagrangians as one reference and found that the corresponding mirror space is a non-commutative Landau-Ginzburg model, which is given as quiver algebras with relations together with a potential function as well as mirror functors. The approach in our paper is different in that we try to find relations between formal deformation spaces of each reference Lagrangian, whereas in \cite{CHL2}, they are not identified and they are regarded as different spaces. 

%The upshot of Section \ref{sec:flop} is that a stability (a choice of quadratic top-form) provides Lagrangian skeleton, and the immersed Lagrangians representing this skeleton can be used to construct mirrors. For different choice of stabilities, we find that the commutative mirror spaces are related by a flop in the case of 4-punctured sphere.
 
\subsection*{Notations for the Novikov field}

$$\Lambda = \left\{\sum_{i=0}^\infty a_i T^{A_i} \mid \lim_{i \to \infty} A_i = \infty, a_i \in \C \right\}$$
is the universal Novikov field where $T$ is a formal parameter.
We have the valuation function
$$\mathrm{val}:\, \sum_{i=0}^\infty a_i T^{A_i} \mapsto \textrm{Min}_{i} \{A_i\} \textrm{ and set } \mathrm{val}(0) = +\infty.$$
This also defines a filtration on $\Lambda$  (and $\Lambda$-modules in the same way) by 
$$F^\lambda (\Lambda) = \{ \alpha \in \Lambda \mid \mathrm{val}(\alpha) \geq \lambda \}$$
We denote by $\Lambda_0 = F^0(\Lambda)$ and by $\Lambda_+$ its maximal ideal.
%
%$$\Lambda_0 = \{ \sum_{i=0}^\infty a_i T^{A_i}  \in \Lambda \mid  \forall i, A_i  \geq 0\}, \;\;\;  \Lambda_+ = \{ \sum_{i=0}^\infty a_i T^{A_i}  \in \Lambda \mid  \forall i, A_i  > 0 \}$$

\subsection*{Acknowledgment}
We express our gratitude to Kaoru Ono, Conan Leung and Dongwook Choa for interesting discussions. The work of C.H. Cho is supported by Samsung Science and Technology Foundation under Project Number SSTF-BA1402-05.
 The work of S.C. Lau in this paper is partially supported by the Simons collaboration grant. The second named author would like to thank Bong Lian and Shing-Tung Yau for constant support and encouragement. His work is substantially supported by Simons Collaboration Grant on Homological Mirror Symmetry. 

\section{Localized mirror functor formalism}\label{sec:lmf}
In this section, we recall how to construct localized mirrors and mirror functors with respect to choices of  Lagrangians $\bL$
from \cite{CHL} and \cite{CHLtoric}. In this formalism, different choices of $\bL$ in a given symplectic manifold $M$ provide different
local mirrors, and our strategy later will be to take several Lagrangians $\bL$ in the same deformation class, each of which provides a local chart of a global mirror of $X$. We refer readers to the appendix for a brief explanation on algebraic preliminaries.

Given a filtered $\AI$-category  $\mathcal{A}$ over $\Lambda$, we
assign to an unobstructed object $
\bL$ a potential function $W_\bL$ on its Maurer-Cartan space, together with an $\AI$-functor $\cF:\CA \to \MF_{\AI}(W_\bL)$ 
from $\CA$ to the dg-category of matrix factorizations of $W_\bL$ (\cite{CHL}). Let us explain in more detail.

\subsection{Local mirror space}
We define a local mirror chart from $\bL$ by considering the formal deformation space of a given object $\bL$. 
More precisely, we will solve a Maurer-Cartan equation which governs the deformation of an object $\bL$.
\begin{defn}
Given an object $\bL$ of a filtered $\AI$-category  $\mathcal{A}$,  
an element $b \in F^+Hom^{\mathrm{odd}}(\bL,\bL)$ is called a  weak Maurer-Cartan element if
\begin{equation}\label{eqn:wMC}
m_0^b= m_0 (e^b) =m_0(1) + \sum_{k=1}^{\infty} m_k (b, \cdots,b ) = c \cdot \be_\bL
\end{equation}
Here $\be_\bL$ is a unit of $\textrm{Hom}(\bL,\bL)$.
In this case $(\bL,b)$ is said to be weakly unobstructed. If $c=0$, then it is called Maurer-Cartan element.
\end{defn}
If the convergence of the above infinite sum is guaranteed, we can extend $b$ to $F^0Hom^{\mathrm{odd}}(\bL,\bL)$.

Such Maurer-Cartan elements can be used to deform the original $A_\infty$-structure to a new one.
\begin{equation}\label{eqn:fdefb}
m_k^{b} (x_1, \cdots, x_k):= \sum_{\substack{i_1,\cdots, i_{k+1} \\ i_1 + \cdots + i_{k+1} = l}} m_{k+l} (\overbrace{b, \cdots, b}^{i_1} ,x_1, b, \cdots, b , x_k , \overbrace{b, \cdots, b}^{i_{k+1}}).
\end{equation}
It is easy to see that $\{m_k^b\}$ defines an $A_\infty$-algebra \cite{FOOO}, and in the case that $b$ is
weakly unobstructed,  $m_0^b = c \cdot \be_\bL$, and in particular we have $(m_1^b)^2 = 0$ from the definition of the unit $\be_\bL$.

We denote the space of such $b$'s by $\mathcal{MC}(L)$ for Maurer-Cartan elements, and by $\mathcal{MC}_{weak} (L)$ for weak Maurer-Cartan elements. In general, one should consider also gauge equivalence relations between Maurer-Cartan solutions, but we omit them as they will be trivial in our examples.

For a symplectic manifold $(X,\omega)$ and  a compact Lagrangian $L$ in $X$, a gapped filtered $A_\infty$-algebra $\left(CF(L,L),\{m_k\}_{k \geq 0} \right)$ possibly with a nontrivial curvature $m_0$ was constructed, and the deformation theory above was introduced to define deformed Lagrangian Floer homology, which is the homology of $m_1^b$ (\cite{FOOO} for smooth Lagrangians and \cite{AJ} for immersed Lagrangians).

% Note that deformations by different elements in $\mathcal{MC}(L)$ could lead to isomorphic objects.  We shall identify them in the construction of moduli spaces.  Isomorphisms are discussed in Section \ref{sec:gluing}.
%$\mathcal{MC} (L)$ can be interpreted as a (formal) deformation space of $L$.  For instance, in the case of toric manifolds, $\mathcal{MC}_{weak} (L)$ agrees with the first cohomology of $L$ which encodes geometric deformation of $L$ (see \cite{FOOO-T}).
%
%In case there is no danger of confusion, we will sometimes call $\mathcal{MC} (L)$ (or $\mathcal{MC}_{weak} (L)$) deformation space or moduli space, omitting ``formal" in front of them.
%
%We can regard $c$ of \eqref{eqn:wMC} as a function on $\mathcal{MC} (L)$  as follows.
%
%

We can divide the applications into two types, one for an immersed Lagrangian, and the other for a Lagrangian torus
(one can also consider a mixed type which we omit).
For the immersed case, we pick odd immersed generators $X_1,\cdots,X_n \in CF^1(\bL,\bL)$ and form a formal linear combination
$$b = x_1 X_1 + \cdots x_n X_n, x_i \in \Lambda_+ ( \textrm{or} \;\; \Lambda_0 \;\textrm{with convergence assumptions}).$$
We regard $x_1,\cdots,x_n$ as formal smoothing parameters at the corresponding immersed points.

For the case of a Lagrangian torus,  one can vary the holonomy of a flat line bundle over a Lagrangian $L$ to deform $L$ (as an object of Fukaya category). More precisely, if we equip $L$ with a  line bundle $E$ with a $\C^\times$-flat connection $\nabla$, the $A_\infty$-operations on $CF(L,L)$ are deformed as follows:
$$m_k^{(L,\nabla)} (x_1, \cdots, x_k) =  \sum_{\beta \in \pi_2(M,L)} \left( hol_{\partial \beta} \nabla \right) m_{k,\beta} (x_1, \cdots, x_k) T^{\omega(\beta)}$$
where $m_{k,\beta}$ is the contribution from holomorphic disks in class $\beta$ to the original $m_k$-operator on $CF(L,L)$. By considering all possible holonomies, we obtain a (formal) moduli of objects in $\Fuk(M)$ isomorphic to $\left(\C^\times\right)^{\dim_\R H_1 (L;\R)}$, which is nothing but the space of all $\C^\times$ flat line bundles on $L$ modulo equivalence.

Now, for $\theta_i \in H^1(L,\C)/H^1(L,\Z)$, we set $ b = \sum_{i=1}^n x_i \theta_i$,
and consider a flat $\C$-line bundle with holonomy given by $\rho^b : \pi_1(L) \to \C^*$ given by
$$\rho^b(\gamma) = exp (2\pi \sqrt{-1}(b,\gamma))$$
Furthermore, we will always choose a special representative of $\C^\times$-line bundles whose connection behaves like a delta function. More precisely, for $L \cong \R^n / \Z^n$, we fix (oriented) hyper-tori $H_i = \epsilon_i + \R \langle e_i \rangle$ for $\epsilon_i \in \R/\Z$ so that the parallel transport over a path $\gamma$ is given by multiplying $z_i^{\pm}$ whenever $\gamma$ runs across $H_i$ where the sign in the exponent is determined by the parity of the intersection $\gamma \cap H_i$. 
(The it is essential that we choose such a flat connection to construct mirror functor. See \cite{CHLtoric} for more details.)
One can slightly enlarge this space by considering $(\Lambda_0)^\times$-line bundle (or combining $\C^\times$-bundle and boundary deformation by elements in $H^1(L,\Lambda_+)$), where the holonomy is still required to have valuation zero.

\subsection{Localized mirror functor}
Denote by $\MF_{A_\infty} (W)$ the $A_\infty$-category obtained from the dg-category $\MF(W)$
as explained in Appendix \ref{sec:dgsign}.
We recall the construction of localized mirror functor from \cite{CHL}. The convention in this paper is slightly different as  we will put the reference Lagrangian in the second slot of $\Hom( \, \cdot \, , \,\cdot)$. This makes the mirror functor covariant
and the signs are much simpler (We thank Sangwook Lee for this observation).
\begin{defn}\label{def:lmf}
There exists an $\AI$-functor $\mathcal{F}^{\bL}$ which assigns to a Lagrangian $L$ of $\Fuk(X)$ a matrix factorization $M_L$ :
$$\mathcal{F}^{\bL}:\Fuk(X) \to \MF_{A_{\infty}} \left(W_{\bL} \right) \qquad L \mapsto \left(\Hom (L, (\bL,b) ), { -} m_1^{0,b} \right).$$
Higher components of $\mathcal{F}^{\bL}$ are defined using $m_k$-operations
$$ \mathcal{F}^{\bL}_k : \Hom (L_1, L_2) \otimes \cdots \otimes \Hom (L_{k-1}, L_k) \to \Hom_{\MF_{A_\infty}} (M_{L_1}, M_{L_k})$$
given by 
$$(x_1, \cdots, x_k) \mapsto  \sum_{i \geq 1} m_{k+i} (x_1, \cdots, x_k, \bullet, b, \cdots,b)  $$
where $m_{k+i} (x_1, \cdots, x_k, \bullet, b,\cdots,b) $ is a map sending $y \in M_{L_k}$ to $m_{k+i} (x_1, \cdots, x_k, y,b \cdots, b)$.  
\end{defn}
Recall that we are using the convention
$$ \Hom_{\MF_{A_\infty}} (M_{L_1}, M_{L_k}) = \Hom_{\MF_{dg}} (M_{L_k},M_{L_1}).$$

%The functor can be interpreted as a family version of the Yoneda embedding in the sense that we consider the hom-functor $\Hom(\cdot, (\bL,b))$, but $b$ is varying over the (weak) Maurer-Cartan space. We examine 

\begin{thm}\cite{CHL}\label{lem:familyyoneda}
This defines a covariant $A_\infty$-functor (with no further sign correction).
\end{thm}
For reader's convenience and to check signs, we give a proof.
\begin{proof}
Set $\mathcal{C} := \Fuk(X)$ and $\mathcal{D} := \MF_{A_\infty} (W)$.
Let us first consider the immersed case.
For a tuple $(x_1, \cdots, x_k)$ of composable morphisms in the Fukaya category, we have to check
$$ m_1^{\mathcal{D}} ( \mathcal{F}^{\bL}_k ( \mathbf{x}) ) + \sum m_2^{\mathcal{D}} (\mathcal{F}^{\bL} (\mathbf{x}^{(1)}), \mathcal{F}^{\bL} (\mathbf{x}^{(2)}) )= \sum (-1)^{|\mathbf{x}^{(1)}|'} \mathcal{F}^{\bL} (\mathbf{x}^{(1)}, m^{\mathcal{C}} (\mathbf{x}^{(2)}), \mathbf{x}^{(3)}).$$
where $\mathbf{x}^{(1)}$, $\mathbf{x}^{(2)}$ on the left hand side have positive lengths, and so does $\mathbf{x}^{(2)}$ on the right hand side.
We plug $y$ into both sides of the equation. Terms in the left hand side gives
\begin{eqnarray*}
 m_1^{\mathcal{D}} ( \mathcal{F}^{\bL}_k ( \mathbf{x}) ) (y) &=& \delta (m^{\mathcal{C}} ( \mathbf{x}, y, e^b) ) - (-1)^{|\mathcal{F}^{\bL}_k (\mathbf{x})|} m^{\mathcal{C}} (\mathbf{x}, \delta(y), e^b ) \\
 &=& {-} m^{\mathcal{C}} ( m^{\mathcal{C}} ( \mathbf{x},y, e^b), e^b) { -} (-1)^{|\mathbf{x}|'} m^{\mathcal{C}} (\mathbf{x}, m^{\mathcal{C}} (y, e^b), e^b)
\end{eqnarray*}
where $\delta = { - } m_1^{0,b}$, and
\begin{eqnarray*}
 m_2^{\mathcal{D}} (\mathcal{F}^{\bL} (\mathbf{x}^{(1)}), \Psi (\mathbf{x}^{(2)}) ) (y) &=& (-1)^{| \mathcal{F}^{\bL} (\mathbf{x}^{(1)})|} \mathcal{F}^{\bL} (\mathbf{x}^{(1)}) \left( \mathcal{F}^{\bL} (\mathbf{x}^{(2)}) (y) \right) \\
 &=& (-1)^{| \mathcal{F}^{\bL} (\mathbf{x}^{(1)})|} m^{\mathcal{C}} (\mathbf{x}^{(1)}, m^{\mathcal{C}} (\mathbf{x}^{(2)},y,e^b ), e^b) \\
 &=& { - } (-1)^{|\mathbf{x}^{(1)}|'} m^{\mathcal{C}} (\mathbf{x}^{(1)}, m^{\mathcal{C}} (\mathbf{x}^{(2)},y,e^b ), e^b) 
 \end{eqnarray*}
whereas the right hand side will read
\begin{eqnarray*}
(-1)^{|\mathbf{x}^{(1)}|'} \mathcal{F}^{\bL} (\mathbf{x}^{(1)}, m^{\mathcal{C}} (\mathbf{x}^{(2)}), \mathbf{x}^{(3)}) (y) &=& (-1)^{|\mathbf{x}^{(1)}|'} m^{\mathcal{C}} (\mathbf{x}^{(1)},m^{\mathcal{C}} (\mathbf{x}^{(2)}), \mathbf{x}^{(3)},y,e^b ).
\end{eqnarray*}
%{ From the computation, we see that the minus sign for $m_1^{0,b}$ is essentially due to the Koszul signs in ($A_\infty$-)dg-setting does not involve degree shifting.}
In the toric case, we need to assume that all the relevant Lagrangian intersections are away from the fixed hyper-tori.
In that case, $m_k^{b,0,\cdots,0}$ can be interpreted as recording the holonomy along the arc in $\bL$, which
can be topologically described by the intersection number with hyper-tori. The rest of the proof is similar and omitted.
\end{proof}

\section{Geometric approach for gluing local mirrors}\label{sec:geoM}
In this section, we provide an intuitive way to glue local mirror spaces.  It aims to give the readers an intuitive understanding of the gluing.  In order to justify the gluing formula and to glue the mirror functors, we shall use Floer theory and algebraic methods in the next section.
%This approach is rather topological and good for constructing a mirror space, but it is difficult to use it for gluing local mirror functors. Therefore, in later sections  we will introduce algebraic approach for gluing, which gives much more precise and functorial method.

Recall that our local mirror chart was defined by Maurer-Cartan solution space. The idea of geometric approach is to consider a family of Lagrangians $\bL_t$ such that the Maurer-Cartan spaces of $\bL_0$ and $\bL_1$ are connected via the Maurer-Cartan spaces of $\bL_t$.   In a sense this is how Strominger-Yau-Zaslow constructs their mirror given a Lagrangian torus fibration, and this works well for toric manifolds for example as  any two different torus fibers can be connected via a family of torus fibers.  The Floer theory remains almost the same since the moduli spaces of holomorphic discs are isomorphic and only the symplectic area of holomorphic discs changes (\cite{Fukaya-famFl}). In this case their Maurer-Cartan spaces $H^1(L_t,\mathbb{\C})$ can be easily identified.

But in our approach, we have two new features.   Firstly,  in our case  the Lagrangians in the family may not be diffeomorphic
to each other as we allow immersed Lagrangians and  Lagrangian surgery at immersed points in the process. Secondly, we incorporate the notion of gauge hyper-surface of \cite{CHLtoric}, which provides additional freedom for coordinate change.

In the case of punctured Riemann surfaces, we will see that Lagrangian surgery is necessary to move the immersed Lagrangian in one pair of pants to another.
\subsection{Lagrangian surgery}
Consider an immersed Lagrangian $\bL$, with weak bounding cochain $b = \sum x_i X_i$.
Suppose $X_1$ is an immersed generator, and we will consider a Lagrangian surgery of $\bL$ at $X_1$ 
to obtain $\WT{\bL}$. 
The corresponding variable $x_1$ may be regarded as a formal surgery parameter  at $X_1$.
After surgery, we would like to introduce a holonomy parameter $\WT{x}_1$ for $\WT{\bL}$, which we want to identify with a non-zero $x_1$.

Let us recall the setup of Lagrangian surgery from Seidel \cite{Sei_graded} (see also Polterovich \cite{Po1})
Let $L_1,L_2$ be Lagrangian submanifolds in $M$, intersecting at a point $p \in L_1 \cap L_2$.
We can find an embedding $j:B^{2n} \subset \C^n \to M$ with $j(0) = p$ and $j^{-1} (L_1) = \R^n \cap B^n,
j^{-1} (L_2) = \sqrt{-1} \R^n\cap B^n$, $j^*\omega = \epsilon \cdot \omega_{std}$.
In $B^{2n}$, Lagrangian handle can be defined by 
$$H  = \cup_{t \in \R} \gamma(t) S^{n-1} \subset \C^n $$
where $\gamma:\R \to \C$ is an embedding with $\gamma(t)=t$ for $t \leq -1/2, \gamma(t)=\sqrt{-1}t$ for $t\geq 1/2$
and $\gamma(\R) \cap - \gamma(\R) = \emptyset$.
Then, Lagrangian surgery,  $L_1 \sharp L_2$,  is defined by taking out neighborhoods $L_i \cap j(B^{2n})$ of $p$ from $L_i$ for $i=1,2$ and attaching the Lagrangian handle $j(H \cap B^{2n})$. 
In $L_1 \sharp L_2$,  we have a codimension one submanifold $ \{0\} \times S^{n-1}$, which
is the {\em vanishing cycle $\bS$} of the Lagrangian surgery.

In our applications, we would like to perform Lagrangian surgery for Lagrangian intersections of index one.
In fact, Seidel have shown that Lagrangian surgery can be carried out for graded Lagrangian intersections
of index one:  on $\C^n$, quadratic complex $n$-form $ \Omega ^{\otimes 2} = (dz_1 \wedge \cdots \wedge dz_n)^{\otimes 2}$ can be used
to defined graded Lagrangian submanifold $L$, which means an additional data of grading $gr_L: L \to \R$ such that
$e^{2\pi i gr_L(x)} =\Omega ^{\otimes 2} (T_xL)$ for all $x$. For $gr_{L_1} \equiv 0$ and
$gr_{L_2} = 1-n/2$, the index at the origin equals one. 
\begin{lemma}[Lemma 2.13 \cite{Sei_graded}]
There exist a grading $gr_H$ for the Lagrangian handle $H$, which agree with that of  $L_1$ and $L_2$.
\end{lemma}
We take this Lagrangian surgery as a local model in our applications.
In the setting of Fukaya-Oh-Ohta-Ono in Chapter 10 \cite{FOOO}, it is constructed as a graph of $df$ where $f(x) = \epsilon log |x|$,
and the above construction corresponds to the case $\epsilon <0$ thereof.  Moreover, they have shown the following theorem
(for a more precise formulation, we refer readers to the reference).
\begin{thm}[Theorem 55.5. \cite{FOOO}]\label{thm:sur_disc}
Let $u$ be an isolated Fredholm-regular $J$-holomorphic triangle with boundary on $L_0,L_1,L_2$ with one of its corner at $p \in L_1\cap L_2$ with multiplicity one.
Then for $\epsilon <0$, there exists a unique Fredholm-regular holomorphic strip $\WT{u}$ with boundaries on  $L_0, L_1 \sharp L_2$ and close to $u$ for sufficiently small $\epsilon$.
\end{thm}
They also mentioned that this theorem should extend to more general polygons in a straightforward way.
We remark that in the case of $\epsilon >0$ they show that there are $S^{n-2}$ dimensional family of corresponding Fredholm-regular holomorphic strips.

Now, we perform Lagrangian surgery at $X_1 \in \bL$. Suppose there exist a $J$-holomorphic polygon $u$ which contributes
to the Maurer-Cartan equation or the potential function for $\bL$, and they are of multiplicity one at $X_1$.
Then in view of the above theorem, it is natural to expect to that there exist a corresponding $\WT{u}$ which continues to contribute
to the Maurer-Cartan equation of the potential function for the Lagrangian $\WT{\bL}$ after surgery too.
In $\bL$, the fact the $u$ has corner at $X_1$ is recorded by the variable $x_1$. Let us introduce a corresponding holonomy variable $\WT{x}_1$ for $\WT{\bL}$ as follows.
\begin{defn}\label{def:vc}
Consider the vanishing cycle $\bS \subset \WT{\bL}$ of the surgery at $X_1$.
Consider a flat $\C^*$-connection whose holonomy is concentrated near this codimension
one submanifold $\bS$.
Namely, consider a flat $\C$-bundle (or in general $\Lambda_0^*$-bundle) on $\WT{\bL}$ whose 
holonomy is trivial away from the neighborhood of $\bS$, and is $\WT{x}_1$ along a curve transverse to
$\bS$ from $L_0$ to $L_1$. 
\end{defn}

Note that $\WT{x}_1$ is non-zero since it is holonomy, whereas $x_1$ could be zero.  In view of Theorem \ref{thm:sur_disc}, the gluing between the formal deformations of $\bL$ and that of its smoothing at $X_1$ is naively $\WT{x}_1=x_1$.  However, in general there are corrections since there could be polygons bounded by $\bL$ passing through $p$ without turning (changing branches), and there could also be polygons bounded by $\WT{\bL}$ passing through the vanishing cycle but whose corresponding polygons pass through the complement of $X_1$ instead of $X_1$.  All these are  automatically taken care of by the algebraic method introduced in the next section.

Let us illustrate the construction in the case of Riemann surfaces.
In this case $\bL$ is an oriented immersed curve, and depending on the choices of branches at the immersed
point there are two generators for Floer theory of $\bL$.
One is of odd degree, while the other is even.  As in a typical deformation theory, the odd-degree part governs deformations while the even-degree part governs obstructions. 

\begin{figure}[htb!]
    \includegraphics[scale=0.45]{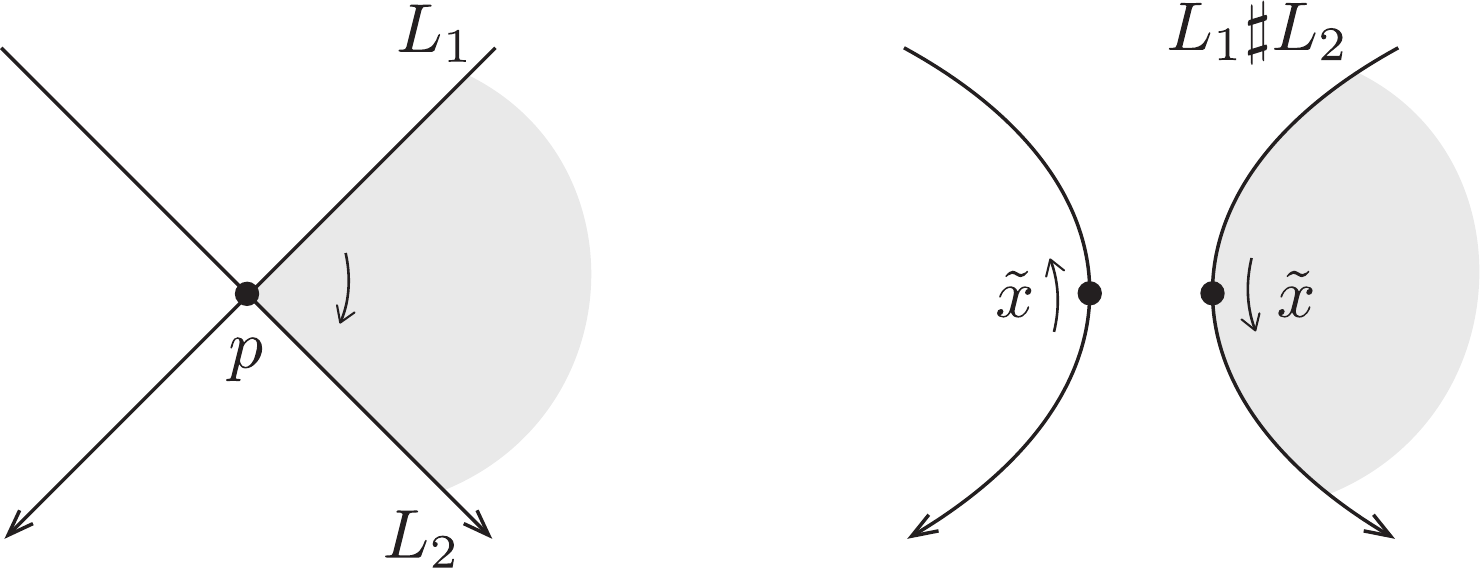}
    \caption{Smoothing at odd immersed point and the corresponding holonomy variable $\WT{x}$}
		\label{fig:smooth_out}
\end{figure}

We take a surgery at the odd immersed generator $X$ of $\bL$. The vanishing sphere $\bS$ is given by two points $S^0$ in $\WT{\bL}$. Thus we consider a $\C$ or $\Lambda_0$ line bundle on $\WT{\bL}$ whose holonomy across the points $S^0$
are illustrated in Figure \ref{fig:smooth_out}. Note that for any polygon $u$ in the surface which intersect $X$ only once with
acute angle as illustrated in the figure, there is an obvious corresponding polygon $\WT{u}$ which boundary on $\WT{\bL}$ (This corresponds to the theorem \ref{thm:sur_disc} in $\dim \bL =1$). The formal variable contribution near $p$ for $u$ and $\WT{u}$ are given by $x$ and $\WT{x}$ respectively.

\subsection{Gauge hypersurface and Floer isomorphisms}\label{sec:isotg}
The previous construction introduces a gauged hyper-surface on $\bL$.
In this subsection, we analyze the phenomenon when we move the  hypersurface $\bS$ in $\bL$ by a smooth isotopy.
The resulting flat connections have the same holonomy, and related by gauge equivalences as shown in \cite{CHLtoric}.
In \cite{CHLtoric} Lemma 5.3, it is shown that the resulting mirror matrix factorizations are isomorphic for these type of gauge equivalences.
Similarly, we have the following lemma for  $\AI$-algebra of Morse-model $CF(\bL,\bL)$ as in \cite{Seidel-g2} and \cite{Sheridan11}, where
generators of Floer theory is given by either self-intersection points or Morse critical points.
\begin{lemma}\label{lem:gc}
Let $\bL$ be an immersed Lagrangian and let $\bS_0,\bS_1$ are two choices of gauge hypersurfaces that are smoothly isotopic and  avoids immersed and Morse critical points. Let $\WT{x}$ be the holonomy across $\bS_i$. Then there exist an $\AI$-isomorphism between their respective $\AI$-algebras of $\bL$.
\end{lemma}
\begin{proof}
We may assume that $\bS_0,\bS_1$ are smoothly isotopic and the isotopy intersects only one of immersed or Morse critical points,
say $p$. For Morse flows, there is no holonomy contribution as we are computing self Hom's.
Thus holonomy contribution for $\bS_i$ is given by the intersection multiplicities of arcs between immersed generators of
$J$-holomorphic polygons. Therefore, it is easy to see that the differences of holonomy for $\bS_0$ and $\bS_1$ appear for arcs starting or ending at $p$. If $p$ is a corner of a $J$-holomorphic polygon, there are two boundary arcs connected to $p$ which is
affected by the smooth isotopy.  It is not difficult to check that  the correspondence $ p \mapsto \WT{x}^{\pm 1} \cdot p$
provides the desired isomorphism (precise sign can be chosen from the orientations of smooth isotopy).
\end{proof}

%We need to glue up these local charts to obtain a global mirror model.
%As the local mirror is obtained via Floer theory, we need to understand  the
%`communications' of Floer theory among the immersed Lagrangians.
%
%Seidel \cite{Sei-spec} speculated the application of pair-of-pants decompositions to prove homological mirror symmetry.
% A recent work of Lee \cite{Lee} proved homological mirror symmetry for punctured Riemann surfaces using such decompositions.
%
%But our approach is rather different in that we do not need to decompose Fukaya category, but just use this decomposition of the symplectic manifold to construct
%an immersed Lagrangian submanifold. (It is in general difficult to have decent property for Fukaya category) In our approach, we choose an immersed Lagrangian corresponding to each pair-of-pants.

Let us illustrate this in the case of Riemann surfaces
Consider two curves intersecting at $P$ and $Q$ with $\bS_0$, $\bS_1$ illustrated as in Figure \ref{fig:movehol},
where the smooth isotopy moves $\bS_0$ to $\bS_1$ along the minimal path crossing $P$.
Then the Floer differential  $\delta$, which is
$$\delta_{\bS_0}(P) = \WT{x} Q, \quad \delta_{\bS_1}(P) = Q$$
Therefore, isomorphism of chain complexes are obtained by setting $P' = \WT{x}P$, so that $\delta_{\bS_1}(P') = \WT{x}Q$.
For the bounding cochains, the chain isomorphism sends $pP$ to $pP' = p\WT{x}P$. Thus we may set
$p' = p\WT{x}$. If the strip contributes to the potential, the corresponding terms are $\WT{x}pq$ for $\bS_0$, and $p'q$ for $\bS_1$. Note that the potential remains the same with the relation $p'= p \WT{x}$.

\begin{figure}[htb!]
    \includegraphics[scale=0.45]{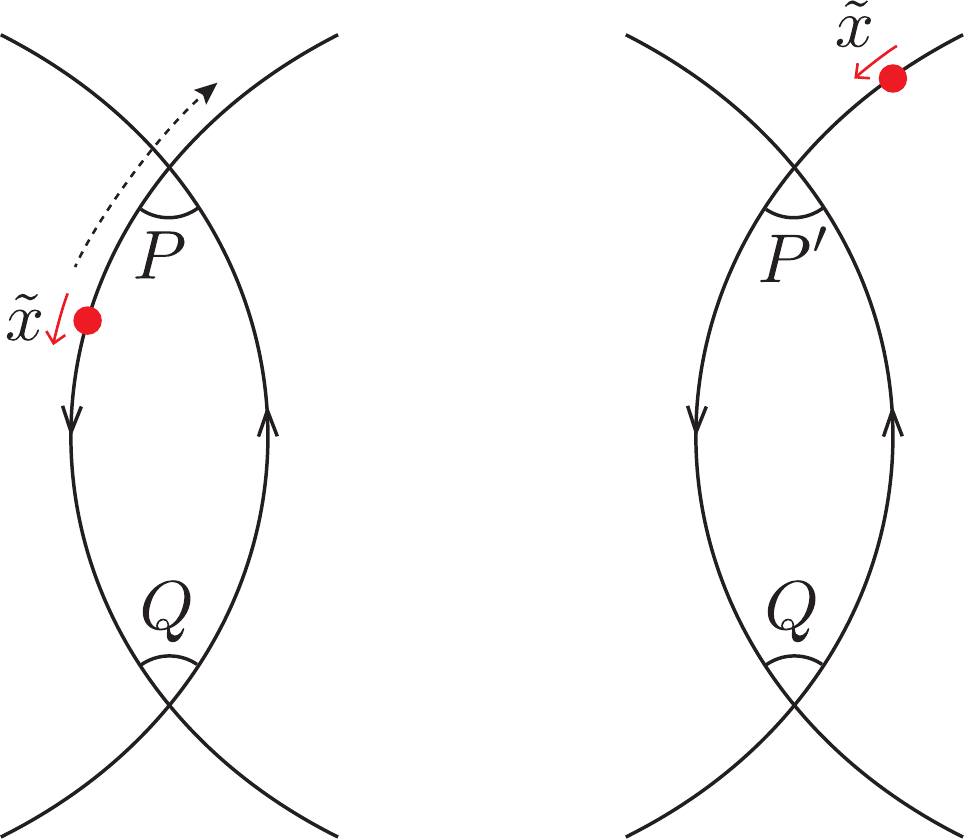}
    \caption{}
		\label{fig:movehol}
\end{figure}

For more complicated smooth isotopy, a decomposition of isotopy provides a composition of the above chain isomorphisms.
We can observe that  the choice of a smooth isotopy between $\bS_0$, $\bS_1$ can give rise to different chain isomorphisms,
and we will see that this is related to a choice of mirror Landau-Ginzburg mirror model (the choice of $k \in \Z$ in the next subsection).

\subsection{Illustration for the case of 4 punctured sphere}\label{sec:4sp}
Let us illustrate the geometric approach for the case of 4-punctured sphere (denote it by $X$), which is given by the gluing of
two pairs of pants. Each pair of pants will give rise to mirror $\C^3$-charts, and for any given $k \in \Z$, we will show
how to obtain a coordinate change between these charts. As a result, we obtain  a  mirror  Landau-Ginzburg model 
on a toric Calabi-Yau 3-fold  $$\check{X} = \cO_{\bP^1}(-k) \oplus \cO_{\bP^1}(k-2),$$
which is given by $U_i = \C^3 = \mathrm{Spec} \, \C[x_i,y_i,z_i]$ for $i=1,2$, whose intersection $U_1\cap U_2$ is $\C^* \times \C^2$ \begin{align}\label{eq:kk2}
&\mathrm{Spec} \, \C[x_1,x_1^{-1},y_1,z_1] \to \mathrm{Spec} \, \C[x_2,x_2^{-1},y_2,z_2], \nonumber \\
&x_1 = x_2^{-1}, y_1 = z_2 x_2^{k}, z_1 = y_2 x_2^{2-k}.
\end{align}
The Landau-Ginzburg superpotential $W:\check{X} \to \C$ is given by 
$$W = x_1y_1z_1 = x_2y_2z_2.$$

\begin{figure}[htb!]
    \includegraphics[scale=0.45]{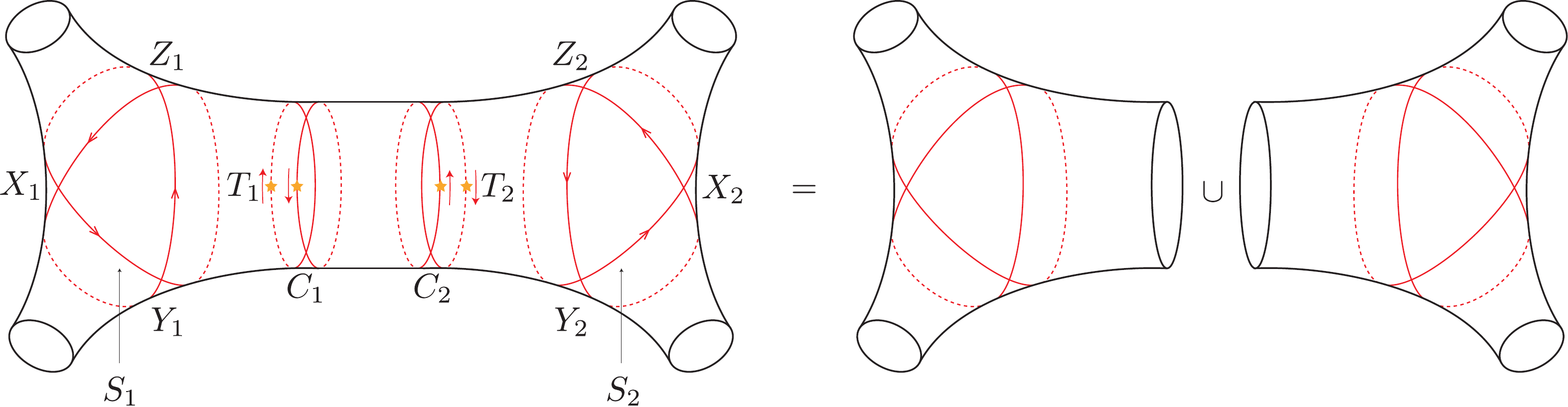}
    \caption{A pair-of-pants decomposition of the four-punctured sphere and immersed Lagrangians.}
		\label{fig:pp-decomp}
\end{figure}

Consider a pair-of-pants decomposition on a 4-punctured sphere as in Figure \ref{fig:pp-decomp}, consider
immersed $S^1$, denoted as  $\bL_1, \bL_2$ in each pair-of-pants.
Recall the such an immersed Lagrangian was first considered by Seidel \cite{Seidel-g2} to prove mirror symmetry for a genus two surface (or its orbifold quotient $\mathbb{P}^1_{5,5,5}$).
There are three transverse immersed points, giving the immersed generators $X,Y,Z$ in odd degree and $\bar{X},\bar{Y},\bar{Z}$ in even degree.  The Floer complex is $\CF(L,L) = \mathrm{Span}\{\mathbf{1},X,Y,Z,\bar{X},\bar{Y},\bar{Z},\mathrm{pt}\}$ as a vector space.

In \cite{CHL}, we have shown that the boundary deformation $b=x_iX_i+y_iY_i+z_iZ_i$ (for $i=1$ or $2$) satisfies weak Maurer-Cartan equationin the sense of Fukaya-Oh-Ohta-Ono, and defined a localized mirror functor, which extends
to hold in 4-punctured sphere $X$.
\begin{lemma}[c.f. \cite{CHL}] \label{lem:pp}
We equip $L_i$ a flat $\C$-bundle of holonomy $(-1)$ or equivalently non-standard spin structure. Then $b=x_iX_i + y_iY_i + z_iZ_i$ satisfies weak Maurer-Cartan equation for $x_i,y_i,z_i \in \Lambda_0$.
\end{lemma}

\begin{figure}[htb!]
    \includegraphics[scale=0.5]{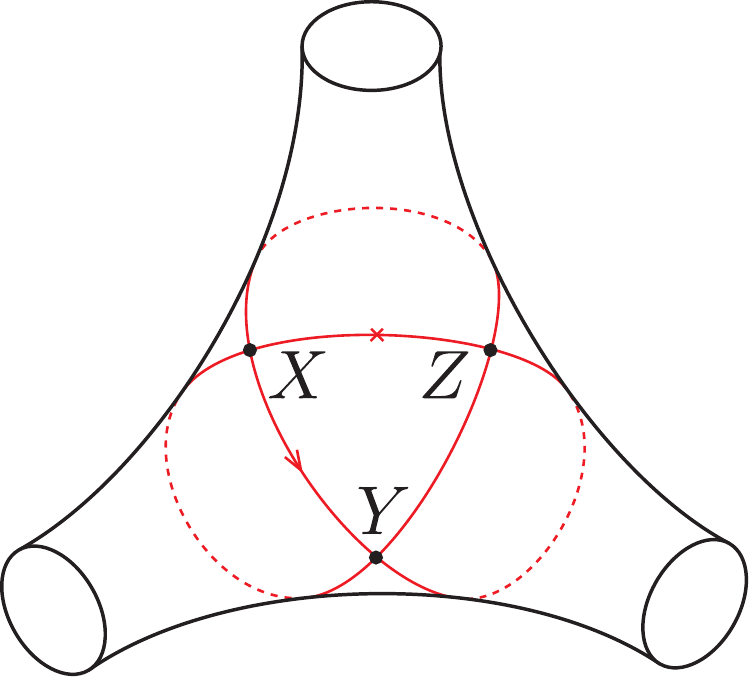}
    \caption{Seidel Lagrangian in a pair-of-pants}\label{fig:Seidel_Lag}
\end{figure}

\begin{proof}
Consider $m_0^b=\sum_{k>0} m_k(b,\ldots,b)$ which has even degree.  We need to argue that the coefficients of $\bar{X}_i,\bar{Y}_i,\bar{Z}_i$ are zero.  It is easy to see that constant polygons do not contribute to $m(e^b)$ because only allowed inputs are $X_i,Y_i,Z_i$.  Also, it is not difficult to see that only polygon with boundary on $\bL_i$ with corners $X_i,Y_i,Z_i$are two $X_iY_iZ_i$-triangles, say $\Delta_1,\Delta_2$ where $\Delta_1,\Delta_2$ are in the upper and lower hemisphere respectively. With trivial spin structure on $L$ or flat bundle of trivial holonomy, we have 
$$m_2(X_i,Y_i) = \bar{Z}_i T^{\omega(\Delta_1)}, m_2(X_i,Y_i) = \bar{Z}_i T^{\omega(\Delta_2)}.$$ Thus they do not cancel out in $m(e^b)$.
But if we impose non-trivial spin structure on $L_i$, or a flat $\C$-bundle of holonomy $(-1)$, $m_2(X_i,Y_i)$ and $m_2(Y_i,X_i)$ cancel out if 
$\omega(\Delta_1) = \omega(\Delta_2)$. The same goes for $\bar{X}_i,\bar{Y}_i$ output. As the remaining even output is $\be_i$, we obtain the claim.
\end{proof}
%
%In \cite{CHL} we focus on the valuation-zero part (which is $\C^3\subset \Lambda_0^3$).  $\C^3$ is the part that takes account of the exact Lagrangians (while $\Lambda_0^3$ also include the non-exact ones). {\color{red}(Shouldn't we work over $\Lambda_+$? The way we reduce to $\C$ is not using this inclusion...)}
%
%
Let us only consider $\C^3$-part, i.e. $x_i,y_i,z_i \in \C$ to obtain $\C$-valued mirror space.
For convenience in sign computation, we take  a non-trivial spin structure on $L_i$, which is represented by a generic point {\color{red}$\times$} on the segment between $X,Z$ in Figure \ref{fig:Seidel_Lag}.

We will perform surgery at $X_1$ of $\bL_1$ in Figure \ref{fig:pp-decomp}).
The immersed Lagrangian becomes a union of two circles after the smoothing, which we will call as a pair-of-circles. 
By translating a pair-of-circles toward the second pair-of pants and we want to identify it with the surgery of $\bL_2$ in $X_2$. For a precise identification, we will need an isotopy of gauge-hypersurfaces as we explain below.
We remark that in this way,  a pair-of-pants decomposition induces a family of immersed Lagrangians over a trivalent graph
in this case.

For a pair-of-circles $\tilde{L}$(  $C_1$  shown in Figure \ref{fig:decomp}), we denote by $t$ the holonomy coordinate corresponding to the flat $\C^\times$-connection on $\tilde{L}$ with holonomy concentrated at the vanishing cycle $\bS$. Here $\bS$ is just a union of two points $S^0$ (see Definition \ref{def:vc}, and Figure).

We can also consider its weak Maurer-Cartan equation as in the case of Seidel Lagrangian.
The following lemma can be proved as in Lemma \ref{lem:pp}
\begin{lemma}\label{lem:dc}
For a pair-of-circles $\tilde{L}$, $b = yY +zZ$ with holonomy $\nabla^t$ (for $t \in \C^\times$)  satisfies the
weak Maurer-Cartan equation if it is $\Z/2$-symmetric.
\end{lemma}
We remark that we need to choose $\tilde{L}$ with standard spin structures on each circles, 
which is related to sign convention with holonomy contribution. Hence, the resulting deformation space of $\tilde{L}$ is 
$$\mathrm{Spec} \, \C[t,t^{-1},y,z] = \C^\times \times \C^2.$$

For the potential function of $\tilde{L}$, the $XYZ$- triangle bounded by $\bL$ corresponds to a bi-gon bounded 
by $\tilde{L}$, with corners $Y,Z$ and passing through the vanishing $\bS^0$.  From this observation, we assert the coordinate change
$$ x_1 = t. $$
Thus the deformation space $\C^3$ of the Seidel Lagrangian $\bL_1$ is glued with the deformation space $\C^\times \times \C^2$ of its smoothing $\tilde{L}$ at the immersed point $x$ via $x = t$, with the other two variables $y$ and $z$ unchanged.
We can perform the same construction for the Seidel Lagrangian $\bL_2$ and obtain the pair-of-circles $C_2$.

Now let us compare the two pair-of-circles $C_1,C_2$. 
Each of them have a   flat $\C^\times$-connections on $\tilde{L}$ with holonomy concentrated at the vanishing $\bS^0$.  A key observation is that the positions of the vanishing $\bS^0$ on the two smoothings are different.  
Let's denote them by $T_1$ and $T_2$ respectively
We need to move one to the other by smooth isotopy of gauge-hypersurfaces in the previous subsection.

As explained in Lemma \ref{lem:gc}, when a gauge-hypersurface moves across an immersed point (say $Y$) of the Lagrangian, there is a non-trivial change of coordinates in the Floer theory, and observe that collection of holomorphic polygons which pass through the gauge-hypersurface changes too.  The coordinate change is given by
$$ y \mapsto ty $$
(while $x$ and $t$ remains unchanged) in this case. 
Similarly, one can move another point in a gauge-hypersurface of $C_1$ through $Z$ to obtain the change
of coordinate $z \mapsto tz$. 
More precisely, we can move $T_1$ to $T_2$ in a way such that the immersed points marked by $Y_1$ and $Z_1$ are passed through exactly once. Now, after matching the position of $T_1,T_2$, we find that the holonomies are opposite at the point. Hence we obtain the relation $t_1 = t_2^{-1}$.
 Therefore, the gluing of the deformation spaces is 
$$\mathrm{Spec} \, \C[t_1,t_1^{-1},y_1,z_1] \to \mathrm{Spec} \, \C[t_2,t_2^{-1},y_2,z_2], t_1 = t_2^{-1}, y_1 = z_2 t_2, z_1 = y_2 t_2. $$
Therefore, we obtain the resolved conifold $\check{X} = \cO_{\bP^1}(-1) \oplus \cO_{\bP^1}(-1)$ after gluing the four deformation spaces together, see Figure \ref{fig:mirror_glue}.

\begin{figure}[htb!]
    \includegraphics[scale=0.55]{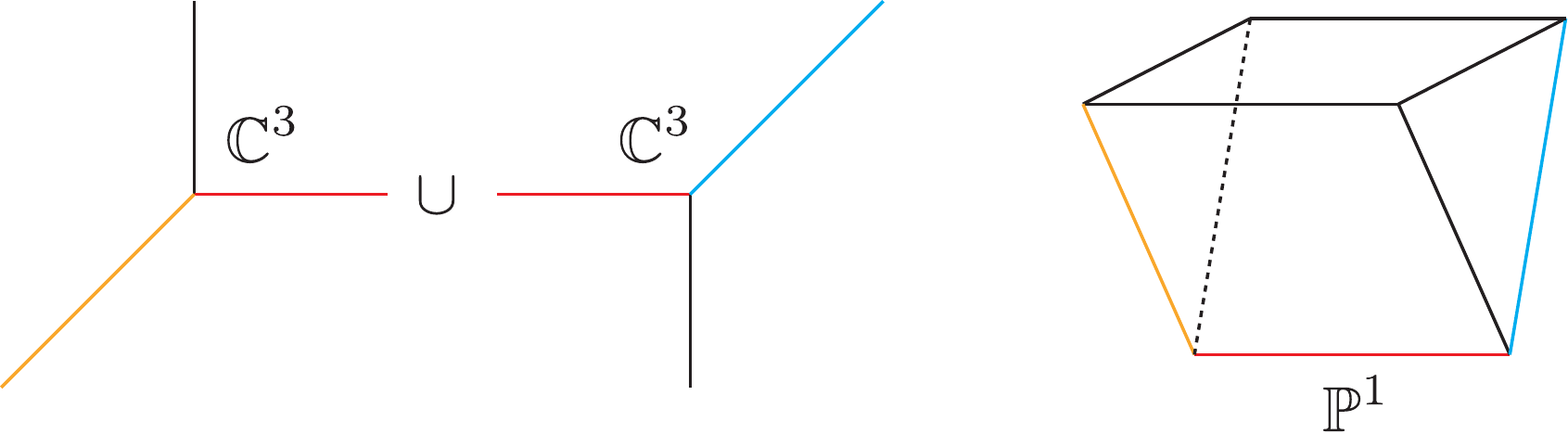}
    \caption{The mirror space for $k=1$ by gluing the deformation spaces of the immersed Lagrangians, is $\cO_{\bP^1}(-1) \oplus \cO_{\bP^1}(-1)$.  The right hand side is its toric diagram.}
		\label{fig:mirror_glue}
\end{figure}

Note that there are many choices of isotopies to match gauge-hypersurfaces of a pair-of-circles $C_1$ to $C_2$.
For example, one can consider a smooth isotopy of moving $T_1$ to $T_2$ in a way such that the immersed point $Y_1$ is passed through $k$ times, while $Z_1$ is passed through $2-k$ times.  Then the gluing then becomes
\eqref{eq:kk2} and we obtain $\cO_{\bP^1}(-k) \oplus \cO_{\bP^1}(k-2)$.

In this way, these difference choices of smooth isotopy of gauge-hypersurfaces \emph{result in different mirror models}.  
Throughout the process, the superpotential $W$ remains the same. 
$$W = z_1 y_1 x_1 = z_1 y_1 t_1 = y_2  t_2 z_2 = y_2 x_2 z_2.$$
Thus the mirror is a Landau-Ginzburg model $(\check{X},W)$.  The resulting Landau-Ginzburg model are equivalent
to each other for different $k \in \Z$. 

In conclusion, we have illustrated that the mirror can be obtained by gluing deformation spaces of
Seidel Lagrangian $L_1, L_2$ in each pair of pants and that of a pair-of-circles obtained by surgery of $L_1,L_2$,
where the coordinate change formula can be geometrically explained.
$$ (\C^3,W^{S_1}) \overset{\small\textrm{surgery}}\longrightarrow (\C^\times \times \C^2,W^{(C,T)}) \overset{\small\textrm{gauge change}}\longleftrightarrow (\C^\times \times \C^2,W^{(C,T')}) \overset{\small\textrm{surgery}}\longleftarrow (\C^3,W^{S_2}) $$
where the first map is $x_1=t,y_1=y_0,z_1=z_0$, the second map is $t'=t^{-1}$, $y_0' = t^ay_0$ and $z_0'=t^bz_0$ with $a+b=2$, the third map is $x_2=t',y_2=y_0',z_2=z_0'$. 

But for more precise formulation, we will turn to algebraic formalism of gluing, which can be used to make the above observations much more precise, and also gives rise to related $\AI$-functors and homotopies between them.

%It is straightforward to see that the maps automatically preserve the superpotentials $W^{S_i}=x_iy_iz_i$, $W^{(C,T)}=ty_0z_0$ and $W^{(C,T')}=t'y'_0z'_0$.
%The resulting moduli space is $(\cO(-1))\oplus\cO(-1)),W)$. In general, if we make more complicated choices for gauge cycles as in \ref{subsec:twistgauge}, we obtain 
%$$(\cO(-1-a))\oplus\cO(-1+a)),W) $$

\section{Algebraic approach for gluing local mirrors and functors}\label{sec:gluingthmalg}
In this section, we provide a precise algebraic method for the gluing. We will introduce a main criterion of
gluing two charts. This simple criterion turns out to give all the gluing data that we need. In particular,
we will obtain the coordinate change rule  on the intersection of two charts.  
On the intersection, the functors are shown to be quasi-isomorphic, with explicit homotopy data. Moreover, we construct a global $\AI$-functor from Fukaya category to the homotopy fiber product of two matrix factorization categories of mirror charts.

Recall that our mirror charts are given by deformation spaces of  immersed Lagrangians $\bL_i$'s.
In this paper, let us suppose that  each mirror chart $U_i$ is given by $\Lambda_0^n$. i.e. each immersed Lagrangian is
weakly unobstructed for all values of $\Lambda_0^n$.
The use of Novikov ring $\Lambda_0$ is rather essential in this story as we will see.

Even though $L_0, L_1$ are immersed Lagrangians which are not equivalent to each other, 
there might exist weak bounding cochains $b_0, b_1$ such that the deformed object $(L_0,b_0)$ and $(L_1,b_1)$
are equivalent in Fukaya category.  
In such a case, there should be open subsets $U_{01} \subset U_0, U_{10} \subset U_1$ which 
gives rise to such equivalences for $b_0 \in U_{01}, b_1 \in U_{10}$. 
As the coordinates of $U_0$, $U_1$ are mirror variables, the correspondence of $b_0$ and $b_1$ provides coordinate change formulas.  In this way, $U_0$ and $U_1$ can be identified in their common intersection $U_{01}=U_{10}$.

This is different from the approach using Lagrangian torus fibration. Recall that each Lagrangian torus
corresponds to a point in the valuation of the deformation space, and thus, one needs to consider
family of Lagrangian tori to construct a (non-Archimedean) chart. But  using immersed Lagrangians has a merit 
that a deformation space of a single immersed Lagrangian can be used to define a mirror chart.
This enables us to choose only finitely many immersed Lagrangians to construct the mirror, which
provides quite an advantage  over the case of torus fibers.

Let us remark that if $L_0$ and $L_1$ are disjoint from each other, $(L_0,b_0)$ and $(L_1,b_1)$ cannot be equivalent to each other, i.e. 
$$L_0 \cap L_1 = \emptyset \Longrightarrow U_0 \cap U_1 = \emptyset.$$
To relate $U_0$ and $U_1$, we will have to find a  sequence of Lagrangians $L_2, \ldots,L_k$ for $k \geq 2$ 
such that $$ U_0\cap U_2 \neq 0, U_2 \cap U_3 \neq 0, \cdots,  U_k \cap U_1 \neq 0.$$

For example, in the case of 4-punctured sphere,  Seidel Lagrangians $L_i$'s contained in each pair of pants
have deformation spaces $U_0, U_1$ with $U_0 \cap U_1 = \emptyset$, and we will make Lagrangian isotopy of $L_0$ to the other pair of pants to obtain Seidel Lagrangian $L_2$ so that $U_2$ has non-trivial intersection with both $U_0,U_1$.

We also remark that even when $L_0$ and $L_1$ are disjoint from each other, we can find a Hamiltonian-equivalent Lagrangian $L_0' \cong L_0$ which intersect with $L_1$, by moving a very small neighborhood of a point toward $L_1$.
But this ad hoc intersection will not make the deformation spaces intersect, since Floer theory is invariant under Hamiltonian diffeomorphisms. 

We will use  the notion of isomorphisms of Fukaya category to find the required coordinate change.

\subsection{Algebraic preliminaries}
Recall the following notion of isomorphisms in $\AI$-category.
\begin{defn}\label{def:iso}
Morphisms $\alpha \in \Hom_{\mathcal{C}} (L_0,L_1)$ and $\beta \in \Hom_{\mathcal{C}} (L_1,L_0)$ are called isomorphisms if 
$$m_1(\alpha) = m_1 (\beta)=0,  \quad m_2(\alpha,\beta)=\one_{L_0} + m_1(\gamma_1), \quad m_2(\beta, \alpha)=\one_{L_1}+ m_1(\gamma_2)$$ 
for some $\gamma_1 \in \Hom_{\mathcal{C}} (L_0,L_0)$ and $\gamma_2 \in \Hom_{\mathcal{C}} (L_1,L_1)$.  
%Notice that we allow their compositions to be identities up to homotopy. 
If $\gamma_1=\gamma_2=0$,then $\alpha$ and $\beta$ are called strict isomorphisms.
\end{defn}
This should be well-known to experts, but it is difficult to give a precise reference. 
We give a proof that these provide a correct notion of isomorphism by using Yoneda embedding
in Section \ref{sec:pfgluethmy}.

\begin{thm}\label{thm:ye}
The Yoneda functors $\mathcal{Y}^{0}$ and $\mathcal{Y}^{1}$ are quasi-isomorphic 
for $\alpha,\beta$ given in Definition \ref{def:iso}.
In particular, the two objects $L_0$ and $L_1$ are quasi-isomorphic in the original category $\mathcal{C}$ 
since Yoneda embedding is fully faithful.
\end{thm}
 The same scheme of proof will be applied to boundary deformed objects and this will provide the natural setup to glue localized mirror functors.
 \begin{defn}\label{def:fuctqisom1}
Two $\AI$-functors $\mathcal{F}_1,\mathcal{F}_2: \mathcal{C} \to \mathcal{D}$ are said to be quasi-isomorphic to each other if there exists a natural transformation
$N_{12}: \mathcal{F}_1 \to \mathcal{F}_2, N_{21}: \mathcal{F}_2 \to \mathcal{F}_1$ such that $N_{12}\circ N_{21}, N_{21}\circ N_{12}$ is cohomologous to the identity natural transformations on $\mathcal{C},\mathcal{D}$ respectively.
i.e. $$N_{12} \circ N_{21} - id = M_1(H_1), N_{21}\circ N_{12} - id = M_1(H_2).$$
\end{defn}

\begin{remark}
Any $\AI$-natural transformation induces a natural transformation of corresponding homology functors.
\end{remark}

Now, let us recall the notion of homotopy fiber products of two dg-categories. We will use the following explicit model from Tabuada \cite{Tabu}, but our sign convention is different from that of \cite{Tabu}.
\begin{defn}
Let $B,C,D$ be dg-categories with dg-functors $G:B \to D, L : C \to D$.

The homotopy fiber product  $B \times^h_D C$ is a dg-category which is defined as follows.
\begin{itemize}
\item The objects of  $B \times_D^h C$ are given by 
\begin{equation*}
 \Big\{ M \in B, N \in C, \phi\in D^0 (G(M),L(N)) \,\,\mbox{with invertible} \, \, [\phi] \,\, \mbox{in} \,\, H^0 (D) \Big\},
\end{equation*}
\item  
$\Hom^i \big((M_1,N_1, \phi_1), (M_2, N_2, \phi_2)\big)$ for two objects $(M_1,N_1, \phi_1), (M_2, N_2, \phi_2)$ is
\begin{equation*}
 B^i (M_1,M_2) \oplus C^i (N_1,N_2) \oplus D^{i-1} (G(M_1),L(N_2))
\end{equation*}
\item  The differential $d$ is defined as 
$$d(\mu,\nu, \gamma)=(d\mu, d\nu, - d\gamma - \phi_2 G(\mu) + L(\nu) \phi_1).$$
\item  The composition of morphisms is defined by
$$(\mu',\nu',\gamma') (\mu,\nu,\gamma) = (\mu' \mu, \nu' \nu, \gamma' G(\mu) +  (-1)^{i'} L(\nu') \gamma).$$
\end{itemize}
\end{defn}
Pictorially, the objects may be considered as a diagram
$$M\to\,\,G(M) \stackrel{\phi}{\to} L(N) \,\, \leftarrow N. $$
and a morphism $(\mu,\nu,\gamma)$ fits in to the following diagram.
\begin{equation*}
\xymatrix{ G(M_1) \ar[dr]^{\gamma}\ar[r]^{\phi_1} \ar[d]_{G(\mu)}& L(N_1) \ar[d]^{L(\nu)}\\
G(M_2) \ar[r]^{\phi_1} & L(N_2)
}
\end{equation*}
Note that $d$-closedness imposes the commutativity of the square diagram up to homotopy $\gamma$.
One can check by elementary computation that this gives a well-defined dg-category. See \ref{sec:dghptyfp}.

%
%Before we start,  we mention the difference of conventional SYZ approach using Lagrangian torus fibers and
%our approach of using immersed Lagrangian. In SYZ approach, one considers all the non-singular torus fibers to construct the mirror space. If two torus fibers are distinct, each fiber belongs to a different Hamiltonian equivalence class and  their deformation space $\Lambda_0^*$ do not intersect. Note that
%the valuation of elements of $\Lambda_0^*$ are all the same.

%Unlike torus fibers, the (formal) deformation space of
%immersed Lagrangian is rather big. For example, the deformation space of a Seidel Lagrangian
%is given by $\Lambda_0^3$ for the Novikov ring $\Lambda_0$. Therefore, it could represent Lagrangians in a different Hamiltonian equivalence class. 

\subsection{Main gluing theorem}
Let us explain the main theorem of gluing mirror charts and respective $\AI$-functors.

Consider two Lagrangian $\bL_0$, $\bL_1$ in a symplectic manifold $X$. By localized mirror formalism, 
we have a mirror $W_0:U_0 \to \Lambda_0$, $W_1:U_1 \to \Lambda_0$, together with $\AI$-functors $\mathcal{F}^{\bL_i}: \Fuk(X) \to \MF(W_i)$ for $i=0,1$.  Let us assume that $U_i \cong \Lambda_0^n$ for simplicity.

\begin{defn}\label{defn:iss}
We say that $(\bL_i,W_i:U_i \to W_i)$ for $i=0,1$ are isomorphic on $U_0 \cap U_1$ if we have subsets   $V_i \subset U_i$ for $i=0,1$ and a bijection $\phi:V_0 \to V_1$ such that
for each $b_0 \in V_0$, and $b_1 = \phi(b_0) \in V_1$, we have isomorphisms $\alpha,\beta$
$$\alpha \in \Hom_{\mathcal{C}} \big((\bL_0,b_0),(\bL_1,b_1)\big), \beta \in \Hom_{\mathcal{C}} \big((\bL_1,b_1),(\bL_0,b_0) \big).$$
i.e. we have 
$$m_1^{b_0,b_1}(\alpha) = m_1^{b_1,b_0} (\beta)=0,  \quad m_2^{b_0,b_1,b_0}(\alpha,\beta)=\one_{\bL_0} + m_1(\gamma_0), \quad m_2^{b_1,b_0,b_1}(\beta, \alpha)=\one_{\bL_1}+ m_1(\gamma_1)$$ 
for some $\gamma_i \in \Hom_{\mathcal{C}} \big( (\bL_i,b_i), (\bL_i,b_i)\big)$ with $i=0,1$.
We will identify $V_0$ and $V_1$ via $\phi$ and write it as $U_0 \cap U_1$.
\end{defn}

Here is the main gluing theorem.
\begin{thm}\label{thm:maingluealg}
Suppose $(\bL_i,b_i)$ for $i=0,1$ are isomorphic on $U_0 \cap U_1$.
Then  the following holds.
\begin{enumerate}
\item Their potential functions agree. i.e. we have $$W_0(b_0) = W_1(b_1), \;\; \textrm{for} \;\; b_0 \in V_0, b_1 = \phi(b_0) \in V_1.$$
\item  The $\AI$-functors $\mathcal{F}^{\bL_i}: \Fuk(X) \to MF(W_i)$ composed with restrictions $r_i: MF(W_i) \to MF(W_i|_{V_i})$ for $i=0,1$ are quasi-isomorphic to each other. i.e. we have
$$ r_0 \circ \mathcal{F}^{\bL_0} \cong r_1 \circ \mathcal{F}^{\bL_1} $$
Moreover, the required natural transformations as well as homotopies for the quasi-isomorphism are explicitly given using $\alpha,\beta$.
\item There exists a global $\AI$-functor 
$$\mathcal{F}: \Fuk(X) \to MF(W_0) \times^h_{MF(W_{01})} MF(W_1)$$
from the Fukaya category of $X$ to the homotopy fiber product of the
dg-categories $MF(W_0)$ and $MF(W_1)$, where we denote by $W_{01}$ the potential function on $V_0$ (or equivalently $V_1$).
\end{enumerate}
\end{thm}
Note that once we identify the isomorphism between two deformed Lagrangians, the rest of the constructions such as quasi-isomorphisms between functors and the explicit construction of the global $\AI$-functor are given canonically. 
\begin{proof}
\begin{enumerate}
\item This follows from the $\AI$-identity.
Namely, from the $\AI$-identity that for any $x \in CF(\bL_0,\bL_1)$, we have
 $$(m_1^{b_0,b_1})^2 (x)  + m_2^{b_0,b_0,b_1}(m_0^{b_0},x) + (-1)^{|x|'} m_2^{b_0,b_1,b_1}(x,m_0^{b_1}) =0.$$
 From the weak MC equation and the definition of a unit, we have
 $$ (m_1^{b_0,b_1})^2 (x) + W_0(b_0) x - W_1(b_1)x =0$$
 Now, for $x = \alpha$, the first term vanishes, and we thus obtain $W_0(b_0) = W_1(b_1)$. 
 
\item 
%See Definition \ref{def:fuctqisom} for the precise definition of quasi-isomorphism between two $\AI$-functors.
The proof of (2) will be given in Section \ref{sec:pfgluethm} by explicitly constructing natural transformations and homotopies from $\alpha$ and $\beta$.

\item 
The proof of (3) will be given in Section \ref{sec:pfgluethm} (see Proposition \ref{prop:3rditemmain}). We only give a definition of $\mathcal{F}$ here (recall that we use the $\AI(MF)$ convention in  \ref{sec:dgsign}). We identify $V_0$ and $V_1$ via $\phi$.

For an object $L \in \Fuk (X)$, the image $\mathcal{F} (L)$ is defined as
\begin{equation*}\label{eqn:hptyfpobjfunc}
\begin{array}{c}
\left(\mathcal{F}^{\bL_0}(L), \quad \mathcal{F}^{\bL_1}(L),\quad  \mathcal{F}^{\bL_0 }(L)|_{V_0} \stackrel{N_{01}}{\longrightarrow}   \mathcal{F}^{\bL_1}(L)|_{V_1} \right)  := \\
 \left(CF(L, (\bL_0 , b_0 )), \quad   CF(L,(\bL_1,b_1)) ,\quad \xymatrix{ CF(L, (\bL_0 , b_0 ))|_{V_0} & CF(L, (\bL_1, \phi(b_0) ))|_{V_1} \ar[l]_{ N_{01}(L)} }
%\tiny (-1)^{\epsilon} m(\bullet,e^{b'(b)},\beta,e^{b})} } 
 \right)
 \end{array}
\end{equation*}
where 
%$N=\{N_k\}_{k\geq 0}$ is the natural transformation between $\mathcal{F}^{\BL}$ and $\mathcal{F}^{\bL_1}$ constructed as in the previous section.
%(Recall that $A_\infty$-functors $\mathcal{F}^{\BL} : \Fuk(X) \to \MF(W)$ and $\mathcal{F}^{\bL_1}$ are given as family Yoneda embeddings.)
%Notice that 
$$N_{01}(L)=  (-1)^{|\bullet|}  m (\bullet,e^{\phi(b_0)},\beta, e^{b_0}) \in \Hom_{MF_{A_\infty}} ( \mathcal{F}^{\bL_0 }(L)|_{V_0 } ,   \mathcal{F}^{\bL_1}(L)|_{V_1} ).$$
This map induces an isomorphism on between $m_1^{\phi(b_0)}$ and $m_1^{b_0}$ cohomologies since a similarly defined map using $\alpha$ induces its inverse on the cohomology level.

For a tuple of composable morphisms 
$$(a_1, \cdots, a_k) \in \Hom_{\Fuk(X)} (L_1,L_2) \otimes \cdots \otimes \Hom_{\Fuk(X)} (L_k,L_{k+1}),$$
we define $\mathcal{F}_k (a_1, \cdots, a_k)$ for $k \geq 1$ to be 
\begin{equation*}\label{eqn:fcttohptypd}
\begin{array}{c}
(\mathcal{F}_k^{\bL_0} (\mathbf{a}),\mathcal{F}_k^{\bL_1} (\mathbf{a}), (N_{01})_k  ( \mathbf{a}) ) := \\
\left( m (\mathbf{a}, \bullet, e^{b_0}), \quad m (\mathbf{a}, \bullet, e^{b_1}), \quad  (-1)^{|\mathbf{a}|'} (-1)^{|\bullet|} m (\mathbf{a},\bullet, e^{\phi(b_0)} , \beta, e^{b_0})|_{b_0 \in V_0} \right)
\end{array}
\end{equation*}
We will see later that $\{(N_{01})_k\}_{k \geq 1}$ comes from a natural transformation between two functors $\mathcal{F}^{\bL_0}$ and $\mathcal{F}^{\bL_1}$ over $V_0$ and $V_1$.

\end{enumerate}
\end{proof}

\begin{defn} \label{def:class}
	Let $L$ and $L'$ be two objects in the Fukaya category.  They are said to be in the same deformation class if there exist objects $L_1,\ldots,L_k$, weakly unobstructed formal deformations $(\nabla,b)$ for $L$, $(\nabla',b')$ for $L'$, $(\nabla_{i,1},b_{i,1})$ and $(\nabla_{i,2},b_{i,2})$ for $L_i$, and the following isomorphisms in the Fukaya category: $(L,\nabla,b) \cong (L_1,\nabla_{1,1},b_{1,1})$, $(L_i,\nabla_{i,2},b_{i,2}) \cong (L_{i+1},\nabla_{i+1,1},b_{i+1,1})$ for $i=1,\ldots,k-1$, $(L_k,\nabla_{k,2},b_{k,2}) \cong (L',\nabla',b')$.
\end{defn}

We can easily generalize the above construction to the case of many charts with no non-trivial triple intersections.
Namely, let $\Gamma$ be a (directed) finite graph, with the set of vertices $\Gamma_0$ and the set of edges $\Gamma_1 \subset \Gamma_0 \times \Gamma_0$.
Suppose with have $\bL_i$ for each $i \in \Gamma_0$, which defines localized mirror functors $\mathcal{F}^{\bL_i}:U_i \to W_i$,
with $U_i \cong (\Lambda_+)^n$. 

\begin{assumption}\label{as:manychart}
 Let us suppose that $(\bL_i,W_i:U_i \to W_i)$ and $(\bL_j,W_j:U_j \to W_j)$
are isomorphic on non-trivial subset $U_i \cap U_j$ (in the sense of Definition \ref{defn:iss}) if and only if $(i,j)$ or $(j,i)$ is
in $\Gamma_1$. And there are no non-trivial triple intersections for any distinct $i,j,k \in \Gamma_0$.
Furthermore, for any loop in $\Gamma$, we assume that the composition of coordinate changes along the loop is
identity. i.e. we assume that there is no monodromy of coordinate changes.
\end{assumption}
In the case of punctured Riemann surfaces, this assumption can be met by Proposition \ref{prop:Lag-coll}.
\begin{cor}\label{cor:manychart}
We can define homotopy fiber products of dg categories $\prod_{i \in \Gamma_0}^h MF(W_i)$.
There exist a global $\AI$-functor from Fukaya category of $X$ to the homotopy fiber product  $\prod_{i \in \Gamma_0}^h MF(W_i)$.
\end{cor}
\begin{proof}
The fiber product in the case of $U_0, U_1$ corresponds to the graph with two vertices $v_0,v_1$ and an arrow from $v_0$ to $v_1$.
The fiber product as well as $\AI$-functor can be easily extended to the case of general graph $\Gamma$.
We remark that the monodromy assumption is for the well-definedness of a global mirror space but we remark that
homotopy fiber product of dg-categories as well as the $\AI$-functor are still well-defined without this assumption
\end{proof}
In the general case with more than double intersections, we need higher homotopy data and we will discuss it elsewhere.

%***** For gluing functors via `pseudo'-isomorphisms, NEED TO RESTRICT to the Fukaya subcategory of Lagrangians whose $m_k$ with the reference pseudo-objects converge over $\Lambda$.  When we restrict back to the actual moduli (rather than the pseudo-one), we have convergence for the whole category and hence the functor is defined on the whole category.  We define in this way for good topology (to glue sheaves).  *****

\section{Lagrangian isotopy and local coordinate change}\label{sec:localisot}
To illustrate the gluing construction, we give a simple example for a pair of pants $M$.
We pick two immersed Lagrangians $\bL_0$ and $\bL_1$ in $M$  as in Figure \ref{fig_protosei}.
Note that they are not hamiltonian isotopic to each other. In fact,  we push two immersed points $Y$ and $Z$ of $\bL_0$ toward a puncture by a Lagrangian isotopy and slightly perturb it to obtain $\bL_1$( so that $\bL_0$ and $\bL_1$ intersect transversely).
The areas of the regions are labeled as $a_1,\cdots, a_8$.
We set $b=xX + y Y+zZ$ for $\bL_0$ and $b'=x'X' + y' Y' + z' Z'$ for $\bL_1$.
We require that $\bL_0$ and $\bL_1$ satisfy the reflection symmetry so that $b$ and $b'$ satisfy weak Maurer-Cartan equations
for $(x,y,z) \in \Lambda_+^3$ and $(x',y',z') \in \Lambda_+^3$. Thus, we set $U_0 =\{ (x,y,z) \in \Lambda_+^3\}$ and
$U_1 = \{ (x',y',z') \in \Lambda_+^3 \}$. 
\begin{figure}[htb!]
    \includegraphics[scale=0.45]{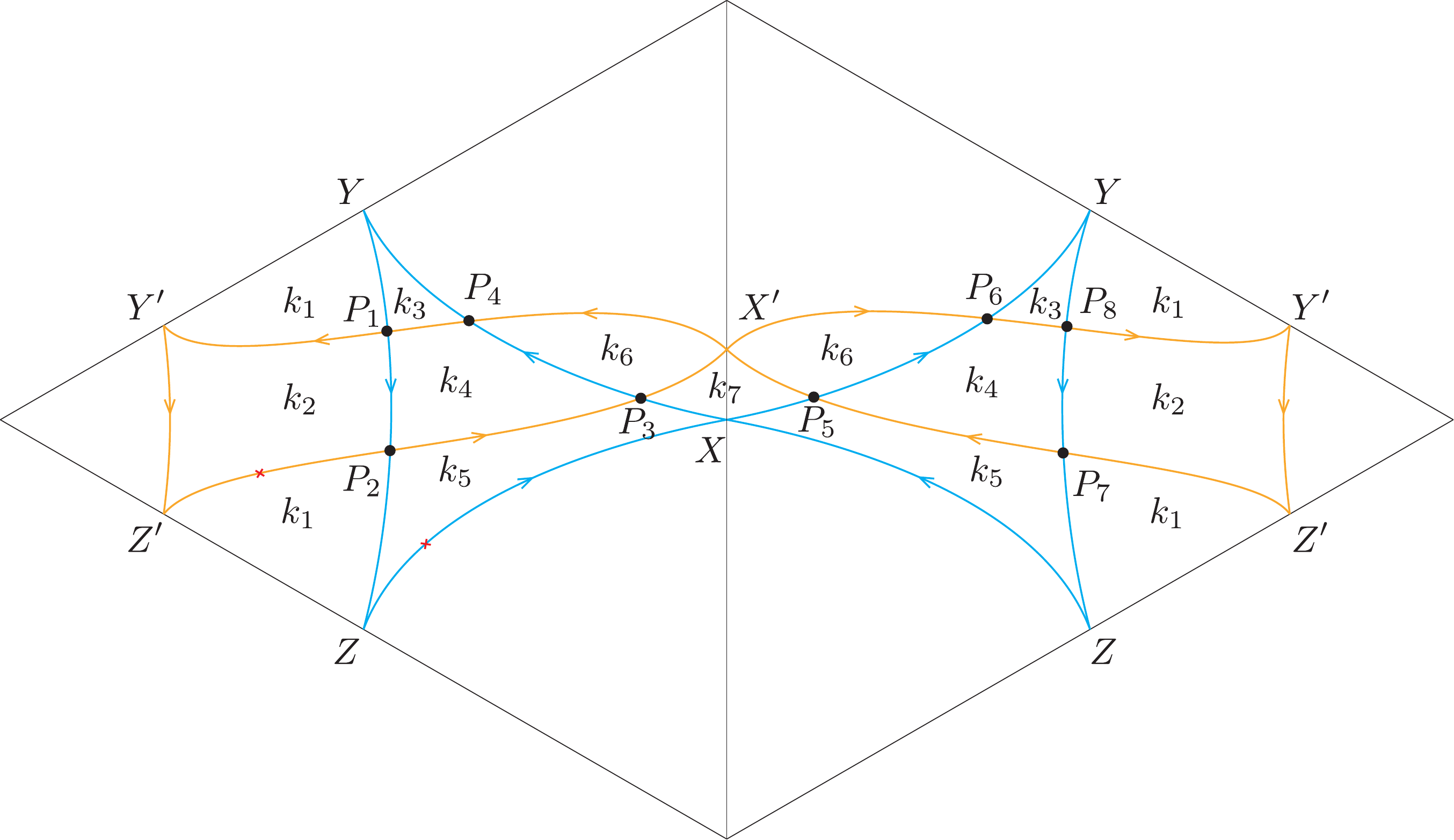}
    \caption{}\label{fig_protosei}
\end{figure}
\begin{prop}\label{prop:lc}
The formal deformation spaces $U_0$ and $U_1$ can be partially identified by the relations
\begin{equation}\label{eq:co}
\begin{cases} x' = T^{2 \delta} x, \\  y' = T^{- \delta} y,  \\ z' =  T^{-\delta} z,
\end{cases}
\end{equation}
for $\delta = 2k_1+k_2-k_5-k_6-k_7$.
Here we assume that $k_2= 5 k_5+3 k_6+4 k_7$ and $k_3 = 2(k_5+k_6+k_7)$.
\end{prop}
\begin{remark}
The last area condition is not essential and  is only to make the coordinate change formula look nice as above.
\end{remark}
To see that  $U_0 \cap U_1$ are proper subsets of $U_0$ and $U_1$ (after identification),
we look at the valuations of \eqref{eq:co}. 
\begin{equation}
\begin{cases} \val (x')  =  \val (x) +2\delta  , \\  \val(y') = \val (y) - \delta,  \\ \val(z') = \val(z) - \delta .
\end{cases}
\end{equation}
Suppose that $a_1$ is sufficiently bigger than $a_2,\cdots,a_8$, and therefore we have $\delta >0$.
Since the valuation of deformation parameters should be positive, we see that two formal deformation spaces $ \{ (x,y,z) \in \Lambda_+^3\}$ and $ \{ (x',y',z') \in \Lambda_+^3\}$ overlap on a large subset (non-compact region)
$$  \{\val (x) > 0, \val (y) > \delta, \val (z) > \delta \} \subset U_0, \quad   \{ \val(x') > 2 \delta, \val(y') >0, \val(z') > 0 \}\subset U_1.$$
%
%Thus, one can see that the formal deformation space of $\bL_0$ and $\bL_1$ overlap. 
%On $x,y,z$ coordinate, the valuation of the formal deformation space of $\bL_0$ is $x,y,z \geq 0$, and
%that of $\bL_1$ is 
%Thus they overlap on .
%Therefore we may view \eqref{eq:co} as a coordinate change formula between two local charts defined by formal-deformation spaces of $\bL_0$ and $\bL_1$.

%For a suitably chosen deformed Seidel Lagrangian $\bL_1$, there exists a fixed constant $\delta>0$ such that $m_1^{b,b'} (P_6) = m_1^{b',b} (P_4) =0$ if we set

\begin{proof}
Note that $\bL_0 \cap \bL_1$ has 8 intersection points, which we labeled as $P_1,\cdots,P_8$.  We claim that $P_6 \in CF((\bL_0,b),(\bL_1,b')$ is an quasi-isomorphism with $P_4 \in CF((\bL_1,b'),(\bL_0,b)$ being its inverse.  Let us first compute $m_1^{b,b'}$ of $P_6$ and $P_4$ whose vanishing will deduce the coordinate change relation. As drawn in Figure \ref{fig:lagiso_x} and Figure \ref{fig:lagiso_m2}, there are eight different polygons contributing to $m_1^{b,b'} (P_4)$. Firstly, the two depicted in \ref{fig:lagiso_m2} have the same output $P_4$ and the same area, but one can check that they admit opposite signs, and hence cancel each other (the one above has a positive sign, and the one below is negative). These polygons also contribute to $m_2$ between $P_6$ and $P_4$, and we will see that the two contributions add up  in this case.
\begin{figure}[htb!]
    \includegraphics[scale=0.28]{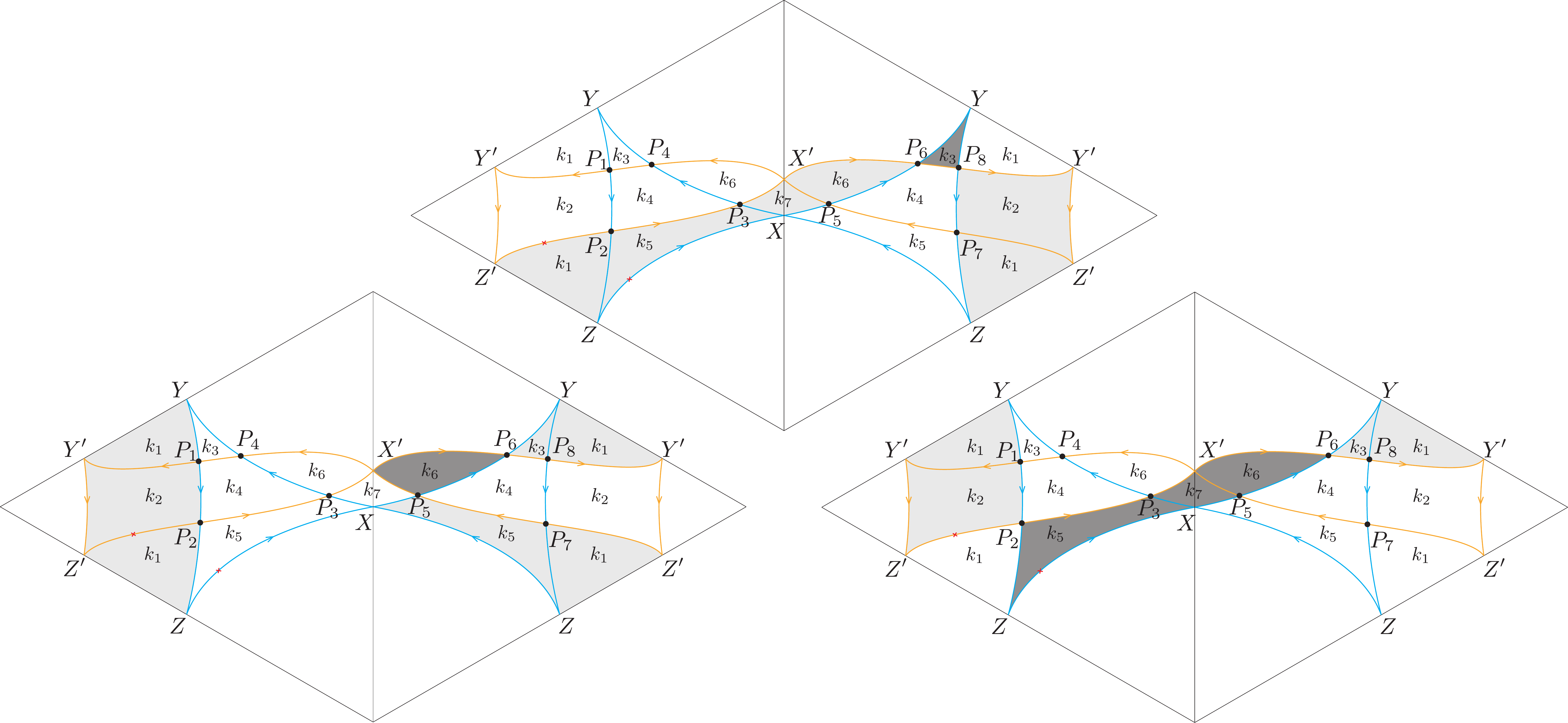}
    \caption{}\label{fig:lagiso_x}
\end{figure}

The pairs of shaded polygons in Figure \ref{fig:lagiso_x} give rise to the terms
$$ (-xT^{4k_1+ k_2 + k_3 + k_5} + x' T^{k_6} )P_5 + (-y T^{k_3} + y' T^{2 k_1 + k_2 + k_5 + k_6 + k_7} )P_8  + (z T^{k_5 + k_6+k_7} -z' T^{2 k_1 + k_2}) P_2$$
respectively, and hence $m_1^{b,b'} (P_6)$ is given by the sum of these. It follows that $m_1^{b,b'} (P_6) = 0$ if and only if
\begin{equation*}
\left\{
\begin{array}{l}
 x' = x T^{4k_1 + k_2 + k_3 + k_5 - k_6} \\
 y' = y T^{- (2 k_1 + k_2 + k_5 + k_6 + k_7 - k_3 }\\
 z' = z T^{-(2 k_1 + k_2 -k_5 -k_6 -k_7 )} 
 \end{array} \right.
 \end{equation*}
Therefore the area conditions in the statement of Proposition implies the desired coordinate change relations \eqref{eq:co}.
$m_1^{b',b} (P_4)$ can be computed in a similar way, and vanishes under the same condition.

%After suitable adjustment of areas $a_i$, the above condition can be simplified as
%\begin{equation*}
%\left\{
%\begin{array}{l}
%x' = x T^{2 \delta} \\
% y' = y T^{ -\delta }\\
% z' = z T^{- \delta } 
% \end{array} \right.
% \end{equation*}
% for some positive constant $\delta$ (namely, $a_2, a_3,a_5,a_6,a_7$ are negligible). 

%We remark that the coefficient of $P_5$ in $m_1(P_6)$ vanishes as we have 
%$4a_1+a_2+2a_4 = 4a_1+a_2+2a_3+2a_6+2a_7$.
%However, these two strips plays an important role in pseudo-isomorphism later on.

%Namely, with the above identification $m_1(P_i)$ has only non-linear polynomials as non-trivial coefficients.

% Note that the higher order terms in $m_1^{b,b'}$ are still non-zero, which means that the intersection points may not be Floer-cycles. But this is quite natural in terms of Floer theory of $\bL_0$ and $\bL_1$.
% Recall that the Floer potential for $\bL_0,\bL_1$ are $xyz$ or $x'y'z'$, and $(\bL,xX+yY+zZ)$ has non-vanishing Floer cohomology for critical points of $W=xyz$, which is the union of three coordinate axis, and Floer cohomology is trivial otherwise. 
% 
% If we restrict ourselves to the case of non-trivial Floer cohomology (or $(x,y,z)$ lies in the coordinate axis), then
% two of $x,y,z$ are zero, and hence any quadratic or higher order polynomials vanish.
% And only in these cases, we have non-trivial Floer cohomology.
%
%
%Now, pseudo-isomorphism in this case is given by $P_5$ and $P_6$.

\begin{lemma}
Let $\alpha = P_6 \in CF((\bL_0,b),(\bL_1,b'))$, $\beta  = P_4 \in CF((\bL_1,b'),(\bL_0,b))$.
Then $\alpha$ and $\beta$ define quasi-isomorphisms between $\bL_0$ and $\bL_1$.
Namely, 
$$\begin{cases}
m_2(\alpha,\beta) =  T^k \cdot \be_{\bL_0}, \\
m_2(\beta,\alpha)= T^k \cdot \be_{\bL_1}.
\end{cases}$$
where we have $k = 4k_1 + k_2 + k_3 + k_5 + k_6 + k_7$. 
\end{lemma}

\begin{figure}[htb!]
    \includegraphics[scale=0.29]{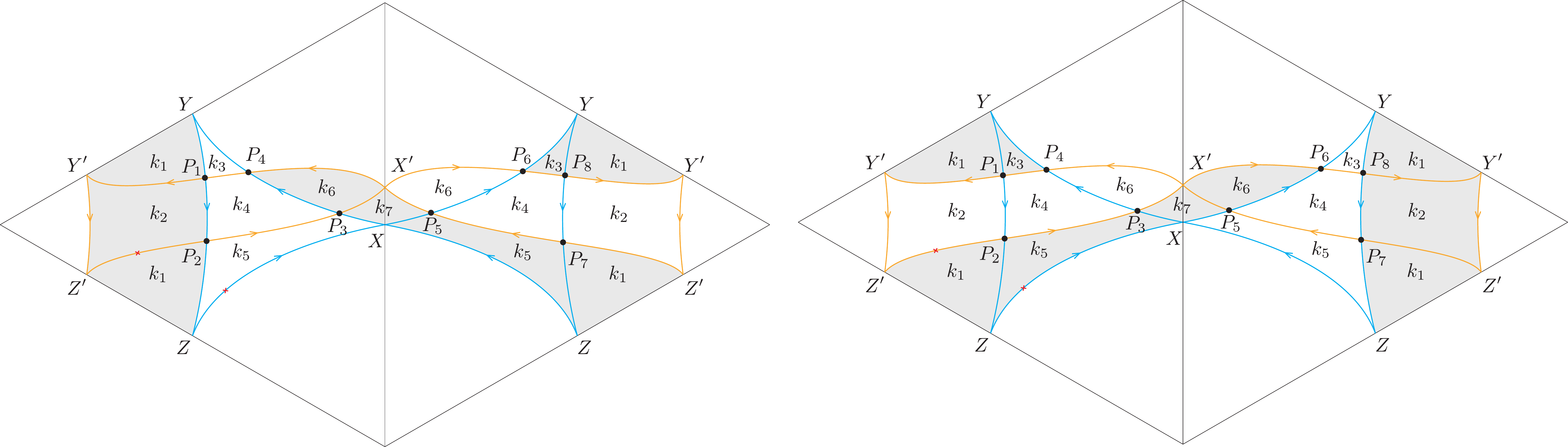}
    \caption{}\label{fig:lagiso_m2}
\end{figure}

\begin{proof}
In Figure \ref{fig:lagiso_m2}, we illustrate the polygons contributing the above identity. As we now think of these polygons as contributing to $m_2$, one can check that these two polygons have the same sign, contrary to the computation of $m_1^{b,b'} (P_6)$,
%Also, we have other contributions of $m_2(\alpha,\beta)$, which corresponds to higher order terms of $m_1(P_5)$.
%But as we mentioned before, these terms should lie in $m_1$ since the Floer cohomology vanishes if we are away from the coordinate axis. (And on axis, these extra terms vanish).
\end{proof}

Thus we have proved Proposition \ref{prop:lc}.
\end{proof}

\section{Gluing two pairs of pants}\label{sec:glue2s}
We now consider the 4-punctured sphere which is  the union of two pairs-of-pants. Consider two Seidel Lagrangians $\bL_0$, $\bL_1$
sitting in each pair-of-pants component.  They are disjoint from each other ($\bL_0 \cap \bL_1 = \emptyset$).
Since Floer homology between $\bL_0$ and $\bL_1$ is trivial, one can say that there is no relation between them.
In particular, their formal deformation spaces $U_0, U_1$ do not overlap ($U_0 \cap U_1 = \emptyset$).

On the other hand, we can take a deformation $\WT{\bL_0}$ of $\bL_0$ toward the other pair of pants. We have seen that
the their formal deformation spaces  overlap non-trivially ($\WT{U}_0 \cap U_0 \neq \emptyset$) in the previous section \ref{sec:localisot}.

We find that if we push $\WT{\bL_0}$ enough to the other pair of pants so that it intersects $\bL_1$ as in Figure \ref{fig:twosei},
then the formal deformation spaces of $\WT{\bL_0}$ and $\bL_1$ also overlap non-trivially ($\WT{U}_0 \cap U_1 \neq \emptyset$). Therefore, this
provides a way to go from $U_0$ to $U_1$ via $\WT{U}_0$.

Note that $\WT{\bL_0}$ is twisted once along the neck region of the 4-punctured sphere as we deform.  The twisting produces the desired coordinate change. If we do not make this twisting, then
the resulting coordinate change switches $y$ with $z'$ and $z$ with $y'$.
%
%A Seidel Lagrangian $\WT{\bL_0}$ is obtained by pushing a pair of edges (joining $Y'$ and $Z'$) of the standard Seidel Lagrangian $\bL_0$ (sitting the pair-of-pants containing $X'$) from left pair-of-pants to the right pair of pants (keeping reflection symmetry). See Figure \ref{fig:twosei}. 
Both of the Lagrangians have nontrivial spin structures which are represented by generic points (marked as {\color{red}$\times$} in Figure \ref{fig:twosei}) as before.
Intersections between $\WT{\bL_0}$ and the standard Seidel Lagrangian $\bL_1$ in the other pair-of-pants are drawn in Figure \ref{fig:twosei}. 

\begin{figure}[htb!]
    \includegraphics[scale=0.35]{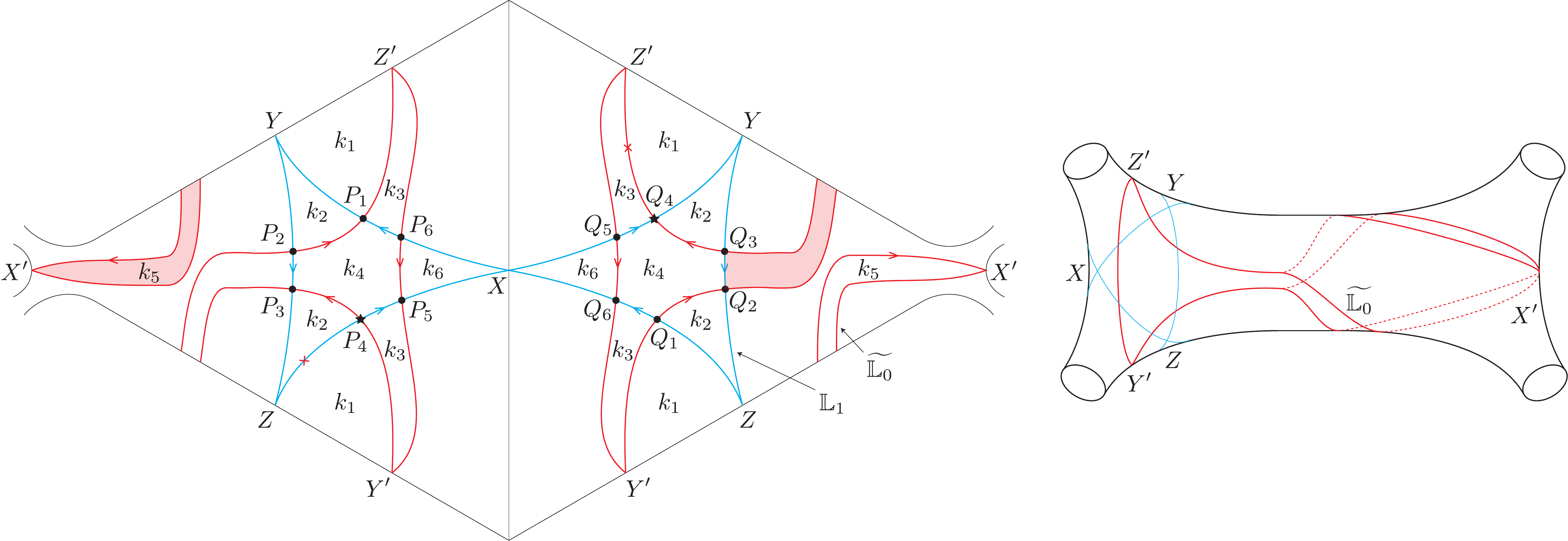}
    \caption{}
		\label{fig:twosei}
\end{figure}

We find a condition for $(\WT{\bL_0}, b'= x'X'+y'Y'+z'Z')$ and $(\bL_1,b=xX+yY+zZ)$ to be (quasi-)isomorphic. Let us look into their Floer complex $CF((\WT{\bL_0},b'),(\bL_1,b))$, which is generated by their 12 intersection points. We label them as $P_1, \cdots, P_6$ and $Q_1,\cdots, Q_6$ (see Figure \ref{fig:twosei}). Areas of regions enclosed by the Lagrangians are marked in Figure \ref{fig:twosei}. Here, the region $a_5$ starts from the $X'$ corner, goes over to the opposite side triangle and ends on the edge $\overline{Q_2 Q_3}$. See the shaded region.

\begin{prop}
Suppose $x x' \neq 0$, and define $\alpha = P_4 - Q_4 \in CF^0(\WT{\bL_0},\bL_1)$, $\beta = Q_1+P_1 \in CF^0(\bL_1,\WT{\bL_0})$. Then $m_1^{b',b}(\alpha) =0$ if (and only if)
\begin{equation*}
\left\{
\begin{array}{l}
x' = x^{-1}  T^{\delta} \\
y'= x y T^{-\epsilon} \\
z'= x z T^{ -\epsilon  },
\end{array} \right.
\end{equation*}
for $\delta=4k_1 + 2 k_2 + 2 k_3 - k_5 - k_6>0$, $\epsilon =  2k_1 +  2k_3- k_6 >0$. Furthermore, $\alpha,\beta$ provide isomorphisms between $(\WT{\bL_0},b')$ and $(\bL_1,b)$ if $\val(b'), \val(b) >0$.
\end{prop}
\begin{proof}

Our candidate for an isomorphism in $CF((\WT{\bL_0},b'),(\bL_1,b))$ is a combination of $P_4$ and $Q_4$. Their Floer differentials are given as follows, where the contributing polygons to $m_1^{b',b} ( P_4)$ are depicted in Figure \ref{fig:twosei1234} and Figure \ref{fig:m2twosei}. (Those for $m_1^{b,b'} (Q_4)$ are  obtained by rotating these polygons about the point $X$ in the figure.)
\begin{eqnarray*}
m_1^{b',b} ( P_4) &=&  \left( - xx' T^{k_4 + k_5 + k_6} + T^{4k_1 + 2 k_2 + 2 k_3 + k_4} \right) P_1 + xy T^{k_2 + k_4 + k_6} P_3 \\
&& - z' T^{2 k_1 + k_2 + 2 k_3 + k_4} Q_3 -  yT^{2k_1 + 2 k_2 + k_3 + k_4} Q_5 + x' z' T^{k_3 + k_4  + k_5} P_5 \\
m_1^{b',b} (Q_4) &=&  \left( xx' T^{k_4 + k_5 + k_6} - T^{4 k_1 + 2 k_2 + 2 k_3 + k_4} \right)Q_1  +  y'T^{2 k_1 + k_2 + 2 k_3 + k_4} P_3 \\
&& - xz T^{k_2 + k_4 + k_6} Q_3 - x'y' T^{k_3 + k_4 + k_5}  Q_5 + z T^{2k_1 + 2k_2 + k_3 + k_4} P_5
\end{eqnarray*}

\begin{figure}[htb!]
    \includegraphics[scale=0.25]{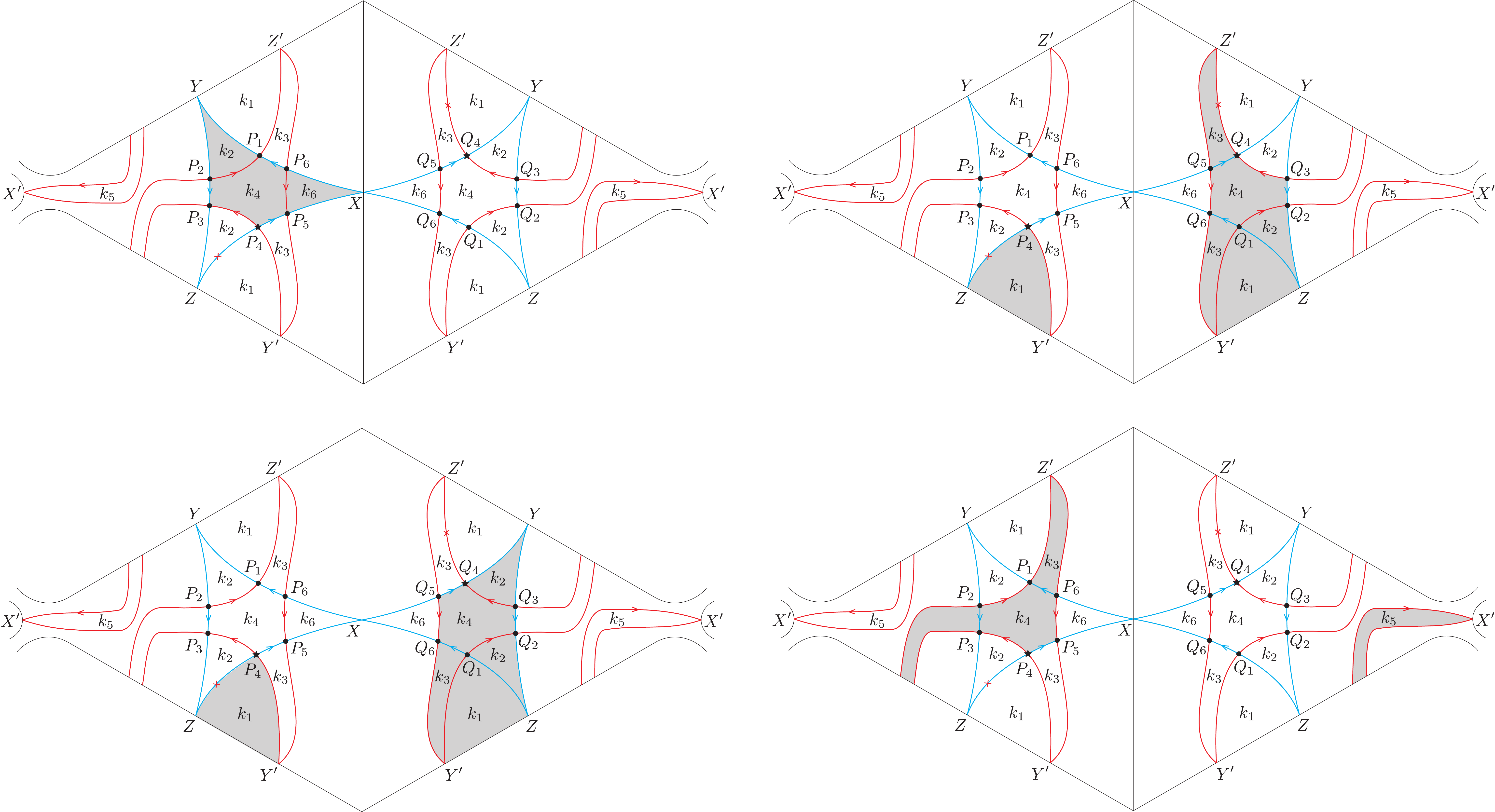}
    \caption{}
		\label{fig:twosei1234}
\end{figure}

It follows that $\alpha:= P_4 -Q_4$ becomes a cocycle if and only if
\begin{equation}\label{eqn:coordtwosei}
\left\{
\begin{array}{l}
x' = x^{-1}  T^{4k_1 + 2k_2 + 2k_3 - k_5 -k_6} \\
y'= x y T^{k_6 - 2k_1 - 2k_3} \\
z'= x z T^{ k_6 - 2k_1  - 2k_3  }\\
\end{array} \right.
\end{equation}
One can easily check that $\alpha$ admits an inverse $\beta = P_1 +  Q_1$ under this condition.
%(, where the Novikov constants $c$ and $c'$ depends on the choice of Morse functions on $\WT{\bL_0}$ and $\bL_1$.
 Thus it gives an isomorphism between $(\WT{\bL_0}, b')$ and $(\bL_1 , b)$ provided that both $b$ and $b'$ have positive valuations. The polygons counted for the composition (either $m_2^{b',b,b'}$ or $m_2^{b,b',b}$) of $P_1$ and $P_4$ are drawn in Figure \ref{fig:m2twosei}. In fact, their $m_2$ count exactly the same polygons contributing to coefficients $P_1$ and $Q_1$ in $m_1^{b',b} (\alpha)$, and the boundaries of these polygons in fact sweep over $\WT{\bL_0}$ and $\bL_1$ once.

\begin{figure}[htb!]
    \includegraphics[scale=0.45]{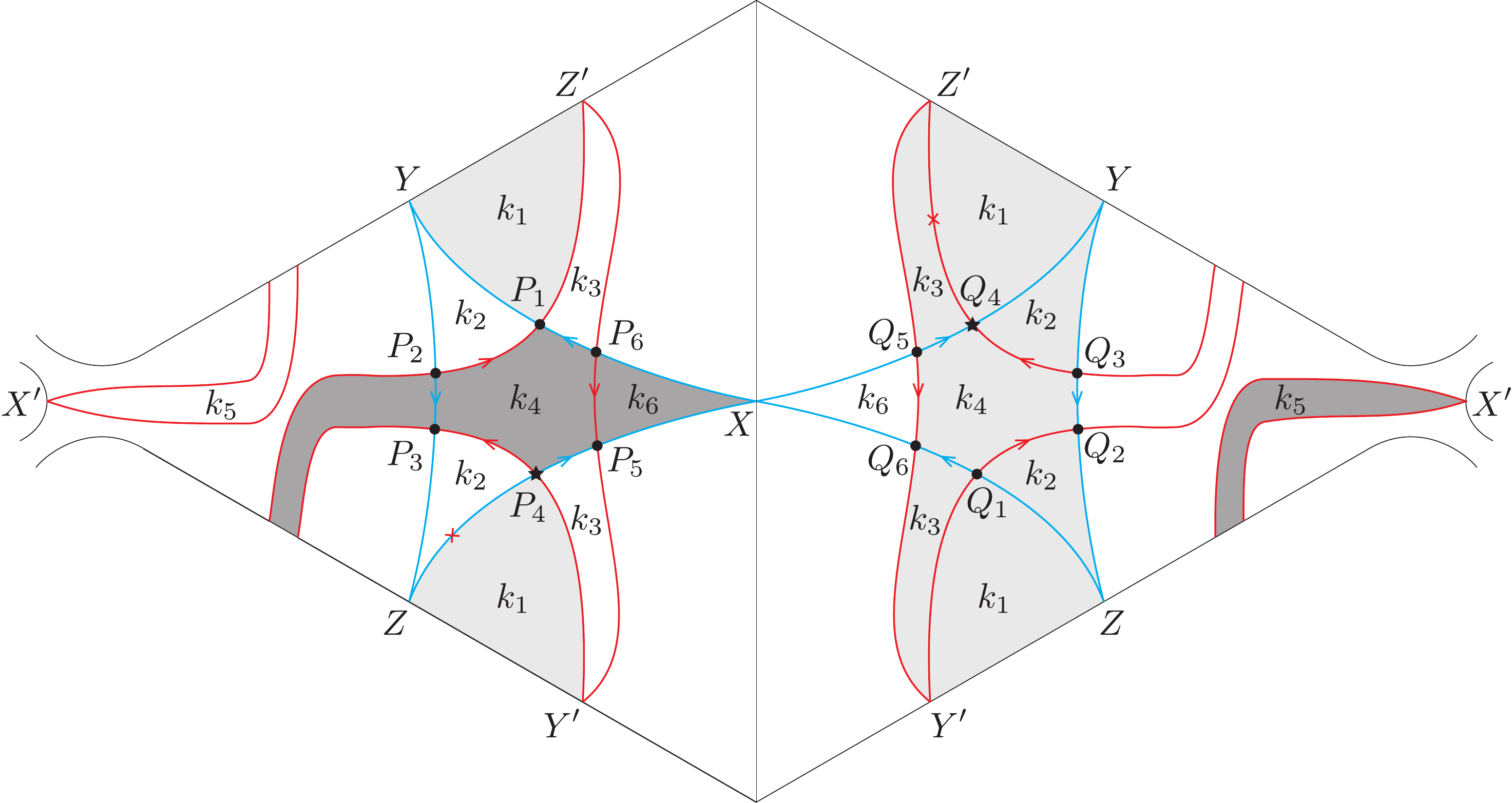}
    \caption{}
		\label{fig:m2twosei}
\end{figure}

We next analyze the valuations of the variables more precisely. Set $\delta:= 4k_1 + 2 k_2 + 2 k_3 - k_5 -K_6$, which is obviously positive from the picture. (Roughly, it is more or less the area of the cylinder that wrap around the neck region of the 4-punctured sphere once.), and $\epsilon:= - ( k_6 - 2k_1 - 2k_3)$ which is also positive. Then taking valuation of \eqref{eqn:coordtwosei} gives
\begin{equation*}
\left\{
\begin{array}{l}
\val (x') = -\val (x) +\delta \\
\val (y')= \val (x) + \val (y)-\epsilon \\
\val (z')= \val (x) + \val (z)  - \epsilon \\
\end{array} \right.
\end{equation*}
Therefore, we should have 
$$\val (x) < \delta ,\quad  \val (x) + \val (y) > \epsilon , \quad  \val (x) + \val (z) > \epsilon$$ 
in order to have a nontrivial overlap of two charts $\{(x,y,z) \in \Lambda_+^3  \}$ and $\{(x',y',z') \in \Lambda_+^3  \}$, and there are similar estimates for $(x',y',z')$.
%For simplicity, one may assume $\delta = \epsilon$ (or equivalently, $a_5 = 2a_1 + 2a_2$), and in this case,
Therefore, we glue two charts over the regions
\begin{equation*}
\begin{array}{l}
 \{ (x,y,z) \in \Lambda_+^3 \mid \val (x) < \delta , \val (x) + \val (y) > \epsilon , \val (x) + \val (z) > \epsilon \}, \\
 \{ (x',y',z') \in \Lambda_+^3 \mid \val (x') < \delta , \val (x') + \val (y') > \delta - \epsilon , \val (x') + \val (z') > \delta - \epsilon \}.
 \end{array}
 \end{equation*}
The above discussion proves proposition.
\end{proof}

Note that if we move $Y',Z'$ appropriately, then our coordinate change formula may not involves any area term.
Recall that the critical locus of the superpotential $xyz$ is the union of the coordinate axis.
Hence if $xx' \neq 0$, then we should be $y=z=0$ and well as $y'=z'=0$ to have non-trivial Floer cohomology. In this case, one can easily check that the resulting homology of the above complex has rank 8, which is the same as the rank of self-Floer homology of the Seidel Lagrangian. 

Combining the above with \eqref{eq:co}, we have a precise change of coordinate formula. For the general cases of
punctured Riemann surfaces, we will generalize this idea to find the global mirror and prove homological mirror symmetry.

\subsection{Other types of coordinate changes}\label{subsec:generalatwosei}
We have chosen $\alpha=P_4 - Q_4$ as an isomorphism in the previous section to deduce the coordinate change \eqref{eqn:coordtwosei}. However, this is not the only possible choice. In fact, one can choose different isomorphisms, which
results in different (but equivalent) Landau-Ginzburg models.
Such a phenomenon also appeared in Section  \ref{sec:isotg} as a choice of isotopy of gauge points.

Let us consider the morphism
$\alpha_{a} : = x^{a-1} P_4 - Q_4 \in CF( (\WT{\bL_0},b'),(\bL_1,b))$
(so, $\alpha= \alpha_1$ is one of its special cases). This is allowed since $a$ never vanished over the gluing region. In this case, $m_1^{b',b} (\alpha_{a})=0$ is equivalent to
\begin{equation*}
\left\{
\begin{array}{l}
x' = x^{-1}  T^{\delta} \\
y'= x^{a} y T^{-\epsilon} \\
z'= x^{2-a} z T^{ -\epsilon  },
\end{array} \right.
\end{equation*}
and one can check that $\alpha_{a}$ is also an isomorphism under this condition, where its inverse $\beta$ should be modified accordingly. The gluing region is also affected by this change, since we have new inequalities
\begin{equation*}
\left\{
\begin{array}{l}
0<\val (x') = -\val (x) +\delta \\
0<\val (y')= a \val (x) + \val (y)-\epsilon \\
0<\val (z')= (2-a) \val (x) + \val (z)  - \epsilon,
\end{array} \right.
\end{equation*}
which still gives a nonempty gluing region.
%\begin{remark}
%As in \ref{subsec:twistgauge}, the mirror Landau-Ginzburg model can have different underlying spaces, although the critical loci remain isomorphic so that the associated category of matrix factorization is invariant. One can see the same phenomenon in the computation of direct isomorphisms between two Seidel Lagrangians by twisting one of Seidel Lagrangians along the neck region of the 4-punctured sphere as follows.
%
%{\color{red} ADD FIGURE}
%\end{remark}
%
%{\color{red}
%\begin{remark}
%The reason for the existence several different ways of gluing is that $(\WT{\bL_0},b')$ and $(\bL_1,b)$ are trivial objects if we fully turn on boundary deformations. However, $(\WT{\bL_0},x'X')$ and $(\bL_1, xX)$ are still nontrivial, and this is reflected in the fact that the identification $x' = x^{-1}  T^{\delta}$ does not depend on the choices we make (i.e. it does not contain $a$).
%\end{remark} 
%}
%

\section{Double-circles and Seidel Lagrangians}\label{sec:doublecircles}
In this section, we provide the relation between the (deformed) Seidel Lagrangian and the pair-of-circles.
As explained in Section \ref{sec:4sp}, if we smooth out one of the immersed points of the Seidel Lagrangian
it becomes a union of two circles which we call a pair-of-circles. For example, see the two red circles in Figure \ref{fig:cone}
or $C_1,C_2$ in Figure \ref{fig:pp-decomp}.
 
 We use a slightly different deformation of the Seidel Lagrangian from the one studied in Section \ref{sec:localisot} and \ref{sec:glue2s}.  The intersection of such a deformed Lagrangian with a Seidel Lagrangian in a neighboring pair-of-pants is simpler, namely they can be made to intersect at eight points (while the one used in Section \ref{sec:glue2s}) has 12 intersection points shown in Figure \ref{fig:twosei}).  We find a precise isomorphism between the pair-of-circles and the Seidel Lagrangian in this section.  This gives the gluing formula stated in Section \ref{sec:4sp}.
 
Let $S_1$ be a Seidel Lagrangian in a pair-of-pants.  We take a deformation $S_1^x$ as shown in Figure \ref{fig:cone} (which is different from the one shown in Figure \ref{fig_protosei} of Section \ref{sec:localisot}).

\begin{figure}[htb!]
    \includegraphics[scale=0.4]{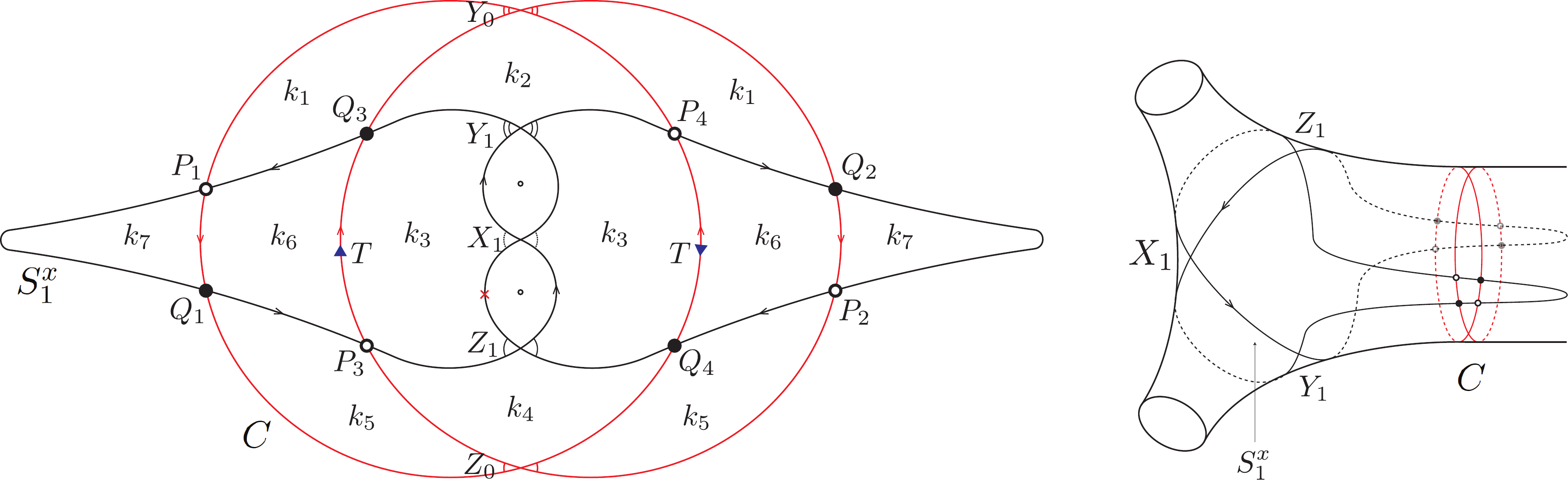}
    \caption{The deformed Seidel Lagrangian $S_1^x$ and the pair-of-circles $C$.  On the left hand side, one of the three punctures of the pair-of-pants is located at infinity, and the other two are located at the middle of the figure.  $k_i$ label the areas of the regions.}\label{fig:cone}
\end{figure}

% For simplicity, let us say that our ambient space is 4-punctured sphere $X$, and we consider one pair of pants which contains the Seidel Lagrangian $S_1$ ( the general case is exactly the same).

Let $C$ be a pair-of-circles $S^1 \cup \phi_H(S^1)$ where  $\phi_H$ is a  Hamiltonian perturbation, and $C$ is preserved by the $\Z/2$-reflection (about the equator of the pair-of-pants).  Moreover the intersection $S^1 \cap \phi_H(S^1)$, which consists of two points, is arranged to be fixed by the reflection.

The Seidel Lagrangian $S_1$ does not intersect the pair-of-circles $C$, see Figure \ref{fig:pp-decomp}.  To obtain the Floer-theoretical relations, we take a deformation $S_1^x$ which intersects at eight points with $C$.

We shall see that we need to make a big deformation of $S_1$.
Namely, the region labeled $k_7$ is required to be larger than the cylindrical region between $S_1$ and $C$.
In particular, $S_1^x$ is NOT a Hamiltonian perturbation of $S_1$.
(Hamiltonian diffeomorphism gives a Floer theoretically isomorphic object which still does not have non-trivial morphism with $C$.)

Recall from Lemma \ref{lem:dc}, that $C$ has the bounding cochain $b = yY + zZ$.
\begin{defn}
Let $\nabla^{tT}$ be a flat connection on $C$ whose holonomy across the vanishing cycle of the surgery is $tT$. (See also Definition \ref{def:vc}.)
Denote by $b_1 := x_1X_1+y_1Y_1+z_1Z_1$ and $b_0:=(\nabla^{tT},y_0Y_0+z_0Z_0)$ the bounding cochains of $S_1^x$ and $C$ respectively, where $(t,y_0,z_0) \in \Lambda_0^\times \times \Lambda_+^2$
and $(x_1,y_1,z_1) \in \Lambda_+^3$ .

\end{defn}
The disc potential functions for $C$ and $S_1$ are $ty_0z_0$ and $T^A x_1y_1z_1$ respectively, where $A$ is the area of one of the two triangles bounded by $S_1$.  The relation between $S_1^x$ and $C$ is given as follows.

%Intuitively we argued in Section \ref{sec:4sp} that the relation between $b_1, b_0$ are given by $x_1=t,y=y_0,z=z_0$ by considering family of smoothings.

\begin{prop} \label{prop:cocycle_4punc}
Let $$\alpha = P_1+P_2 \in  CF(C, S_1^x), \beta \in Q_1 - Q_2 \in CF(S_1^x,C) $$
where $P_i, Q_i$ are intersection points given in Figure \ref{fig:cone}.
$(\alpha,\beta)$ gives an isomorphism between the objects $(S_1^x,b_1)$ and $(C,b_0)$ for $b_1 \in (\Lambda_+)^3, b_0 \in  \Lambda_0^\times \times \Lambda_+^2$ if and only if 
\begin{equation}\label{eq:ds}
\begin{cases} x_1=t T^{\delta} \\y_1=y_0 T^{-\eta_1} \\ z_1=z_0 T^{-\eta_2} \end{cases} \;\;\;\;\;\; \textrm{where} \;\; \begin{cases}  \delta = k_7 - k_1 - k_2 - k_3 - k_4 - k_5 \\  \eta_1 =  k_7-2k_1-k_2 \\ \eta_2 = k_7 -k_4-2k_5. \end{cases}
\end{equation}
\end{prop}

\begin{remark}
	We shall take the limit $k_6=k_3=k_1=k_5=0$, so that we only leave with $k_7, k_2, k_4 > 0$.  The condition $\delta \geq 0$ reduces to $k_7 \geq k_2 + k_4$, where $k_2+k_4$ is the cylinderical area bounded byetween $S_1$ and $C$.  When $k_7 \leq k_2+k_4$ isomorphism does not exist.  This shows that the amount of stretching depends on the
	location of the pair-of-circles $C$.  Farther away $C$ is from $S_1$, bigger $k_7$ is required to be.
\end{remark}

\begin{proof}
We will see that $m_1^{b_0,b_1}(P_1 + P_2) = 0$, $m_1^{b_1,b_0}(Q_1 -  Q_2)=0$ provides the exact coordinate change as follows.
First, we count holomorphic strip with input $P_1,P_2$, with possible insertions of $b_0$ (in the upper boundary) and
$b_1$(in the lower boundary). In Figure \ref{fig:puncsph-discs}, the first diagram in the upper left corner is an honest holomorphic strip from $P_1$ to $Q_1$, and another strip with an insertion of $X_1$ is drawn in upper right corner. And these two are the only
contributions from $P_1$ to $Q_1$. By computing their areas and signs, we see that $Q_1$ component of $m_1^{b_0,b_1} (P_1)$ vanishes if and only if
\begin{equation}\label{eqn:dcx1t}
T^{k_7} - x_1 t^{-1}T^{k_1 + k_2 + k_3 + k_4 + k_5}=0
\end{equation}
This gives rise to the coordinate change formula between $x_1$ and $t$.
\begin{figure}[htb!]
    \includegraphics[scale=0.75]{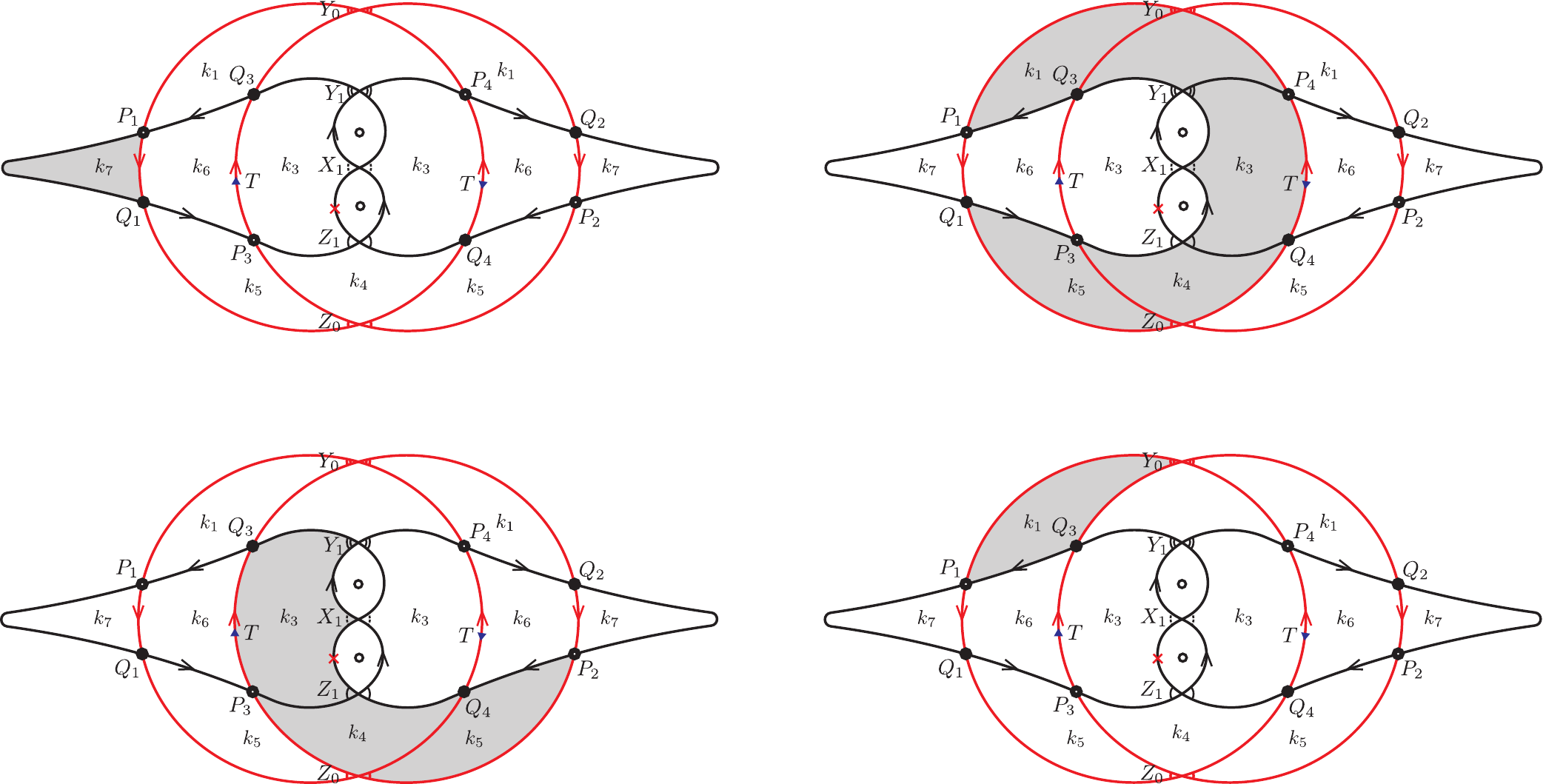}
    \caption{Holomorphic strips used in Proposition \ref{prop:cocycle_4punc}.}
		\label{fig:puncsph-discs}
\end{figure}

Similarly, $m_1^{b_0,b_1} (P_1)$ has $Q_3$ as an output (which is drawn in the lower right corner) and
note that we have $m_1^{b_0,b_1}(P_2)$ has $Q_3$ as an output also (drawn in the lower left corner).
Thus the vanishing of $Q_3$ contribution in  $m_1^{b_0,b_1} (P_1 + P_2)$ is equivalent to
$$y_0 T^{k_1} -  x_1 y_1 t^{-1} T^{k_3 + k_4 + k_5} = y_0 T^{k_1}  -y_1 T^{k_7 - k_1 - k_2} $$
where we used the identity  \eqref{eqn:dcx1t}, and the vanishing of this expression gives rise to the second coordinate change formula.

In the same way, we can check that the vanishing of $Q_4$ component in $m_1^{b_0,b_1} (P_1 + P_2)$
which gives  $$ T^{k_1+k_2+k_3}t^{-1}x_1z_1= z_0 T^{k_5}.$$ 
The computation for $Q_1- Q_2$  is similar and gives rise to the same coordinate changes.

In addition, we show that  $\alpha, \beta$  are
inverses to each other in the following lemma, which proves the proposition. \end{proof}
\begin{lemma}
Assume that \eqref{eq:ds} holds. The products of $\alpha=P_1+P_2$ and $\beta=Q_1-Q_2$ are given by
$$\begin{cases}
m_2^{b_0,b_1,b_0} (\alpha,\beta)= T^{k_7}\one_{S_1^x} \\
m_2^{b_1,b_0,b_1} (\beta,\alpha)=T^{k_7} \one_{C} 
\end{cases}$$
\end{lemma}
\begin{proof}
The same holomorphic polygons in the upper left and right of Figure \ref{fig:puncsph-discs} contribute to the above $m_2$-computation. The additional marked point will be an output marked point on the edges of the polygons, thus
contributing to  the deformed $m_2$-operation. Note that the union of boundaries of these two shaded regions (for $P_1,Q_1$)
and its reflection image (for $P_2,Q_2$) covers the Seidel Lagrangian $S_1^x$ and the pair-of-circles $C$ exactly once.
In this way, we can show that the outputs are the units.
We leave the detailed check as an exercise.

Now, since $t$ is a holonomy parameter, hence it has valuation $0$. To make $x_1= t T^{\delta}$ lie in $\Lambda_0$ (so that it defines
a bounding cochain of $S_1^x$ in Floer theory), we need $\delta \geq 0$. Also, since the critical locus of $W=xyz$ are given by coordinate axis, if we set two of the deformation parameter  $x_i,y_i,z_i$ to be non-zero, then the corresponding deformed Lagrangian turns out to be trivial. As we are considering the case of non-trivial $x_1$ (as a surgery parameter), we may set $y_i$ and $z_i$ to be zero. These assumptions are in fact not necessarily.  If $y_0,y_1,z_0,z_1 \in \Lambda_0$ and non-zero, then it should still
define (trivial) isomorphisms between zero objects.
\end{proof}

%\begin{remark} %%%\tilde{a} and a%%%
%Similarly to \ref{subsec:generalatwosei}, one can choose $ t^{a} e_1 + e_2$ as an isomorphism for any $a \in \Z$. By essentially the same counting as above, we see that the relation $\tilde{t} = t^{-1}$ remains unchanged, but the other two become
%$$y_0' = t^{1+a} y_0, \quad z_0' = t^{1-a} z_0,$$
%which coincides with the above computation. 
%\end{remark}

In the same manner, we can compute the relation between the Seidel Lagrangian $S_1$ and its deformation $S_1^x$.
\begin{figure}[htb!]
	\begin{center}
		\includegraphics[scale=0.7]{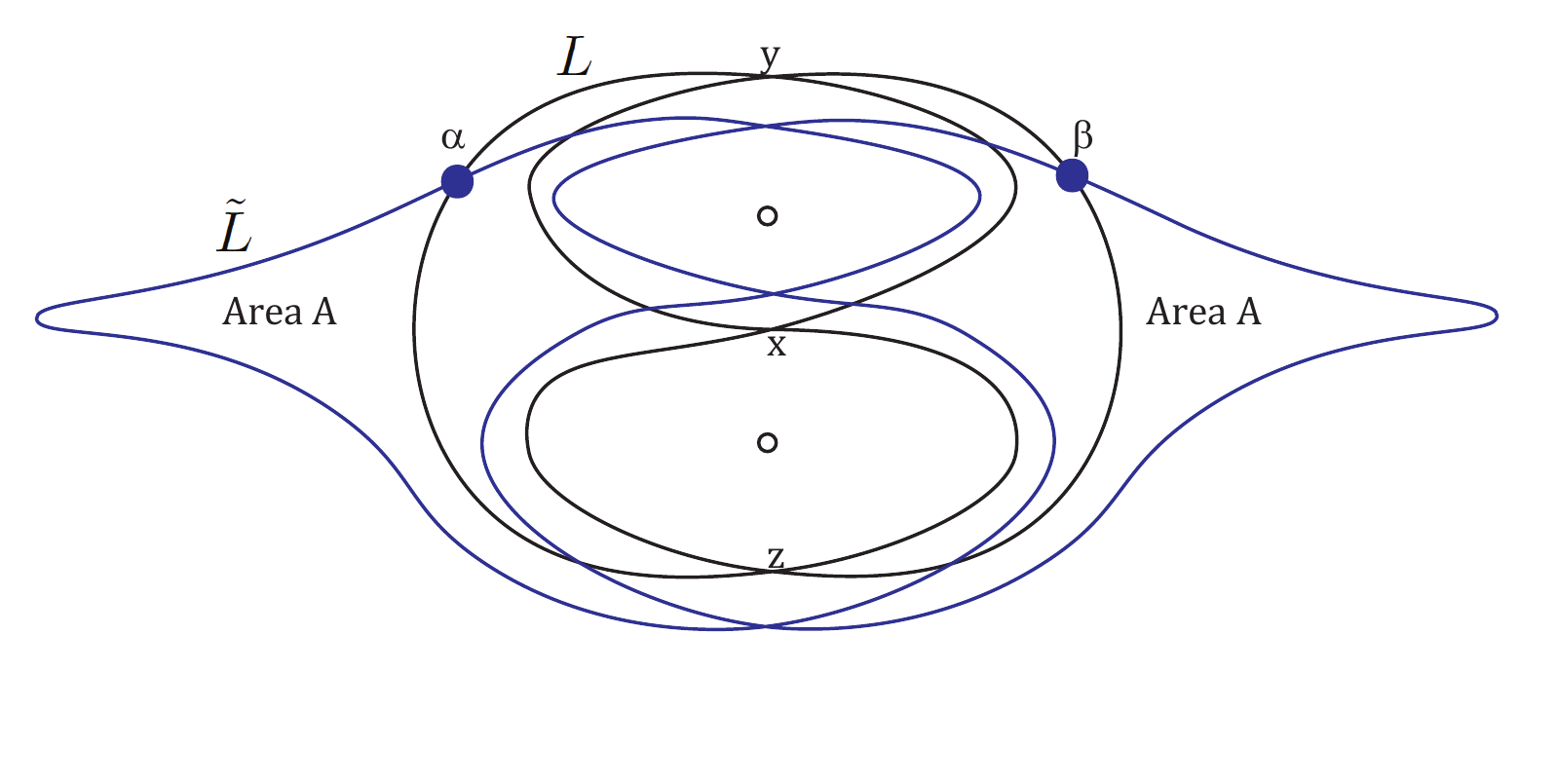}
		\caption{Comparing $S_1^x$ and $S_1$}
		\label{fig:deform-Seidel}
	\end{center}
\end{figure}
The coordinate change can be obtained using $\alpha$ and $\beta$ in Figure \ref{fig:deform-Seidel}, which
we leave as an exercise. A simplified formula in  the limiting case can be easily obtained.
\begin{prop} \label{prop:move-var}
Consider the case where Seidel Lagrangian $S_1$
limits to the skeleton of a pair of pants, so that the minimal triangle has zero area.
Then $(S_1, b=xX+yY+zZ)$  and $(S_1^x, \tilde{b}= \tilde{x}\tilde{X}+\tilde{y}\tilde{Y}+\tilde{z}\tilde{Z})$
are related via an isomorphism if
	$$ \tilde{x}=T^A x, \tilde{y}=T^{-A} y, \tilde{z}=T^{-A} z  $$
\end{prop}
We will use the above formula in Section \ref{sec:cr}.

\begin{remark}
	The disc potentials of $L,\tilde{L}$ are $xyz$, and $T^A \tilde{x}\tilde{y}\tilde{z}$ respectively, since $\tilde{L}$ bounds a triangle with area $A$, while $L$ bounds a triangle with area $0$ (in the limit).  The potentials respect the coordinate changes given in the above proposition. When $\val(\tilde{x})=0$, $\val(x)=-A<0$.  It indicates that from $S_1$ to $S_1^x$, we are deforming in the direction of $-\val(x)$.
\end{remark}

Now consider a 4-punctured sphere which is a union of two pair-of-pants.  (We shall study punctured Riemann surfaces in Section \ref{sec:mir-surf} and \ref{sec:HMS-surf}.) We have two Seidel Lagrangians $S_1,S_2$ in each pair of pants.  

Similar to Proposition \ref{prop:cocycle_4punc}, $S_1$ can be stretched to $S_1^x$ which intersects $S_2$ at eight intersection points (see Figure \ref{fig:Lag-winding} or \ref{fig:gluing-edges}).  We have a similar isomorphism $(\tilde{\alpha}=\tilde{P}_1+\tilde{P}_2,\tilde{\beta}=\tilde{Q}_1-\tilde{Q}_2)$ given by intersection points $\tilde{P}_i, \tilde{Q}_i$ for $i=1,\ldots,4$ between $S_1^x$ and $S_2$, under the relations
$$x_1 = T^{\tilde{k}_7 - \tilde{k}_2 - \tilde{k}_4} x_2^{-1}, y_1 = T^{-(\tilde{k}_7 - \tilde{k}_2)} y_2, z_1 = T^{-(\tilde{k}_7 - \tilde{k}_4)} x_2^2 z_2$$
where $\tilde{k}_2, \tilde{k}_4, \tilde{k}_7$ are areas of regions similar to above.  (We shall not care too much about these areas since we shall use exactness of $S_1^x,S_2$ to absorb these area terms into exact variables).  Such an isomorphism can be understood as the composition of the above isomorphism from $S_1^x$ to $C$, and that from $C$ to $S_2$ (assuming $C$ is taken to intersect $S_2$ at eight points as in Figure \ref{fig:Lag-winding}).  

As in Section \ref{subsec:generalatwosei}, we can take a different isomorphism $\tilde{\alpha}=x_1^{a_1}\tilde{P}_1+x_1^{a_2}\tilde{P}_2$ to get different coordinate relations
$$x_1 = T^{\tilde{k}_7 - \tilde{k}_2 - \tilde{k}_4} x_2^{-1}, y_1 = T^{-(\tilde{k}_7 - \tilde{k}_2)} x_2^{a_1-a_2} y_2, z_1 = T^{-(\tilde{k}_7 - \tilde{k}_4)} x_2^{2+a_2-a_1} z_2.  $$
This can be understood as the composition of the isomorphism from $S_1^x$ to $C$, a gauge change on $C$, and the isomorphism from $C$ to $S_2$.  We shall use such an isomorphism between $S_1^x$ and $S_2$ in Section \ref{sec:mir-surf} and \ref{sec:HMS-surf}.  The choice of $a_1,a_2\in \Z$ would be a part of the input data for the mirror construction.

Note that while the isomorphism $\tilde{\alpha}$ is valid for $\val(x_1)$ belonging to a non-empty open set, the interpretation via the pair-of-circle $C$ is only valid for a particular value of $\val(x_1)$ since $\val(t)=0$ where $t$ is the holonomy parameter for $C$.  In order to make the interpretation via $C$ valid for an open set of $\val(x_1)$, one can take a family of pair-of-circles $C$.  

It suggests an alternative approach of taking a family of (infinitely many) pair-of-circles interpolating between the immersed Lagrangians $S_1$ and $S_2$.  This is similar to the SYZ fibrations and is closer to the family Floer theory \cite{Fukaya-famFl, Tu-FM, Ab-famFl2}.  It involves infinitely many mirror charts and the gluing of functors becomes rather complicated.  Moreover the pair-of-circles are not exact in general.  

To implement this alternative approach for Riemann surfaces, 
one may take a countable set of generators of the exact Fukaya category, such that these generators  have the same intersection pattern with any of the pair-of-circles $C_t$ in this family $t \in I$.  This will provide a (small) rigid analytic chart $U_C$ (See (b) of Figure \ref{fig:modovernov} for an illustration) which can be glued with that of $S_1, S_1^x, S_2^{x'}, S_2$.

In this paper we only use the isomorphisms between $S_1, S_1^x, S_2$ which involve only finitely many charts, so that gluing of functors can be computed in a very efficient way.  The pair-of-circles $C$ provides a good interpretation of the choice of isomorphisms between $S_1^x$ and $S_2$, and also serves as a convenient intermediate step to compute the gluing between $S_1^x$ and $S_2$ in Section \ref{sec:HMS-surf}.  However we shall not use the chart of $C$ in our mirror construction.

%by pushing both $S_1$ and $S_2$ toward each other and considering the pair-of-circles $C$ in the middle, so that both $S_1^x, S_2^{x'}$ become isomorphic to $C$.  But this approach has a drawback in that the holonomy parameter $t$ for $C$ has a fixed valuation 0 (whereas immersed parameter has valuation $\geq 0$). Thus we need to use infinitely many pair-of-circles to construct a mirror chart.  This is contrast to the construction in section \ref{sec:glue2s} where we used only 3 immersed Lagrangian to construct mirror charts.  This is the benefit of using immersed Lagrangian.

\begin{figure}[htb!]
    \includegraphics[scale=0.45]{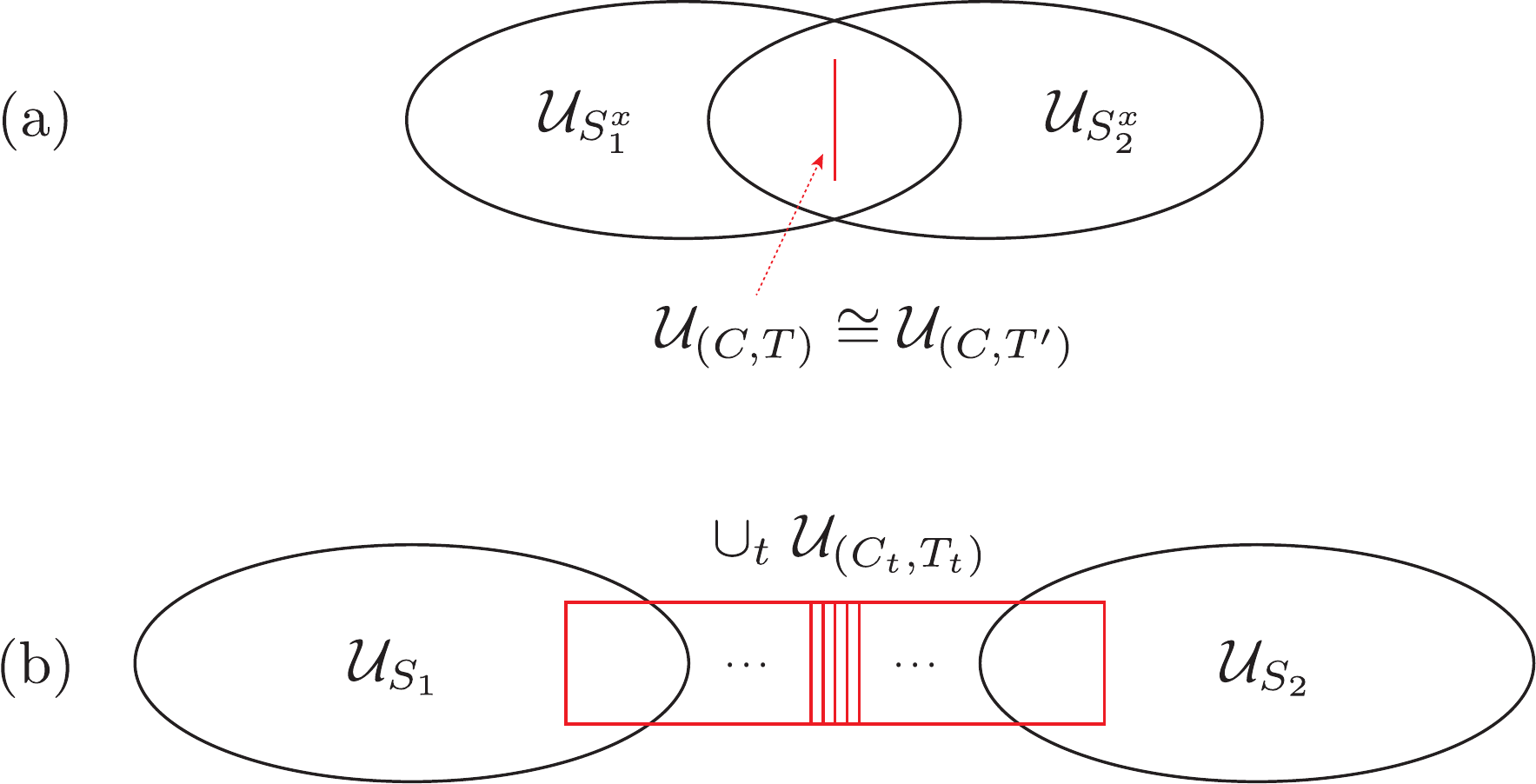}
    \caption{The moduli space over $\Lambda$.}
		\label{fig:modovernov}
\end{figure}

\section{$\C$-valued homological mirror functor}\label{sec:C}
In this section we explain how to obtain $\C$-valued theories from $\Lambda$-valued constructions in previous sections.
The localized mirror functor machinery works only with $\Lambda_+$ or $\Lambda_0$ coefficients.  It cannot be directly applied to Fukaya category with coefficients in $\C$.

To obtain $\C$-valued theory, we proceed in the following way. 
We first embed the exact Fukaya category $\mathcal{F}^{\C}(M)$ to 
 $\Lambda$-valued Fukaya category  $\mathcal{F}^\Lambda(M)$.
 In $\mathcal{F}^\Lambda(M)$, we obtain mirror charts, gluing data and homological mirror functors in $\Lambda$.
In fact, these data come from Lagrangian Floer theory between exact Lagrangian submanifolds.
Therefore, we can absorb all the area terms and obtain $\C$-valued theory using the embedding of exact Fukaya category
one more time.  

We do need to assume certain convergences so that the $\C$-reduced charts are
given by polynomial ring $\C[x_1,\cdots,x_n]$ instead of power series ring $\C[[x_1,\cdots,x_n]]$ and so on.
Also non-trivial coordinate changes as in Section \ref{sec:localisot} in $\Lambda$ may become trivial coordinate changes in $\C$.
Therefore, $\C$-reduced mirror could have multiple charts which are identical, but mirror functor and homotopies on those charts are not identical. Thus we still need to keep them to have the theory working. 

\subsection{Embedding of exact Fukaya category} \label{sec:exact}
The following is well-known and we review it to set up a convention. See \cite{Seidel-book} for more details.
Let $(M, \omega=d\theta)$ be a Liouville manifold. 
A Lagrangian submanifold is exact if $\theta |_L = d f_L$ for some $f_L$.
In addition, an exact non-compact Lagrangian submanifold $L$ is conical at infinity (invariant under the Liouville flow outside a compact set)
and required to have $\theta |_{L}$ vanishes outside a compact set.
Consider two exact Lagrangian submanifold  $L_0, L_1$ such that $\alpha|_{L_i} = d f_i$ for $i=0,1$.
We consider its path space $\Omega(L_0,L_1) :=\{ \gamma:[0,1] \to M \mid \gamma(0) \in L_0, \gamma(1) \in L_1\}$.
An action functional is defined as $A_{L_0,L_1}(\gamma)= - \int_\gamma \alpha   - f_0(\gamma(0))+f_1(\gamma(1))$.

We denote by $\mathcal{F}^{\C}(M)$ the exact (wrapped) Fukaya category of $M$ with coefficients in $\C$.
and by $\mathcal{F}^\Lambda(M)$ the $\Lambda$-valued Fukaya category of $M$ with the same set of objects.
(We may allow more general non-exact objects in $\mathcal{F}^\Lambda(M)$, but this is not needed for our purpose in this section).
We assume that the perturbation scheme of $\mathcal{F}^{\C}(M)$ and $\mathcal{F}^\Lambda(M)$ are the same.
Namely, for $\mathcal{F}^{\C}(M)$, one uses Floer datum for each pair of Lagrangians and compatible system of
domain dependent perturbation. This has been extended to the Morse-Bott setting and allowing also exact immersed Lagrangians in
\cite{Seidel-g2} and \cite{Sheridan11}, and we use the same perturbation scheme for  $\mathcal{F}^\Lambda(M)$.
For an immersed Lagrangian $L$, which is given by an immersion $i: \WT{L} \to X$, 
an immersed generator $X \in Hom(L,L)$ may be considered as an intersection of local Lagrangian branches 
in the order of $\WT{L}_0$  and $\WT{L}_1$. Recall that switching the branches gives the opposite generator $\bar{X}$. $L$ is exact if $i^*\theta = d(f_{\WT{L}})$. 
Each immersed generator $X$ has two pre-images $X_i \in \WT{L}_i$ for $i=0,1$.

We consider the inclusion functor following Lekili \cite{Lekili12} Lemma 1.2.
\begin{prop}
There is a fully faithful $\AI$-functor
$$\ee: \mathcal{F}^{\C}(M) \otimes_\C \Lambda_\C \to \mathcal{F}^{\Lambda}(M)$$
linear over $\Lambda_\C$. This is the identity on objects. 

For $X \in \Hom(L_0,L_1) = CF(\phi(L_0),L_1)$ where the Hamiltonian $\phi$ is from the Floer datum of the pair $(L_0,L_1)$ the functor $\ee$ is given by
$\ee(X) = T^{A_{\phi(L_0),L_1}(X)}X$.
For a Morse critical point $p \in \Hom(L,L)$, we take $\ee(p) = p$.
 For an immersed generator $X \in Hom(L,L)$,
 we take $\ee(X) = T^{-f_{\WT{L}}(X_0) + f_{\WT{L}}(X_1)}X$.
 
 The higher $\AI$-terms of $\ee$ are identically zero.
\end{prop}
The proof is similar to  that of \cite{Lekili12} and omitted.

To distinguish generators, we make the following definition.
\begin{defn}
A generator $X \in Hom(L_0,L_1)$ in $\mathcal{F}^{\Lambda}(M)$ is denoted as $X_{geo}$ in this section.
If $X$ is an immersed generator, $X_{geo}$ is given by a constant path at the immersed point.

For a generator $X \in Hom(L_0,L_1)$ in $\mathcal{F}^{\C}(M)$,
we denote its image under the inclusion functor $\ee$ by $X_{ex}$.
Namely, we call 
$$ X_{ex}:= \ee(X)$$
to be an exact generator which is a morphism in $\mathcal{F}^{\Lambda}(M)$.
For an immersed generator $X$, we have
$$X_{ex} = T^{-f_{\WT{L}}(X_0) + f_{\WT{L}}(X_1)} X_{geo}$$
\end{defn}
We just note that since $\ee$ is a functor with vanishing higher $\AI$-terms, we have
\begin{equation}\label{exsa}
m_k^\C(Y_1,\cdots,Y_k) = Y_0 \quad \Longrightarrow \quad m_k^\Lambda((Y_1)_{ex}, \cdots, (Y_k)_{ex}) = (Y_0)_{ex}
\end{equation}
Thus if we write inputs and outputs in exact generators (elements in the image of $\ee$), then
the coefficients lie in $\C$ (not just in $\Lambda_0$). 
This means that these exact data (such as $\theta, f$) assigns $x_{ex}$ at the canonical energy level (which is not of valuation zero in general)
so that all the transition functions can be $\C$-valued. This transition data defines $\C$-valued mirror and functors. Let us describe it in more detail.

\begin{defn}
Define dual variables $x_{ex} \in \Lambda$ and $x_{geo} \in \Lambda_0$ by 
$$ x_{ex} X_{ex} = x_{geo} X_{geo}.$$
\end{defn}
For $X$ immersed, we have
\begin{equation}\label{eq:exgeo}
 x_{ex} T^{-f_{\WT{L}}(X_0) + f_{\WT{L}}(X_1)}  = x_{geo}.
 \end{equation}
For Maurer-Cartan formalism, we need $\val(x_{geo}) \geq 0$  and hence we should have
$$\val (x_{ex}) \geq f_{\WT{L}}(X_0) - f_{\WT{L}}(X_1).$$
This inequality may be somewhat confusing at first. Note that the valuation of exact variables $x_{ex}$ is not necessarily zero or
sometimes cannot be zero! (note that the valuation of $\C \subset \Lambda_0$ is zero).
On the other hand, we will see that all the transition data are written in $\C$ in terms of these exact variables.

\subsection{$\C$-valued localized mirror functor}
Let us first define $\C$-valued localized mirror functor (see Section \ref{sec:lmf}).
Let $\bL$ be an exact immersed Lagrangian, with  bounding cochain from immersed generators
$$b=  \sum x_{ex,i} X_{ex,i} = \sum x_{geo,i} X_{geo,i}, \;\; \textrm{for}\;\; x_{geo,i} \in \Lambda_0$$
If we want to emphasize that we use exact variables, we may also denote $b$ by $b_{ex}$.

By our localized mirror construction, we obtain $\Lambda$-valued potential function  $W^\Lambda$ on $x_{geo}$-variables.
We have $W^\Lambda \in \Lambda_0 \ll x_{geo,1} ,\cdots, x_{geo,k} \gg$ which is
the completion of polynomial ring with respect to the filtration of $\Lambda_0$.
If we change variables of $W^\Lambda$ to exact variables using \eqref{eq:exgeo}, then obtain
a $\C$-valued potential function.
\begin{defn}
We define the exact potential $W^\C := W^\Lambda(b_{ex})$.
In other words,
since $\ee(\be^\C) = \be^\Lambda$, we have 
$$W^\C \cdot \be^\C = \sum_k \sum_{i_1,\cdots,i_k} x_{ex,i_1}\cdots x_{ex,i_k} m^\C_k (X_{i_1},\cdots, X_{i_k}) \in  \C[[x_{ex,1},\cdots,x_{ex,n}]] \cdot \be^\C$$
\end{defn}

% and the following 
%$$m_k^\Lambda(x_{geo,1} X_{geo,1}, \cdots, x_{geo,k} X_{geo,k}) =x_{geo,1}\cdots x_{geo,k} m^\Lambda_k (X_{geo,1},\cdots, X_{geo,k}),$$
%
%On the other hand, we have
%\begin{eqnarray}\label{ef}
%m_k^\Lambda(x_{geo,1} X_{geo,1}, \cdots, x_{geo,k} X_{geo,k})  &=& m_k^\Lambda(x_{ex,1} X_{ex,1}, \cdots, x_{ex,k} X_{ex,k})  \nonumber \\
%&=& x_{ex,1}\cdots x_{ex,k} m^\Lambda_k (\ee(X_1),\cdots, \ee(X_k) ) \nonumber \\
%&=& x_{ex,1}\cdots x_{ex,k} \ee \big( m^\C (X_{1},\cdots, X_{k}) \big)
%\end{eqnarray}
%The last  line follows since $\ee$ is a functo

%\begin{lemma}
%The potential function $W^\Lambda$ on $\{x_{geo,i}\}$ can be written 
% in terms of $\{x_{ex,i}\}$ by the relation \eqref{eq:exgeo}, and
%we have 
%$$W^\Lambda (x_{ex,i_1}\cdots x_{ex,i_n}) = W^\C. $$
%{\color{red} $W^\Lambda (T^* x_{ex,i_1}\cdots T^* x_{ex,i_n}) = W^\C$ to kill $T$'s in $W^\Lambda$?} 
%\end{lemma}

For example, for a 3-punctured sphere, we have
$$W^\Lambda = T^{A} x_{geo}y_{geo}z_{geo} = x_{ex}y_{ex}z_{ex}=:W^\C: \C^3 \to \C$$
%\begin{proof}
%Note that if 
%$$m_k^\Lambda (X_{geo,1},\cdots, X_{geo,k}) =\sum_Y c^\Lambda_Y Y_{geo}$$
%we have
%$$m_k^\Lambda (X_{ex,1},\cdots, X_{ex,k})  = 
%m_k^\Lambda (\ee(X_1), \cdots, \ee(X_k)) 
%= \ee(m_k^\C(X_1,\cdots,X_k))
%= \ee(\sum_Y \lambda_Y^\C Y) = \sum_Y \lambda_Y^\C \ee(Y)$$
%
%From $m^\Lambda(e^b) = W(b) \cdot \be$, $\ee (\be^\C) = \be^\Lambda$ and \eqref{ef}
%we have 
%$$m^\Lambda(e^b) =\sum_k \sum_{i_1,\cdots,i_k} x_{ex,1}\cdots x_{ex,k} \ee \big( m^\C (X_{1},\cdots, X_{k}) $$
%
%
%\end{proof}

Now, similar idea can be used to show that we have a $\C$-valued analogue of localized mirror functor
in Definition \ref{def:lmf}.
\begin{defn} 
The $\C$-valued  localized mirror functor $F^{\bL, \C} : \mathcal{F}u^\C(M) \to MF(W^\C)$ is defined as follows.
Let $R= \C[[x_{ex,1},\cdots, x_{ex,n}]]$. For a Lagrangian $L$, mirror matrix factorization $F^{\bL,\C}(L)$ is given by $\Z/2$-graded $R$-module
$$ L  \mapsto (Hom(L,\bL)_{ex}, -m_1^{0,b_{ex}}).$$
Also, we define
$$F^{\bL, \C}(Y_1,\cdots,Y_k)  = \sum_{i \geq 1} m_{k+i}(Y_1,\cdots, Y_k, \bullet, b_{ex},\cdots,b_{ex}).$$
\end{defn}
From the discussion in \eqref{exsa}, the above functor is well-defined. We will make the following convergence assumption.

\begin{assumption}\label{as:cov}
Note that the $\C$-reduction naturally lies in the formal power series ring $\C[[x_{ex,1},\cdots,x_{ex,n}]]$ instead of polynomial ring $\C[x_{ex,1},\cdots, x_{ex,n}]$.
We assume that the potential $W^\C$ as
well as the functor $\mathcal{F}^\C$ can be written in the matrix factorization category with $R=\C[x_{ex,1},\cdots, x_{ex,n}]$.
\end{assumption}

\begin{lemma}
A Seidel Lagrangian $\bL$ in a punctured Riemann surface $\Sigma$ satisfies  the above assumption \ref{as:cov}.
In particular, $m_k^{0,b}$ is trivial if $b$ is inserted more than 3 times.
\end{lemma}
\begin{proof}
From the exact condition, given inputs and an output, the energy of possible holomorphic polygon is fixed, and hence there are only finitely many contributions from the Gromov-compactness. But note that we allow infinitely many deformation by $b$-insertions for $m_k^{0,b}$. Thus, it is enough to show that if number of $b$ is more than 3, then $m_k^{0,b}$ is 0.
Since $b$'s are given by the sum of immersed generators $X, Y,Z$ a holomorphic curve component meeting these corners  is given by an immersed
polygon on $\Sigma$. Due to punctures, there are no big polygons with boundary on $\bL$. And also note that if we turn at immersed corners ($X,Y,Z$) twice, then the polygon already close up a minimal triangle for $\bL$.
Thus the claim follows.
\end{proof}
Thus, we may use $\C[X,Y,Z]$ for $\C$-valued mirror symmetry of punctured Riemann surfaces instead of $\C[[X,Y,Z]]$.

In the same way, gluing of charts and functors are defined over $\C$ if we use exact generators, and
we omit the details. We notice the following is a new feature.
Recall that in Section \ref{sec:localisot}, we moved $\bL_0$ to $\bL_1$ in non-Hamiltonian way, and the resulting coordinate
change between $U_0$ and $U_1$ is non-trivial in $(\Lambda_0)^3$. But the $\C$-reduction of this coordinate change is trivial,
and hence $U_0^\C = U_1^\C = \C^3$. 
%We remark that $\bL_0, \bL_1$ are isomorphic in exact Fukaya category, but not so in $\Lambda$-valued Fukaya category.
Even though $U_0$ and $U_1$ is identical, we keep it as a separate chart. This is because the $\AI$-functor on these two charts are not identical (but only homotopic). One can observe this by considering an exact Lagrangian which intersect $\bL_0, \bL_1$ in different patterns. Thus we need to keep as many charts as in $\Lambda$-valued mirror symmetry, and this is what $\C$-reduction of
our construction in Corollary \ref{cor:manychart} gives.

\section{Mirror construction for punctured Riemann surfaces} \label{sec:mir-surf}
In this section, we construct mirrors of punctured Riemann surfaces by the method introduced in previous sections, namely gluing deformation spaces of immersed Lagrangians via isomorphisms.  The starting data is a tropical curve for the punctured surface, and an integer for each finite edge (which is the choice of an isomorphism coming from gauge change).

First we construct the `pseudo' mirror space $Y(\Lambda)$ consisting of `pseudo' deformations of the Seidel Lagrangians over vertices of the tropical curve.  It serves as the ambient space in which the actual mirror sits.  We shall see that $Y(\Lambda)$ is a $\Lambda$-valued toric Calabi-Yau obtained by taking dual fan of the tropical curve.  The $m_0$-terms of the Seidel Lagrangians glue to give a disc potential $W$ over $Y(\Lambda)$.

Next we recall and use the exact variables of the immersed Lagrangians explained in Section \ref{sec:exact}.  Using exact variables as charts, the change of coordinates will not involve the Novikov parameter $T$.  By restricting the exact variables to be $\Lambda_0$,$\Lambda_+$, or $\C$-valued, we obtain $Y(\Lambda_0) \supset Y(\Lambda_+) \supset Y(\C)$.  The disc potential $W$ written in exact variables do not involve the Novikov parameter.  $W$ restricts as $W_{Y(\Lambda_0)},W_{Y(\Lambda_+)},W_{Y(\C)}$ respectively, where $W_{Y(\C)}$ is a $\C$-valued function over the toric CY $Y(\C)$.  However, the change from exact variables to actual immersed variables of the immersed Lagrangians involve $T^{-A}$ for $A>0$ in general.  Thus none of these spaces is formed by merely formal deformations of the Seidel Lagrangians over the vertices. 

It is crucial to use formal deformations rather than pseudo deformations for construction of the mirror functor, because 
symplectic geometry makes sense only for formal deformations and also there are  convergence issues of the $A_\infty$ operations for pseudo-deformations.  We will choose a collection of Lagrangian immersions, which consists of (actual) deformations of the Seidel Lagrangians over the vertices, such that the formal deformation spaces of these objects gives the covering of  the critical locus of $W_{Y(\Lambda_0)}$.  

In Section \ref{sec:localisot} and \ref{sec:glue2s}, we have studied a particular way of deforming the Seidel Lagrangian to obtain relations between two adjacent pair-of-pants.  In Section \ref{sec:doublecircles}, we give another way to deform the Seidel Lagrangian, whose intersection with other immersed Lagrangians is slightly simpler.  One can choose either way to construct the mirror.  In this and the next section, we will use the deformed Lagrangians in 
Section \ref{sec:doublecircles} (denoted as $S_1^x$ in that section).

Moreover, to simplify the construction of the mirror functor, we choose the collection in such a way that their formal deformation spaces do not have triple intersection near the critical locus of $W_{Y(\Lambda_0)}$.  (Namely for any three different objects in the collection, the common intersection of their formal deformation spaces and a fixed neighborhood of the critical locus is empty.)  We shall denote the glued space of the formal deformations of this collection by $Y_+ \subset Y(\Lambda)$. $(Y_+,W_{Y_+})$ will be the mirror for the Fukaya category supported in a compact neighborhood of the finite edges, while $(Y(\C),W_{Y(\C)})$ will be the mirror for the wrapped Fukaya category.

\subsection{Novikov-valued toric CY}
We start with a Laurent polynomial in two variables which gives a convex Newton polytope in $\R^2$.  It defines a punctured Riemann surface in $(\C^\times)^2$.  In a tropical limit the surface is approximated by a tropical curve, which forms a dual graph of the triangulation.  To take the tropical limit, we take the coefficients of the polynomial to be generic t-powers $t^{\nu}$ for $\nu \in \R$, and the tropical curve is given by the $|t|\to+\infty$-limit of the image of the complex curve under $\log_t|\cdot|: (\C^\times)^2 \to \R^2$.  We make a generic choice such that the tropical curve is a trivalent graph.  This fixes the tropical curve which is our starting data.

The dual of the tropical curve gives the fan of a toric CY.  Namely, the tropical curve (as a graph) is dual to a triangulation of the Newton polytope by standard integral triangles (the triangles with vertices in $\Z^2$ and with affine area $1/2$).  Embedding $\R^2$ as $\R^2 \times \{1\} \subset \R^3$ and taking cone over the triangulated polygon, we obtain a fan defining a toric CY.

$(\C^\times)^2$ is equipped with the standard symplectic form which is the differential of the one-form $r_1 d\theta_1 + r_2 d\theta_2$, where $z_j = t^{r_j + i\theta_j}$ are the standard coordinates.  Here $t$ is a fixed complex number with $|t|$ very large such that the curve is well approximated by the union of cylinders in $(\C^\times)^2$ corresponding to the edges and coamoebas in torus fibers of $\log_t: (\C^\times)^2 \to \R^2$ over the vertices of the tropical curve.  

We take $t$ to be real, so that the curve is invariant under complex conjugation, which is an anti-symplectic involution on $(\C^\times)^2$.  

The tropical limit induces a pair-of-pants decomposition of the complex curve.  Moreover, over each vertex of the tropical curve, we have a Seidel Lagrangian which is exact and invariant under the complex conjugation.  The Lagrangian is taken limit to the `Y-shape' such that the two triangles it bounds have zero area.  Then the disc potential for such a Lagrangian is given by $W=xyz$ where $x,y,z \in \Lambda_+$ are the immersed variables.

We want to glue the formal deformation spaces of Seidel Lagrangians over the vertices.  However, since these Lagrangians are disjoint, their formal deformation spaces (which are copies of $\Lambda_+^3$) are disjoint from each other.   In a later subsection, we shall take a collection of immersed Lagrangians, which are (actual) deformations of the Seidel Lagrangians (rather than just the Seidel Lagrangians over the vertices), such that their formal deformation spaces have common intersections via isomorphisms in the Fukaya category.

In this subsection, we enlarge the formal deformation space by allowing $\Lambda$ (instead of $\Lambda_+$) valued deformations.  We call them pseudo deformations.  

\begin{defn}
	Let $L$ be a spin oriented immersed Lagrangian, and $X_i$ its degree-one immersed generators.  Let $\nabla$ be a flat $\C^\times$ connection (with a prescribed gauge).  $(\nabla,b)$, where $b=\sum_i x_i X_i$, is called a formal deformation for $x_i \in \Lambda_+$.   We have the $A_\infty$ algebra $\left(\mathrm{CF}(L),\left\{m_k^{(\nabla,b)}\right\}_{k=0}^\infty\right)$.
	
	$(\nabla,b)$ is called a {\em pseudo deformation} if $x_i\in \Lambda$, $\nabla$ is a flat $\Lambda_0^\times$ connection, and $m_k^{(\nabla,b)}$ are convergent over $\Lambda$ for all $k$ and all inputs.  As before, it is called to be weakly unobstructed if $m_k^{(\nabla,b)}$ is a multiple of the unit.
\end{defn}
The enlarged spaces may have  intersections even when the formal deformations spaces do not overlap and hence can be glued together. 
We will use this to define an ambient space in which the various spaces we shall define live, namely $Y(\Lambda_{0}), Y(\C), Y_+ \subset Y(\Lambda)$.  As we mentioned, $\Lambda$-valued boundary deformations with negative valuations are not allowed in Fukaya category, and we will not use it in the construction of the mirror functor.
%  Nevertheless they give.)

For the case of a Seidel Lagrangian $L$ in a punctured Riemann surface, since $m_k^b$ only consists of finitely many terms for each $k$, it still converges for $x,y,z \in \Lambda$.  In other words, we can make sense of the $A_\infty$ category consists of the objects $(L,b)$ for $b=xX+yY+zZ$ where $x,y,z\in\Lambda$. The space of weakly unobstructed pseudo deformations is $\Lambda^3$. 

Then we glue these copies of $\Lambda^3$ over different vertices via chains of pseudo-isomorphisms.

\begin{defn}
	Let $\;\bL=(L,\nabla,b)$ and $\bL'=(L',\nabla',b')$ be Lagrangians with pseudo deformations.  $\alpha \in \CF^0(L,L';\Lambda)$
	is called a pseudo isomorphism from $\bL$ to $\bL'$ if the $A_\infty$ operations between $\bL$ and $\bL'$ converge over $\Lambda$, $m_1^{\bL,\bL'}(\alpha)=0$, and there exists $\beta \in \CF^0(\bL',\bL;\Lambda)$ with  $m_1^{\bL',\bL}(\beta)=0$  such that 
	$m_2(\alpha,\beta)$ and $m_2(\beta,\alpha)$ equals to identity up to image of $m_1$.
\end{defn}

In other words, we can make sense of the $A_\infty$ category consisting of the objects $\bL=(L,\nabla,b)$ and $\bL'=(L',\nabla',b')$, and a pseudo isomorphism is an isomorphism of objects in this category in the usual sense.  We warn that pseudo-isomorphisms do not compose well due to convergence issue.  Hence we use chains of pseudo-isomorphisms instead of a single one.

Now take a finite edge $\overline{v_1v_2}$ of the tropical curve, and consider the two Seidel Lagrangians $S_i$ over the vertices $v_i$ for $i=1,2$.  We first prove
\begin{prop} \label{prop:gluing_S12}
	Let $S_1$ and $S_2$ be Seidel Lagrangians over adjacent vertices.
	We have a chain of pseudo-isomorphisms between $(S_1,x_1X_1 + y_1Y_1+z_1Z_1)$ and $(S_2,x_2X_2 + y_2Y_2+z_2Z_2)$ with the coordinate change
	$$ x_2 = T^{-A} x_1^{-1}, y_2 = x_1^{-a_1+a_2+2} T^{A_y}y_1, z_2 = x_1^{a_1-a_2} T^{A_z}z_1$$
	where $A$ is the cylindrical area bounded between $S_1$ and $S_2$, $A_y,A_z \in \R$ are certain constants with $A_y + A_z = A$. Here $(x_i,y_i,z_i) \in \Lambda^3$ for $i=1,2$ and the coordinate change occurs when $x_1,x_2 \neq 0$.
\end{prop}
The cylindrical area $A$ bounded between $S_1$ and $S_2$ is given by the affine length of the edge (namely $(v_2-v_1) \cdot v$ where $v$ is the primitive vector in the direction $v_2-v_1$).
\begin{proof}
We will use a pair-of-circles $C$ in the neck region corresponding to the edge to have a chain of pseudo-isomorphisms and a gauge-change.  The pair-of-circles is a union of two isotopic circles (surrounding the neck) intersecting with each other at two points, and we assume that the two strips bounded by $C$ have area zero (in the limit).

$S_i$ and $C$ are disjoint from each other.  We need to connect them by chains of pseudo isomorphisms.  Deform $S_1$ to $S_1^x$ which intersects with $C$ as shown in Figure \ref{fig:cone}.  Here we take the areas of regions to be $0=k_3=k_6=k_1=k_5=k_7$ in the limit.  So we are only left with $k_2,k_4 > 0$.  Note that $k_2+k_4$ is the area of the cylinder bounded between $S_1^x$ and $C$ .Recall that $T$ in the expression $\nabla^{tT}$ stands for the two gauge points in $C$.

By Proposition \ref{prop:cocycle_4punc}, we have the pseudo-isomorphism   $\alpha=P_1+P_2$ between $(S_1^x,x_1X_1 + y_1Y_1+z_1Z_1)$ and $(C,\nabla^{tT}, y_Y+zZ) $ if
$x_1=tT^{-k_2-k_4}, y_1=yT^{k_2}, z_1=zT^{k_4}$ for $x_1,y_1,z_1,t,y,z \in \Lambda$.
Note that we allow some of the variables to have negative valuations. These define an actual isomorphism if
$x_1,y_1,z_1,y,z \in \Lambda_0$ and $t \in \Lambda_0^\times$. The former is the enlargement via pseudo-deformation.

Also $S_1$ and $S_1^x$ intersect at eight points and they are indeed isomorphic, see Figure \ref{fig:deform-Seidel} where $\alpha$ is an isomorphism and see Proposition \ref{prop:move-var}, where $A=0$ here.  Combining, we have fixed a chain of pseudo-isomorphisms between $(S_1,x_1X_1 + y_1Y_1+z_1Z_1)$ and $(C,\nabla^{tT}, yY+zZ)$ with the above change of coordinates.

Now consider gauge change which can also be understood as isomorphism, see Section \ref{sec:doublecircles}.  Namely we fix the isomorphism $t^{a_2} e_1 + t^{a_1} e_2$ from $(C,\nabla^{tT}, yY+zZ)$ to $(C,\nabla^{t'T'}, y'Y+z'Z)$ where $t=t', y = t^{a_1-a_2-2}y', z=t^{a_2-a_1}z'$.  $a_i$ for $i=1,2$ are the number of times the two gauge points pass through the immersed point $z$ in $C$. This coordinate change follows from the condition for $t^{a_2} e_1 + t^{a_1} e_2$ to be an isomorphism, and in particular closed under the corresponding Floer differential $d$ which counts Morse trajectories with suitable weights from holonomies. See Figure \ref{fig:a1a2cpxv} where the pair-of-circles represents $C$, and $e_i$ are the degree zero Morse critical points.

\begin{figure}[h]
	\begin{center}
		\includegraphics[scale=0.5]{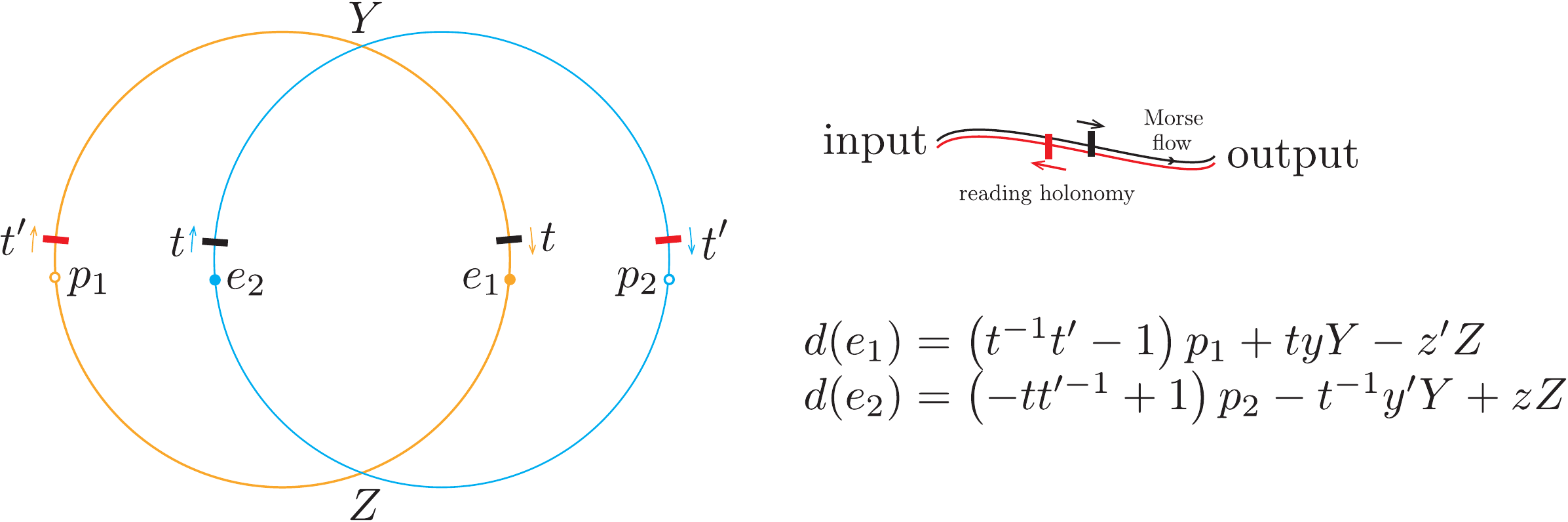}
		\caption{Isomorphism between $(C,\nabla^{tT}, yY+zZ)$ and $(C,\nabla^{t'T'}, y'Y+z'Z)$}
		\label{fig:a1a2cpxv}
	\end{center}
\end{figure}

The data of the tropical curve fixes the number $a_2-a_1$.  Namely we take this number from the relation between the three outward directions $\vec{x}_i,\vec{y}_i,\vec{z}_i$ at the adjacent vertices $v_i$: 
$$\vec{x}_2=-\vec{x}_1, \, \vec{y}_1 =\vec{y}_2 +(a_1-a_2-2) \vec{x}_1, \, \vec{z}_1 =\vec{z}_2 +(a_2-a_1)  \vec{x}_1.$$  We still have the freedom to choose the individual number $a_2 \in \Z$.  This is the number attached to each finite edge mentioned in the beginning of this section.  It does not affect the mirror space, but it will appear in the mirror functor.

$(S_2,x_2X_2 + y_2Y_2+z_2Z_2)$ and $(C,\nabla^{t'T'}, y'Y+z'Z)$ are related by a similar chain of pseudo-isomorphisms.   Combining, we obtained the proposition.
\end{proof}

Thus we have chosen a chain of pseudo-isomorphisms between $S_1$ and $S_2$ which gives the coordinate changes (up to Novikov factors) of the toric Calabi-Yau (for the two corresponding coordinate charts) given by the dual fan of the tropical curve.  Inductively we have fixed a chain of pseudo-isomorphisms between Seidel Lagrangians over any two different vertices.  (If the tropical curve has genus, we may have two chains of pseudo-isomorphisms between a pair of vertices; they are compatible in the sense that they give the same coordinate changes since we make the choice according to the tropical curve.)  Thus we obtain a $\Lambda$-valued toric Calabi-Yau $Y(\Lambda)$ (see Proposition \ref{propLCY}).

The constants $A_y, A_z, A$ for the finite edges may appear annoying.  In the following subsection we introduce exact variables, which make these constants much clearer.

\subsection{Exact variables}

We can keep track of the area terms systematically using the exactness of an immersed Lagrangian $S$.  Namely, the exact one-form $r_1 d\theta_1 + r_2 d\theta_2$ restricted to $S$ equals to $df$ for a function $f$ on the normalization of $S$.  Each odd immersed generator $X$ has its preimage being two points $X_+,X_-$ in the normalization of $S$ (where $X$ is turning from the branch of $X_-$ to that of $X_+$).  Recall that we have the exact variable $x_\ex := T^{-f(X_+) + f(X_-)} \cdot x $ from Equation \eqref{eq:exgeo}.
For a holomorphic polygon with corners $X^{(1)},\ldots,X^{(K)}$, the area $A$ is given by
$\sum_{i=1}^K \left(-f(X^{(i)}_+) + f(X^{(i)}_-)\right)$
by Stokes theorem.  Thus
$x^{(1)}_\ex \ldots x^{(K)}_\ex = T^A \cdot x^{(1)}\ldots x^{(K)}$.  In other words, all the Novikov factors have been absorbed to the exact variables and we no longer need to keep track of areas separately.

The exact variables for Seidel Lagrangians are easily obtained from the tropical data.

\begin{prop} \label{prop:exact-var}
	For the Seidel Lagrangian over a vertex $(a,b)\in \R^2$ of the tropical curve with the primitive edge directions (emanated from the vertex) $v_x,v_y,v_z\in \Z^2$, the exact variables and the immersed variables are related by
	$$ x_\ex = T^{-(v_x,(a,b))} \cdot x$$
	and similarly for $y_\ex$ and $z_\ex$.
\end{prop}

\begin{proof}
	The smoothing of the Seidel Lagrangian over $(a,b)$ at the odd immersed generator $X$ produces a union of two circles.  Consider one of the circles, which is homotopic to a circle in the fiber $T_{(a,b)}$ of $\log |\cdot|:(\C^\times)^2\to\R^2$ in the class $v_x\in H^1(T_{(a,b)})\cong \Z^2$.  The exact one form restricts as $a d\theta_1 + b d\theta_2$, and the integration over the circle gives $(v_x,(a,b))$.  Hence the difference of the exact function values in the two branches of $X$ is $(v_x,(a,b))$.
\end{proof}

\begin{example}
	Consider the Riemann surface given by $\{1+x+y+t^a/xy=0\} \subset (\C^\times)^2$ for $a\gg 0$.  Figure \ref{fig:exact-vari} shows the corresponding tropical curve, and the relation between exact and immersed variables.  Its mirror is the toric Calabi-Yau manifold $K_{\bP^2}$ with a superpotential.
\end{example}

\begin{figure}[h]
	\begin{center}
		\includegraphics[scale=0.8]{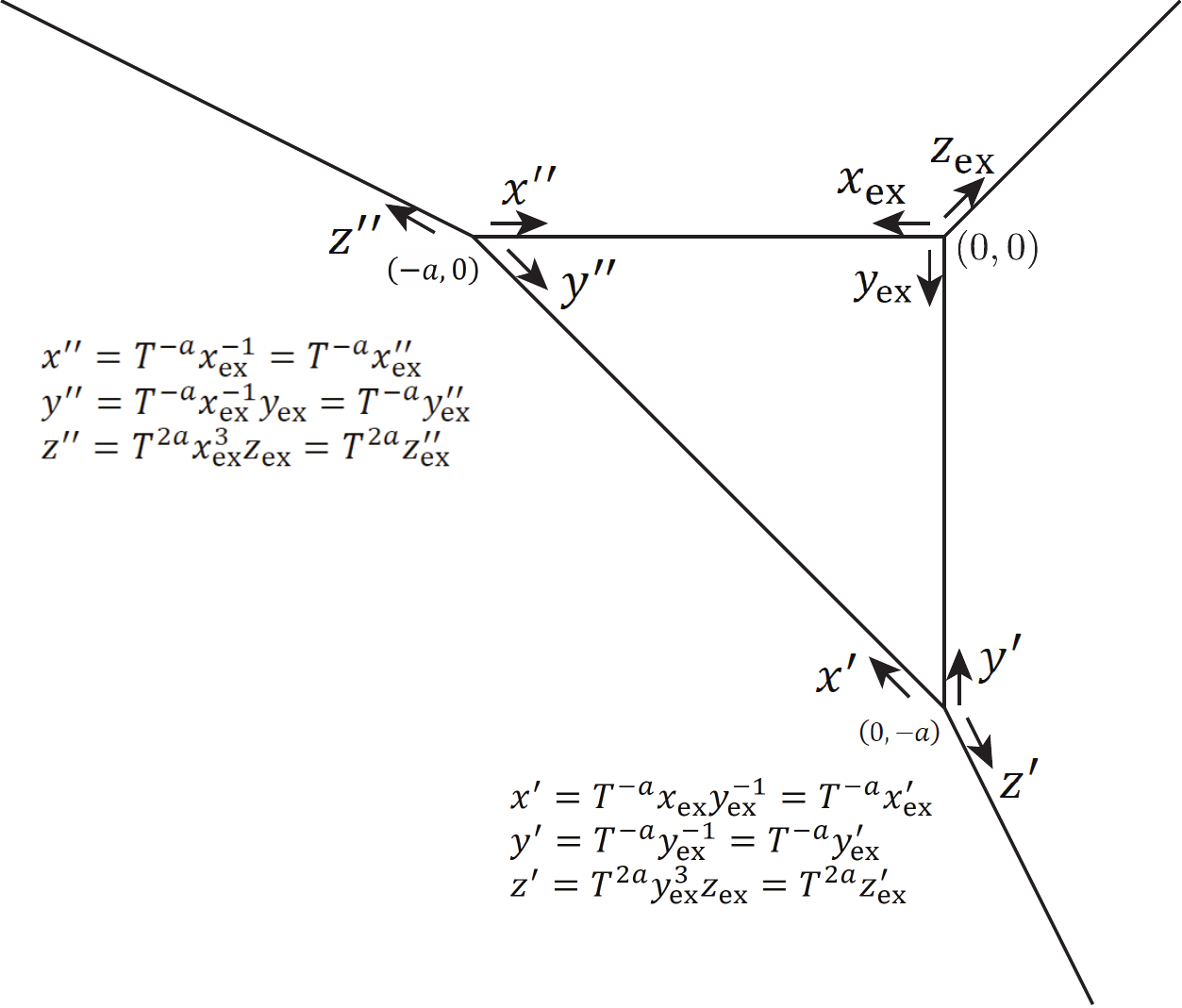}
		\caption{An example to illustrate the relation between exact and immersed variables.}
		\label{fig:exact-vari}
	\end{center}
\end{figure}

\begin{remark}
	Since $v_x+v_y+v_z=0$, the immersed Lagrangian over each vertex has the disc potential $W=xyz=x_\ex y_\ex z_\ex$ as expected.
\end{remark}

\begin{remark}
	A circle fiber over $(a,b)\in \R^2$ of the hypersurface is exact if and only if $(v,(a,b))=0$ where $v\in\Z^2$ is the direction of the fiber.  %For instance if the circle fiber is over $(0,0)$, then it is exact.
\end{remark}

Using exact generators and exact variables, all the Novikov factors are absorbed away.  Thus the coordinate change in Proposition \ref{prop:gluing_S12} simplifies to
$$ (x_2)_\ex = (x_1)_\ex^{-1}, (y_2)_\ex = (x_1)_\ex^{-a_1+a_2+2} (y_1)_\ex, (z_2)_\ex = (x_1)_\ex^{a_1-a_2}(z_1)_\ex.  $$
We conclude that
\begin{prop}\label{propLCY}
	The space glued from the chosen chains of pseudo-isomorphisms is the $\Lambda$-valued toric CY $Y(\Lambda)$ whose fan is given by the dual of the tropical curve.
\end{prop}

By restricting the exact variables $x_\ex,y_\ex,z_\ex$ for each Seidel Lagrangian to be $\Lambda_+$, $\Lambda_0$ or $\C$, we obtain $Y(\Lambda_+), Y(\Lambda_0), Y(\C)$ respectively.  The disc potential is simply $W=x_\ex y_\ex z_\ex$ in each chart, and so it is $\C$-valued on $Y(\C)$.  We have the Landau-Ginzburg model $W_{Y(\Lambda_0)}$.  Its critical locus is crucial for the study of HMS.

Note that $x = T^{(v_x,(a,b))} \cdot x_\ex$ may involve negative power of $T$.  Thus even if we restrict $x_\ex \in \Lambda_+$, $x$ can still have negative valuation.  So points in $Y(\Lambda_+)$ are still pseudo deformations (rather than formal deformations) of a Seidel Lagrangian over a vertex.

Let's fix a vertex $(a,b)$ and consider the exact variables of the Seidel Lagrangian $S_0$ over this vertex.  The valuation images of $(x_\ex,y_\ex,z_\ex) \in (\Lambda_+-\{0\})^3$ gives the cone $\R_+^3 \subset \R^3$.  ($\R^3$ can be interpreted as the valuation image of $(\Lambda^\times)^3$, namely the open $\Lambda^\times$ orbit of the pseudo deformation space.)  The valuation image of  the formal deformation space is a translated cone $\R_+^3 - \{((v_x,(a,b)),(v_y,(a,b)),(v_z,(a,b)))\}$.  

We can express the exact and immersed variables of Seidel Lagrangians over other vertices in terms of the exact variables of $S_0$.  Restricting their exact variables to have valuations in $\R_+$, the valuation images form cones in $\R^3$ (passing through the origin).  They form the fan picture of the toric CY.  The valuation image of the formal deformation spaces give disjoint translated cones of the fan.  See (2) of Figure \ref{fig:Cones-KP2} for an example.

Pseudo deformations give a convenient way to understand the mirror spaces.  However, for the purpose of HMS, we need to consider isomorphisms of formal deformations (in $\Lambda_+$ rather than in $\Lambda$) to ensure convergence of the Floer theory in the whole Fukaya category.  On the other hand, the formal deformation spaces of the Seidel Lagrangians over different vertices are disjoint from each other.  In the next subsection, we shall make a collection of immersed Lagrangians which intersecting properly so that there are isomorphisms between their formal deformations.

\subsection{A collection of Lagrangian immersions covering the critical locus}\label{sec:cr}

In this subsection we choose a collection of Seidel Lagrangians whose deformation spaces cover the critical locus.  Such a collection of Lagrangians plays the role of a Lagrangian fibration in the SYZ formulation.

The critical locus of the disc potential (over $\Lambda_0$ for the exact variables) defined over the mirror toric CY is a union of the subsets $\{(x_\ex,0,0): x_\ex \in \Lambda_0\}\cup \{(0,y_\ex,0): y_\ex \in \Lambda_0\} \cup \{(0,0,z_\ex): z_\ex \in \Lambda_0\}$ in the toric charts $\Lambda_0^3$.  (Recall that gluing charts in exact variables does not involve any Novikov parameters.
Also note that $\val(x=0)=+\infty$.)  

From now on, we assume the tropical curve has the property that the position $(a,b)$ of the vertex adjacent to each infinite edge in (outgoing) direction $v$ satisfies $v \cdot (a,b) \geq 0$.  This condition is easily satisfied if we take the tropical curve large enough such that all these vertices lie outside the Newton polygon (and the origin lies in the Newton polygon).

First of all, the non-compact part of the critical locus which corresponds to infinite edges (of the moment polytope of the toric CY) is covered as the following proposition shows.  Then it remains to cover the parts of the critical locus corresponding to finite edges.

Recall that given a Seidel Lagrangian $S$, we considered two different deformations of $S$, one in 
Proposition \ref{prop:lc}, which we denote by $S'$, and the other in  Proposition \ref{prop:move-var}, which we denote by $\tilde{S}$
from now on.
\begin{prop} \label{prop:inf-edge}
	Consider a non-compact one-strata of the toric Calabi-Yau, which corresponds to an infinite edge in the polytope picture.  Let $S$ be the Seidel Lagrangian at the vertex incident to this edge.  Let $z$ be the deformation direction corresponding to the infinite edge.
	Consider the deformed Lagrangian in $z$-direction $S', \tilde{S}$.  Then $\{x=y=0, \val(z_\ex)\geq 0\}$ (which is a part of the critical locus) is a subset of the formal deformation spaces of $S'$ and $\tilde{S}$.
\end{prop}

\begin{proof}
	Denote by $(a,b)$ the position of the vertex incident to the infinite edge (with outgoing direction $v$).  By Proposition \ref{prop:exact-var}, $z_\ex=T^{-(v,(a,b))} \cdot z$, and we have taken such that $(v,(a,b))>0$.  Hence $\{x=y=0, \val(z_\ex)\geq 0 \}$ is contained in the formal deformation space of $S$.  
	By Proposition \ref{prop:move-var}, the formal deformations of $\tilde{S}$ (or $S'$) and $S$ are related by $\tilde{z}=T^{A} z = T^{(v,(a,b)) + A} z_\ex$ (or $z'=T^{(v,(a,b)) + 2A} z_\ex$) where $A\geq 0$.  Hence $\{x=y=0, \val(z_\ex)\geq 0 \}$ is contained in the deformation space of $\tilde{S}$ (or $S'$).
\end{proof}

The following is a corollary of Proposition \ref{prop:move-var}.
\begin{cor} \label{cor:cone}
	Let $S_0$ be the Seidel Lagrangian over the vertex $(0,0)$ (in the tropical curve) and denote its variables by $\underline{x}_\ex,\underline{y}_\ex,\underline{z}_\ex$.
	Consider a Seidel Lagrangian $S$ over a vertex and denote its formal deformations by $(x,y,z)$ (where $\val(x),\val(y),\val(z)\geq 0$).  The valuation image of the formal deformation space of $S$ forms an  affine cone $C = p + \R_{\geq 0}\{X,Y,Z\} \subset \R^3_{\underline{x}_\ex,\underline{y}_\ex,\underline{z}_\ex}$, where $X,Y,Z$ are primitive vectors dual to $x,y,z$ (which are vectors in $(\R^3_{\underline{x}_\ex,\underline{y}_\ex,\underline{z}_\ex})^\vee$).  Denote by $\bar{C}\subset \R^2$ the image cone under projection along the $Z$-direction.  
	
	If $S$ is deformed in $x$-direction (or $y$-direction) to $\tilde{S}$, then $\bar{C}$ is translated in direction of $-X+Y$ (or $-Y+X$ resp.).  If $S$ is deformed in $z$-direction to $\tilde{S}$, then $\bar{C}$ is translated in direction of $X+Y$. 
	
	Similarly, 	if $S$ is deformed in $x$-direction (or $y$-direction) to $S'$, then $\bar{C}$ is translated in direction of $-2X+Y$ (or $-2Y+X$ resp.).  If $S$ is deformed in $z$-direction to $S'$, then $\bar{C}$ is translated in direction of $X+Y$. 
\end{cor}

Proposition \ref{prop:exact-var} gives the following corollary.

\begin{cor} \label{cor:cone-pos}
	Assume the same notations as in Corollary \ref{cor:cone}.  Let $S_1,\ldots,S_k$ be Seidel Lagrangians at vertices adjacent to the same face (which is possibly unbounded) in the tropical curve.  Denote by $z^{(i)}$ the immersed variables of $S_i$ corresponding to the edges which are not adjacent to the face.  (See for instance Figure \ref{fig:exact-vari}.)  Consider the cones $\bar{C}_i$ correpsonding to $S_i$ in projection along the $Z^{(i)}$ direction.
	
	Then the vertices of the cones $\bar{C}_i$ are given by taking negative of the corresponding vertices of the tropical curve (up to translation).
\end{cor}

In the following we use the deformation $\tilde{S}$.  The same construction can be done using $S'$ and we shall not repeat.

\begin{example} \label{ex:charts-KP2}
	Consider the pair-of-pants decomposition that is mirror to $K_{\bP^2}$.  See Figure \ref{fig:exact-vari}.  We have three disjoint Seidel Lagrangians $S_0,S_1,S_2$ sitting over the three vertices $(0,0), (0,a), (a,0)$ respectively.  Their formal deformation spaces give three disjoint copies of $\Lambda_0^3$.  The valuation images of the formal deformation spaces (in terms of the deformation parameters of $S_0$) are given by cones $C_i \cong \R_{\geq 0}^3$ for $i=0,1,2$, whose positions are given by Corollary \ref{cor:cone-pos} as shown in (2) of Figure \ref{fig:Cones-KP2} near the toric divisor $z=0$.  To get the planar figure, we have taken the projection along the $z$-direction, and the cones $C_i$ project to $\bar{C}_i \subset \R^2$.
	
	The $\Lambda_{0}$-valued critical locus in local charts $(x^{(i)}_\ex,y^{(i)}_\ex,y^{(i)}_\ex)$ are given by $$\{(x^{(i)}_\ex,0,0): \val (x^{(i)}_\ex) \geq 0\}\cup \{(0,y^{(i)}_\ex,0): \val (y^{(i)}_\ex) \geq 0\}\cup\{(0,0,z^{(i)}_\ex): \val (z^{(i)}_\ex) \geq 0\}.$$
	
	By Proposition \ref{prop:inf-edge}, the infinite edges of the $\Lambda_{0}$-valued critical locus are covered by the formal deformation spaces of $S_0,S_1,S_2$.  It remains to cover the part of the critical locus contained in the toric divisor $\{z=0\}\cong \bP^2$, which are the toric divisors of $\{z=0\}$.  They are located at infinity in (2) of Figure \ref{fig:Cones-KP2}.
	
	Now deform $S_1$ in the $y'$ direction to $\tilde{S}_1$ so that the formal deformation space of $\tilde{S}_1$ intersect with that of $S_0$.  By Corollary \ref{cor:cone}, $\bar{C}_1$ moves to $\bar{C}_1 + (2,1) h$, where $h$ is taken such that $h-A>0$.  Similarly, deform $S_2$ (and $S_0$) in the $y''$ direction (and $x$ direction resp.) to $\tilde{S}_2$ ($\tilde{S}_0$ resp.) such that $\bar{C}_2$ shifts to $\bar{C}_2 - (1,2)k$ for suitably chosen $k>0$ which intersects with $\bar{C}_1$ but not  $\bar{C}_1 + (2,1) h$ (resp. $\bar{C}_0$ shifts to $\bar{C}_0 + (-1,1) p$ for some suitably chosen $p>0$ which intersects with $\bar{C}_2$ but not $\bar{C}_2 - (1,2)k$.).
	See (3) of Figure \ref{fig:Cones-KP2}.
	
	The formal deformation spaces of $\tilde{S}_0,S_0,\tilde{S}_1,S_1,\tilde{S}_2,S_2$ cover a neighborhood of the whole critical locus.  Also at most two of them intersect with each other by construction.
\end{example}

\begin{figure}[h]
	\begin{center}
		\includegraphics[scale=0.5]{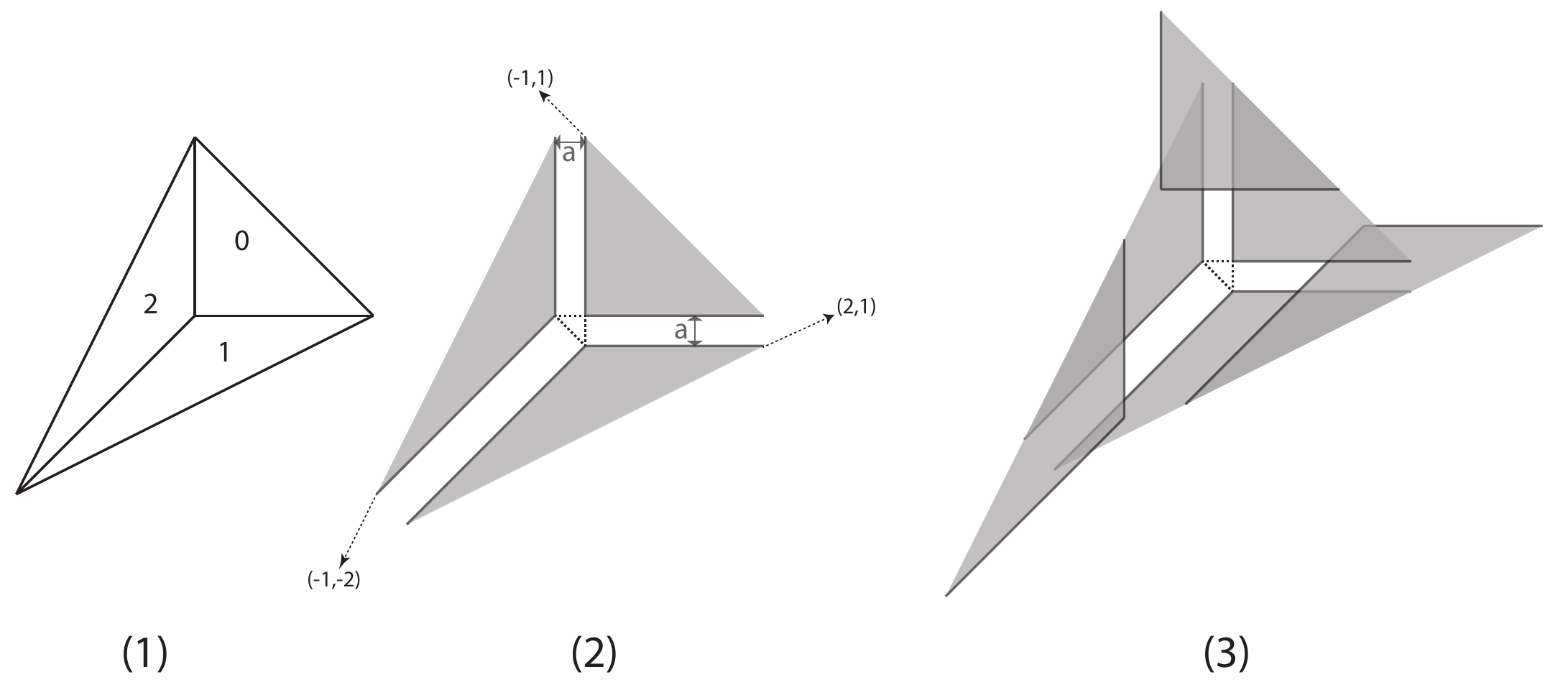}
		\caption{Deformation spaces covering the critical locus in $K_{\bP^2}$.}
		\label{fig:Cones-KP2}
	\end{center}
\end{figure}

\begin{example}
	Consider the pair-of-pants decomposition shown in Figure \ref{fig:toricCYeg}, which is mirror to the toric Calabi-Yau whose fan is shown in (1) of Figure \ref{fig:Cone-toricCYeg}.  Let's consider the formal deformation spaces of $S_i$ (sitting inside the mirror toric Calabi-Yau) near the divisor $y=0$.  We only need to consider $S_0,S_2,S_3,S_4$, since we can take a neighborhood of $y=0$ which is disjoint from the formal deformation space of $S_1$.
	
	Using Corollary \ref{cor:cone-pos}, the valuation images of $S_i$ for $i=0,2,3,4$, projected along $y^{(i)}$-direction, are shown in (2) of Figure \ref{fig:Cone-toricCYeg}.  Note that they are disjoint from each other.
	
	First consider $S_0,S_1,S_2$ whose formal deformation spaces are disjoint and lie in the neighborhood of the divisor $z=0$.  It is similar to Example \ref{ex:charts-KP2}.  We deform $S_2$ towards $S_1$ to get $\tilde{S}_2$, deform $S_0$ towards $\tilde{S}_2$ to get $\tilde{S}_0$, and deform $S_1$ towards $\tilde{S}_0$ to get $\tilde{S}_1$.
	 Thus we have the Lagrangians $S_1,\tilde{S}_2,\tilde{S}_0,\tilde{S}_1$ whose formal deformation spaces intersect with each other and cover the part of the critical locus lying in $\{z=0\}$.  See (3) of Figure \ref{fig:Cone-toricCYeg} for their valuation images projected along $y^{(i)}$-direction.
	
	We still need to deform $S_3$ and $S_4$ in order to cover the whole critical locus.  We deform $S_4$ in the $x^{(4)}$ direction to make $\tilde{S}_4$ intersect with $S_3$, deform $S_3$ in the $z^{(3)}$ direction to make $\tilde{S}_3$ intersect with $\tilde{S}_2$, and also deform $\tilde{S}_0$ in the $z^{(3)}$ direction to make $\tilde{\tilde{S}}_0$ intersect with $S_4$.  The movement of the cones are given by  Corollary \ref{cor:cone}.  We choose the deformations such that $\tilde{S}_0, \tilde{S}_2, \tilde{S}_3, S_3, \tilde{S}_4, S_4, \tilde{\tilde{S}}_0$ cover the part of critical locus contained in the divisor $y=0$, and such that at most two of their formal deformation spaces intersect.  See (4) of Figure \ref{fig:Cone-toricCYeg}.
\end{example}

The same procedure works in general and gives the following.

\begin{figure}[h]
	\begin{center}
		\includegraphics[scale=0.5]{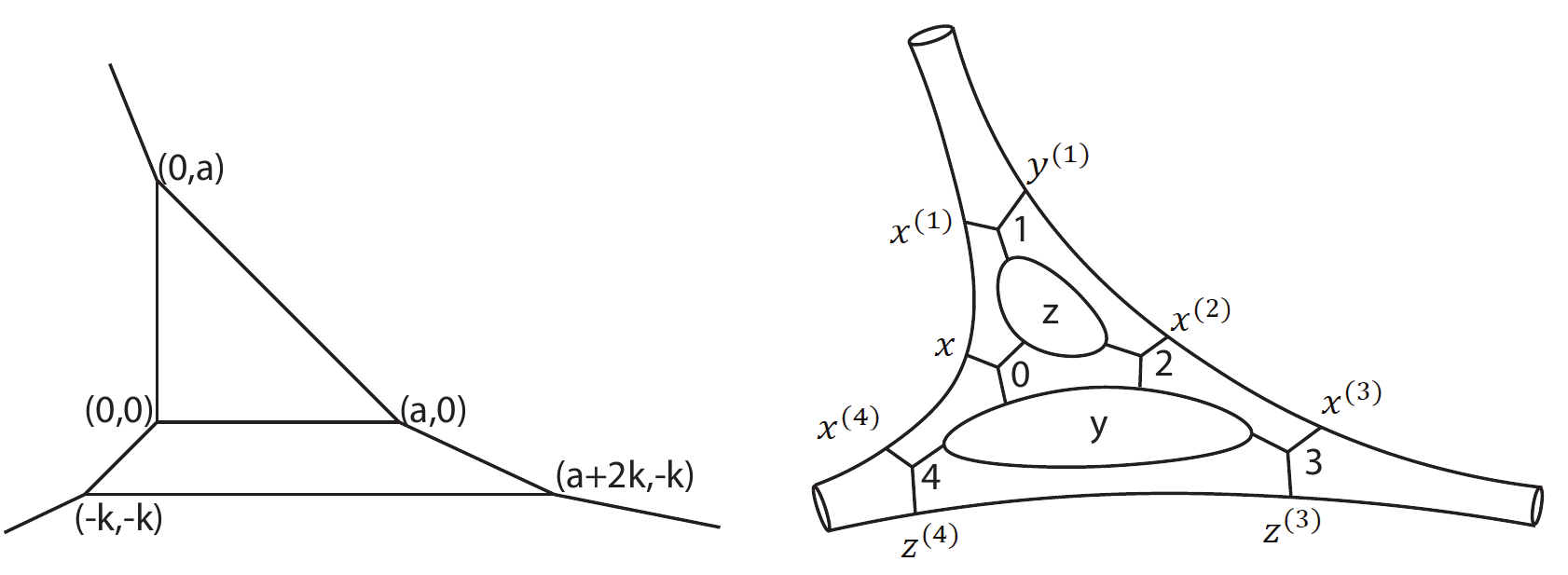}
		\caption{A more complicated example of a pair-of-pants decomposition.  The Seidel Lagrangians sitting over vertices are doubling of the $Y$-shapes (showing only the front).}
		\label{fig:toricCYeg}
	\end{center}
\end{figure}

\begin{figure}[h]
	\begin{center}
		\includegraphics[scale=0.5]{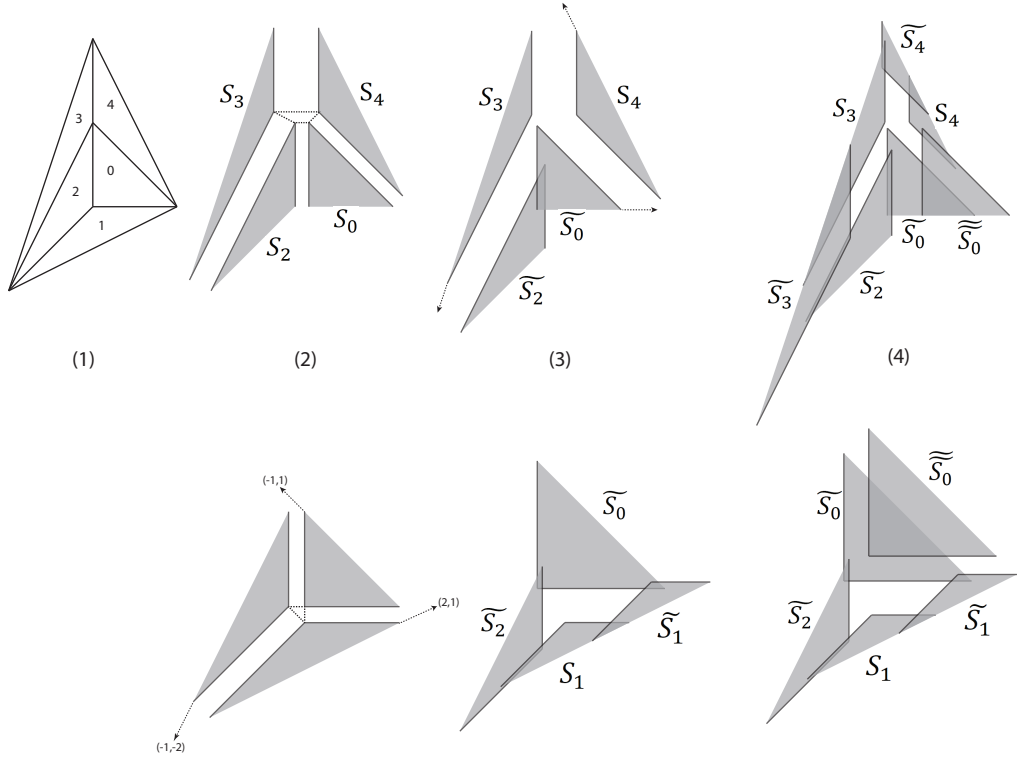}
		\caption{Charts for a more complicated example.}
		\label{fig:Cone-toricCYeg}
	\end{center}
\end{figure}

\begin{prop} \label{prop:Lag-coll}
	There exists a collection of Seidel Lagrangians whose formal deformation spaces cover the whole $\Lambda_0$-valued critical locus of $W$, and at most two intersect when restricted to a neighborhood of the critical locus.  
\end{prop}

\begin{proof}
	We construct the collection of Lagrangians inductively on the ray generators of the fan of the mirror toric Calabi-Yau.  First, we start with the collection of Seidel Lagrangians over the vertices which are adjacent to at least one infinite edge in the tropical curve.  Their formal deformation spaces cover the infinite edges of the critical locus.  We will expand our collection step-by-step to cover other parts of the critical locus.
	
	At this stage, for each maximal cone of the fan there is at most one corresponding Seidel Lagrangian included in our collection.  In general, in each step we make sure the following: if there are two Lagrangians in our collection corresponding to the same maximal cone (in particular one is the deformation of the other), the part of the critical locus corresponding to the three facets of the maximal cone has already been covered (by our collection).  This helps to avoid including three Seidel Lagrangians corresponding to the same maximal cone in our collection in the next step (which would violate the requirement that at most two charts intersect with each other).
	
	Now consider a ray $r$ of the fan which is contained in the interior of the support of the fan, with the property that at least one of its adjacent maximal cones corresponds to
	a Seidel Lagrangian already included in our collection.   (For instance a Seidel Lagrangian corresponding to a maximal cone with one of its facets lying in the boundary of the fan polytope is included in our collection in the very first step.)  Denote the maximal cones adjacent to $r$ in clockwise order by $C_1,\ldots,C_k$.  At least one of these maximal cones correspond to deformed Seidel Lagrangians already in our collection. 
	
	Let us denote the (deformed) Seidel Lagrangians which is already in the collection by $S_{i_1},\ldots,S_{i_p}$
	and the corresponding  valuations of formal deformation spaces by $C_{i_1}+v_1,\ldots,C_{i_p}+v_p$ where $i_j < i_{j+1}$ and $v_j \in \R^3$.  
	 For each $j$, we pick deformed Seidel Lagrangians $S_{i_j+1},\ldots,S_{i_{j+1}-1}$ corresponding to the cones $C_{i_j+1},\ldots,C_{i_{j+1}-1}$ such that the deformation space of $S_{i_j+l}$ intersects with that of $S_{i_j+l-1}$ for each $l=1,\ldots,i_{j+1}-i_j-1$.  Then we also deform $S_{i_{j+1}}$ to $\tilde{S}_{i_{j+1}}$ such that the formal deformation space of $\tilde{S}_{i_{j+1}}$ intersect with that of $S_{i_{j+1}-1}$.  As a result, the formal deformation spaces of $\{\tilde{S}_{i_j},S_{i_j},S_{i_j+1},\ldots,S_{i_{j+1}-1}:j=1,\ldots,p\}$ cover the part of the critical locus contained in the divisor corresponding to the ray $r$.  We add these Lagrangians into our collection.
	
	In the above step, the cones that correspond to more than one Lagrangians in our collection are $C_{i_j}:j=1,\ldots,p$.  For each of these cones, the facet which is not adjacent to $r$ is contained in the boundary of the fan polytope, whose corresponding part of the critical locus has been covered.  The two facets which are adjacent to $r$ are also covered by the above step.  Thus we have made sure that all the three facets of the maximal cone have been covered.
	
	Then we proceed inductively to the next ray $r'$ contained in the interior of the support of the fan, with at least one of its adjacent maximal cones corresponds to a Seidel Lagrangian already included in our collection.  In the previous steps we have made sure any maximal cone that has two deformed Seidel Lagrangians corresponding to it in our collection has all its three facets already covered.  Thus those Seidel Lagrangians that we need to deform in this step originally have at most one copy in our collection.  This ensures there are still at most two Lagrangians corresponding to the same maximal cone in our expanded collection after this step.  Also for each cone that corresponds to two Lagrangians in the expanded collection after this step, the facet which is not adjacent to $r'$ is already covered in previous steps (since the cone is already included in our collection before this step).  And the two facets which are adjacent to $r'$ are also covered in this step.  So such a cone has all its three facets covered.  %(Moreover for the deformations we need to deform in large enough extent such that the formal deformation spaces of the deformed Lagrangians do not intersect with those of others for other rays, see (4) of Figure \ref{fig:Cone-toricCYeg}.)
	
	In this way the part of critical locus contained in any compact divisor has been covered.  Then we proceed to rays that lie in the boundary of the fan polytope (but which are not extremal).  Proceeding in the same procedure, we obtain the required collection of Lagrangians whose formal deformation spaces cover the whole critical locus, and at most two of them intersect with each other near the critical locus.
\end{proof}
The figure \ref{fig:Cone-toricCYeg} illustrates the induction where we start with the ray corresponding to the lower interior vertex,
and followed by the ray corresponding to the upper interior vertex.

We denote by $Y_+ \subset Y(\Lambda)$ the space glued from the formal deformation spaces of our collection of Lagrangians.  $Y_+$ contains the critical locus of $W_{Y(\Lambda_0)}$.

$Y(\Lambda)$ naturally comes with a sheaf of rigid analytic functions.  The structure sheaf of $Y(\Lambda)$ restricts to give the structure sheaves of $Y(\Lambda_+)$,$Y(\Lambda_0)$ and $Y_+$.  Recall that in the work of Fukaya-Oh-Ohta-Ono on mirror symmetry for toric varieties \cite{FOOO-T3}, they defined the completion of Laurent polynomial ring using the moment map polytope.  We define the sheaf of analytic functions in analogous to their formulation.  The main difference is that our variables take value in $\Lambda$ whose valuation can be $+\infty$, while the variables in the toric case \cite{FOOO-T3} take value in $\Lambda^\times$ which have valuation in $\R$.

$\Lambda[x_\ex,y_\ex,z_\ex]$ can be regarded as the ring of Novikov-convergent functions on the space $\Lambda^3$.  
For an open subset $U=\val^{-1}(U') \subset \Lambda^3$ where $U'$ is open in $(-\infty,+\infty]^3$, let $M_{U'}$ be the monoid of Laurent monomials $x_\ex^py_\ex^qz_\ex^r$ for $p,q,r \in \Z$, whose valuations are well-defined and not equal to $-\infty$.  (It means if $U'$ contains any point of the form $(+\infty,b,c)$, then $p\geq 0$, and similar for $q,r$.)  $\mathscr{O}(U)$ is defined as the completion of $\Lambda[M_{U'}]$ with respect to the norm $e^{-\val_{U'}}$, where
$$\val_{U'} = \inf_{u\in U'} \val_u, \, \val_u(x_\ex)=u_1, \, \val_u(y_\ex)=u_2, \, \val_u(z_\ex)=u_3$$
and $u=(u_1,u_2,u_3)$.
For each $f\in \Lambda[x_\ex,y_\ex,z_\ex]$, $\val_u(f)$ is continuous on $u \in \Lambda^3$. 

%In particular we have the ring $\mathscr{O}(\Lambda_+^3)$.
It is obvious that for $V' \subset U'$, $\Lambda[M_{U'}] \subset \Lambda[M_{V'}]$; for $f \in \Lambda[M_{U'}]$, 
$\val_{V'}(f) \geq \val_{U'}(f)$ and so $e^{-\val_{V'}(f)} \leq e^{-\val_{U'}(f)}$.  Thus there is a canonical restriction homomorphism $\mathscr{O}(U) \to \mathscr{O}(V)$ (where $V=\val^{-1}(V')$).

$Y(\Lambda)$ is glued from charts which are copies of $\Lambda^3$, and the overlapping regions between two charts are $\Lambda^{3-i} \times (\Lambda^\times)^{i} \subset \Lambda^3$ for some $i=1,2,3$.  Open subsets in $Y(\Lambda)$ are generated by $U=\val^{-1}(U')$ in the charts where $U' \subset (-\infty,\infty]^3$ are relatively compact.  Such an open set $U$ is said to be relatively compact.  (Since the gluing maps on $\left(\val_u(x_\ex),\val_u(y_\ex),\val_u(z_\ex)\right)$ are linear and hence continuous, relative compactness is preserved.)

Suppose $U'$ is contained in the intersection of (the valuation images of) two charts.
The toric gluing identifies a Laurent monomial $T^A (x^{(1)}_\ex)^{i_1}(y^{(1)}_\ex)^{j_1}(z^{(1)}_\ex)^{k_1}$ in a chart to a Laurent monomial $T^A (x^{(2)}_\ex)^{i_2}(y^{(2)}_\ex)^{j_2}(z^{(2)}_\ex)^{k_2}$ in the other chart via the change of coordinates.  This identifies the monoids $M_{U'}^{(i)}$ for the two charts.  Thus we have the ring of regular functions over $U$, which is $\Lambda[M_{U'}^{(i)}]$ in coordinates.  ($U$ is the valuation preimage of $U'$ in the first chart which is identified with that of the second chart under the gluing.)  The gluing identifies $u^{(i)} \in (-\infty,+\infty]^3$ for the two charts.  By definition $\val_{u^{(1)}}(T^A (x^{(1)}_\ex)^{i_1}(y^{(1)}_\ex)^{j_1}(z^{(1)}_\ex)^{k_1})=\val_{u^{(2)}}(T^A (x^{(2)}_\ex)^{i_2}(y^{(2)}_\ex)^{j_2}(z^{(2)}_\ex)^{k_2})$.  Hence the metrics from the two charts are the same, and so we have a well-defined completion $\mathscr{O}(U)$ for the ring of regular functions.  Thus the above defines the structure sheaf $\mathscr{O}$ of $Y(\Lambda)$. \footnote{For general non-toric varieties, the coordinate changes are polynomials or even series living in the completion.  One has to directly deal with the completion in that case.}

By Theorem \ref{thm:maingluealg}, we have a functor from the (wrapped) Fukaya category of the punctured Riemann surface to the category glued from local matrix factorizations of $W_{Y_+}$.  In the next section, we shall show that the derived functor on $W\Fuk$ is an equivalence to its image, which is the category of $\C$-valued matrix factorizations of $W_{Y(\C)}$.

\section{Explicit computations for the mirror functor of punctured Riemann surfaces} \label{sec:HMS-surf}

Homological mirror symmetry for punctured two-spheres was proved by \cite{AAEKO}, and that for punctured Riemann surfaces was proved by \cite{HLee}.  In their works, they found and matched generators of the categories on the two sides, and proved that the morphism spaces are isomorphic.  

In this section, we compare our functor obtained from the gluing construction with the results of \cite{AAEKO,HLee}.   We make explicit computations of our functor in object and morphism levels for the generating set of the wrapped Fukaya category.  The result is as follows.

\begin{thm} \label{thm:surf}
	Let $X$ be a punctured Riemann surface associated to a tropical curve (with finitely many edges) in $\R^2$.  
	Let $\{L_i: i \in I\}$ be the collection of Lagrangian immersions constructed in Proposition \ref{prop:Lag-coll}, and we have fixed isomorphisms between their formal deformations by assigning an integer to every finite edge.  Denote by $Y(\C) \subset Y(\Lambda_0) \subset Y(\Lambda)$ the $\C$-valued (and $\Lambda_0$-valued, $\Lambda$-valued  respectively) toric CY dual to the tropical curve, and $Y_+ \subset Y(\Lambda)$ the glued space from $\{L_i: i \in I\}$.  Let $W_{Y_+}$ be the disc potential of $\{L_i: i \in I\}$.
	
	We have a functor $\mathrm{WFuk}(X) \to \MF_{\mathrm{glued}}(W_{Y_+})$.  The derived functor on $D\mathrm{WFuk}(X)$ is a quasi-equivalence to its image which is  $D\MF(W_{Y(\C)}) \subset D\MF_{\mathrm{glued}}(W_{Y_+})$.
\end{thm}

The integer assigned to each edge will be denoted by $a_1^e$ or simply $a_1$ throughout this section.  We have another integer $d^e$ associated to each finite edge of a tropical curve: the tropical curve near the finite edge gives the toric diagram of $\mathcal{O}(-d^e-2,d^e)$.  Then we define $a_2^e := a_1^e + d^e$.  The two numbers $a_1^e$ and $a_2^e$ can be understood as a choice of gauge change on a pair-of-circles in the cylinder corresponding to $e$.  Namely $a_i$ for $i=1,2$ are the numbers of times that the two gauge points pass through the immersed point $z$.  See also Figure \ref{fig:gluing-obj-gauge}.

Note that $W_{Y_+}$ can be extended as $W_{Y(\Lambda)}$ over $Y(\Lambda)$ and is $\C$-valued when restricted as $W_{Y(\C)}$ on $Y(\C)$.  
Also recall that $Y_+$ contains the critical locus of $W$ over $Y(\Lambda_0)$ (even though $Y_+$ does not contain the whole $Y(\Lambda_0)$ in general).

In this section, many computations are done by a direct count of holomorphic polygons in the punctured surface, which will be illustrated by figures.  The counts of holomorphic polygons involved are rather direct and simple, since the Seidel Lagrangians in the pair-of-pants are parts of the boundary conditions, and they help to localize the polygons in a small region (which is either a pair-of-pants or a four-punctured sphere).	 We remark that bigger polygons may contribute to the higher part of the mirror functor, but we do not need to count them for the proof.

\subsection{Objects}
First consider the mirror functor in the object level.  The tropical curve in $\R^2$ produces a planar diagram which has a number of faces, edges, and vertices.
Fix a (possibly non-compact) face $f$, and let $L$ be a (possibly non-compact) oriented Lagrangian path circulating around $f$ and winding about the cylinders corresponding to finite edges adjacent to $f$.  It is an object used in \cite{HLee}.

The easiest case is that $f$ has no adjacent finite edge (implying that it has only two adjacent infinite edges).  Let $v$ be the only vertex adjacent to $f$.  By the construction in the last section, the Seidel Lagrangian $S$ over $v$ is in our collection.  It has three immersed variables $x,y,z$, and the disc potential is $W=xyz$ in its formal deformation space.  Let $z$ be the immersed variable corresponding to the face $f$ (so the divisor $D_f$ corresponding to $f$ in the mirror is defined by $z=0$).

$L$ intersects $S$ at two points $A,B$.
The image matrix factorization of $L$ over the formal deformation space of $S$ is $(\mathrm{Span}_R\{A,B\},\delta)$ where $\delta$ is given by $A \mapsto z \cdot B, B \mapsto xy \cdot A$, $R = \Lambda_0[x,y,z]$. The cokernel is given by $\mathrm{Span}_R\{B\} / \langle z \cdot B \rangle$.  It simply reduces to a computation in the pair-of-pants, see Figure \ref{fig:MF-simplest}.

\begin{figure}[h]
	\begin{center}
		\includegraphics[scale=0.5]{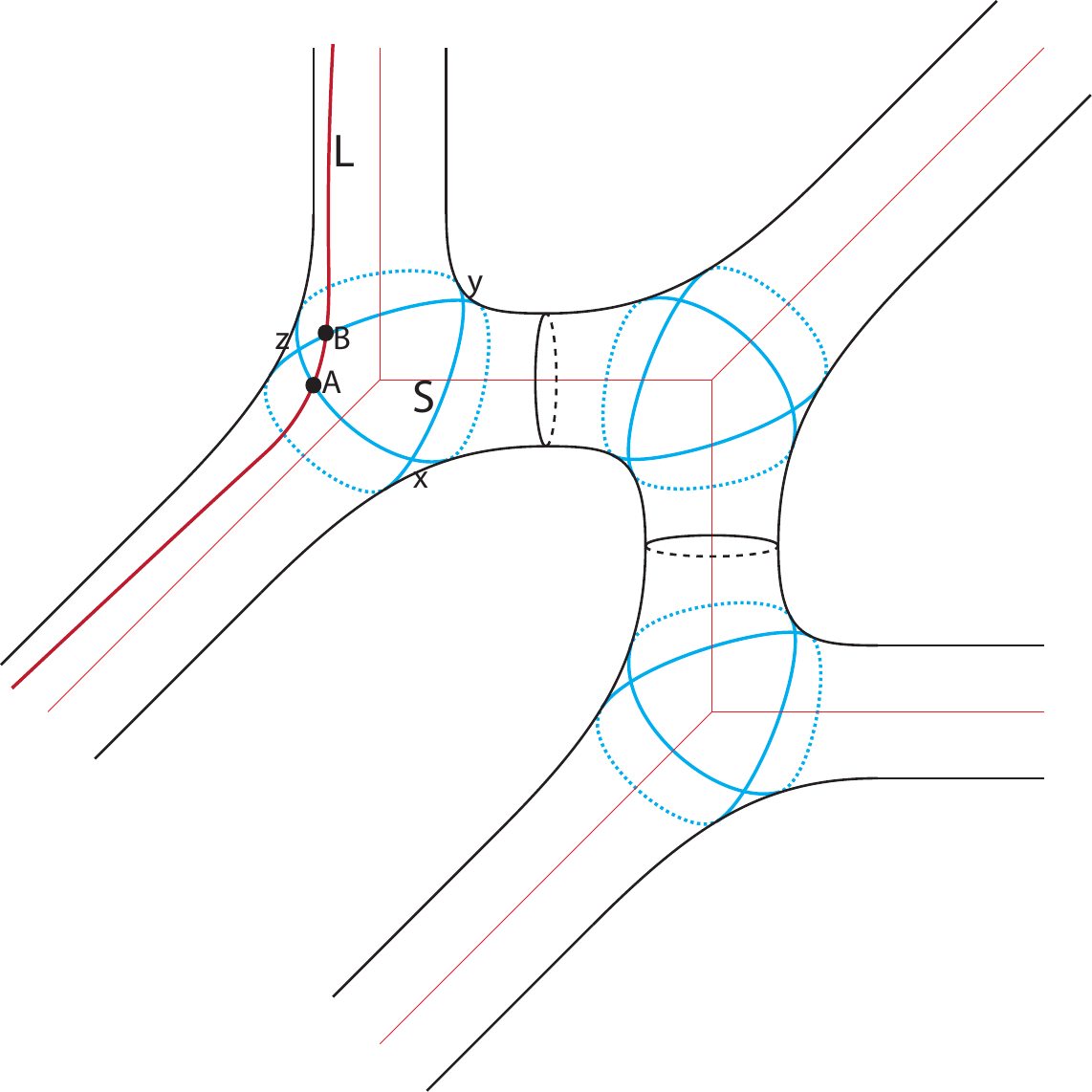}
		\caption{The easiest case.  $L$ divides the triangle $zxy$ into two parts, which corresponds to the factorization $z$ times $xy$.  The image of $L$ is $\mathscr{O}_{D_f}$ in $D\mathrm{Sing}(W_{Y(\C)})$.}
		\label{fig:MF-simplest}
	\end{center}
\end{figure}

For any object $S'$ other than $S$ in our collection, the image of $L$ transformed by $S'$ is trivial (namely it has trivial cokernel in $D\mathrm{Sing}$, since the strips bounded by $S'$ and $L$ either has no corner or involves all the three corners $x,y,z$ once).  Thus the image object of $L$ is supported in $D_f$, and indeed is simply the push forward of $\mathscr{O}_{D_f}$ in $D\mathrm{Sing}(W_{Y(\C)})$.

From now on we consider the case that $f$ has at least one adjacent finite edge.
Fix a finite edge $e=\overline{v_1v_2}$ adjacent to $f$, and suppose $L$ winds $m$ times around the cylindrical part corresponding to $e$.  

Consider the two pair-of-pants corresponding to the two adjacent vertices $v_1, v_2$.  The union of these two pair-of-pants gives a four-punctured sphere.  We have explained the gluing for two pair-of-pants in Section \ref{sec:glue2s}.

In our collection of Lagrangian immersions we have two objects $S_1,S_2$ in the two pair-of-pants, whose formal deformation spaces cover the part of the ($\Lambda_{0}$-valued) critical locus corresponding to the finite edge.  The deformation variables of $S_i$ are denoted by $x_i,y_i,z_i \in \Lambda_+$ for $i=1,2$, where 
deformation in the edge $e$-direction corresponds to the variables $x_i$.  Also the notations are chosen such that $z_i$ are the variables corresponding to the given face $f$.  (Namely the divisor $D_f$ in the mirror corresponding to $f$ is given by $z_i=0$. )

In the following we assume that $S_1$ is the deformation along the $x_1$-direction of the Seidel Lagrangian over the vertex $v_1$ (see Figure \ref{fig:Lag-winding}), and $S_2$ is exactly the Seidel Lagrangian over the vertex $v_2$.   $S_2$ bounds a triangle of area $0$ (in the limit).  For simplicity assume
$S_1$ bounds a triangle of area $A$, where $A$ is the area of the cylinder bounded by the Seidel Lagrangians over $v_1,v_2$.  Under this area condition we have $x_1 = x_2^{-1}$.  Note that no Novikov parameter appears in this equality.  In general, $S_1$ and $S_2$ can be some other different deformations of the Seidel Lagrangians along the $x_1$-direction.  The computations and results will be the same.

$L$ is taken such that it intersects $S_2$ at two points, denoted by $A$ and $B$, while it intersects $S_1$ at $2+2m$ points, denoted by $C_i$ and $D_i$ for $i=0,\ldots,m$.  See Figure \ref{fig:Lag-winding}.  $A$ and $C_i$ are odd morphisms, while $B$ and $D_i$ are even.  The intersection points $C_0, D_0$ and $A,B$ occur very close to $z_1$ and $z_2$ respectively (so that the areas of the triangles with corners $(z_1, C_0, D_0)$ and $(z_2, A,B)$ can be taken to be zero in the limit).

\begin{figure}[h]
	\begin{center}
		\includegraphics[scale=0.5]{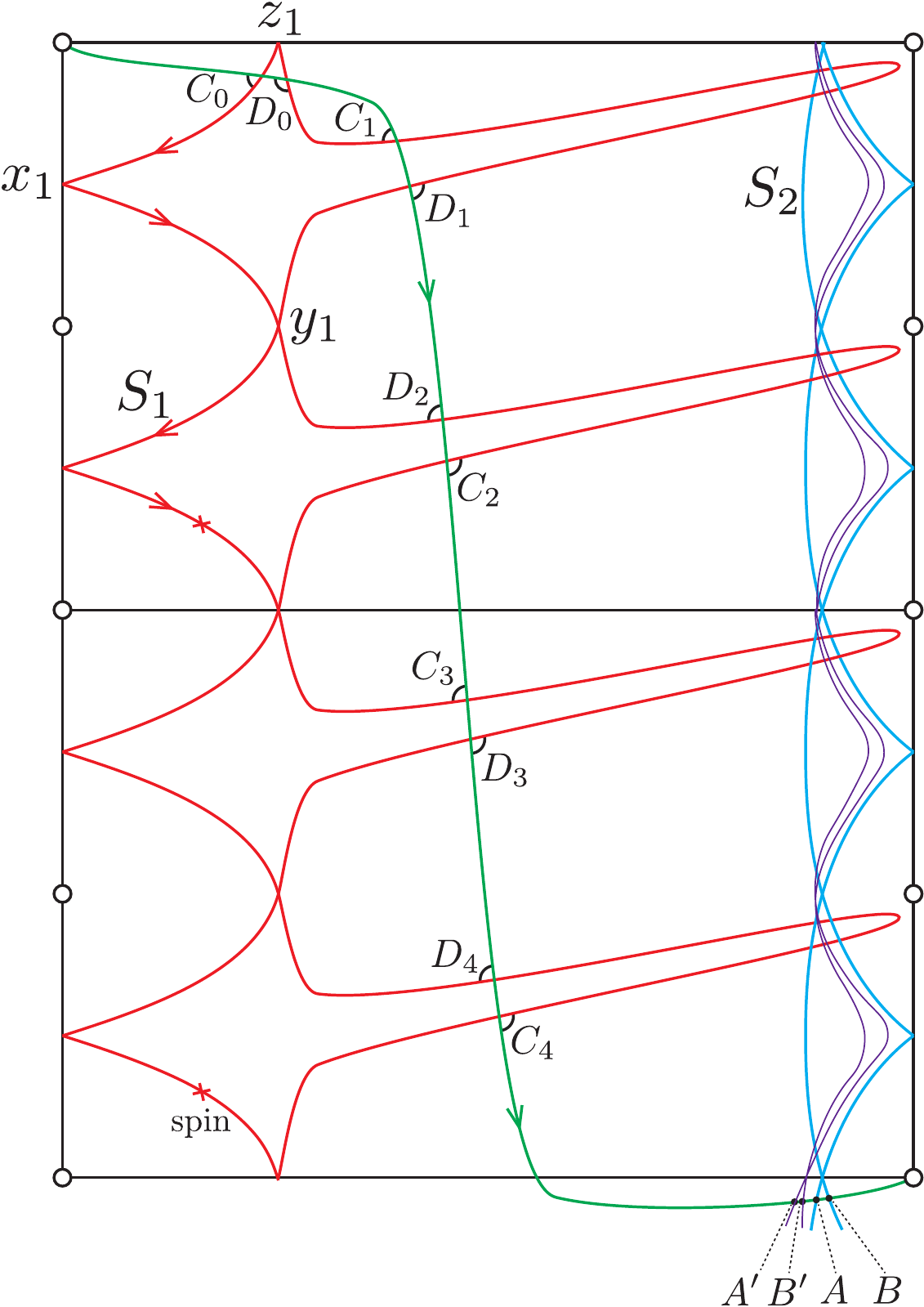}
		\caption{The Lagrangian path to be transformed shown in the cylindrical part corresponding to a chosen edge.  The figure shows the associated four-punctured sphere whose fundamental domain is given as a rectangle.  (In this figure $L$ has winded around the cylindrical part twice.)}
		\label{fig:Lag-winding}
	\end{center}
\end{figure}

We have an isomorphism between $S_1$ and $S_2$ as explained in the end of Section \ref{sec:doublecircles}.  To be more geometric, we also consider a pair-of-circles $C$ which is a smoothing of $S_2$ at the odd generator $x_2$.  The isomorphism from $S_1$ to $S_2$ can be obtained as composition of that from $S_1$ to $C$, gauge change on $C$, and that from $C$ to $S_2$.  The immersed variables of $C$ are denoted by $y,z$.  (Note that the circle components of $C$ may not be exact.)

$C$ is taken very close to $S_2$, so that the strips bounded by $C$ and $S_2$ can be assumed to have zero area in the limit.  $C$ intersects $L$ at two points, denoted by $A'$ and $B'$ as shown in Figure \ref{fig:Lag-winding}.  $A'$ is an odd morphism and $B'$ is even.  Again $A',B'$ occur very close to $z$, so that the area of the triangle with corners $(z, A',B')$ can be taken to be zero in the limit.

Recall that the gauge cycles of $C$ are important.  Denote by $P=\{p_1,p_2\} \subset C$ the gauge cycle which is the vanishing cycle corresponding to the smoothing of $S_1$ at $x_1$, and $P'=\{p_1',p_2'\} \subset C$ the gauge cycle corresponding to the smoothing of $S_2$ at $x_2$.  The gauge change, which depends on the way of moving $P$ to $P'$, is determined by the given tropical curve as explained in Section \ref{sec:mir-surf}.

The following lemma for local matrix factorizations is given by a direct computation of $m_1$ between $(S_1,x_1X_1+y_1Y_1+z_1Z_1)$ and $L$, see Figure \ref{fig:MF-loc}.

%Below $R$ is defined to be the coordinate ring of $\Lambda_+^3$, namely the subspace in $\Lambda[x,y,z]$ whose elements are convergent over $\Lambda$ when we subsitute $x,y,z\in\Lambda_+$.  We should define this (for different local spaces like $\Lambda_+^3, \Lambda_0^\times, \Lambda_0^3, \Lambda^3$) in the very beginning (like FOOO).

\begin{lemma} \label{lem:MF_S1}
	The local matrix factorization mirror to $L$ transformed by $(S_1,x_1X_1+y_1Y_1+z_1Z_1)$ equals to $(\mathrm{Span}_\Lambda \{C_i,D_i: i=0,\ldots,m\} \otimes_\Lambda \mathscr{O}(\Lambda_+^3),\delta_1)$, where $\delta_1$ is given by
	\begin{align*}
	C_0 \mapsto& z_1 D_0\\
	D_0 \mapsto& T^A x_1 y_1 C_0\\
	C_1 \mapsto& T^A D_1 - D_0\\
	D_1 \mapsto& x_1y_1z_1 C_1 + x_1 y_1 C_0\\
	C_{2k} \mapsto& z_1x_1(-y_1 D_{2k}+D_{2k-1}) \textrm{ for } k=1,\ldots,m\\
	D_{2k} \mapsto& -T^A C_{2k} + x_1 (z_1 C_{2k-1}+C_{2k-2}) \textrm{ for } k=1,\ldots,m\\
	C_{2k+1} \mapsto&T^A D_{2k+1} + x_1y_1 D_{2k} - x_1 D_{2k-1} \textrm{ for } k=1,\ldots,m-1\\
	D_{2k+1} \mapsto& x_1y_1z_1 C_{2k+1} + x_1 y_1 C_{2k} \textrm{ for } k=1,\ldots,m-1.
	\end{align*}
	The local matrix factorizations transformed by $(C,\nabla_{tP},yY+zZ)$, $(C,\nabla_{t'P'},y'Y'+z'Z')$, and $(S_2,x_2X_2 + y_2Y_2 + z_2Z_2)$ equal to
	$$\left\{ \begin{array}{ll}
	A' \mapsto& z B' \\
	B' \mapsto& ty A',
	\end{array}
	\right.$$
	$$\left\{ \begin{array}{ll}
	A' \mapsto& z' B' \\
	B' \mapsto& (t')^{-1}y' A',
	\end{array}
	\right.$$
	and
	$$\left\{ \begin{array}{ll}
	A \mapsto& z_2 B \\
	B \mapsto& x_2 y_2 A
	\end{array}
	\right.$$
	respectively.
\end{lemma}

\begin{figure}[h]
	\begin{center}
		\includegraphics[scale=0.3]{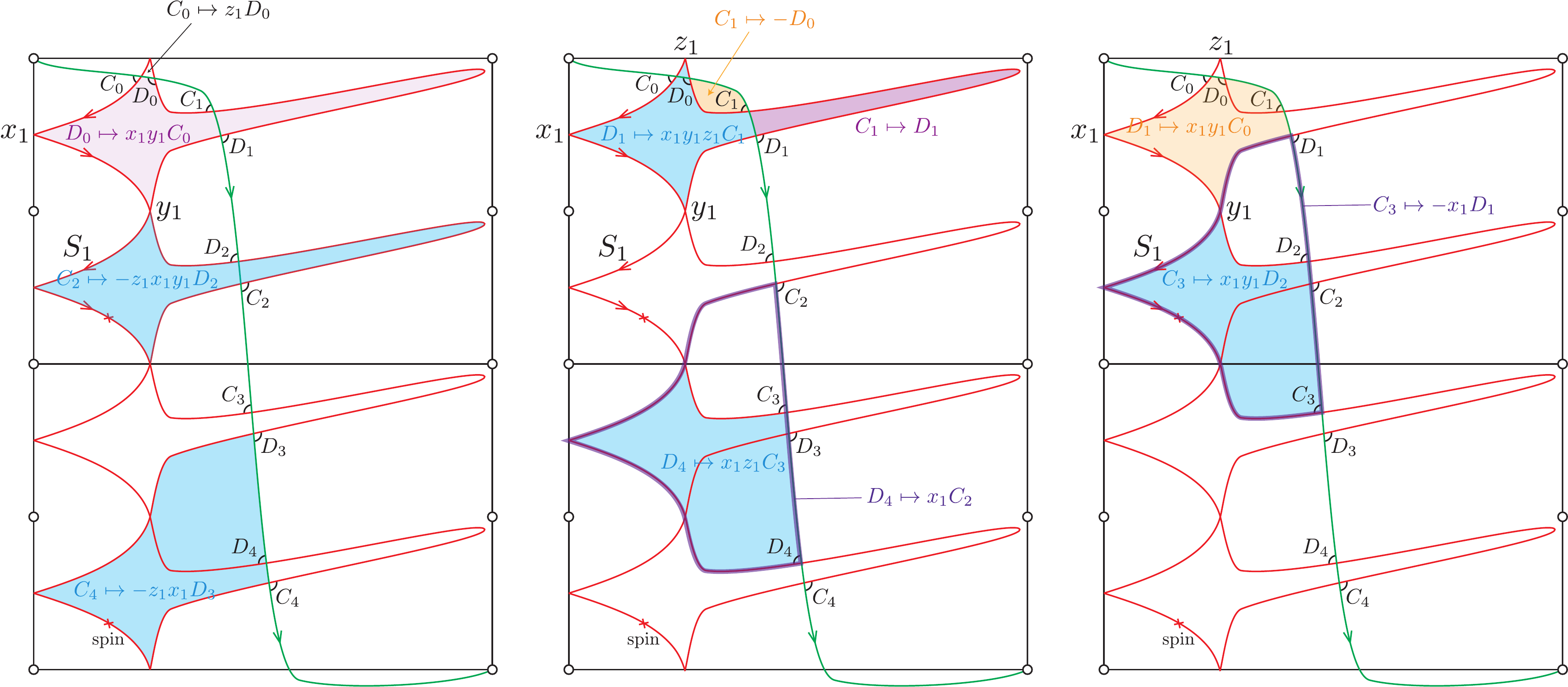}
		\caption{The strips contributing to the matrix factorization transformed by $S_1$ in Lemma \ref{lem:MF_S1}.}
		\label{fig:MF-loc}
	\end{center}
\end{figure}

Recall that the coordinate change obtained from isomorphisms between $(S_1,x_1X_1+y_1Y_1+z_1Z_1)$ and $(C,\nabla_{tP},yY+zZ)$ is given by
$$
x_{1}= t,
T^{A_1}y_{1}=y,
T^{A_2}z_{1}=z.
$$
where $A_1,A_2 \geq 0$ and $A_1+A_2=A$.  ($A_i$ are determined by the tropical curve.)

%In the above we have taken the deformation $S_1$ of the Seidel Lagrangian over the vertex $v_1$ such that the deformed part divide the cylinder bounded between the standard Seidel Lagrangian over $v_1$ and $C$ in two halves.  It gives the above coefficient $T^{A/2}$. 
We take the Lagrangian $L$ very close to the union of the real Lagrangian around the given face and the circle around the neck (corresponding to the finite edge), which occurs very close to $S_1$.  The real Lagrangian divides the strip bounded by $S_1$ and $C$ that has area $A_1$ into half.  It accounts for the term $A_1/2$ appearing below.

We shall compute the gluing induced by the isomorphism between the local matrix factorization transformed by $S_1$ and that transformed by $S_2$.  The isomorphism between $S_1$ and $S_2$ has been explained in the end of Section \ref{sec:doublecircles}.  To be more geometric, we compute it via composing the isomorphisms $S_1\to C$, gauge change on $C$, and $C\to S_2$.

First consider $S_1$ and $C$.  The isomorphism that we have taken between $(S_1,x_1X_1+y_1Y_1+z_1Z_1)$ and $(C,\nabla_{tP},yY+zZ)$ gives the following gluing between the corresponding local matrix factorizations.  See Figure \ref{fig:gluing-obj-sm} for the holomorphic polygons involved, which are terms in $m_2$.  

\begin{lemma}
	The isomorphism induces the following maps between the two matrix factorizations
	$$\left(\mathrm{Span}_\Lambda \{C_i,D_i: i=0,\ldots,m\} \otimes_\Lambda \mathscr{O}(\Lambda_+^3),\delta_1\right) \leftrightarrow
	 \left(\mathrm{Span}_\Lambda \{A',B'\} \otimes_\Lambda \mathscr{O}(\Lambda_0^\times \times \Lambda_+^2), \delta \right) 
	$$ 
where the former is transformed by $(S_1,x_1X_1+y_1Y_1+z_1Z_1)$ and the latter by $(C,\nabla_{tP},yY+zZ)$):
	\begin{align*}
	C_{2m} \leftrightarrow& -T^{-A_2-A_1/2} A'\\
	T^{A_1/2} x_1(-y_1 D_{2m} + D_{2m-1}) \leftarrow& B'\\
	T^{A_2+A_1/2} D_{2m-1} \rightarrow& t^{-1} B'
	\end{align*}
	and all $C_i$ for $i\not=2m$ and $D_j$ for $j\not=2m-1$ are sent to zero.
\end{lemma}

\begin{figure}[h]
	\begin{center}
		\includegraphics[scale=0.4]{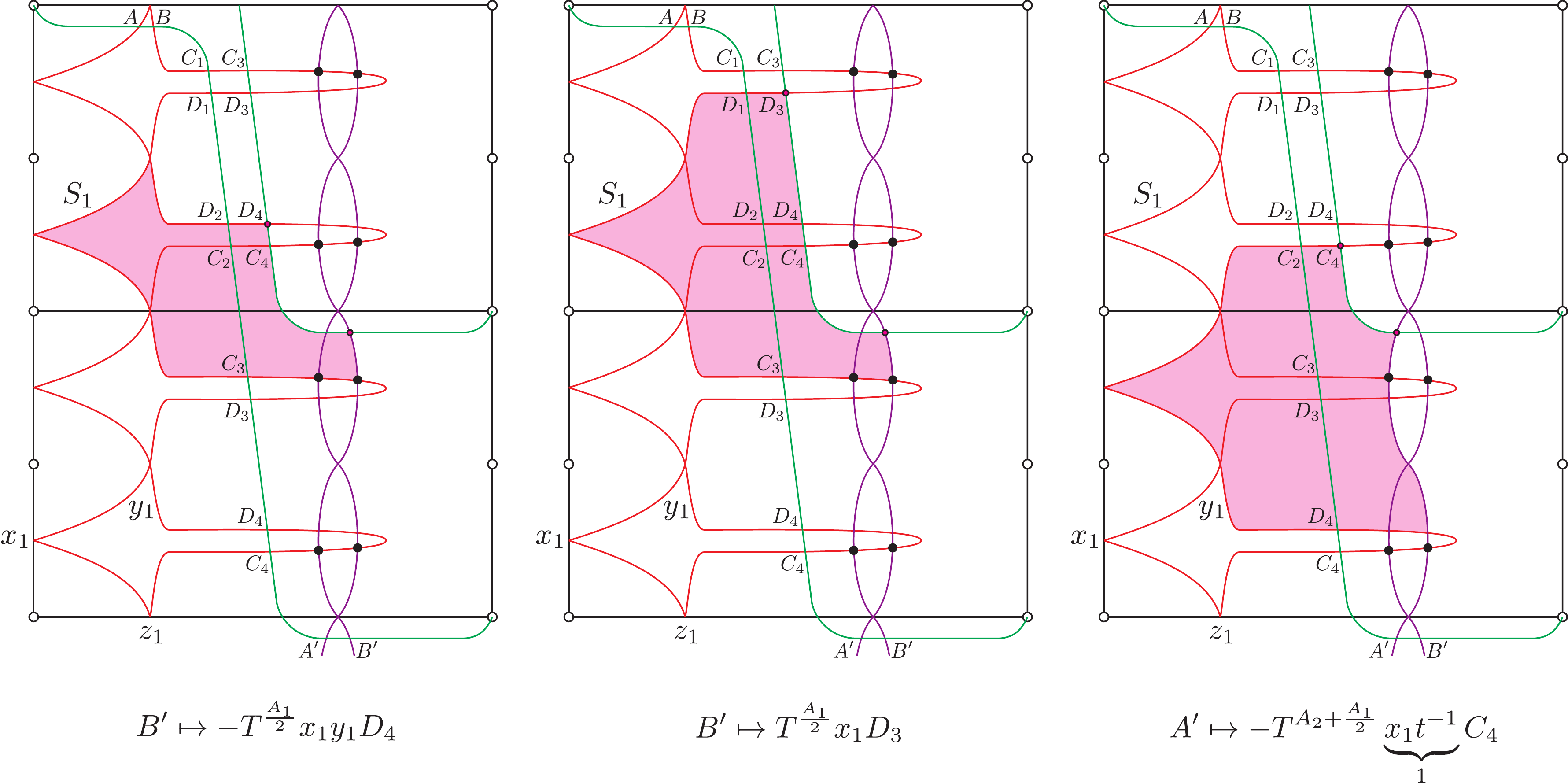}
		\caption{The holomorphic triangles contributing to the maps between the matrix factorizations transformed by $S_1$ and $C$.}
		\label{fig:gluing-obj-sm}
	\end{center}
\end{figure}

Now we consider the gauge change over $C$ from $P=\{p_1,p_2\}$ to $P'=\{p_1',p_2'\}$.  Denote by $a_i\in\Z$ the number of times (counted with signs) that $p_i$ passes through the immersed point $z$ (in order to change to $P'$).  Then the number of times that $p_i$ passes through the immersed point $y$ is $a_1-1$ and $a_2+1$ respectively.  $a_1,a_2$ are fixed in the very beginning for every finite edge of the tropical curve.

Recall that we have taken the intersection points $A',B'$ to be very close to the immersed point $z$. Whenever $p_i$ passes through $z$, it passes through $A$ for $i=1$ and $B$ for $i=2$.  Below we use the convention that $p_i$ passes through the immersed points in upward direction in Figure \ref{fig:gluing-obj-gauge}.  The coordinate change is inversed for downward direction.

\begin{figure}[h]
	\begin{center}
		\includegraphics[scale=0.4]{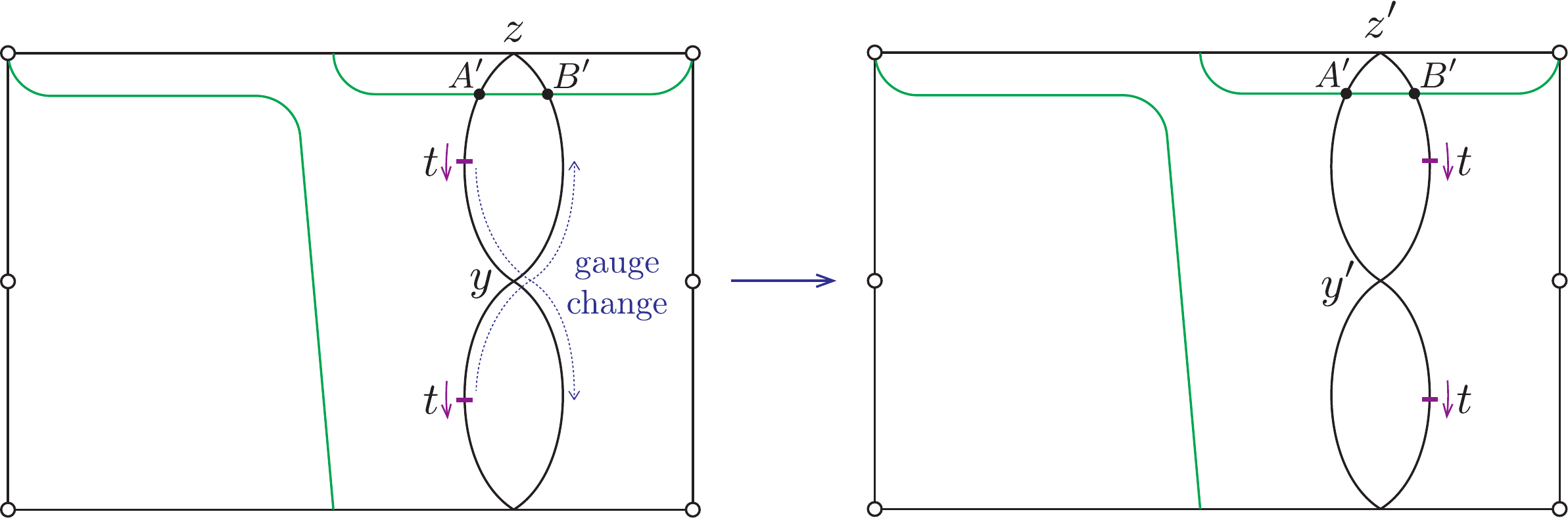}
		\caption{Gauge change of the pair-of-circles $C$.}
		\label{fig:gluing-obj-gauge}
	\end{center}
\end{figure}

For each time $p_1$ passes through $z$, the coordinate change and gluing are 
$$
A' \leftrightarrow t^{-1}A',
B' \leftrightarrow B',
z=t^{-1}z', t=t', y = y'.
$$

For each time $p_2$ passes through $z$,
$$
A' \leftrightarrow A',
B' \leftrightarrow t^{-1} B',
z=tz', t=t', y = y'.
$$

For each time $p_i$ passes through $y$, $y=ty'$ for $i=1$ and $y=t^{-1}y'$ for $i=2$.  Moreover, $z=z', t=t'$ and the gluing on $A',B'$ is trivial.

It results in an overall coordinate change and gluing summarized as follows.

\begin{lemma}
	The gauge change on $C$ from $P$ to $P'$, where $p_i$ passes through the immersed point $z$ by $a_i$ times, produces the gluing
$$
\left\{ \begin{array}{lll}
z&=t^{a_{2}-a_{1}}z', \\
	y&=t^{a_{1}-a_{2}-2}y', \\
	t&=t'
		\end{array}
	\right.  \hspace{1cm}
	\left\{ \begin{array}{ll}
	A'&\leftrightarrow t^{-a_{1}}A',\\
	B'&\leftrightarrow t^{-a_{2}}B'. 
		\end{array}
	\right.
$$
%
%	\begin{align*}
%	z&=t^{a_{2}-a_{1}}z', \\
%	y&=t^{a_{1}-a_{2}-2}y', \\
%	t&=t',\\
%	A'&\leftrightarrow t^{-a_{1}}A',\\
%	B'&\leftrightarrow t^{-a_{2}}B'. 
%	\end{align*}
\end{lemma}

Finally, the gluing between $(C,\nabla_{t'P'},y'Y'+z'Z')$, and $(S_2, x_2X_2 + y_2Y_2 + z_2Z_2)$ resulted from the isomorphism is straightforward:
$$ x_2=(t')^{-1}, y_2 = y', z_2=z', A' \leftrightarrow A, B' \leftrightarrow B. $$

Composing all the above, we obtain the gluing between $(S_1, x_1X_1 + y_1Y_1 + z_1Z_1)$ and $(S_2, x_2X_2 + y_2Y_2 + z_2Z_2)$.

\begin{prop} \label{prop:glueMF_12}
	Let $L$, $S_1$ and $S_2$ be Lagrangians given in the beginning of this subsection.  The gluing between the matrix factorizations 
	$(\mathrm{Span}_\Lambda \{C_i,D_i: i=0,\ldots,m\} \otimes_\Lambda \mathscr{O}(\Lambda_+^3),\delta_1)$ and $(\mathrm{Span}_\Lambda \{A,B\} \otimes_\Lambda \mathscr{O}(\Lambda_+^3),\delta_2)$,
	which are mirror to $L$ and transformed by $S_i$, is given by
	\begin{align*}
	C_{2m} \leftrightarrow& -T^{-A_2-A_1/2} t^{-a_1} A\\
	T^{A_1/2} x_1(-y_1 D_{2m} + D_{2m-1}) \leftarrow& t^{-a_2}B\\
	T^{A_2+A_1/2} D_{2m-1} \rightarrow& t^{-a_2-1} B
	\end{align*}
	and all $C_i$ for $i\not=2m$ and $D_j$ for $j\not=2m-1$ are sent to zero.  The coordinate changes are
	$x_1=t=x_2^{-1}, T^{A_1} y_1=t^{a_1-a_2-2}y_2, T^{A_2} z_1=t^{a_2-a_1}z_2.$
\end{prop}

In the above proposition, if we use exact variables $(x_i)_\ex,(y_i)_\ex,(z_i)_\ex$ and exact generators $(C_i)_\ex,(D_i)_\ex,A_\ex,B_\ex$, then all area terms are absorbed and gone.  In other words, the gluing for these exact variables and generators is obtained by setting $A=0$ in the above proposition which simplifies the expressions.  We shall use exact variables and exact generators from now on.

The space glued from $\C^3_{((x_1)_\ex,(y_1)_\ex,(z_1)_\ex)}$ and $\C^3_{((x_2)_\ex,(y_2)_\ex,(z_2)_\ex)}$ defined by the above coordinate change is the total space of $\mathscr{O}_{\bP^1}(a_1-a_2-2,a_2-a_1)$ (which is CY since $(a_1-a_2-2)+(a_2-a_1)=-2$).

\begin{prop} \label{prop:glue-obj}
	Over the total space of $\mathscr{O}_{\bP^1}(a_1-a_2-2,a_2-a_1)$ with the disc potential $W$, the image of $L$ in $D\mathrm{Sing}(\{W=0\})$ under the derived functor is the push-forward of a line bundle over the hypersurface $D_f = \bigcup_{i=1,2} \{(z_i)_\ex=0\}$.  The line bundle corresponds to the divisor $(a_2+m)\cdot \{(x_1)_\ex=0\}$ in $D_f$.
\end{prop}

\begin{proof}
	The corresponding object in $D\mathrm{Sing}(\{W=0\})$ is obtained by taking cokernel of $\delta_i$ from the odd part to the even part.  Over $\C^3_{((x_1)_\ex,(y_1)_\ex,(z_1)_\ex)}$, using 
	$$\{(D_0)_\ex, (D_1)_\ex - (D_0)_\ex, (D_{2k})_\ex, (D_{2m})_\ex, (D_{2k+1})_\ex+(x_1)_\ex ((y_1)_\ex (D_{2k})_\ex - (D_{2k-1})_\ex) \textrm{ for } k=1,\ldots,m-1\}$$ 
	as a basis of the even part, the cokernel of $\delta_1$ equals to
	$$ \left(\C[(x_1)_\ex, (y_1)_\ex, (z_1)_\ex] \big/ \langle (z_1)_\ex \rangle \right) (D_0)_\ex \oplus \left( \bigoplus_{k=1}^m \left(\C[(x_1)_\ex, (y_1)_\ex, (z_1)_\ex] \big/ \langle (x_1)_\ex (y_1)_\ex (z_1)_\ex\rangle \right) (D_{2k})_\ex \right).  $$
	Note that $R_1 / \langle (x_1)_\ex (y_1)_\ex (z_1)_\ex \rangle$ is trivial in $D\mathrm{Sing}(\{ (x_1)_\ex (y_1)_\ex (z_1)_\ex=0\})$.  Thus the corresponding local object is the skyscraper sheaf supported on the divisor $(z_1)_\ex=0$.
	
	Over $\C^3_{((x_2)_\ex,(y_2)_\ex,(z_2)_\ex)}$, the cokernel of $\delta_2$ equals to
	$\left(\C[(x_2)_\ex, (y_2)_\ex, (z_2)_\ex] / \langle (z_2)_\ex \rangle \right) B_\ex$.  By Proposition \ref{prop:glueMF_12}, the gluing is given by sending $B_\ex$ to
	\begin{align*}
	(x_1)_\ex^{a_2} (x_1)_\ex (-(y_1)_\ex (D_{2m})_\ex + (D_{2m-1})_\ex) \sim& (x_1)_\ex^{a_2+1} (D_{2m-1})_\ex \\
	\sim& (x_1)_\ex^{a_2+2} (-(y_1)_\ex (D_{2m-2})_\ex + (D_{2m-3})_\ex)\\
	\vdots& \\
	\sim& (x_1)_\ex^{a_2+m} (D_1)_\ex \sim  (x_1)_\ex^{a_2+m} (D_0)_\ex
	\end{align*}
	mod $ (x_1)_\ex (y_1)_\ex (D_{2k})_\ex$ (since we don't care the components $\left(R_1 / \langle (x_1)_\ex (y_1)_\ex (z_1)_\ex\rangle \right) (D_{2k})_\ex$) and the image of $\delta_1$ (see Lemma \ref{lem:MF_S1}).
	
	Hence the object is the push-forward of a line bundle over the divisor $(z_i)_\ex=0$.  The line bundle has a section $B_\ex = (x_1)_\ex^{a_2+m}(D_0)_\ex$ which has a unique zero of multiplicity $a_2+m$ at $(x_1)_\ex=0$.
\end{proof}

We have computed the effect of winding around a finite edge $e=e_1$.  Now we need to glue with the edge $e_2$ which is adjacent to the face $f$ and also the vertex $v=v_1$.
$e_2$ could either be a finite or infinite edge.  Suppose  $e_2=\overline{v_1v_3}$ is a finite edge.
Consider the two immersed Lagrangians in our collection that cover the part of $\Lambda_0$-valued critical locus corresponding to $e_2$.  There are two possibilities.  The first case is that $S_1$ and $S_3$, a deformation of the Seidel Lagrangian over $v_3$, cover $e_2$.  The second case is that another deformation $S_1'$ of the Seidel Lagrangian over $v_1$ and $S_3$ cover $e_2$.  See Figure \ref{fig:gluing-edges}.
For the situation that $e_2$ is an infinite edge, the Seidel Lagrangian $S_1'$ over $v_1$ is in our collection.  $S_1'$ and $S_1$ cover the part of $\Lambda_0$-valued critical locus corresponding to $e_2$.  This is similar to the second case above and so we do not separately consider this.

\begin{figure}[h]
	\begin{center}
		\includegraphics[scale=0.4]{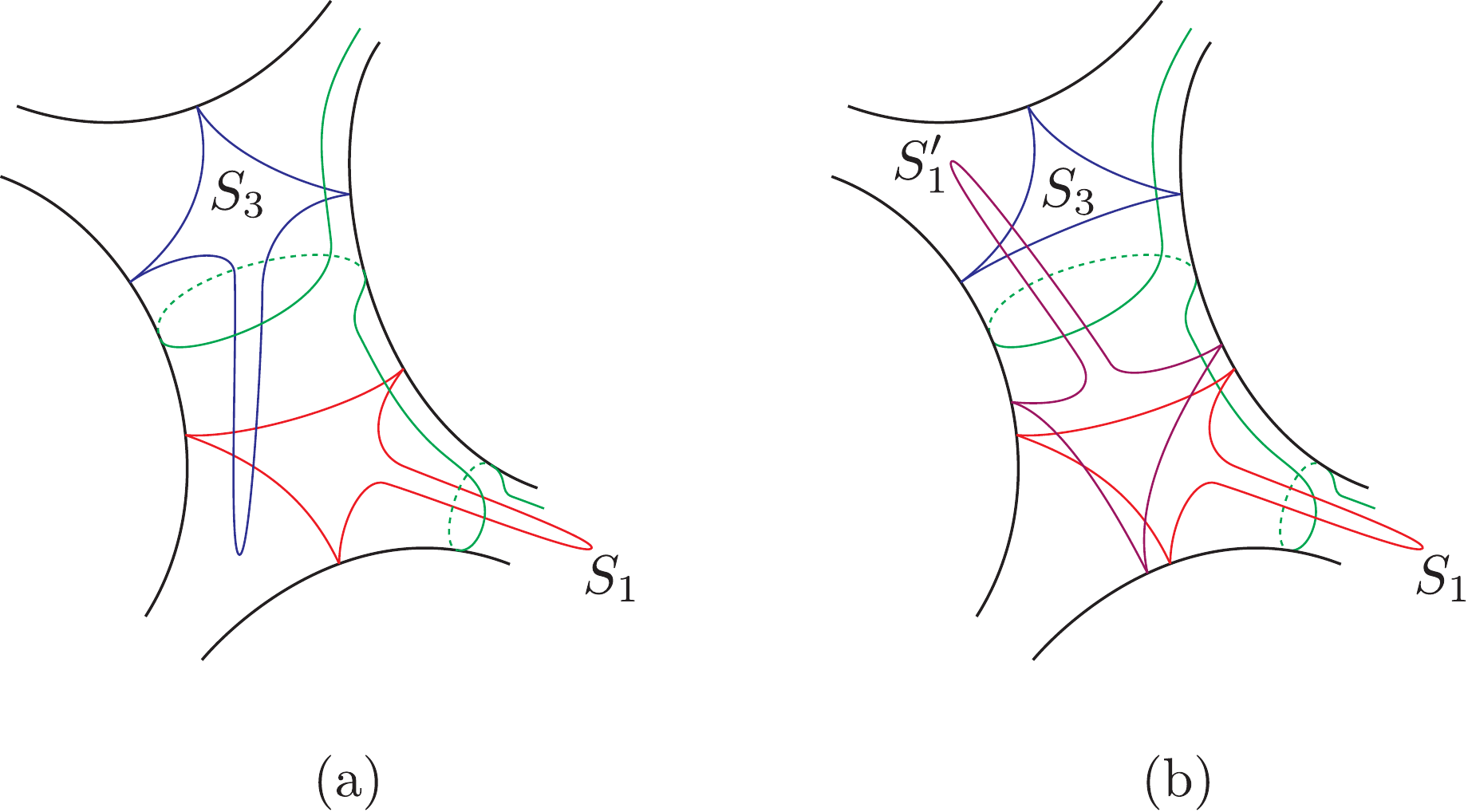}
		\caption{The two possibilities of gluing with another edge.}
		\label{fig:gluing-edges}
	\end{center}
\end{figure}

In the first case, $L$ intersects $S_3$ at ($2+2m'$) points, where $m'$ is the number of times that $L$ winds around $e_2$.  (Recall that $L$ winds around $e_1$ $m$-times, and it intersects $S_1$ at ($2+2m$) points.)  The immersed variables of $S_3$ are $x_3,y_3,z_3$, where $y_3$ corresponds to deformations along the edge $e_2$, and $z_3$ corresponds to the face $f$.  The intersection points between $L$ and $S_3$ are denoted as $C_i', D_i'$ for $i=0,\ldots,m$ as shown in Figure \ref{fig:gluing-edges1}.  The cokernel of the local matrix factorization transformed by $S_3$ equals to 
$$ \left(\C[(x_3)_\ex, (y_3)_\ex, (z_3)_\ex] / \langle (z_3)_\ex \rangle \right) (D_0')_\ex \oplus \left( \bigoplus_{k=1}^m \left(\C[(x_3)_\ex, (y_3)_\ex, (z_3)_\ex] / \langle (x_3)_\ex (y_3)_\ex (z_3)_\ex\rangle \right) (D_{2k}')_\ex \right) $$
where $\C[(x_3)_\ex, (y_3)_\ex, (z_3)_\ex] / \langle (x_3)_\ex (y_3)_\ex (z_3)_\ex\rangle$ is trivial in $D\mathrm{Sing}$.  Thus the object is supported over $(z_3)_\ex=0$.

\begin{figure}[h]
	\begin{center}
		\includegraphics[scale=0.5]{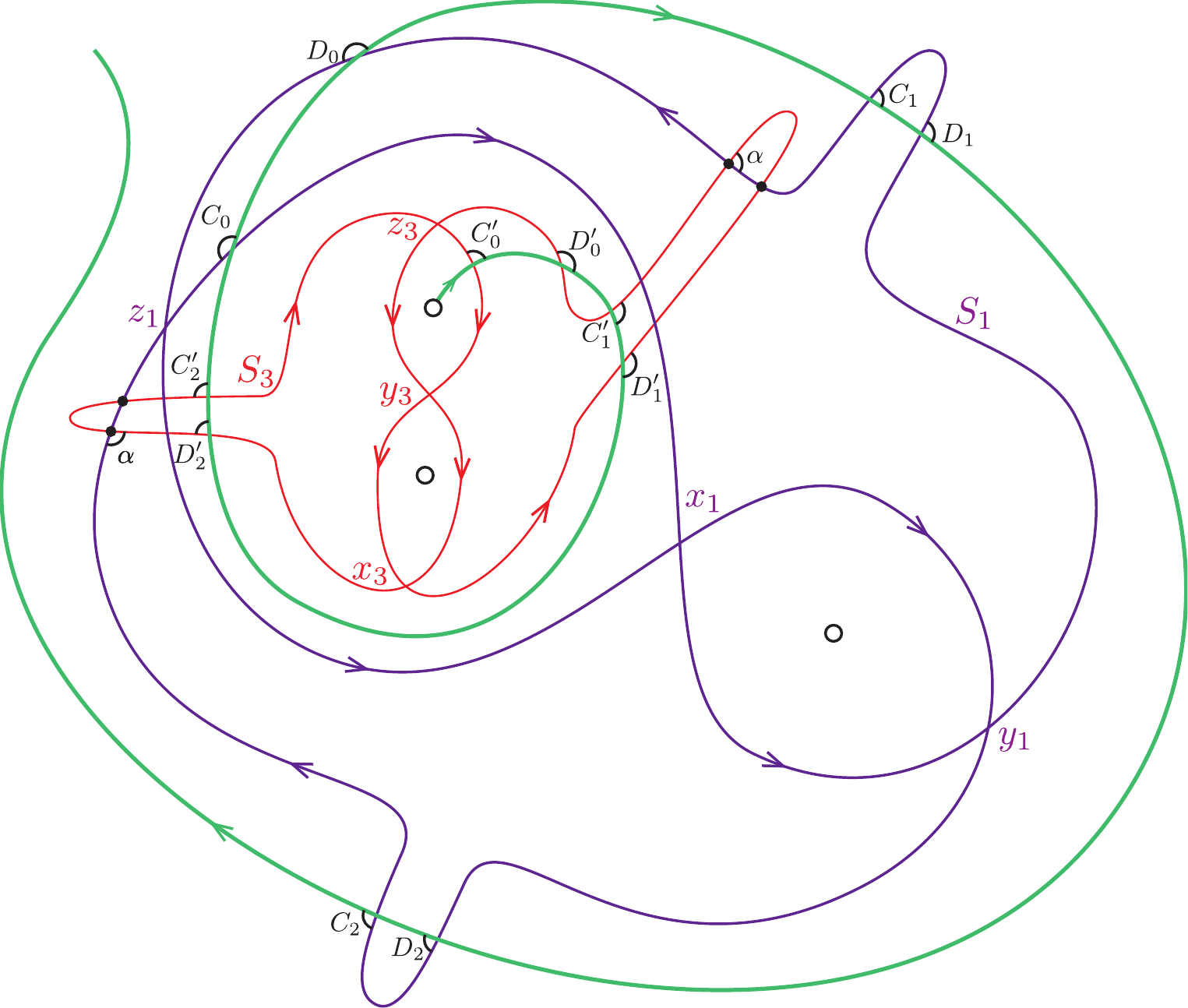}
		\caption{The first case of gluing with another edge.  It shows the four-punctured sphere (which is identified with a three-punctured plane) that contains both $S_1$ and $S_3$.}
		\label{fig:gluing-edges1}
	\end{center}
\end{figure}

The gluing between $S_3$ and $S_1$ is similar to what we have done for that between $S_1$ and $S_2$.  Namely we compose the isomorphism from $S_3$ to a pair-of-circles $C_{31}$ over $e_2$, gauge change for $C_{31}$ (determined by the tropical curve), and isomorphism from $C_{31}$ to $S_3$.  It can be checked that the counting of strips is essentially the same as that for $(S_1,S_2)$ in Proposition \ref{prop:glueMF_12} (see Figure \ref{fig:gluing-edges1-strips}), and hence the resulting gluing is given by
\begin{align*}
(y_{3})_\ex^{-a'_{1}}(C_{0})_\ex \mapsto& -(C_{2m}')_\ex\\
(y_{3})_\ex^{-a'_{2}}(D_{0})_\ex \mapsto& (y_{3})_\ex (-(x_{3})_\ex (D_{2m}')_\ex +(D_{2m-1}')_\ex ) \\ 
(z_{3})_\ex =&(y_{3})_\ex^{a_{2}'-a_{1}'}(z_{1})_\ex,\text{~ }(x_{3})_\ex=(y_{3})_\ex^{a_{1}'-a_{2}'-2}(x_{1})_\ex,\text{~ }(y_{3})_\ex=(y_{1})_\ex^{-1}
\end{align*}
where $a_i'\in\Z$ (for $i=1,2$) is the number of times (counted with signs) that the gauge points $p_i'$ in $C_{31}$ passes through the immersed point $z'$ (corresponding to $z_3$ and $z_1$).  ($C_i,D_i$ for $i>0$ are sent to zero.)

\begin{figure}[h]
	\begin{center}
		\includegraphics[scale=0.35]{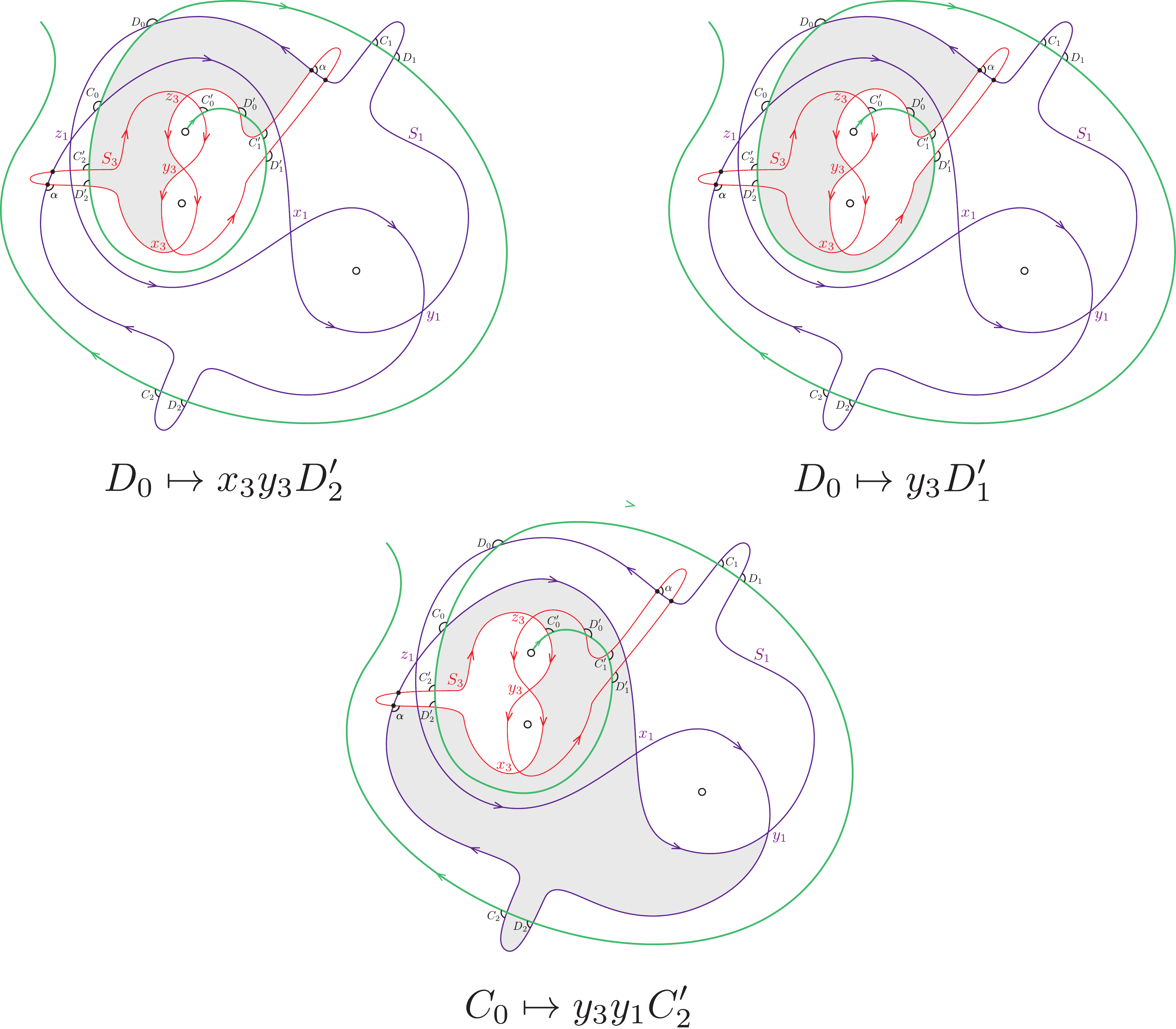}
		\caption{The strips that contribute to the gluing between $S_1$ and $S_3$.}
		\label{fig:gluing-edges1-strips}
	\end{center}
\end{figure}

As in the proof of Proposition \ref{prop:glue-obj}, we see that $(D_0)_\ex$ is glued with $(y_3)_\ex^{a_2'+m'} (D_0')_\ex$.  Combining with that $B_\ex$ is glued with $(x_1)_\ex^{a_2+m}(D_0)_\ex$, we conclude that the glued object is (the push-forward of) the divisor line bundle over $\{(z_i)_\ex=0\}$ corresponding to the divisor $\left((a_2+m)\cdot \{(x_1)_\ex=0\} +(a_2'+m') \cdot \{(y_3)_\ex=0\}\right)$ in $\{(z_i)_\ex=0\}$.

In the second case, we need to compute the gluing between $S_1'$ and $S_1$.  Denote the immersed variables of $S_1'$ by $x_1',y_1',z_1'$.
$L$ intersects $S_1'$ at $(2+2m')$ points and intersects $S_3$ at two points.  Denote the intersection points between $S_1'$ and $L$ by $C_i',D_i'$ for $i=0,\ldots,m'$.  See Figure \ref{fig:gluing-edges2}.  

\begin{figure}[h]
	\begin{center}
		\includegraphics[scale=0.55]{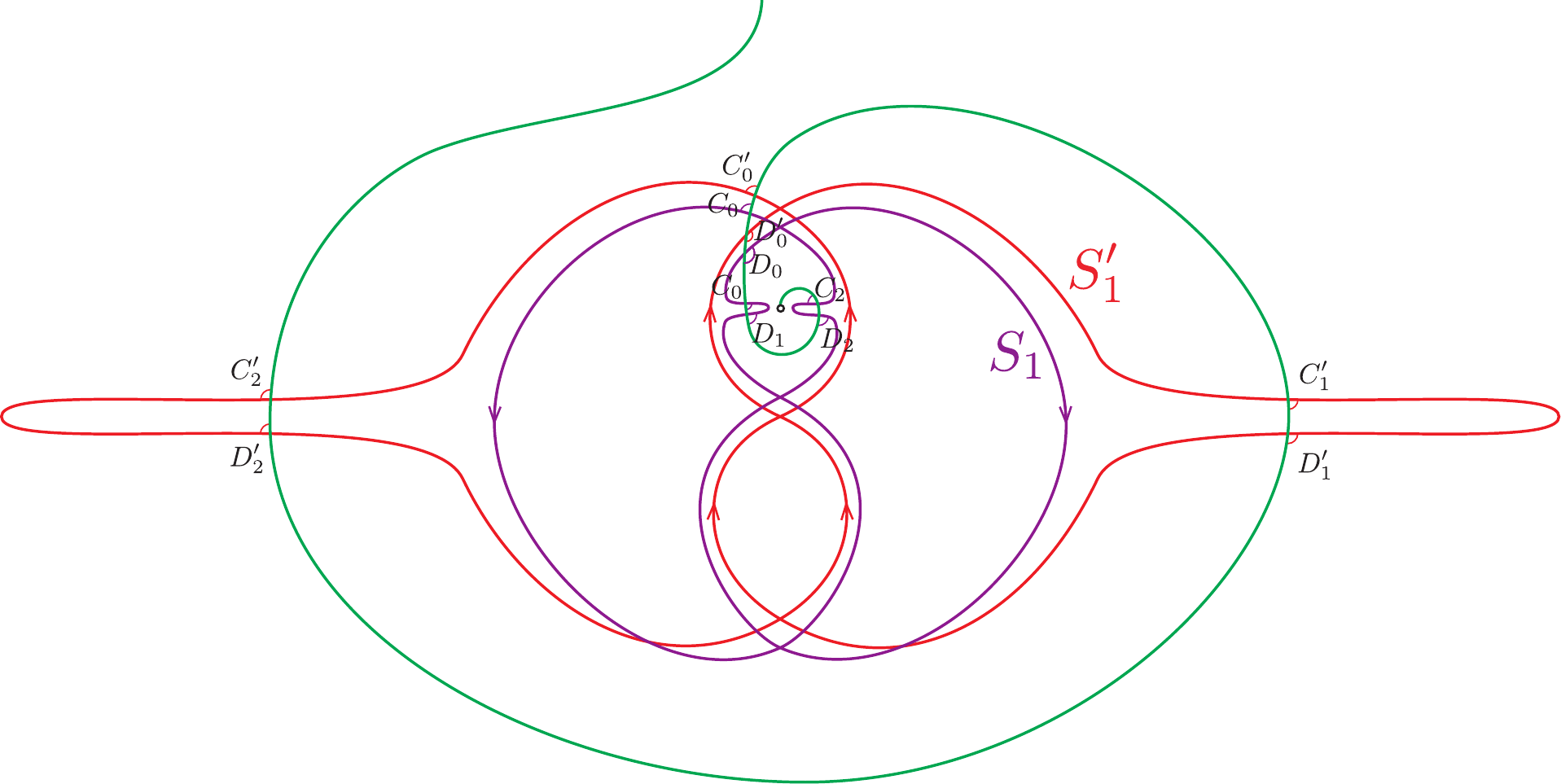}
		\caption{The second case of gluing with another edge.  It shows the pair-of-pants that contains both $S_1$ and $S_1'$.}
		\label{fig:gluing-edges2}
	\end{center}
\end{figure}

The local matrix factorization transformed by $S_1'$ is similar to that given in Lemma \ref{lem:MF_S1} for $\delta_1$.  So the cokernel is again
$$ \left(\C[(x_1')_\ex, (y_1')_\ex, (z_1')_\ex]/ \langle (z_1')_\ex \rangle \right) (D_0')_\ex \oplus \left( \bigoplus_{k=1}^m \left(\C[(x_1')_\ex, (y_1')_\ex, (z_1')_\ex] / \langle (x_1')_\ex (y_1')_\ex (z_1')_\ex\rangle \right) (D_{2k}')_\ex \right).  $$
It can be checked that the gluing between the matrix factorizations transformed by $S_1$ and $S_1'$ is simply given by
$$ (C_0)_\ex \mapsto (C_0')_\ex, (D_0)_\ex \mapsto (D_0')_\ex$$
and all $C_i,D_i$ for $i>0$ are mapped to zero (see Figure \ref{fig:gluing-edges2-strips}).  (Also the coordinate change is trivial: $(x')_\ex=x_\ex, (y')_\ex=y_\ex, (z')_\ex=z_\ex$.)  Hence the gluing is simply identity.  By Proposition \ref{prop:glueMF_12} we have the divisor line bundles $(a_2+m)\cdot \{(x_1)_\ex=0\}$ and $(a_2'+m') \cdot \{(y_3)_\ex=0\}$ for the edges $e_1$ and $e_2$ respectively.  They are simply glued by identity in the common intersection, and hence we again get the divisor line bundle $\left((a_2+m)\cdot \{(x_1)_\ex=0\} +(a_2'+m') \cdot \{(y_3)_\ex=0\}\right)$ in $\{(z_i)_\ex=0\}$.

\begin{figure}[h]
	\begin{center}
		\includegraphics[scale=0.4]{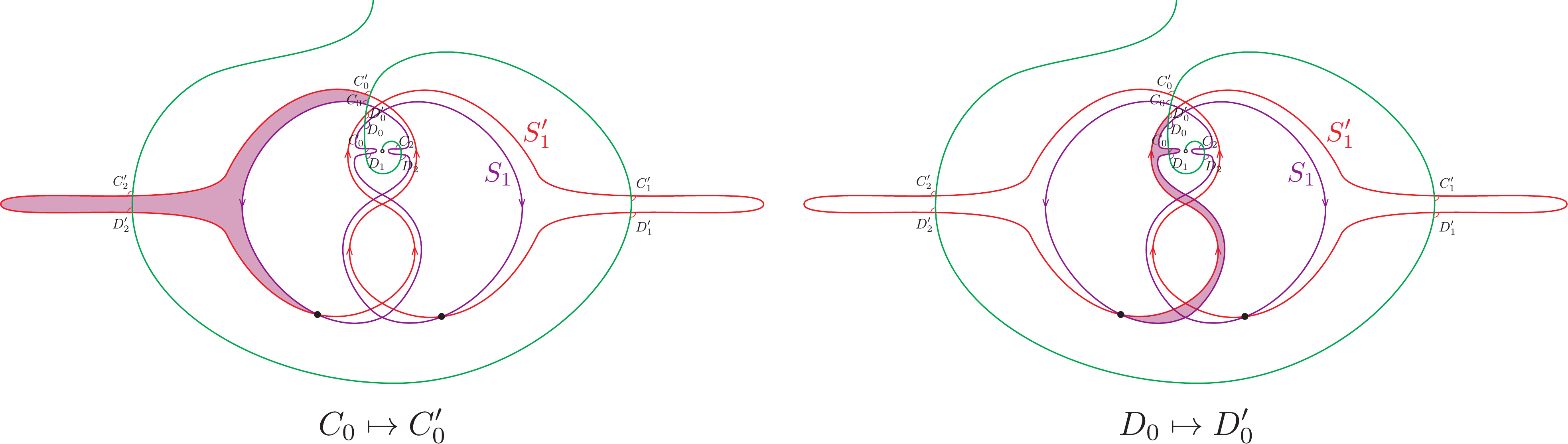}
		\caption{The strips that contribute to the gluing between $S_1$ and $S_1'$.}
		\label{fig:gluing-edges2-strips}
	\end{center}
\end{figure}

By gluing all the edges adjacent to the face $f$, we obtain a divisor line bundle over $\{(z_i)_\ex=0\}$.
For immersed Lagrangians $S$ in our collection corresponding to vertices not adjacent to $f$, the strips bounded by $L$ and $(S, xX+yY+zZ)$ either never involve the $x,y,z$ angles, or involve all the $x,y,z$ angles once.  Hence the cokernel of the corresponding matrix factorization is a direct sum of $\C[x_\ex, y_\ex, z_\ex]/\langle x_\ex y_\ex z_\ex\rangle$ which is trivial in $D\mathrm{Sing}$.  Thus the mirror object is merely supported over $\{(z_i)_\ex=0\}$.  We conclude the following.

\begin{thm} \label{thm:mir_obj}
	Let $L$ be a Lagrangian path around a face $f$ of the tropical curve, and let $m^e$ be the winding numbers of $L$ around the cylindrical part of the surface corresponding to the finite edges $e$ adjacent to $f$.  Its image in $D\mathrm{Sing}(W_{Y(\C)})$ under our functor is the push-forward of a divisor line bundle over the toric divisor corresponding to $f$, where the divisor line bundle is given by
	$$ \sum_e (a_2^e + m^e)\cdot \{z_e = 0\}$$
	where the sum is over finite edges $e$ adjacent to $f$, $z_e$ is the toric variable corresponding to the primitive vector parallel to the edge and along the counter-clockwise boundary orientation of $f$, and $a_2^e$ is the number of times that the gauge point $p_2$ of the pair-of-circles in $e$ passes through the immersed point $z$.  
\end{thm}

The numbers $a_2^e \in \Z$ for finite edges $e$ are fixed in the very beginning, which corresponds to the choice of isomorphisms between Lagrangians in our collection.

\begin{remark} \label{rmk:Lee}
	In the work of H. Lee \cite{HLee}, the moment-map polytope (which is given by the tropical curve) defines a global line bundle $\cL_Y$ over the toric CY $Y(\C)$.  Let $D$ be an irreducible toric divisor (corresponding to the face $f$ in the above notation).  $O_D(k)$ is the restriction of $\cL_Y^{\otimes k}$ to $D$, which is equivalent to the divisor 
	$$ \sum_e (k \cdot n^e)\cdot \{z_e = 0\}$$
	in the above notation, where $n^e$ is the affine length of the edge $e$ in the polytope.  By the above theorem, $O_D(k)$ is mirror to the Lagrangian $L$ circulating around the face $f$ with winding numbers $m^e = k \cdot n^e - a_2^e$ around each finite edge adjacent to $f$.  
	
	Comparing with the notations in \cite{HLee}, we set $f=\alpha$, and the edge $e$ to be the intersection of the faces $\alpha$ and $\beta$.  Moreover, set
	\begin{equation*}
	a_{1}^e=-\delta _{\beta,\alpha}+1, a_{2}^e=-\delta _{\alpha,\beta}.
	\end{equation*}
	Then
	$a^e_{2}-a^e_{1}=\delta _{\beta,\alpha}-\delta _{\alpha, \beta}-1 = d_{\alpha,\beta}$ (which is responsible for the coordinate change in $z$) ;  $a^e_{1}-a^e_{2}-2=\delta _{\alpha,\beta}-\delta _{\beta,\alpha}-1=d_{\beta,\alpha}$  (which is responsible for the coordinate change in $y$).  Thus $m^e = k \cdot n_{\alpha\beta} + \delta_{\alpha,\beta}$, which agrees with the result of \cite{HLee} that such an $L$ is mirror to $O_D(k)$.  This implies our functor sends generators of $D\mathrm{WFuk}(X)$ to generators of $D\MF(W_{Y(\C)})$.
\end{remark}

\subsection{Morphisms}
Next we shall show that the functor is an isomorphism on morphism spaces between objects.  Let's denote by $\cL \in D\MF(W_{Y(\C)})$ the mirror object of a Lagrangian $L$ given in Theorem \ref{thm:mir_obj}.  For two such mirror objects $\cL_1,\cL_2$, the morphism space is non-zero only when the corresponding facets in the toric diagram intersect at an edge.  Moreover, the morphism space is explicitly known.  We shall compute our functor on morphism spaces and show that they are isomorphisms.

The computations in this subsection are over $\Z_2$, namely we just compute up to $\pm$ sign.  Although everything is defined over $\Z$, we do not bother about explicit signs since we just need to determine the mirror morphisms up to sign in order to show the isomorphisms.

First consider endomorphisms of $L$ given as in Theorem \ref{thm:mir_obj}.  Take a Hamiltonian perturbation $\phi(L)$ which wraps around punctures as shown in Figure \ref{fig:morphisms-inf-01}, such that intersection points between $\phi(L)$ and $L$ occur only in cylindrical parts corresponding to infinite edges of the tropical curve.  In particular, if $L$ is circulating around a compact face, then its endomorphism space is trivial.  So we assume that $L$ is circulating around a non-compact face (whose boundary has two infinite edges).

Let's orient $L$ counter-clockwisely around the non-compact face.  Consider a pair-of-pants containing the non-compact edge that $L$ ends (or begins) with.  Let $S_2$ (or $S_1$ resp.) be the Seidel Lagrangian over the vertex $v_2$ (or $v_1$ resp.) adjacent to this non-compact edge.  See Figure \ref{fig:morphisms-inf-01} which depicts $S_k$, $L$ and $\phi(L)$ in the pair-of-pants.  Denote by $P_i$ for $i\in \Z_{> 0}$ (or $i \in \Z_{\leq 0}$ resp.) the morphisms from $\phi(L)$ to $L$ as shown in the figure.  The endomorphism space of $L$ is spanned by $P_i$ for $i\in\Z$.

The mirror matrix factorization $\cL$ restricted to the formal deformation space of $S_k$ is given by $\mathrm{Span}\{A,B\}$.  Consider the endomorphisms $P_0$ and $P_1$.
By counting triangles bounded by $(S_k,xX+yY+zZ),\phi(L),L$ as shown in Figure \ref{fig:morphisms-inf-01}, we obtain the mirror endomorphisms of $\cL$ restricted to the formal deformation space of $S_k$ as follows.

\begin{figure}[h]
	\begin{center}
		\includegraphics[scale=0.4]{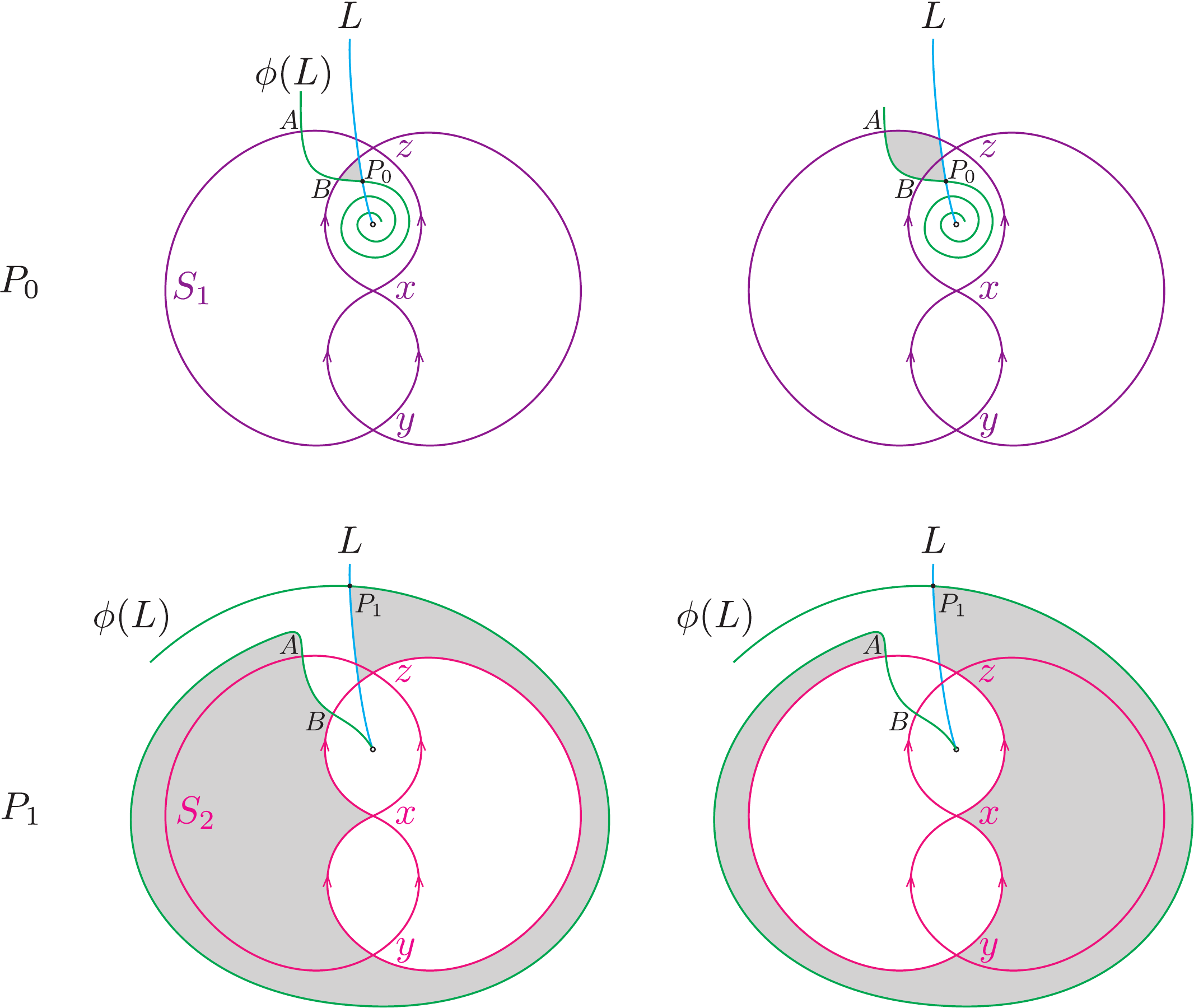}
		\caption{Transforming the morphisms $P_0$ and $P_1$ in an infinite edge.}
		\label{fig:morphisms-inf-01}
	\end{center}
\end{figure}

\begin{lemma} \label{lem:mor-inf-edge}
	Let $P_0$ and $P_1$ be the endomorphisms of $L$ given above.  Under the mirror functor, their images restricted to $\C^3$, the formal deformation space of $S_k$ (where $k=1$ for $P_0$ and $k=2$ for $P_1$), are given by $\mathrm{Id}$ and multiplication by $x$ on $\cL|_{\C^3}=(\mathrm{Span}\{A,B\},\delta)$ respectively.
\end{lemma}

On the other hand, the endomorphism space of $\cL$ is explicitly known, which is $$\mathrm{Span}\left(\{\mathbf{x}^i: i\in\Z_{> 0}\} \cup \{\mathbf{y}^j: j\in\Z_{\geq 0}\}\right).$$  
The restriction of $\mathbf{x}^i$ (or $\mathbf{y}^j$) on the chart $\C^3$ corresponding to the vertex $v_2$ (or $v_1$ resp.) is given by multiplication by $x^i$ (or $y^j$ respectively).  Thus $P_0,P_1$ must be mapped to $\mathbf{y}^0$ and $\mathbf{x}$ respectively.

The module structures on the endomorphism spaces of $L$, $\cL$ are known: $m_2 (P_i, P_j) = P_{i+j}$ (already descended to cohomology level); $\mathbf{x}^i \cdot \mathbf{x}^j =\mathbf{x}^{i+j}$ (for $i,j>0$), $\mathbf{y}^i \cdot \mathbf{y}^j = \mathbf{y}^{i+j}$ (for $i,j\geq 0$), and $\mathbf{x}^i \cdot \mathbf{y}^j$ (where $i > 0, j \geq 0$) equals to $\mathbf{x}^{i-j}$ for $i>j$ and $\mathbf{y}^{j-i}$ for $i\leq j$.
Since the functor preserves compositions of morphisms, it follows that $P_i$ is mapped to $\mathbf{x}^i$ for $i>0$ and  $\mathbf{y}^{|i|}$ for $i \leq 0$.  Alternatively we can also directly check that the morphism corresponding to $P_i$ is multiplication by $x^i$ if $i \geq 0$ and $y^{|i|}$ if $i<0$, see Figure \ref{fig:morphisms-inf-other}.

\begin{figure}[htb!]
	\centering
	\begin{subfigure}[b]{0.3\textwidth}
		\centering
		\includegraphics[width=\textwidth]{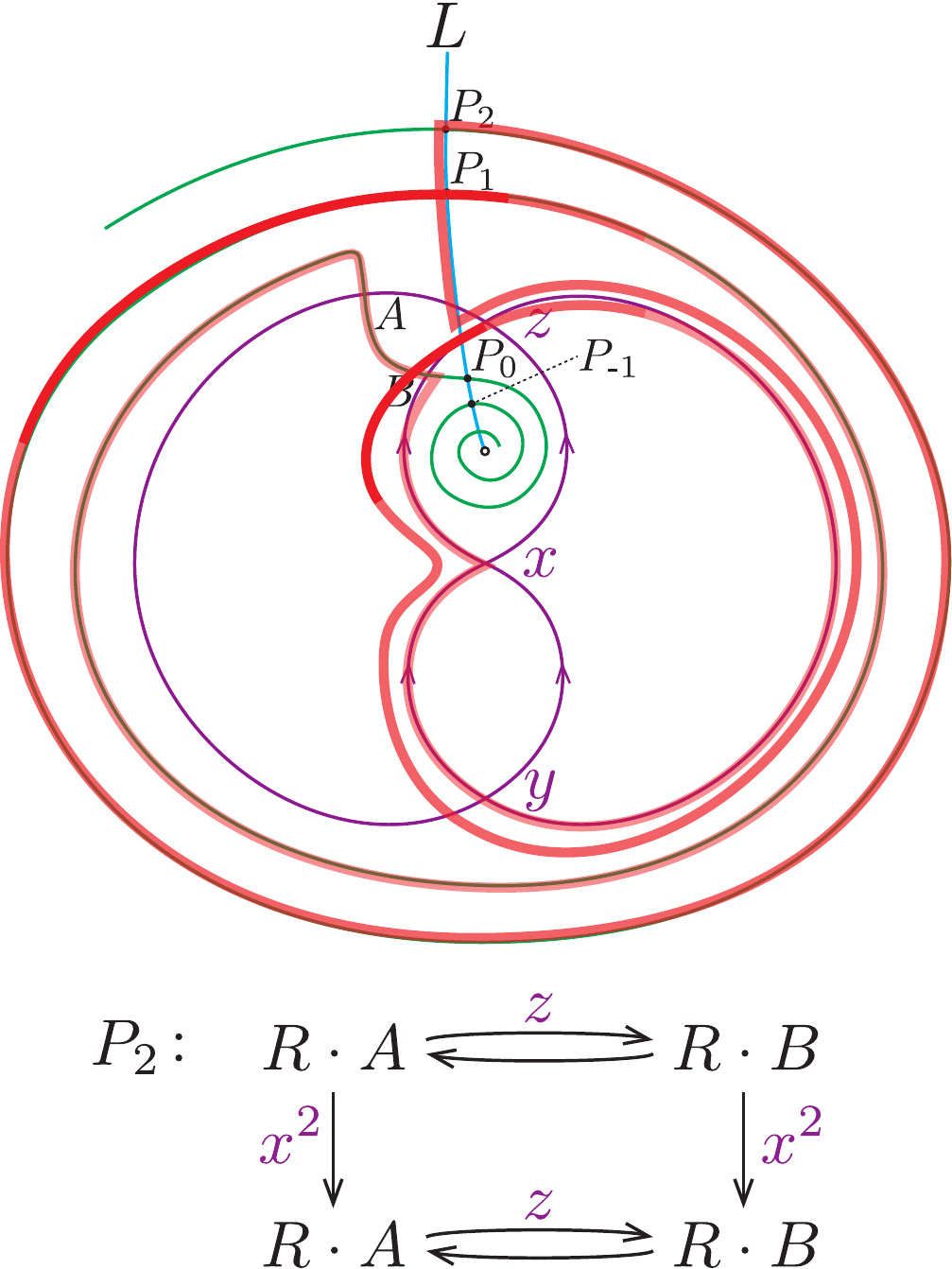}
		\caption{The morphism $P_2$ corresponds to multiplication by $x^2$.}
	\end{subfigure}
	\hspace{10pt}
	\begin{subfigure}[b]{0.55\textwidth}
		\centering
		\includegraphics[width=\textwidth]{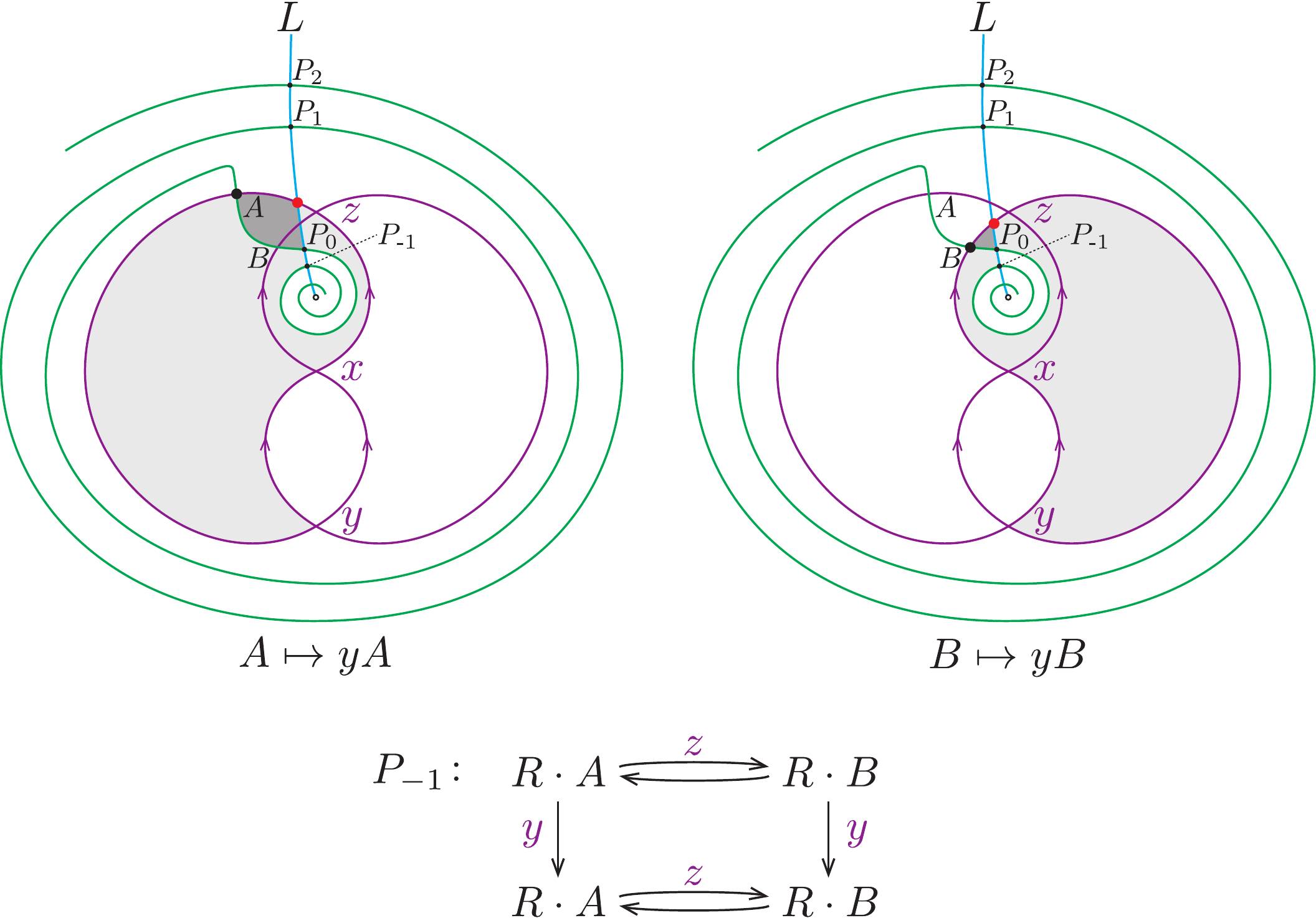}
		\caption{The morphism $P_{-1}$ corresponds to multiplication by y.}
	\end{subfigure}
	\caption{Mirrors of the morphisms $P_2$ and $P_{-1}$.  $R=\C[x_\ex,y_\ex,z_\ex]$.}
	\label{fig:morphisms-inf-other}
\end{figure}

As a consequence we have the following.
\begin{cor}
	The mirror functor derives isomorphisms on endomorphism spaces.
\end{cor}

Next we consider the morphism space between $L$ and another Lagrangian $L'$ (which is also circulating around certain face and winds about the adjacent edges).  $L$ and $L'$ may be circulating around either a compact or a non-compact face of the tropical curve.  There are two cases: the faces that $L$ and $L'$ circulate around are the same, or they are distinct.

\noindent {\bf Case 1:  $L$ and $L'$ circulate around the same face.}
The intersection points between $L$ and $L'$ can occur at any of the boundary edges of the face.  For infinite edges the computation is exactly the same as the above for endomorphisms.  Labeling the edges around the face (assumed to be non-compact for the moment) counterclockwisely, the first and last edges
(which are non-compact) have intersection points $P_i$ for $i\leq 0$ and intersection points $P_i$ for $i>0$ respectively.  As above, $P_i$ is mapped to $\mathbf{x}^i$ for $i>0$ and $\mathbf{y}^{|i|}$ for $i \leq 0$.  Hence it is an isomorphism (on the part corresponding to the infinite edges).

For a compact edge adjacent to the face, let $m$ and $m'$ be the winding numbers of $L$ and $L'$ around this edge respectively.  $L$ and $L'$ are arranged such that they intersect at $|m-m'|$ points in the edge labeled by $H_1,\ldots,H_{m'-m}$ when $m'>m$ and $H_{-1},\ldots,H_{m'-m}$ when $m'<m$.  See Figure \ref{fig:morphisms-fin-01} showing the case $m'>m$.  

For the case $m'>m$, let $S$ be the Seidel Lagrangian over one of the boundary vertices of the edge.  $L$ and $L'$ intersect $S$ at the points $A,B$ and $A',B'$ respectively.  By counting strips shown in Figure \ref{fig:morphisms-fin-01}, we obtain the following.

\begin{lemma} \label{lem:cpt-edge}
	Let $L$ and $L'$ be Lagrangians circulating around the same face, and let $m<m'$ be the winding numbers of $L$ and $L'$ around a compact edge as described above.
	For $i>0$, the image of the morphism $H_i$ from $L$ to $L'$ under the mirror functor is the morphism $\mathrm{Span}\{A',B'\} \to \mathrm{Span}\{A,B\}$, $A' \mapsto x^{i-1} A$, $B' \mapsto x^{i-1} B$.
\end{lemma}

The above morphisms correspond to the sections $x^{i-1} \in H^0(\mathcal{O}_{\bP^1}(m'-m))$ for $i=1,\ldots,m'-m$.  They form a basis of $H(\cL,\cL')$.  Hence it is an isomorphism (on the part corresponding to such a finite edge).

\begin{figure}[h]
	\begin{center}
		\includegraphics[scale=0.4]{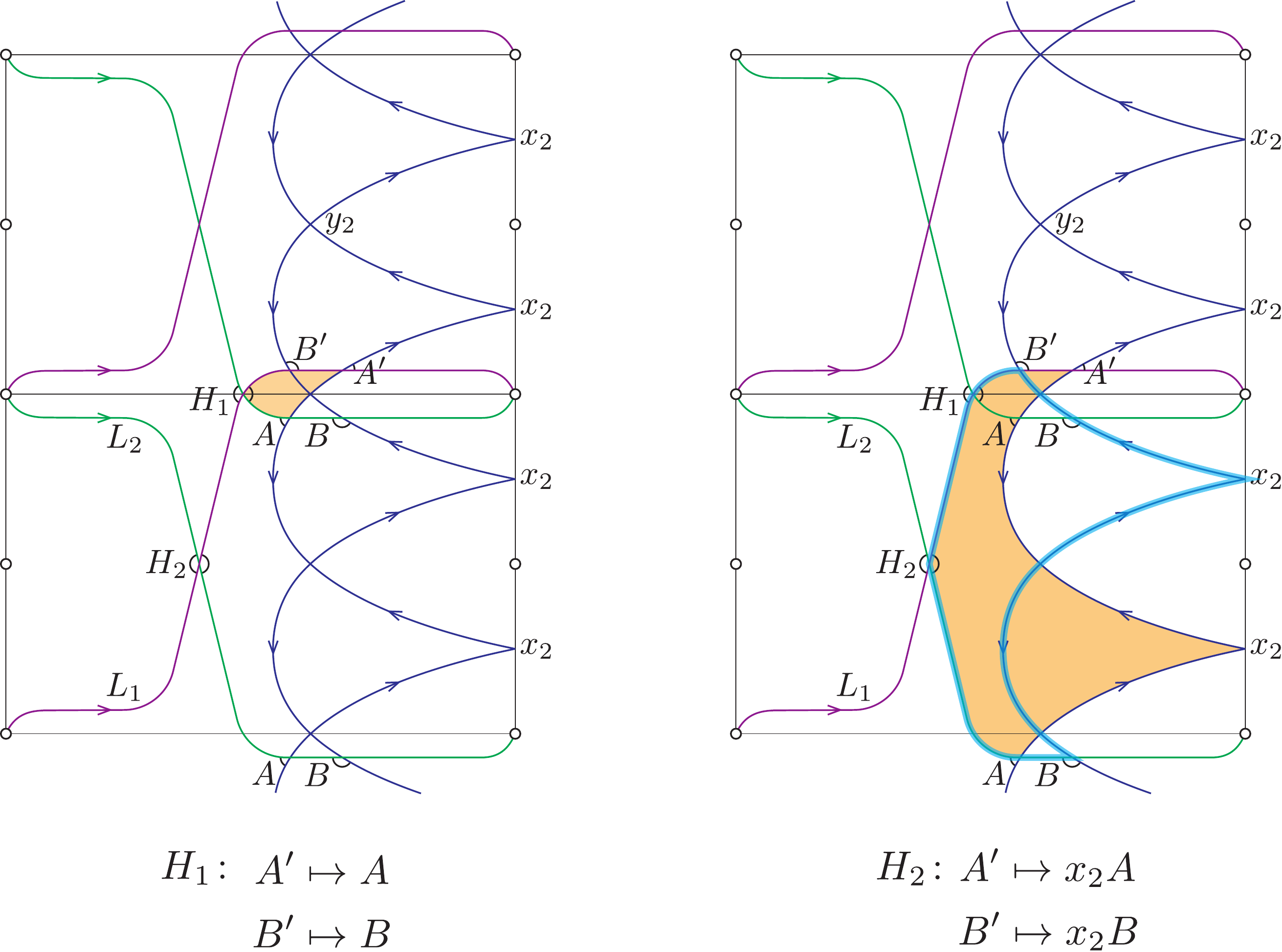}
		\caption{Transforming the morphisms $H_i$ for $i\geq 0$ in a finite edge, where $L$ and $L'$ are circulating around the same face.  In the figure $m=-1$ and $m'=1$.}
		\label{fig:morphisms-fin-01}
	\end{center}
\end{figure}

Now consider the case $m'<m$.  Take $L''$ to be a Lagrangian circulating around the same face as $L$ and winding around the edge $e$ for $m''$ times with $m'' > m > m'$.  For $i=-1,\ldots,m'-m$, we have $m_2(H^{L,L'}_i,H^{L',L''}_j)=H^{L,L''}_{i+j}$ for $j=m-m'+1,\ldots,m''-m'$.  On the mirror side, $H(\cL,\cL')$ has a basis $\{\mathcal{H}^{\cL,\cL'}_i: i=-1,\ldots,m'-m\}$ which has a similar equality for composition: $\mathcal{H}^{\cL',\cL''}_j \circ \mathcal{H}^{\cL,\cL'}_i = \mathcal{H}^{\cL,\cL''}_{i+j}$ for $j=m-m'+1,\ldots,m''-m'$.  The functor preserves compositions on the two sides.  Moreover, as shown above it induces isomorphisms $HF(L',L'') \to H(\cL',\cL'')$  and $HF(L,L'') \to H(\cL,\cL'')$ which send $H^{L',L''}_j$ to $\mathcal{H}^{\cL',\cL''}_j$ and $H^{L,L''}_{k}$ to $\mathcal{H}^{\cL,\cL''}_{k}$ respectively.  As a result $H^{L,L'}_i$ must be sent to  $\mathcal{H}^{\cL,\cL'}_i$ for  $i=-1,\ldots,m'-m$.  Hence it is an isomorphism.

\noindent {\bf Case 2:  $L$ and $L'$ circulate around two different faces.}
Two distinct faces intersect at most along one edge (and they must intersect along one edge if they intersect, since all vertices are trivalent).  If the two faces are not adjacent, morphism spaces on both sides are zero, and the induced map is just zero.  Thus we only need to consider adjacent faces.  The adjacent faces may intersect at either a finite or an infinite edge.  

First consider the case of an infinite edge.  The morphism space from $\phi(L)$ to $L'$ is spanned by $Q_j$ for $j\geq 0$ (which are all of odd degree), and that from $\cL$ to $\cL'$ is spanned by $\mathcal{Q}_j$ for $j\geq 0$.
We can compute explicitly (by counting strips similar to Lemma \ref{lem:mor-inf-edge}) that the image of $Q_0$ under our functor is the morphism $\mathcal{Q}_0$ (which sends $A$ to $B'$ and $B$ to $x\cdot A'$).  We already know that our functor preserves composition of morphisms.  We have $m_2(P_i,Q_0)=Q_i$ for $i>0$ (where $P_i$ are endomorphisms of $L$ defined before Lemma \ref{lem:mor-inf-edge}).  Moreover, $\mathbf{x}^i \cdot \mathcal{Q}_0 = \mathcal{Q}_i$ for $i>0$ (where $\mathbf{x}^i$ are the endomorphisms of $\cL$ right after Lemma \ref{lem:mor-inf-edge}).  Since $P_i$ is mapped to $\mathbf{x}^i$, $Q_i$ must be mapped to $\mathcal{Q}_i$ under our functor.  Hence it is an isomorphism on morphism spaces.

Now consider the case of a finite edge.  $L$ and $L'$ intersect at $|m'-m|$ points in the edge labeled by $H_1,\ldots,H_{m'-m}$ when $m'>m$ and $H_{-1},\ldots,H_{m'-m}$ when $m'<m$.  See Figure \ref{fig:morphisms-fin-diff-01} which shows the case $m'>m$.  Take the Seidel Lagrangian over one of the vertices adjacent to the edge.

For $m'>m$, the images of $H_i$ under our functor are the morphisms $\mathcal{H}_i$, whose restriction to the formal deformation space of $S$ is $A' \mapsto x^i_\ex A,  B' \mapsto x^{i-1}_\ex B$.  See Figure \ref{fig:morphisms-fin-diff-01}.  $\mathcal{H}_i$ form a basis of the morphism space from $\cL$ to $\cL'$ in $D\MF(W)$, and hence we have an isomorphism between the morphism spaces.

\begin{figure}[h]
	\begin{center}
		\includegraphics[scale=0.35]{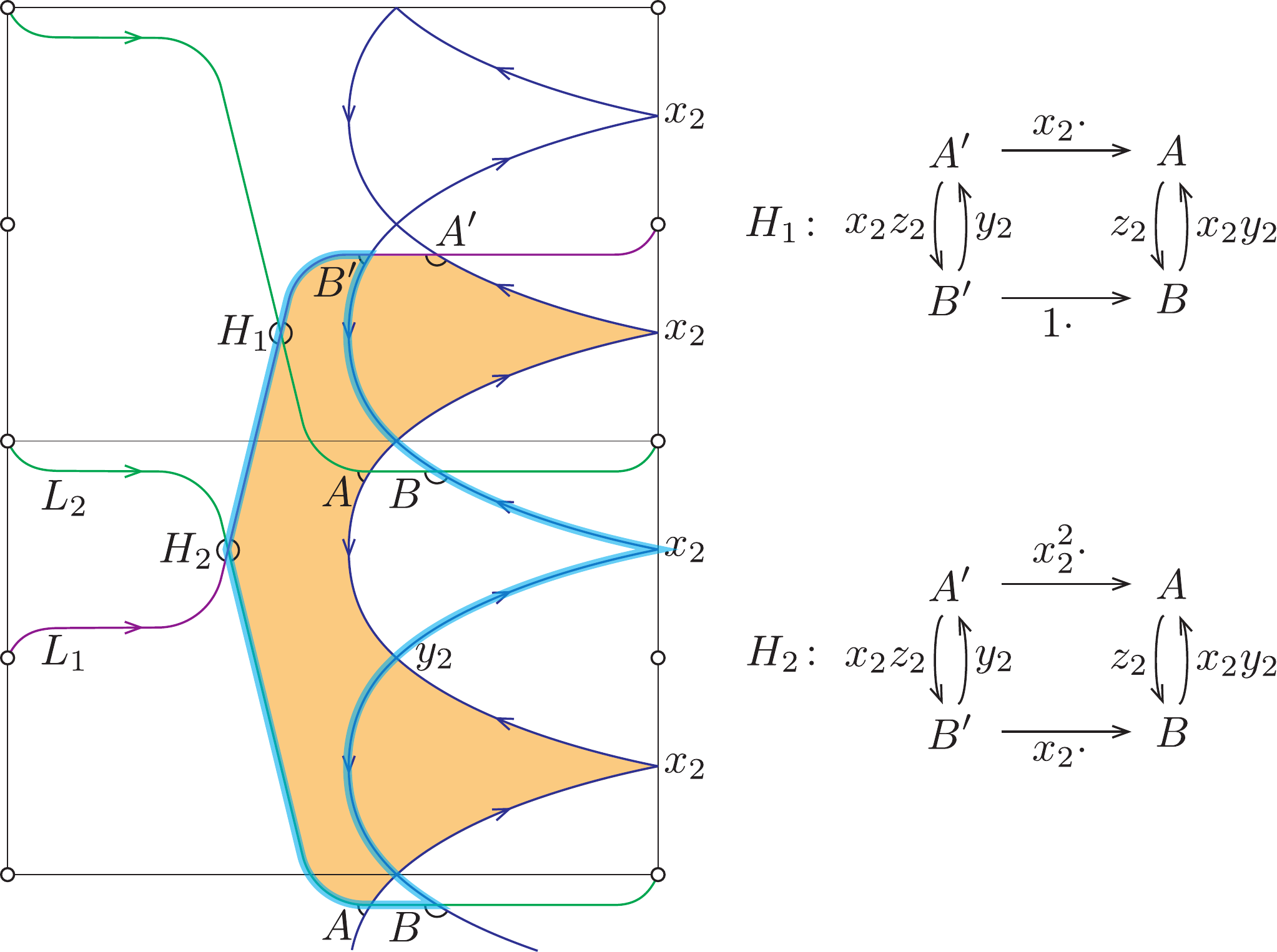}
		\caption{Transforming the morphisms $H_i$ for $i\geq 0$ in a finite edge, where $L$ and $L'$ are circulating around different faces.  In this figure $m=-1$ and $m'=1$.}
		\label{fig:morphisms-fin-diff-01}
	\end{center}
\end{figure}

For the case $m'<m$, the same argument as the last paragraph in Case 1 shows that $H_i$ maps to $\mathcal{H}_i$ and hence it is an isomorphism.  Combining all the cases, we conclude the following.

\begin{thm}
	The functor induces isomorphisms on morphism spaces.
\end{thm}

From Remark \ref{rmk:Lee}, the derived functor sends a generating set of objects in $D\mathrm{WFuk}(X)$ to a generating set of objects in $D\MF(Y(\C))$.  Moreover, it is an isomorphism on morphism level.  Thus the functor derives a quasi-equivalence between $D\mathrm{WFuk}(X)$ and $D\MF(Y(\C))$ as stated in Theorem \ref{thm:surf}.

\section{Relation to stability conditions and flops}\label{sec:flop}
Given a punctured Riemann surface, we can take different choices of pair-of-pants decompositions.  It is related to the choice of quadratic diferentials and stability conditions.  Below we discuss an example of the 4-punctured sphere.  

We shall see that taking a different pair-of-pants decomposition induces the Atiyah flop on the mirror side.  This section is more expository and we do not intend to give a systematic study of stability conditions in this paper.  Stability conditions for punctured Riemann surfaces were studied by Haiden-Katzarkov-Kontsevich \cite{HKK}.  We wish to understand the relation with the work of \cite{HKK} which is left for future investigation.

\begin{figure}[htb!]
    \includegraphics[scale=0.45]{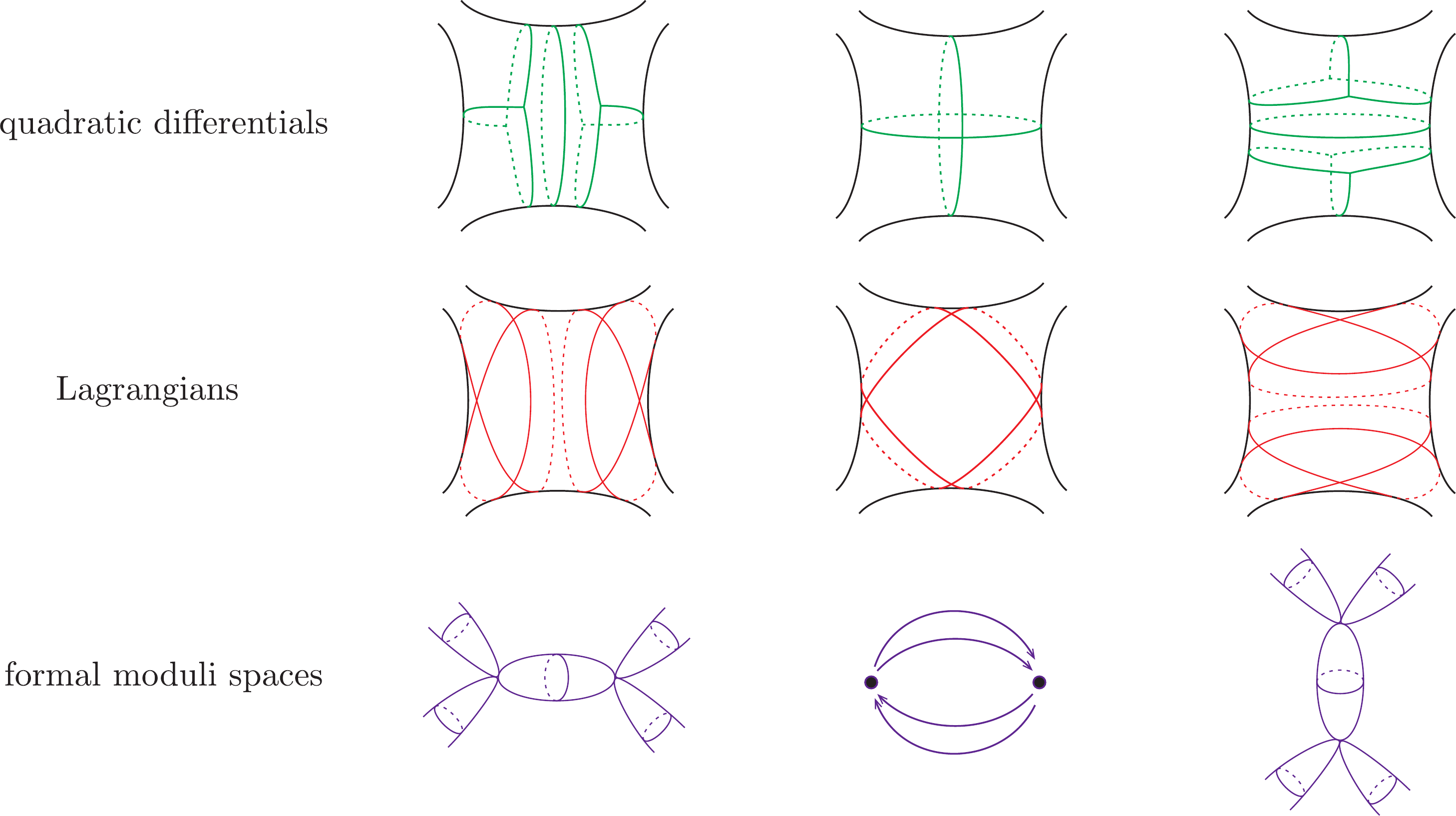}
    \caption{Different choices of collections of immersed Lagrangians leads to different strata of the K\"ahler moduli in the mirror side.
    	different moduli related by flops.}
		\label{fig:flop}
\end{figure}
Figure \ref{fig:flop} depicts two different pair-of-pants decompositions and the corresponding Seidel Lagrangians.  In the middle the two Seidel Lagrangians are merged together to  form a more degenerate immersed Lagrangian (which is a union of two circles).  

Non-commutative homological mirror symmetry for the more degenerate immersed Lagrangian was studied in our previous work \cite{CHL2}.  The resulting mirror is the non-commutative resolution of the conifold corresponding to a quiver (together with a superpotential). 

In \cite{FHLY} we studied ``flop" on a Lagrangian fibration for $T^* \bS^3$, which results in the Atiyah flop on the mirror resolved conifold.  If we take the union of two certain $S^3$'s in $T^* \bS^3$ as a reference Lagrangian, then we produced the non-commutative resolution of the conifold.  
It is a three-dimensional analog of the construction in this section.

%Recall that a special Lagrangian (of phase zero) is  a smooth map $\iota: \hat{L}^n \to M^{2n}$ whose differential is injective everywhere except at zeros of the quadratic top-form, with $\iota^* \omega = 0$ and $\iota^* \mathrm{Re} \Omega = 0$.  By a domain-dependent Hamiltonian isotopy the map $\iota$ is made to avoid the zeros of the quadratic top-form such that its differential is everywhere injective, and it only self-intersects at transverse points.  Then any Lagrangian which is isomorphic to such $\iota$ is said to be stable.
%
%Take the four-punctured sphere as an example.   {\color{red} We will see that the corresponding moduli spaces are related by a flop. In fact, the mirror obtained from deformations of any of (pairs of) Lagrangians in Figure \ref{fig:flop} is a crepant resolution (possibly noncommutative) of a conifold singularity, and the potential functions are compatible in a suitable way. In particular, the corresponding B-model categories are all equivalent, which is natural since we do not change the symplectic structure on A-side.}
%

Let $\bL$   be a union of two circles in $\Sigma:=S^2 \setminus \{\textnormal{4 points}\}$.
We denote four immersed points of $\bL$ by $X,X',Y,Z$ as in Figure \ref{fig:surg}. More precisely, if $\bL= S_1 \oplus S_2$, then we have the following eight immersed generators
$$X,X',\bar{Y},\bar{Z} \in CF(S_1,S_2),  \quad \bar{X},\bar{X'},Y,Z \in CF(S_2,S_1) $$
as well as generators $e_1,[\pt]_1$ and $e_2, [\pt]_2$ from $H(S_1)$ and $H(S_2)$, respectively.

Now, let us consider the following formal deformations of $\bL$ at $X$ and $X'$ respectively
$$\bL_1:= (\bL,T^{\delta'} X'), \quad \bL_2:= (\bL, T^{\delta} X),$$
for some $\delta, \delta' \in \mathbb{R}_{>0}$.
One can check that their self-Floer cohomologies are 8 dimensional. For example $HF(\bL_1,\bL_1)$ has a basis consisting of 
$$\big\{ e_1 + e_2, X,Y,Z, \bar{Y},\bar{Z},\bar{X} \pm \bar{X'}, [\pt]_1 \pm [\pt]_2 \big\}.$$  

Although $\bL_1, \bL_2$ are formal deformations, one can find isomorphisms to two Seidel Lagrangians which may be obtained
as an actual surgery as in Figure \ref{fig:surg}. Moreover, the pair-of-circles on the middle of $\Sigma$ can be
obtained from the further formal deformations.
$$(\mathbb{L}_1,T^{\delta} X) \cong (\mathbb{L}_2,T^{\delta'} X') \cong (\mathbb{L}, X + X'),$$
These are left as exercises as we will not use them.

\begin{figure}[h]
\begin{center}
\includegraphics[height=1in]{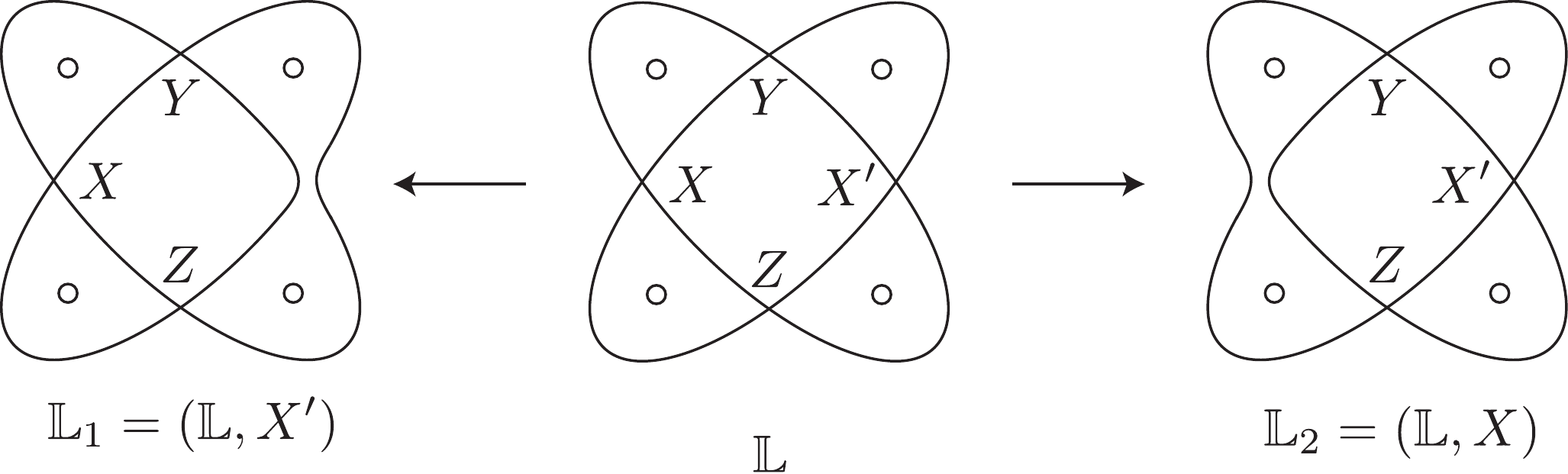}
\caption{two Lagrangians obtained by opposite corners in $\bL$}\label{fig:surg}
\end{center}
\end{figure}

Instead of using Seidel Lagrangians of the pair of pants, we show how to use $\bL_1,\bL_2$ to construct mirrors.\footnote{Readers are warned that the deformation parameters for $\bL$ were taken from a certain quiver algebra in \cite{CHL2}, and we are looking at slightly different boundary deformation of $\bL$, here.}
Let us consider further formal deformations of  $\bL_1$ and $\bL_2$ by $b=xX+yY+zZ$ and $b'=x'X'+y'Y+z'Z$ respectively. It is easy to check that $b$ and $b'$ solve weak Maurer-Cartan equation for any $(x,y,z) \in \Lambda_+^3$ and $(x',y',z') \in \Lambda_+^3$, and the corresponding potentials are $W_1=T^{\alpha + \delta'} xyz$ and $W_2=T^{\alpha +\delta} x'y'z'$, respectively where $\alpha$ is the symplectic area of the rectangle with four corners $X,Y,X',Z$. For instance, $m^{\bL_1}(e^b)$ admits contributions from two rectangles symmetric with respect to the equator, and the following gives one of cancelling pairs in $m^{\bL_1} (e^b)$: 
$$m_3^{\bL} (T^{\delta'} X', yY, xX ) = T^{\alpha + \delta'} xy \bar{Z},  \quad m_3^{\bL} (xX, yY, T^{\delta'} X' ) = - T^{\alpha + \delta'} yx\bar{Z}$$
since $x$ and $y$ commute with each other.

Let us first look at the case $y=z=y'=z'=0$ (i.e. deforming $\bL_1$ and $\bL_2$ only by $xX$ and $x'X'$) and seek for the condition in order for $(\bL_1,xX)$ and $(\bL_2,x'X')$ to be isomorphic to each other. Obviously, if $x=T^\delta$ and $x'=T^{\delta'}$, then these objects are isomorphic, but we rather want to have isomorphic \emph{family} of objects.

Observe that $(\bL_1, xX) \cong (\bL, T^{\delta'} X' + xX)$ and $(\bL_2,x'X') \cong (\bL, T^{\delta}X + x'X)$ by the definition of boundary deformation of $A_\infty$-algebras, and hence 
$$CF((\bL_1, xX) , (\bL_2,x'X')) = CF ((\bL, T^{\delta'} X' + xX), (\bL, T^{\delta} X + x'X')).$$
%If these two objects are isomorphic, the Floer differential $\delta$ should vanish identically.

\begin{prop}
$(\bL_1,xX)$ and $(\bL_2,x'X')$ are isomorphic for $x x' = T^{\delta + \delta'}$ (with $x\neq 0$ and $x' \neq 0$).
\end{prop}

Thus the gluing region in this case is given by $\left\{ x \mid  \val(x) \leq (\delta + \delta') \right\}$.

\begin{proof}
The Floer differential $d$ on $CF ((\bL, X' + xX), (\bL,X + x'X'))$ can be computed as follows:
\begin{eqnarray*}
d(e_1)&=& m_2 ( e_1, T^{\delta} X) + m_2 ( e_1, x'X') =  T^{\delta} X +x' X',\\
d(e_2)&=& m_2 (T^{\delta'} X', e_2) + m_2 (xX, e_2) =  -T^{\delta'} X' - xX,\\
d(Y)&=& m_3 (xX, Y, x'X') + m_3 ( T^{\delta'}X' , Y,  T^{\delta} X) =  (x x' - T^{\delta+ \delta'}) \bar{Z},\\
d(Z)&=& m_3 (xX, Z, x'X') + m_3 ( T^{\delta'} X', Z, T^{\delta} X) =  (x x' - T^{\delta+ \delta'}) \bar{Y},\\
d(\bar{X})&=& m_2 (xX, \bar{X}) + m_2 (  \bar{X}, T^{\delta} X) =  x[\pt]_1 \pm T^{\delta} [\pt]_2,\\
d(\bar{X'})&=& m_2 (\bar{X'}, x'X') + m_2 ( T^{\delta'}X', \bar{X'}) =  x'[\pt]_2 \pm T^{\delta'} [\pt]_1,
\end{eqnarray*}
and $d(\bar{Y}) = d(\bar{Z}) = d([\pt]_1) = d([\pt]_2)=0$.

Therefore $\alpha:=x T^{-\delta} e_1 + e_2 \in CF((\bL_1, xX) , (\bL_2,x'X')) $ and $\beta:= x^{-1} T^{\delta} e_1 +e_2  \in CF((\bL_2,x'X'), (\bL_1, xX)) $ serve as isomorphisms 
%provided that $x x' =1$, 
since they are $d$-closed and
$m_2 ( \alpha, \beta) =e_1 + e_2$
due to unital property (and similar for $m_2 (\beta,\alpha)$).

\end{proof}

If one considers full boundary deformations of $\bL_1$ and $\bL_2$ (i.e. $(\bL_1, b)$ and $(\bL_2,b')$), then one additionally have the condition $xyz=x'y'z'$ since their potentials should match. For example, one can choose the coordinate change to be
\begin{equation}\label{eqn:beforeF}
\left\{
\begin{array}{c}
x'= x^{-1} T^{\delta + \delta'} \\
y'= x y T^{-\delta}\\
z' = x z T^{-\delta}
\end{array}\right..
\end{equation}
In fact, if we compute the resulting Floer differential $\WT{d}$ on  $CF ((\bL, T^{\delta'} X' + xX + y Y + z Z_), (\bL, T^{\delta} X + x'X' + y' Y + z' Z))$, we see that
\begin{eqnarray*}
\WT{d} (e_1)&=& m_2 ( e_1, T^{\delta} X) + x' m_2 ( e_1, X') + y m_2 ( Y,e_1)+  z m_2 ( Z, e_1) \\
&=& T^{\delta} X +x' X' - y Y + zZ,\\
\WT{d} (e_2)&=& m_2 (X', e_2) + x m_2 (T^{\delta'}X, e_2) +  y' m_2 ( e_2, Y)+ z' m_2 (  e_2,Z)   \\
&=&  -T^{\delta'} X' - xX + y' Y - z' Z.
%\\ \WT{\delta} (Y)&=& m_3 (xX, Y, x'X') + m_3 ( X', Y, X) =  (x x' - 1) \bar{Z},\\
%\WT{\delta} (Z)&=& m_3 (xX, Z, x'X') + m_3 ( X', Z, X) =  (x x' - 1) \bar{Y},\\
%\WT{\delta} (\bar{X})&=& m_2 (xX, \bar{X}) + m_2 (  \bar{X}, X) =  x[\pt]_1 \pm [\pt]_2,\\
%\WT{\delta} (\bar{X'})&=& m_2 (\bar{X'}, x'X') + m_2 ( X', \bar{X'}) =  x'[\pt]_2 \pm [\pt]_1,
\end{eqnarray*}
Therefore, $\WT{d} ( \alpha) = \WT{d} (x T^{-\delta} e_1 +e_2) = 0$ gives rise to the above relations.

Let us reduce the coefficients to $\C$ as discussed in Section \ref{sec:C}, and explain how flop appears.
We can play the same game with $(Y,Z)$ instead of $(X,X')$ as in Figure \ref{fig:flop2}. By the same argument as above, the corresponding mirror in this case has two $\C^3$ patches parametrized by $(y,x,x')$ and $(z,\tilde{x}, \tilde{x}')$ respectively, and they glue by the relation
\begin{equation}\label{eqn:afterF}
\left\{
\begin{array}{c}
z= y^{-1} \\
\tilde{x} = y x\\
\tilde{x}' = y x'
\end{array}\right. .
\end{equation}

\begin{figure}[h]
\begin{center}
\includegraphics[height=3.5in]{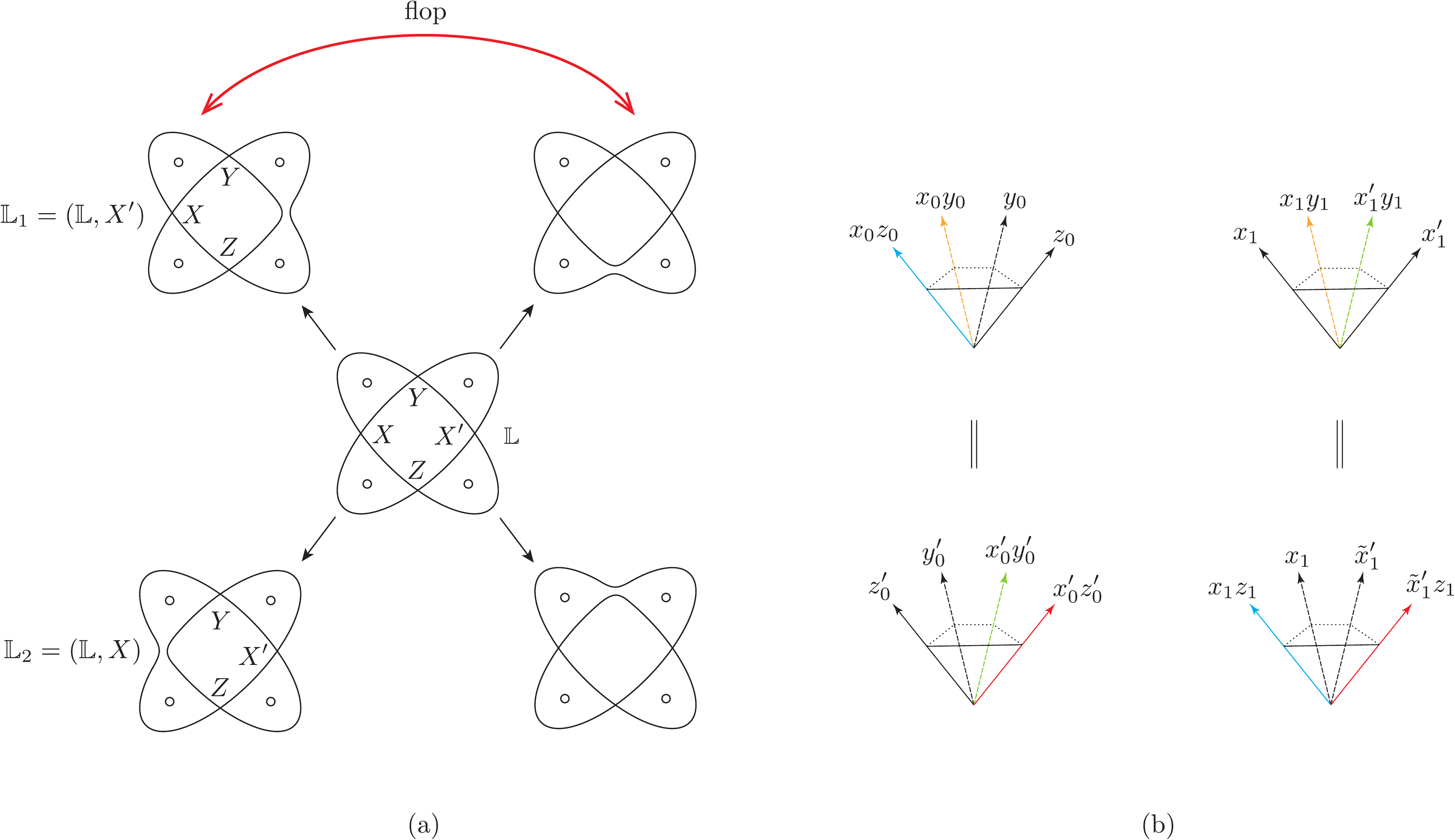}
\caption{two different surgeries of $\BL$ are related by flop.}\label{fig:flop2}
\end{center}
\end{figure}

To distinguish two different spaces \eqref{eqn:beforeF} and \eqref{eqn:afterF}, we put subindices to each set of coordinates in the following way.
\begin{equation}\label{eqn:befaftF}
\left\{
\begin{array}{c}
x'_0= x_0^{-1} \\
y'_0= x_0 y_0\\
z'_0 = x_0 z_0
\end{array}\right.
\qquad 
\left\{
\begin{array}{c}
z_1= y^{-1}_1 \\
\tilde{x}_1 = y_1 x_1\\
\tilde{x}'_1 = y_1 x'_1
\end{array}\right. .
\end{equation}
In terms of these coordinates, one has a natural birational map
\begin{equation}\label{eqn:flopcoordexp}
 (x_0,y_0,z_0) \mapsto (y_1,x_1,x_1')=(y_0 z_0^{-1}, x_0 z_0, z_0)
 \end{equation}
which identifies two spaces away from the zero section $\bP^1$ of $\cO (-1) \oplus \cO(-1)$ (given by $\{y_0=z_0=0\}$ and $\{x_1=x_1'=0\}$). The map extends to other coordinate charts via \eqref{eqn:befaftF}. To see this map is naturally induced, one can contract $\bP^1$ in both of spaces, and compare coordinates of the resulting conifolds. As in (b) of Figure \ref{fig:flop2}, the map \eqref{eqn:flopcoordexp} identifies coordinate axes of the toric diagrams in a natural way. Here, $(x_0 y_0,z_0,x_0 z_0,y_0) = (y'_0,x'_0 z'_0, z'_0, x'_0 y'_0)$ serves as coordinates for one the conifolds, and so does $(y_1 x_1, x'_1, y_1x'_1, x_1) = (\tilde{x}_1, z \tilde{x}'_1, \tilde{x}'_1, z_1\tilde{x}_1)$. Observe that the potentials on both spaces are compatible with \eqref{eqn:flopcoordexp}. (Both of the potentials vanish over the zero sections of $\cO (-1) \oplus \cO(-1)$.) Therefore, two different mirrors are related by the Atiyah flop on the underlying space of the LG model. 

On the other hand, one can think of these two mirrors as being obtained by gluing formal deformation spaces of two different pairs of Seidel Lagrangians as drawn in (a) of Figure \ref{fig:flop2}. This is because, if we boundary-deform $\bL$ by precisely one angle, then it smoothes out to a Seidel Lagrangian, or more precisely, the resulting object is isomorphic to a Seidel Lagrangian.

Notice that each pair of Seidel Lagrangians determines a unique pair-of-pants decomposition of the 4-punctured sphere. Indeed, if we push these Lagrangians two either ends as much as possible, each of them sits in exactly one pair-of-pants in the corresponding decomposition of the 4-punctured sphere. We conclude that taking a different pair-of-pants decomposition induces the Atiyah flop on the mirror side.

%\subsection{Comparison with 3-dimensional picture}
%
%This matches precisely with the picture of a (deformed) conifold. Namely, $\BL$ above corresponds to the union of two 3-dimensional spheres in the deformed conifold, and the surgery at each degree one immersed point produces one of singular torus fibers which are  immersed $S^2 \times S^1$'s. Figure \ref{fig:3dbf} shows the process of surgeries at $X$ and $X'$ to obtain stable (singular) torus fibers before the flop. 
%
%\begin{figure}[h]
%\begin{center}
%\includegraphics[height=3in]{3dbf}
%\caption{3-dimensional analogue (before flop)}\label{fig:3dbf}
%\end{center}
%\end{figure}
%\begin{figure}[h]
%\begin{center}
%\includegraphics[height=4in]{3daf}
%\caption{3-dimensional analogue (after flop)}\label{fig:3daf}
%\end{center}
%\end{figure}
%
%If we smooth out $Y$ and $Z$ instead, then we get another set of singular Lagrangian tori, which become stable after applying flop. See Figure \ref{fig:3daf}.

\section{Proof of isomorphisms in $\AI$-category}\label{sec:pfgluethmy}
In this section, we give a proof of Theorem \ref{thm:ye}, which claims that  Yoneda functors corresponding to $L_0$ and $L_1$ are quasi-isomorphic. Here, we assume that two objects  $L_0$ and $L_1$ of $\mathcal{C}$ are isomorphic via $\alpha \in \Hom_{\mathcal{C}} (L_0,L_1)$ and $\beta \in \Hom_{\mathcal{C}} (L_1,L_0)$ (see Definition \ref{def:iso}). 
 Since Yoneda embedding is fully faithful, this will imply that the two objects are quasi-isomorphic in the original category.

First, we remark that  isomorphisms form equivalence relations in $\mathcal{C}$.
Consider three objects $L_0,L_1,L_2$ in an $A_\infty$-category $\mathcal{C}$, and suppose we are given morphisms 
\begin{equation*}
\xymatrix{ L_0 \ar@<0.5ex>[r]^{\alpha} & L_1  \ar@<0.5ex>[r]^{\gamma} \ar@<0.5ex>[l]^{\beta} & L_2 \ar@<0.5ex>[l]^{\delta}}
\end{equation*}
such that all of them are $m_1$-closed and $m_2 (\alpha,\beta)= id+ m_1 (x)$, $m_2 (\gamma,\delta)= id+  m_1 (y)$ for some $x$ and $y$. Then  
\begin{lemma}\label{lem:isocomp}
 $m_2(\alpha,\gamma)$ and $m_2(\delta,\beta)$ are isomorphisms.
\end{lemma}

%We give the proof of lemma \ref{lem:isocomp}.
\begin{proof}
Set $m_2 (\gamma,\delta) = id + m_1(z)$. We compute $m_2( m_2(\alpha,\beta),m_2(\gamma,\delta))$ to see that it is the identity up to an image of $m_1$:
\begin{eqnarray*}
m_2( \overbrace{m_2(\alpha,\gamma)}^{\star_1},m_2(\overbrace{\delta}^{\star_2},\overbrace{\beta}^{\star_3})) &=& - m_2( m_2(m_2(\alpha,\gamma),\delta),\beta) - m_1 \circ m_3 ( m_2(\alpha,\gamma ), \delta, \beta) \\
&=&  m_2 (m_2 ( \alpha, m_2(\gamma,\delta) ), \beta) \mod \,\, \mathrm{Im} \, m_1\\
&=& m_2 (\alpha, \beta) +  m_2 ( m_2 (\alpha, m_1(z)), \beta)  \mod \,\, \mathrm{Im} \, m_1 \\
&=& id - m_2 (m_1 \circ m_2 (\alpha,z), \beta)   \mod \,\, \mathrm{Im} \, m_1 \\
&=& id - m_1 \circ m_2 ( m_2 (\alpha,z),\beta)   \mod \,\, \mathrm{Im} \, m_1 \\
&=& id    \mod \,\, \mathrm{Im} \, m_1.
\end{eqnarray*}
Computation for $m_2$ in the other direction is similar. 
Thus we see that two morphisms $m_2(\alpha,\gamma)$ and $m_2(\delta,\beta)$ compose to give the identity up to an image of $m_1$. 
\end{proof}

%We begin by reviewing the notion of  Yoneda functor for an $A_\infty$-category, briefly.

In order to justify Definition \ref{def:iso}, we need to consider Yoneda embeddings into chain complexes $\mathcal{CH}_{A_\infty}$
(with sign convention of dg-category of chain complexes as in  \ref{def:dgsign}).
\begin{defn}
The Yoneda functor of $L_0$ by $\mathcal{Y}^{0}: \mathcal{C} \mapsto \mathcal{CH}_{A_\infty}$ is defined as follows.
On the object level, we have
$$\CY^0(C) = \big( \Hom_{\mathcal{C}} (C,L_0),  -m_1\big).$$
On the morphism level, for $a_{i,i+1} \in  \Hom_{\mathcal{C}} ( C_i,C_{i+1}) $, $i=0,\ldots,k-1$, $$\CY^0_k(a_{01},\cdots,a_{(k-1)k}) :\Hom_{\mathcal{C}} (C_k,L_0) \to \Hom_{\mathcal{C}} (C_0,L_0)$$
%(i.e. $(\CY^0_k(a_{01},\cdots,a_{(k-1)k}) \in \Hom_{\mathcal{CH}_{dg}} (\CY^0 (C_k), \CY^0 (C_0) )= \Hom_{\mathcal{CH}_{A_\infty}} (\CY^0 (C_0), \CY^0 (C_k)$)
is defined to be
$$ \bullet \mapsto m_k(a_{01},\cdots,a_{(k-1)k},\bullet).$$
\end{defn}
One can check that  $\CY^0$ is an $A_\infty$-functor as in Theorem \ref{lem:familyyoneda}.
The Yoneda functor $\CY^1$ for $L_1$ is defined in the same way. Our local mirror functor $\mathcal{F}^{\bL_i}$ is nothing but a curved family version of the above, where $L_0$ is replaced by a family of weakly unobstructed objects $(\bL_i,b_i)$ with $b_i$ varying over the (weak) Maurer-Cartan space.

Now, we can compare $L_0$ and $L_1$ by comparing their Yoneda functors.

\begin{lemma}
We have a pre-natural transformation (Appendix \ref{def:nattrans}) $N_{01}$ from $\CY^0$ to $\CY^1$ induced by a morphism $\beta \in \Hom(L_1,L_0)$.
\end{lemma}
Let us define $N_{01}$ first.
\begin{defn}
To a given object $C_0$ in $\mathcal{C}$, 
we assign an element $N_{01} (C_0)$ in 
$ \Hom_{\mathcal{CH}_{A_\infty}}(\CY^0(C_0),
\CY^1(C_0)) =  \Hom_{\mathcal{CH}_{dg}}(\CY^1(C_0),
\CY^0 (C_0))$, 
defined as
$$N_{01}(C_0):\Hom_{\mathcal{C}} (C_0,L_1) \to \Hom_{\mathcal{C}} (C_0,L_0): \bullet \mapsto  {(-1)^{|\bullet|}} m_2(\bullet, \beta).$$
Moreover,
$$N_{01}(C_0,C_1,\ldots,C_k):\Hom_{\mathcal{C}} (C_0,C_1)\times \cdots \times \Hom_{\mathcal{C}} (C_{k-1},C_k)\to \Hom_{\mathcal{CH}_{A_\infty}} (\CY^0(C_0),
\CY^1(C_k))$$
is defined as the map 
$$ (a_{01},\cdots,a_{(k-1)k}) \mapsto  { (-1)^{|\mathbf{a}|'}} { (-1)^{|\bullet|}}m_k(a_{01},\cdots,a_{(k-1)k}, \bullet, \beta)$$
for $\bullet \in \Hom_\mathcal{C} (C_k, L_1)$.
\end{defn}

Note that an element in $ \Hom_{\mathcal{CH}_{A_\infty}} (\CY^0(C_0),
\CY^1(C_k))$ should be a map from $\Hom_\mathcal{C} (C_k, L_1) $ to $\Hom_\mathcal{C} (C_0, L_0)$ due to our convention \eqref{eqn:convadg}. The degree of $N_{01}$ is given by $||N_{01}|| = 0$ since $|\alpha|' = |\beta|' = -1$ and $|m_k|'=1$, and hence $||N_{01}||'=-1$.

%The degree of $N_{01}$ is given by $||N_{01}|| = 0$: $|\alpha|' = |\beta|' = -1$ and $|m_k|'=1$, and hence $||N_{01}||'=-1$.

We have the following lemmas whose proofs will be given in short.

\begin{lemma}\label{lem:NM1close}
If $m_1(\beta)=0$, then $N_{01}$ is a natural transformation i.e., $M_1$-closed. (See 
Appendix \ref{def:nattrans}.) 
\end{lemma}

Similarly, one can define $N_{10}$ which is a natural transformation from $\CY^1$ to $\CY^0$ for a given $m_1$-closed morphism $\alpha \in \Hom_{\mathcal{C}} (L_1, L_0)$.

Here is the proof of theorem \ref{thm:ye}.
\begin{proof}
The composition $M_2( N_{01}, N_{10})$ is a natural transformation from $\CY^0$ to $\CY^0$ which is given explicitly as
%
%$$N_{01}\circ N_{10} (a_{01},\cdots,a_{(k-1)k}):= 
%\sum m(\beta,m(\alpha,\bullet,a_{01},\cdots),a_i,\cdots,a_{(k-1)k}). $$
\begin{equation}\label{eqn:M2twonat}
\begin{array}{lcl}
M_2 (N_{01}, N_{10}) (\mathbf{a})(\bullet) &=&
\sum (-1)^{||N_{10}||' \cdot |\mathbf{a}^{(1)}|'} m_2^{\mathcal{D}} (N_{01} (\mathbf{a}^{(1)} ), N_{10} (\mathbf{a}^{(2)} )) (\bullet) \\
&=& \sum (-1)^{ |\mathbf{a}^{(1)}|'} \overbrace{(-1)^{|N_{01} (\mathbf{a}^{(1)})|}}^{ {\rm dg} \leftrightarrow A_\infty\textnormal{-dg} } N_{01} (\mathbf{a}^{(1)}) \circ N_{10} (\mathbf{a}^{(2)}) (\bullet)\\
&=&  \sum (-1)^{|\bullet|} (-1)^{|\mathbf{a}^{(2)}|'} N_{01} (\mathbf{a}^{(1)}) \left( m^{\mathcal{C}} (\mathbf{a}^{(2)} ,\bullet, \alpha) \right)\\
&=& \sum (-1)^{|\bullet|}  (-1)^{|\mathbf{a}^{(2)}|'}    (-1)^{| m^{\mathcal{C}}(\mathbf{a}^{(2)} ,\bullet, \alpha)|}  (-1)^{  |\mathbf{a}^{(1)}|'} m^{\mathcal{C}}(\mathbf{a}^{(1)},m^{\mathcal{C}} (\mathbf{a}^{(2)} ,\bullet, \alpha),\beta) \\
&=& \sum (-1)^{|\mathbf{a}^{(1)}|'} m^{\mathcal{C}}(\mathbf{a}^{(1)},m^{\mathcal{C}} (\mathbf{a}^{(2)} ,\bullet, \alpha),\beta) \\
\end{array}
\end{equation}
Since $m^{\mathcal{C}}(\mathbf{a}^{(2)} ,\bullet, \alpha)| = |m^{\mathcal{C}}(\mathbf{a}^{(2)} ,\bullet, \alpha)|'+1 = |\mathbf{a}^{(2)}|' + |\bullet|'  $. $M_1$-closedness of $M_2(N_{01}, N_{10})$ simply follows from that of $N_{01}$ and $N_{10}$.

The theorem \ref{thm:ye} is an immediate consequence of the following Lemma \ref{lem:Ncohomtoid1}.
\end{proof}

\begin{lemma}\label{lem:Ncohomtoid1}
If $\alpha$ and $\beta$ are isomorphisms in Definition \ref{def:iso}, then $N_{01}\circ N_{10}$ is cohomologous to identity, and so is $N_{10} \circ N_{01}$.
\end{lemma}

\begin{proof}[Proof of Lemma~\ref{lem:NM1close}]

Since $\mathcal{D}$ is ($A_\infty$-)dg and $||N_{01}||'=-1$, $M_1$ reduces to the following equation,
\begin{equation*}
\begin{array}{rl}
&M_1(N_{01})(a_1,\ldots,a_k) \\
=& m_1^{\mathcal{D}} ( N_{01} (\mathbf{a}) ) + (-1)^{||N_{01}||' \cdot |\mathbf{a}^{(1)}|'} m_2^{\mathcal{D}} (\CY^0 (\mathbf{a}^{(1)}), N_{01} (\mathbf{a}^{(2)}) ) +  m_2^{\mathcal{D}} (N_{01} (\mathbf{a}^{(1)}), \CY^1 (\mathbf{a}^{(2)}) )\\
&{ -} \sum (-1)^{||N_{01}||' + |\mathbf{a}^{(1)}|'} N_{01} ( \mathbf{a}^{(1)}, m^{\mathcal{C}}_*(\mathbf{a}^{(2)}),\mathbf{a}^{(3)})\\
=& m_1^{\mathcal{D}} ( N_{01} (\mathbf{a}) ) +  (-1)^{|\mathbf{a}^{(1)}|'} m_2^{\mathcal{D}} (\CY^0 (\mathbf{a}^{(1)}), N_{01} (\mathbf{a}^{(2)}) ) +  m_2^{\mathcal{D}} (N_{01} (\mathbf{a}^{(1)}), \CY^1 (\mathbf{a}^{(2)}) )\\
&{ +} \sum (-1)^{ |\mathbf{a}^{(1)}|'} N_{01} ( \mathbf{a}^{(1)}, m^{\mathcal{C}}_*(\mathbf{a}^{(2)}),\mathbf{a}^{(3)}).
\end{array}
\end{equation*}
If we plug in $\bullet \in \Hom_\mathcal{C} (C_k, L_1)$ to the above equation, each of first three terms becomes
$$m_1^{\mathcal{D}} (N_{01} (\mathbf{a})) (\bullet) = - { (-1)^{|\mathbf{a}|'}} { (-1)^{|\bullet|}} m_1^{\mathcal{C}} (m^{\mathcal{C}}(\mathbf{a}, \bullet,\beta)) -  { (-1)^{|\bullet|+1}} { (-1)^{|\mathbf{a}|'}}  (-1)^{|N_{01} (\mathbf{a}) |} m^{\mathcal{C}} (\mathbf{a}, - m_1^{\mathcal{C}}(\bullet),\beta)$$
{ (Recall that $|N_{01} (\mathbf{a})| = ||N_{01}|| + |\mathbf{a}|' = |\mathbf{a}|'$.)}
%....
\begin{eqnarray*}
(-1)^{|\mathbf{a}^{(1)}|'}  m_2^{\mathcal{D}} (\CY^0(\mathbf{a}^{(1)}),N_{01} (\mathbf{a}^{(2)})) (\bullet) &=& (-1)^{|\mathbf{a}^{(1)}|'} \overbrace{(-1)^{|\CY^0 (\mathbf{a}^{(1)})|}}^{{\rm dg} \leftrightarrow A_\infty\textnormal{-dg}} { (-1)^{|\mathbf{a}^{(2)}|'}} { (-1)^{|\bullet|}} m^{\mathcal{C}}(\mathbf{a}^{(1)}, m^{\mathcal{C}} (\mathbf{a}^{(2)}, \bullet,\beta))\\
&=& -{ (-1)^{|\mathbf{a}^{(2)}|'} } { (-1)^{|\bullet|}} m^{\mathcal{C}}(\mathbf{a}^{(1)}, m^{\mathcal{C}} (\mathbf{a}^{(2)}, \bullet,\beta)) \\
&=& -  (-1)^{|\mathbf{a}|'} (-1)^{|\mathbf{a}^{(1)}|'} { (-1)^{|\bullet|}} m^{\mathcal{C}}(\mathbf{a}^{(1)}, m^{\mathcal{C}} (\mathbf{a}^{(2)}, \bullet,\beta))
\end{eqnarray*}
%(when $\mathbf{a}^{(2)}= \phi$:
\begin{eqnarray*}
m_2^{\mathcal{D}} (N_{01}(\mathbf{a}^{(1)}),\CY^1(\mathbf{a}^{(2)})) (\bullet) &=& \overbrace{(-1)^{|N_{01}  (\mathbf{a}^{(1)})|}}^{{\rm dg} \leftrightarrow A_\infty\textnormal{-dg}} { (-1)^{|\mathbf{a}^{(1)}|'}} { (-1)^{|m^\mathcal{C} (\mathbf{a}^{(2)}, \bullet)|}} m^{\mathcal{C}} ( \mathbf{a}^{(1)},m^{\mathcal{C}} (\mathbf{a}^{(2)}, \bullet) ,\beta) \\
&=&   { (-1)^{| \mathbf{a}^{(2)}|' +  |\bullet|  + 1}} m^{\mathcal{C}} ( \mathbf{a}^{(1)},m^{\mathcal{C}} (\mathbf{a}^{(2)}, \bullet) ,\beta) \\
&=& - (-1)^{|\mathbf{a}^{(1)}|'  } { (-1)^{| \mathbf{a}| '} (-1)^{  |\bullet|}} m^{\mathcal{C}} ( \mathbf{a}^{(1)},m^{\mathcal{C}} (\mathbf{a}^{(2)}, \bullet) ,\beta) 
\end{eqnarray*}
%(when $\mathbf{a}^{(1)}= \phi$:
%{ THIS TERM ISN'T RIGHT!$\nearrow$}
The last term gives
\begin{equation*}
\begin{array}{l}
\sum (-1)^{ |\mathbf{a}^{(1)}|'} N_{01} ( \mathbf{a}^{(1)}, m^{\mathcal{C}}_*(\mathbf{a}^{(2)}),\mathbf{a}^{(3)}) (\bullet) \\
= \sum (-1)^{|\mathbf{a}^{(1)}|'} { (-1)^{|\mathbf{a}^{(1)}|' + |\mathbf{a}^{(2)}|' +1  + |\mathbf{a}^{(3)}|'}} { (-1)^{|\bullet|}} m^\mathcal{C} (\mathbf{a}^{(1)}, m^{\mathcal{C}} (\mathbf{a}^{(2)} ), \mathbf{a}^{(3)}, \bullet, \beta) \\
= -   (-1)^{|\mathbf{a}|'}  { (-1)^{|\bullet|}} \sum (-1)^{|\mathbf{a}^{(1)}|'}  m^\mathcal{C} (\mathbf{a}^{(1)}, m^{\mathcal{C}} (\mathbf{a}^{(2)} ), \mathbf{a}^{(3)}, \bullet, \beta)
\end{array}
\end{equation*}

Thus $M_1(N_{01})=0$  corresponds to the $\AI$-equation in $\mathcal{C}$ with inputs $\mathbf{a}, \bullet,\beta$ after overall multiplication by $(-1)^{|\mathbf{a}|'} (-1)^{|\bullet|'}$. (The terms involving $m_1(\beta)$ are missing, which are simply zero by our assumption.)
\end{proof}

%\begin{lemma}\label{lem:Ncohomtoid}
%If $\alpha$ and $\beta$ are isomorphisms, then $N_{01}\circ N_{10}$ is cohomologous to identity.
%\end{lemma}
\begin{proof}[Proof of Lemma~\ref{lem:Ncohomtoid1}]
Let us first assume that $\alpha$ and $\beta$ are strict isomorphisms.
We consider the $\AI$-equation with inputs $\mathbf{a}, \bullet,\alpha,\beta$, which involves 
$$(-1)^{|\mathbf{a}^{(1)}|'} m^{\mathcal{C}}(\mathbf{a}^{(1)},m^{\mathcal{C}} (\mathbf{a}^{(2)} ,\bullet, \alpha),\beta).$$
Apart from the term \eqref{eqn:M2twonat}, the rest of the terms are
\begin{equation}\label{eqn:restabalbe}
\begin{array}{l}
(-1)^{|\mathbf{a}^{(1)}|'} m^{\mathcal{C}}(\mathbf{a}^{(1)},m^{\mathcal{C}} (\mathbf{a}^{(2)} ,\bullet, \alpha,\beta)) \\
%$$m(m(\beta,\alpha,\bullet,...) ...)$$
(-1)^{|\mathbf{a}|' + |\bullet|'} m^{\mathcal{C}}(\mathbf{a},\bullet,m^{\mathcal{C}}( \alpha,\beta)) \left\{
\begin{array}{ll} 
= (-1)^{|\bullet|'} (-1)^{|\bullet|} \bullet = - \bullet \,\, \mbox{when there is no} \,\, \mathbf{a}  \\
=0 \,\, \mbox{otherwise}
\end{array} \right.
\\
%$$m(m_2(\beta,\alpha),\bullet, a..)$$
(-1)^{|\mathbf{a}^{(1)}|'} m^{\mathcal{C}}(\mathbf{a}^{(1)},m^{\mathcal{C}} (\mathbf{a}^{(2)} ,\bullet), \alpha,\beta) \\
%$$m(\beta,\alpha,m(\bullet,...),...)$$
(-1)^{|\mathbf{a}^{(1)}|'} m^{\mathcal{C}}(\mathbf{a}^{(1)},m^{\mathcal{C}} (\mathbf{a}^{(2)}), \mathbf{a}^{(3)} ,\bullet, \alpha,\beta)).
\end{array}
\end{equation}
%$$m(\beta,\alpha,\bullet,..m(...)...)$$
%
%The second one vanishes (except when there are not a's) since $m_2(\beta,\alpha)$ is the unit, and
%using the property of the unit.
Therefore, we see that
\begin{equation}\label{eqn:diffidcomp}
\begin{array}{lcl}
 \left(M_2(N_{01},N_{10}) - N_{id} \right) (\mathbf{a}) &=& - \left( \sum (-1)^{|\mathbf{a}^{(1)}|'} m^{\mathcal{C}}(\mathbf{a}^{(1)},m^{\mathcal{C}} (\mathbf{a}^{(2)} ,\bullet, \alpha,\beta))\right. \\
 && + \sum  (-1)^{|\mathbf{a}^{(1)}|'} m^{\mathcal{C}}(\mathbf{a}^{(1)},m^{\mathcal{C}} (\mathbf{a}^{(2)} ,\bullet), \alpha,\beta) \\
 &&\left. + \sum  (-1)^{|\mathbf{a}^{(1)}|'} m^{\mathcal{C}}(\mathbf{a}^{(1)},m^{\mathcal{C}} (\mathbf{a}^{(2)}), \mathbf{a}^{(3)} ,\bullet, \alpha,\beta)) \right).
 \end{array}
 \end{equation}

If we define pre-natural transformation $H$ by
$$H(\mathbf{a}) =   m^{\mathcal{C}} ( \mathbf{a},\bullet,\alpha,\beta)$$
%$$H(\mathbf{a}) = m(\beta,\alpha,\bullet,\mathbf{a})$$
Notice that $||H|| = -1$, so $||H||'$ is even. $M_1(H)$ is given as follows.
\begin{equation*}
\begin{array}{lcl}
M_1 (H) (\mathbf{a}) &=&  m_1^{\mathcal{D}} ( H(\mathbf{a}) ) +  m_2^{\mathcal{D}} (\CY^0 (\mathbf{a}^{(1)}), H (\mathbf{a}^{(2)} ))  + m_2^{\mathcal{D}} (H(\mathbf{a}^{(1)}), \CY^0 (\mathbf{a}^{(2)}) ) \\
&& - \sum (-1)^{|\mathbf{a}^{(1)}|'} H(\mathbf{a}^{(1)}, m^{\mathcal{C}} (\mathbf{a}^{(2)}), \mathbf{a}^{(3)}) \\
&=& - m_1 (   m^{\mathcal{C}} ( \mathbf{a},\bullet,\alpha,\beta)) - (-1)^{|H(\mathbf{a})|} m^{\mathcal{C}} ( \mathbf{a},- m_1^{\mathcal{C}}(\bullet),\alpha,\beta)) \\
 && + \overbrace{(-1)^{|\CY^0 (\mathbf{a}^{(1)})|}}^{ {\rm dg} \leftrightarrow A_\infty\textnormal{-dg} } m^{\mathcal{C}} ( \mathbf{a}^{(1)},  m^{\mathcal{C}} ( \mathbf{a}^{(2)},\bullet,\alpha,\beta)) \\
 && + \overbrace{(-1)^{|H(\mathbf{a}^{(1)}|}}^{{\rm dg} \leftrightarrow A_\infty\textnormal{-dg}}  m^{\mathcal{C}} ( \mathbf{a}^{(1)} ,m^{\mathcal{C}} (\mathbf{a}^{(2)}, \bullet), \alpha,\beta) \\
 &&- \sum (-1)^{|\mathbf{a}^{(1)}|'}   m^{\mathcal{C}}(\mathbf{a}^{(1)},m^{\mathcal{C}} (\mathbf{a}^{(2)}), \mathbf{a}^{(3)} ,\bullet, \alpha,\beta)).
\end{array}
\end{equation*}
Observe that $|\CY^0 (\mathbf{a}) | = |\CY^0 (\mathbf{a}) |' + 1 =|\mathbf{a}|' +1$ and $|H(\mathbf{a})|=||H|| + |\mathbf{a}|' =|\mathbf{a}|' -1 $. Hence the right hand side reads
\begin{equation}\label{eqn:M1H}
\begin{array}{l}
-   m_1^{\mathcal{C}} (  m^{\mathcal{C}} ( \mathbf{a},\bullet,\alpha,\beta)) - (-1)^{|\mathbf{a}|' } m^{\mathcal{C}} ( \mathbf{a}, m_1^{\mathcal{C}}(\bullet),\alpha,\beta) \\
-  (-1)^{|\mathbf{a}^{(1)}|'  }  m^{\mathcal{C}} ( \mathbf{a}^{(1)},  m^{\mathcal{C}} ( \mathbf{a}^{(2)},\bullet,\alpha,\beta)) -(-1)^{|\mathbf{a}^{(1)}|'}  m^{\mathcal{C}} ( \mathbf{a}^{(1)} ,m^{\mathcal{C}} (\mathbf{a}^{(2)}, \bullet), \alpha,\beta) \\
- \sum (-1)^{|\mathbf{a}^{(1)}|'}   m^{\mathcal{C}}(\mathbf{a}^{(1)},m^{\mathcal{C}} (\mathbf{a}^{(2)}), \mathbf{a}^{(3)} ,\bullet, \alpha,\beta))
\end{array}
\end{equation} 
which is precisely \eqref{eqn:diffidcomp}. Thus we see that $M_2(N_{01},N_{10}) - N_{id}= M_1(H)$. \

If $m_2(\alpha,\beta) = \one_{L_0} + m_1(x)$ with nontrivial $m_1(x)$, then the second term in \eqref{eqn:restabalbe} additionally involves  
$$ (-1)^{|\mathbf{a}|' + |\bullet|'} m^{\mathcal{C}} (\mathbf{a}, \bullet, m_1(x) )$$
which by $A_\infty$-relation equals to 
\begin{equation*}
\begin{array}{l}
-m_1^{\mathcal{C}} (m^{\mathcal{C}} (\mathbf{a}  , \bullet, x) ) - (-1)^{|\mathbf{a}|' } m^{\mathcal{C}} (\mathbf{a}  , m_1(\bullet), x) \\
 - (-1)^{|\mathbf{a}^{(1)}|'} m^{\mathcal{C}} (\mathbf{a}^{(1)}, m^\mathcal{C} (\mathbf{a}^{(2)}, \bullet, x)) - (-1)^{|\mathbf{a}^{(1)}|'}  m^{\mathcal{C}}(\mathbf{a}^{(1)}, m^\mathcal{C} (\mathbf{a}^{(2)}, \bullet), x) \\
- \sum (-1)^{|\mathbf{a}^{(1)}|'} m^{\mathcal{C}} (\mathbf{a}^{(1)}, m^\mathcal{C} (\mathbf{a}^{(2)}), \mathbf{a}^{(3)}, \bullet, x)
\end{array}
\end{equation*}
Note that this is precisely the same as \eqref{eqn:M1H} except that $x$ sits at the end of each term instead of $\alpha,\beta$. Therefore, if one sets
$$H=  m^{\mathcal{C}} ( \mathbf{a},\bullet,\alpha,\beta) +  m^{\mathcal{C}} ( \mathbf{a},\bullet,x),$$
the difference between $N_{01} \circ N_{10}$ and the identity natural transformation is  $M_1(H)$, which completes the proof.
\end{proof}

\section{Proof of main gluing theorem \ref{thm:maingluealg}}\label{sec:pfgluethm}
We give a detailed proof of Theorem \ref{thm:maingluealg}.
First note that a similar statement to Theorem \ref{thm:ye} holds for the curved case. Consider two object $L_0$ and $L_1$ in the curved $A_\infty$-category with the curvature $C$. Namely,  $m_0^{L_0} = C \cdot \one_{L_0}$ and $m_0^{\BL_1} = C \cdot \one_{L_1}$ for same $C$. Then the corresponding Yoneda functors lands on the category of curved complexes with the curvature $C$, which is again a dg-category.
Notice that the proofs of Lemma \ref{lem:NM1close} and \ref{lem:Ncohomtoid1} involve $m_k$ always with more than one inputs (at least it has either $(\bullet,\beta)$ or $(\bullet,\alpha,\beta)$ or $(\bullet,x)$). Therefore, inserting $m_0$ terms does not make any difference as they all vanish due to the property of the unit. We conclude that two Yoneda functors are still quasi-isomorphic through natural transformations when two objects in a curved $A_\infty$-category are related by isomorphisms.

We will follow the new convention given in \ref{sec:dgsign} throughout the argument, and we write $\mathcal{B}_{\AI}:=\MF_{A_\infty} (W^{\bL_0})$, $\mathcal{C}_{\AI}:=\MF_{A_\infty} (W^{\bL_1})$, $\mathcal{D}_{\AI}:=\MF_{A_\infty} (W^{\bL_0}|_{V})$ . Recall that the  $A_\infty$-functor
$$\mathcal{F} : \Fuk (X) \to \mathcal{B}_{\AI} \times^h_{\mathcal{D}_{\AI}}  \mathcal{C}_{\AI}$$ 
is defined as follows (where the right hand side is regarded as $A_\infty$-dg category). For an object $L \in \Fuk (X)$, the image $\mathcal{F} (L)$ is defined as
\begin{equation*}\label{eqn:hptyfpobjfunc}
\begin{array}{c}
\left(\mathcal{F}^{\bL_0}(L), \quad \mathcal{F}^{\bL_1}(L),\quad  \mathcal{F}_0^{\bL_0 }(L)|_{V } \stackrel{}{\longrightarrow}  \mathcal{F}_0^{\bL_1}(L)|_{V_1}  \right) = \\
 \left(CF(L, (\bL_0 , b_0 )), \quad   CF(L,(\bL_1,b_1) ,\quad \xymatrix{ CF(L, (\bL_0 , b_0 ))|_{V_0 } & CF(L, (\bL_1, \phi(b_0)))|_{V_1} \ar[l]_{ N_{01}(L)} }
%\tiny (-1)^{\epsilon} m(\bullet,e^{\phi(b_0)},\beta,e^{b})} } 
 \right)
 \end{array}
\end{equation*}
where 
$$N_{01}(L)=  (-1)^{|\bullet|}  m (\bullet,e^{\phi(b_0)},\beta, e^b) \in \Hom_{MF_{A_\infty}} ( \mathcal{F}^{\bL_0 }(L)|_{V_0} ,\mathcal{F}^{\bL_1}(L)|_{V_1} ).$$
%due to our convention of interchanging dg and $A_\infty$-dg. 
The map $N_{01}$ induces an isomorphism on between $m_1^{\phi(b_0)}$ and $m_1^{b_0}$ cohomologies since the similarly defined map $N_{10}$ using $\alpha$ instead of $\beta$ induces its inverse on the cohomology level.
For simplicity, let us write $N_{01}(L)=N_0, N_{10}(L)=N_0'$. We have the following.
\begin{lemma}\label{lem:n0hptyeq}
$m_2^{\mathcal{D}_{A_\infty}}(N_0, N_0')-id= m_1^{\mathcal{D}_{A_\infty}} (H)$ with
$$ H (\bullet)= m(\bullet, e^{b_1},\beta, e^{b_0}, \alpha, e^{b_1}) + m(\bullet, e^{b_1},x, e^{b_1}) $$
where we set $m_2^{b_1,b_0 ,b_1} (\alpha,\beta) = \one_{\bL_1} + m_1^{b_1,b_1}(x)$.
\end{lemma}
\begin{proof}
%Proof is similar to that of Lemma \eqref{lem:Ncohomtoid1}.
We compute $m_2^{\mathcal{D}_{A_\infty}}(N_0, N_0')$ using $A_\infty$-relations as follows. (We set $b_1:=\phi(b_0)$ in the following computation for simplicity.)
\begin{equation}\label{eqn:mmqisom}
\begin{array}{lcl}
m(m(\bullet,e^{b_1}, \beta, e^{b_0}),e^{b_0}, \alpha, e^{b_1}) &=& (-1)^{|\bullet|} m  (\bullet, e^{b_1},  m_2^{b_1,b_0,b_1} (\beta, \alpha), e^{b_1}) - m ( m(\bullet, e^{b_1},  \beta,e^{b_0} , \alpha, e^{b_1}),  e^{b_1}  )\\
&&- m(m(\bullet,e^{b_1}), e^{b_1},\beta, e^{b_0}, \alpha, e^{b_1})   \\
&=&  (-1)^{|\bullet|} m  (\bullet, e^{b_1},  \one_{\bL_1} + m_1^{b_1,b_1} (x), e^{b_1}) - m_1^{0,b_1} ( m(\bullet, e^{b_1},  \beta,e^{b_0} , \alpha, e^{b_1}) )\\
&&- m(m_1^{0,b_1} (\bullet), e^{b_1},\beta, e^b, \alpha, e^{b_1})   \\
&=& \bullet + (-1)^{|\bullet|} m_2^{0,b_1,b_1}  (\bullet,  m_1^{b_1,b_1} (x) ) - m_1^{0,b_1} ( m(\bullet, e^{b_1},  \beta,e^{b_0} , \alpha, e^{b_1}) )\\
&&- m(m_1^{0,b_1} (\bullet), e^{b_1},\beta, e^{b_0}, \alpha, e^{b_1})   \\
&=&  \bullet - m_1^{0,b_1} ( m(\bullet,e^{b_1},x,e^{b_1}) ) - m(m_1^{0,b_1} (\bullet),e^{b_1},x,e^{b_1}) \\
&&-m_1^{0,b_1} ( m(\bullet, e^{b_1},  \beta,e^{b_0} , \alpha, e^{b_1}) )
- m(m_1^{0,b_1} (\bullet), e^{b_1},\beta, e^{b_0}, \alpha, e^{b_1}) 
\end{array}
\end{equation}
%for some $x$ 
using $m(e^{b_1}, \beta, e^{b_0}, \alpha, e^{b_1}) =id$.  Therefore, the difference between $m_2^{\mathcal{D}_{A_\infty}}(N_0, N_0')$ and $id$ is precisely $m_1^{\mathcal{D}_{A_\infty}} (H)$.

%{\color{red} WHY $-H$? Oh, because $-m_1^{0,b}$ is the matrix factorization operator!}

\end{proof}
The same statement holds for the composition $N_0$ and $N_0'$ in the other direction. Therefore, \eqref{eqn:hptyfpobjfunc} is a well-defined object in $\mathcal{B}_{\AI} \times^h_{\mathcal{D}_{\AI}} \mathcal{C}_{\AI}$. 

\begin{remark}
In dg-setting, the category $\mathcal{B}_{\AI} \times^h_{\mathcal{D}_{\AI}} \mathcal{C}_{\AI}$ corresponds to $\mathcal{C}_{dg} \times^h_{\mathcal{D}_{dg}} \mathcal{B}_{dg}$. Notice that the positions of $\mathcal{B}$ and $\mathcal{C}$ are switched.
\end{remark}

For a tuple of composable morphisms 
$$(a_1, \cdots, a_k) \in \Hom_{\Fuk(X)} (L_1,L_2) \otimes \cdots \otimes \Hom_{\Fuk(X)} (L_k,L_{k+1}),$$
we have defined $\mathcal{F}_k (a_1, \cdots, a_k)$ to be 
\begin{equation*}\label{eqn:fcttohptypd}
\begin{array}{c}
(\mathcal{F}_k^{\bL_0} (\mathbf{a}),\mathcal{F}_k^{\bL_1} (\mathbf{a}), N_k  ( \mathbf{a}) ) = \\
\left( m (\mathbf{a}, \bullet, e^{b_0}), \quad m (\mathbf{a}, \bullet, e^{b_1}), \quad   (-1)^{|\mathbf{a}|'} (-1)^{\bullet} m (\mathbf{a},\bullet, e^{\phi(b_0)} , \beta, e^{b_0})|_{b_0 \in V_0} \right)
\end{array}
\end{equation*}
Note that $\bullet \mapsto m(\mathbf{a}, \bullet, e^{\phi(b_0)}, \beta,e^{b_0})$ gives a map
$$CF(L_{k+1}, (\bL_1,\phi(b_0) )) \to CF(L_1, (\bL_0 ,b_0)) $$
which is an element in $\Hom_{\MF_{A_\infty}} ( \mathcal{F}^{\bL_0} (L_1)|_{V_0},\mathcal{F}^{\bL_1} (L_{k+1})|_{V_1})$ again due to our convention. 

$\{N_k\}$ (and $\{N_k'\}$ defined similarly using $\alpha$) gives a natural transformation between two local functors $r_0 \circ \mathcal{F}^{\bL_0} = \mathcal{F}^{\bL_0} |_{V_0}$ and $r_1 \circ \mathcal{F}^{\bL_1}=\mathcal{F}^{\bL_1} |_{V_1}$. Furthermore, the composition of two natural transformations $\{N_k\}$ and $\{N_k'\}$ (in each of directions) is homotopic to the identity transformation. The proof is the same as that of Lemma \ref{lem:Ncohomtoid1} (with $L_0$ and $L_i$ replaced by $(\bL_0,b_0)$ and $(\bL_1,\phi(b_0))$), and we do not repeat. Hence we obtain (2) of Theorem \ref{thm:maingluealg}.

We are now ready to give a proof of (3) of Theorem \ref{thm:maingluealg}. 

\begin{prop}\label{prop:3rditemmain}
$\{\mathcal{F}_k\} : \Fuk (X) \to \mathcal{B}_{\AI} \times^h_{\mathcal{D}_{\AI}} \mathcal{C}_{\AI}$ defines an $A_\infty$-functor.
\end{prop}

\begin{proof}
Put $\mathcal{E}:=\mathcal{B}_{\AI} \times^h_{\mathcal{D}_{\AI}} \mathcal{C}_{\AI}$, and recall from \ref{sec:dgsign} and \ref{sec:dghptyfp} that the $A_\infty$-structure on $\mathcal{E}$ is given by
$$m_1^\mathcal{E} (\mu,\nu, \gamma) = (m_1^{\mathcal{B}} (\mu), m_1^{\mathcal{C}} (\nu), - m_1^{\mathcal{D}} (\gamma) -  m_2^{\mathcal{D}} (\phi_2, G(\mu)) + (-1)^{|\nu|} m_2^{\mathcal{D}} (L(\nu), \phi_1),$$
%$$(-1)^{|\mu'|} m_2^\mathcal{E} ((\mu',\nu', \gamma'),(\mu,\nu, \gamma)) )= $$
%$$( (-1)^{|\mu'|} m_2(\mu', \mu),(-1)^{|\nu'|} m_2(\nu', \nu), (-1)^{|\gamma'|} m_2(\gamma',G(\mu)) + m_2 (L(\nu'),\gamma) ),$$
$$m_2^\mathcal{E} ((\mu',\nu', \gamma'),(\mu,\nu, \gamma)) )=(   m_2^{\mathcal{B}}(\mu', \mu), m_2^{\mathcal{C}}(\nu', \nu), -m_2^{\mathcal{D}}(\gamma',G(\mu)) + (-1)^{|\nu'|} m_2^{\mathcal{D}} (L(\nu'),\gamma) ).$$
(Here, we suppressed the subscript $\AI$ for simplicity)

For given a tuple $\mathbf{a} :=(a_1, \cdots, a_k)$ of morphisms in $\Fuk(X)$, one needs to show that
\begin{equation}\label{eqn:functeqnhpty}
m_1^{\mathcal{E}} ( \mathcal{F} (\mathbf{a}) ) + m_2^{\mathcal{E}} (\mathcal{F} (\mathbf{a}^{(1)}), \mathcal{F}(\mathbf{a}^{(2)}) ) = \sum (-1)^{|\mathbf{a}^{(1)}|'} \mathcal{F} (\mathbf{a}^{(1)}, m (\mathbf{a}^{(2)}), \mathbf{a}^{(3)} )
\end{equation}
where $m$ on the right hand side is the $A_\infty$-operation on $\Fuk(X)$. Note that the first and the second components of \eqref{eqn:fcttohptypd} automatically satisfy the condition since they are simply $\mathcal{F}^{\bL_0}$ and $\mathcal{F}^{\bL_1}$ which has shown to be $A_\infty$-functors earlier. Thus we only need to check the last component. The following is the list of terms that appear in the third components of \eqref{eqn:functeqnhpty}. We write $b_1$ for $\phi (b_0)$ in the computation below for simplicity.
\begin{equation*}
\begin{array}{l}
-m_1^{\mathcal{D}} ( m(\mathbf{a}, \bullet, e^{b_1}, \beta, e^{b_0}) )\\
=-\left( - (-1)^{|\mathbf{a}|'} (-1)^{|\bullet|}   m_1^{0,b_0} ( m(\mathbf{a}, \bullet, e^{b_1}, \beta, e^{b_0})  ) - \overbrace{ (-1)^{|m(\mathbf{a}, \bullet, e^{b_1}, \beta, e^{b_0})|}}^{\textnormal{as a map in} \,\, \bullet}  (-1)^{|\mathbf{a}|'} (-1)^{|\bullet|+1}  m(\mathbf{a},-m_1^{0,b_1} (\bullet), e^{b_1},\beta, e^{b_0}) \right) \\
=   (-1)^{|\mathbf{a}|'} (-1)^{|\bullet|}  m_1 (   m(\mathbf{a}, \bullet,e^{b_1}, \beta, e^{b_0}), e^{b_0}) -  (-1)^{|\mathbf{a}|'} (-1)^{|\bullet|+1}  \overbrace{ (-1)^{|m(\mathbf{a}, \bullet, e^{b_1}, \beta, e^{b_0})|} }^{\textnormal{as a map in} \,\, \bullet} m( \mathbf{a}, m(\bullet, e^{b_1} ) , e^{b_1}, \beta, e^{b_0})  \\
=  (-1)^{|\mathbf{a}|'} (-1)^{|\bullet|}  m_1 (   m(\mathbf{a}, \bullet,e^{b_1}, \beta, e^{b_0}), e^{b_0}) + (-1)^{|\mathbf{a}|'} (-1)^{|\bullet| }  \overbrace{ (-1)^{| \mathbf{a} |'} }^{\textnormal{as a map in} \,\, \bullet} m( \mathbf{a}, m(\bullet, e^{b_1} ) , e^{b_1}, \beta, e^{b_0}) 
\end{array}
\end{equation*}

\begin{equation*}
\begin{array}{lcl}
-m_2^{\mathcal{D}} (\phi_2, G(\mu)) &=&-
m_2^{\mathcal{D}} ( (-1)^{|\star|} m(\star,e^{\phi(b_0)},\beta,e^{b_0}), \overbrace{m(\mathbf{a},\bullet,e^{b_1}) }^{\star} ) \\
&=& - (-1)^{|m(\mathbf{a},\bullet,e^{b_1})  |}  \overbrace{(-1)^{|m(\star,e^{b_1 },\beta,e^{b_0}) |}}^{\textnormal{as a map in} \,\, \star} m(m(\mathbf{a},\bullet,e^{b_1}),e^{b_1 },\beta,e^{b_0}) \\
&=&    (-1)^{| \mathbf{a}|'+ |\bullet|' +1}  m(m(\mathbf{a},\bullet,e^{b_1}),e^{b_1 },\beta,e^{b_0})  \\
&=&   (-1)^{| \mathbf{a}|'} (-1)^{ |\bullet| } m(m(\mathbf{a},\bullet,e^{b_1}),e^{b_1 },\beta,e^{b_0}) 
\end{array}
\end{equation*}

\begin{equation*}
\begin{array}{lcl}
  (-1)^{|\nu|} m_2^{\mathcal{D}} (L(\nu), \phi_1) &=&   (-1)^{|\nu|}  m_2^{\mathcal{D}} (m(\mathbf{a},\star, e^{b_0}),\overbrace{  (-1)^{|\bullet|}  m(\bullet,e^{\phi(b_0)},\beta,e^{b_0}) }^{\star} ) \\
&=&   \overbrace{(-1)^{|m(\mathbf{a}, \star, e^{b_0}) )|}}^{\textnormal{as a map in} \,\, \star}   (-1)^{|\bullet|} \overbrace{(-1)^{|m(\mathbf{a},\star, e^{b_0}) |}}^{\textnormal{as a map in}\,\, \star} m(\mathbf{a}, m(\bullet,e^{\phi(b_0)},\beta,e^{b_0}), e^{b_0})\\
&=&   (-1)^{|\bullet|}  m(\mathbf{a}, m(\bullet,e^{\phi(b_0)},\beta,e^{b_0}), e^{b_0}) \\
&=&    (-1)^{|\mathbf{a}|'} (-1)^{|\bullet|}  (-1)^{|\mathbf{a}|'} m(\mathbf{a}, m(\bullet,e^{\phi(b_0)},\beta,e^{b_0}), e^{b_0})
\end{array}
\end{equation*}

\begin{equation*}
\begin{array}{lcl}
 -m_2^{\mathcal{D}}(\gamma',G(\mu)) &=& -m_2^{\mathcal{D}} (   (-1)^{|\mathbf{a}^{(1)}|'} (-1)^{|\star|}   m(\mathbf{a}^{(1)},\star, e^{b_1}, \beta, e^{b_0}), \overbrace{m(\mathbf{a}^{(2)},\bullet, e^{b_1} )}^{\star})\\
 &=& - \overbrace{(-1)^{|m(\mathbf{a}^{(1)},\bullet, e^{b_1}, \beta, e^{b_0}) |}}^{\textnormal{as a map}}   (-1)^{|\mathbf{a}^{(1)}|'} (-1)^{|m(\mathbf{a}^{(2)},\bullet, e^{b_1} ) |}   m(\mathbf{a}^{(1)},m(\mathbf{a}^{(2)},\bullet, e^{b_1} ), e^{b_1}, \beta, e^{b_0}) \\
  &=& -  (-1)^{| \mathbf{a}^{(1)} |'}   (-1)^{|\mathbf{a}^{(1)}|'} (-1)^{|m(\mathbf{a}^{(2)},\bullet, e^{b_1} ) |}  m(\mathbf{a}^{(1)},m(\mathbf{a}^{(2)},\bullet, e^{b_1} ), e^{b_1}, \beta, e^{b_0})\\
    &=&  (-1)^{| \mathbf{a}^{(1)} |'}    (-1)^{|\mathbf{a}^{(1)}|'} (-1)^{|\mathbf{a}^{(2)}|'+|\bullet|'+1}   m(\mathbf{a}^{(1)},m(\mathbf{a}^{(2)},\bullet, e^{b_1} ), e^{b_1}, \beta, e^{b_0})
 \end{array}
\end{equation*}

\begin{equation*}
\begin{array}{lcl}
(-1)^{|\nu'|} m_2^{\mathcal{D}} (L(\nu'),\gamma) ) &=& (-1)^{|\nu'|} m_2^{\mathcal{D}}
(m(\mathbf{a}^{(1)}, \star, e^{b}), \overbrace{  (-1)^{|\mathbf{a}^{(2)}|'} (-1)^{|\bullet|} m(\mathbf{a}^{(2)}, \bullet, e^{b_1},\beta, e^{b_0})}^{\star} )\\
&=&  (-1)^{|m(\mathbf{a}^{(1)}, \bullet, e^{b_0})|}  (-1)^{|\mathbf{a}^{(2)}|'} (-1)^{|\bullet|} (-1)^{|m(\mathbf{a}^{(1)}, \bullet, e^{b_0})|} m(\mathbf{a}^{(1)}, m(\mathbf{a}^{(2)}, \bullet, e^{b_1},\beta, e^{b_0}), e^{b_0}) \\
&=&  (-1)^{|m(\mathbf{a}^{(1)}, \bullet, e^{b_0})|}   (-1)^{|\mathbf{a}^{(2)}|'} (-1)^{|\bullet|} (-1)^{|m(\mathbf{a}^{(1)}, \bullet, e^{b_0})|} m(\mathbf{a}^{(1)}, m(\mathbf{a}^{(2)}, \bullet, e^{b_1},\beta, e^{b_0}), e^{b_0}) \\
&=&    (-1)^{|\mathbf{a}^{(2)}|'} (-1)^{|\bullet|} m(\mathbf{a}^{(1)}, m(\mathbf{a}^{(2)}, \bullet, e^{b_1},\beta, e^{b_0}), e^{b_0}) 
\end{array}
\end{equation*}
Finally, the third component of $(-1)^{|\mathbf{a}^{(1)}|'} \mathcal{F} (\mathbf{a}^{(1)}, m (\mathbf{a}^{(2)}), \mathbf{a}^{(3)} )$ on the right hand side of \eqref{eqn:functeqnhpty} is given by
$$\sum (-1)^{|\mathbf{a}^{(1)}|'}   (-1)^{|\mathbf{a}|' +1} (-1)^{|\bullet|}  m(\mathbf{a}^{(1)}, m(\mathbf{a}^{(2)}), \mathbf{a}^{(3)}, \bullet, e^{b_1},\beta, e^{b_0} ) $$
$$= -  (-1)^{|\mathbf{a}|' } (-1)^{|\bullet|}  \sum (-1)^{|\mathbf{a}^{(1)}|'}  m(\mathbf{a}^{(1)}, m(\mathbf{a}^{(2)}), \mathbf{a}^{(3)}, \bullet, e^{b_1},\beta, e^{b_0} ) $$

We see that \eqref{eqn:functeqnhpty} is equivalent to the $A_\infty$-equation with inputs $\mathbf{a},\bullet, e^{b_1},\beta, e^{b_0}$ (after overall multiplication by $  (-1)^{|\mathbf{a}|' }$.)
 
\end{proof}

\appendix

\section{Algebraic conventions}
We recall  various algebrac notions, such as  filtered $\AI$-category, its deformation theory of Fukaya-Oh-Ohta-Ono \cite{FOOO}, We refer readers to the further mentioned references for more details.

\subsection{Signs for dg-categories}\label{sec:dgsign}
\begin{defn}\label{def:dgsign}
For a given dg-category $(\mathcal{C}, \circ, d)$, we define $A_\infty$-structure on $\mathcal{C}$ following Sheridan. First, we reverse the directions of all morphisms:
$$\Hom_{A_\infty} (E,F) = \Hom_{dg} (F,E).$$
Then we set
\begin{equation}\label{eqn:convadg}
\begin{array}{l}
m_1 = d \\
m_2(\phi,\psi) = (-1)^{|\phi|} \phi \circ \psi.
\end{array}
\end{equation}
\end{defn}
One can check that this gives an $A_\infty$-structure with respect to the usual Koszul sign rule.

\subsection{Homotopy fiber products of dg-categories is a dg-category}\label{sec:dghptyfp}
Let $B,C,D$ be dg-categories, along with dg-functors
\begin{equation*}
\xymatrix{ & B \ar[d]^G \\
C \ar[r]_{L} & D
}.
\end{equation*}

The (homotopy) fiber product $B \times^h_D C$ is defined to be the category with
\begin{equation*}
{\rm Obj} (B \times_D^h C) = \{ M \in B, N \in C, \phi\in D^0 (G(M),L(N)) \,\,\mbox{with invertible} \, \, [\phi] \,\, \mbox{in} \,\, H^0 (D) \},
\end{equation*}
Pictorially, we may write
$$(M\to)\,\,G(M) \stackrel{\phi}{\to} L(N) \,\, (\leftarrow N). $$
For two objects $(M_1,N_1, \phi_1), (M_2, N_2, \phi_2)$, morphisms are given as
\begin{equation*}
\begin{array}{l}
 (B \times^h_D C)^i ((M_1,N_1, \phi_1), (M_2, N_2, \phi_2)) \\
 = B^i (M_1,M_2) \oplus C^i (N_1,N_2) \oplus D^{i-1} (G(M_1),L(N_2))
\end{array}
\end{equation*}
where a morphism $(\mu,\nu,\gamma)$ fits in to the following diagram.
\begin{equation*}
\xymatrix{ G(M_1) \ar[dr]^{\gamma}\ar[r]^{\phi_1} \ar[d]_{G(\mu)}& L(N_1) \ar[d]^{L(\nu)}\\
G(M_2) \ar[r]^{\phi_2} & L(N_2)
}
\end{equation*}
We do not require that the diagram commutes, but instead define
$$d(\mu,\nu, \gamma)=(d\mu, d\nu, - d\gamma - \phi_2 G(\mu) + L(\nu) \phi_1) $$
so that the $d$-closedness imposes the commutativity up to homotopy $\gamma$
Finally, the composition is defined by
$$(\mu',\nu',\gamma') (\mu,\nu,\gamma) = (\mu' \mu, \nu' \nu, \gamma' G(\mu) +   (-1)^{i'} L(\nu') \gamma).$$

We check that the product is associative.
$$(\mu'',\nu'',\gamma'') ((\mu',\nu',\gamma') (\mu,\nu,\gamma)) = (\mu'',\nu'',\gamma'')  (\mu' \mu, \nu' \nu, \gamma' G(\mu) +  (-1)^{i'} L(\nu') \gamma).$$
\begin{eqnarray*}
3rd &=& \gamma'' G(\mu' \mu) + (-1)^{i''} L(\nu'') (\gamma' G(\mu) +  (-1)^{i'} L(\nu') \gamma)\\
&=& \gamma'' G(\mu' \mu) +(-1)^{i''} L(\nu'') \gamma' G(\mu) + (-1)^{i''+i'} L(\nu'' \nu') \gamma
\end{eqnarray*}

$$((\mu'',\nu'',\gamma'') (\mu',\nu',\gamma')) (\mu,\nu,\gamma) = (\mu'' \mu', \nu'' \nu', \gamma'' G(\mu') +   (-1)^{i''} L(\nu'') \gamma')(\mu,\nu,\gamma).$$
\begin{eqnarray*}
3rd &=& (\gamma'' G(\mu') +   (-1)^{i''} L(\nu'') \gamma') G(\mu) + (-1)^{i''+i'} L(\nu'' \nu') \gamma \\
&=& \gamma'' G(\mu' \mu) +(-1)^{i''} L(\nu'') \gamma' G(\mu) + (-1)^{i''+i'} L(\nu'' \nu') \gamma
\end{eqnarray*}

Other axioms for dg-category can be checked as follows.

 \label{subsec:anpmconvhpty}
\begin{lemma} $d^2 =0$ on $B\times^h_D C$.
\end{lemma}

\begin{proof}
$$d^2 (\mu,\nu,\gamma) = d(d\mu, d\nu, - d\gamma - \phi_2 G(\mu) + L(\nu) \phi_1)$$
whose first and second components are obviously zero. The third component reads
$$-d ( - d \gamma -\phi_2 G(\mu) + L(\nu) \phi_1) - \phi_2 G(d \mu) + L(d \nu) \phi_1.$$
As $|G|=|L|=|\phi_i|=0$, the above equation vanishes as well.
\end{proof}

\begin{lemma} $d$ satisfies the Leibnitz rule.
\end{lemma}
\begin{proof}
Let 
$$A=(\mu,\nu,\gamma): (M_1,N_1,\phi_1) \to (M_2,N_2,\phi_2) ,$$ 
$$B=(\mu',\nu',\gamma'): (M_1,N_1,\phi_2) \to (M_2,N_2,\phi_3),$$ 
be two (composable) morphisms with $\deg A =i$ and $\deg B = i'$. We have to show that
$$ d (BA) = d(B) A + (-1)^{i'} B d(A).$$
The left hand side can be computed as
\begin{eqnarray*}
d(BA) &=& d(\mu' \mu, \nu' \nu, \gamma' G(\mu) +  (-1)^{i'} L(\nu') \gamma) \\
&=& (d(\mu' \mu), d(\nu' \nu), -d(\gamma' G(\mu) +  (-1)^{i'} L(\nu') \gamma) - \phi_3 G(\mu' \mu) + L(\nu' \nu) \phi_1),
\end{eqnarray*}
and hence, the third components reads
 
\begin{equation}\label{eqn:dBA}
\begin{array}{c}
-d \gamma' G(\mu) + (-1)^{i'} \gamma' G(d\mu) + (-1)^{i'+1} L(d \nu') \gamma- L(\nu') d\gamma \\
 - \phi_3 G(\mu' \mu) + L(\nu' \nu) \phi_1.
 \end{array}
\end{equation}

On the other hand,
\begin{eqnarray*}
d(B)A &=& (d\mu',d \nu', -d \gamma' - \phi_3 G( \mu') + L(\nu') \phi_2) (\mu,\nu,\gamma) \\
&=& (d(\mu') \mu, d(\nu') \nu, (-d \gamma' - \phi_3 G( \mu') + L(\nu') \phi_2) G(\mu) +   (-1)^{i'+1} L(d \nu') \gamma)
\end{eqnarray*}
whose third component is 
\begin{equation}\label{eqn:d(B)A}
 -d \gamma'  G(\mu)  + (-1)^{i'+1} L(d \nu') \gamma - \phi_3 G( \mu' \mu)   + L(\nu') \phi_2 G(\mu),
\end{equation}
and
\begin{eqnarray*}
(-1)^{i'}Bd(A) &=& (-1)^{i'} (\mu',\nu',\gamma') (d\mu,d \nu, -d \gamma - \phi_2 G( \mu) + L(\nu) \phi_1)  \\
&=& (-1)^{i'} (\mu'd(\mu) , \nu' d(\nu) , \gamma' G(d \mu) + (-1)^{i'} L(\nu') (-d \gamma - \phi_2 G( \mu) + L(\nu) \phi_1)) \\
&=& ((-1)^{i'} \mu'd(\mu) ,  (-1)^{i'}\nu' d(\nu) , (-1)^{i'} \gamma' G(d \mu) + L(\nu') (-d \gamma - \phi_2 G( \mu) + L(\nu) \phi_1)) 
\end{eqnarray*}
with the third component being
\begin{equation}\label{eqn:BdA}
(-1)^{i'} \gamma' G(d \mu) - L(\nu') d \gamma  + L( \nu' \nu) \phi_1   -L(\nu') \phi_2 G( \mu)
\end{equation}

Therefore, $\eqref{eqn:dBA}= \eqref{eqn:d(B)A} + \eqref{eqn:BdA}$.
\end{proof}

\subsection{Natural Transformations}\label{def:nattrans}
Given two $\AI$-functors $\mathcal{F}_1,\mathcal{F}_2:\mathcal{C} \to \mathcal{D}$, a pre-natural transformation $N$ is defined 
by the following data.
For each integer $k \geq 0$ and $k+1$ tuple of objects $C_0,\cdots, C_d$ of $\mathcal{C}$,
we have a multi-linear map 
$$N(C_0,C_1,\ldots,C_k):Hom(C_0,C_1)\times \cdots \times Hom(C_{k-1},C_k)\to Hom(\mathcal{F}_1(C_0),
\mathcal{F}_2(C_k)).$$
Denote by $Hom(\mathcal{F}_1,\mathcal{F}_2)$ the space of pre-natural transformations from $\mathcal{F}_1$ to $\mathcal{F}_2$.

The differential on  $Hom(\mathcal{F}_1, \mathcal{F}_2)$ is defined by
\begin{equation}\label{eqn:defM1}
\begin{array}{rl}
&M_1(N)(a_1,\ldots,a_n) \\
=&\sum (-1)^{\epsilon_1} m^{\mathcal{D}}_k \big(\mathcal{F}_1(\mathbf{a}^{(1)}),\cdots,\mathcal{F}_1(\mathbf{a}^{(i-1)}),N(\mathbf{a}^{(i)}),\mathcal{F}_2(\mathbf{a}^{(i+1)}), \cdots,\mathcal{F}_2(\mathbf{a}^{(k)})\big)\\
&- \sum (-1)^{\epsilon_2} N( \mathbf{a}^{(1)}, m^{\mathcal{C}}_*(\mathbf{a}^{(2)}),\mathbf{a}^{(3)})
\end{array}
\end{equation}
where $\epsilon_1 = ||N||' \cdot \left( \sum_{j=1}^{i-1} \left|\mathbf{a}^{(j)} \right|'\right)$ and $\epsilon_2 = \left| \mathbf{a}^{(1)} \right|' + ||N||'$. 
Here $||N||'=||N||-1$ is a degree of $N$ as a pre-natural transformation. ($||N|| = |N(C)|$ for any object $C$ of $\mathcal{C}$.) We remark that $||N||$ is the same as the degree of $N(\bullet)$ as a map between graded modules with respect to shifted degrees, i.e. 
$$ ||N||+|\mathbf{a}|' = |N(\mathbf{a}) (x) |' - |x|' = |N(\mathbf{a}) (x) | - |x|= |N(\mathbf{a})|$$
for any $x$.
%($||N||'$ is denoted by $t$ in [Fukaya].)
%{ REDEFINE!}

If the target category of $\mathcal{F}_i$ is ($A_\infty$-)dg, then the above equation reduces to
\begin{equation*}
\begin{array}{rl}
&M_1(N)(a_1,\ldots,a_n) \\
=& m_1^{\mathcal{D}} ( N(\mathbf{a}) ) + (-1)^{||N||' \cdot |\mathbf{a}^{(1)}|'} m_2^{\mathcal{D}} (\mathcal{F}_1 (\mathbf{a}^{(1)}), N (\mathbf{a}^{(2)}) ) +  m_2^{\mathcal{D}} (N (\mathbf{a}^{(1)}), \mathcal{F}_2 (\mathbf{a}^{(2)}) )\\
&-\sum (-1)^{||N||' + |\mathbf{a}^{(1)}|'} N( \mathbf{a}^{(1)}, m^{\mathcal{C}}_*(\mathbf{a}^{(2)}),\mathbf{a}^{(3)}).
\end{array}
\end{equation*}

One can check that $M_1$ is a differential ( See [FukII].
We remark that our sign $M_1$ here differ from Fukaya's one by $-1$.)

A natural transformation is a $M_1$-closed pre-natural transformation.

When $\mathcal{F}_1= \mathcal{F}_2=\mathcal{F}$, identity natural transformation is just given by $N_{id} (C) = \one_{\mathcal{F}(C)}$, where the higher components are all zero. Its degree is given by $||N_{id}||=0$. One can check that $M_1(N_{id})=0$ using the property of the units in the $A_\infty$-category. In fact, the first and the last terms in $M_1(N_{id})$ simply vanishes, and we are only left with
$$ (-1)^{||N_{id}||' \cdot |\mathbf{a}|'}  m_2^{\mathcal{D}} (\mathcal{F} (\mathbf{a}),\one) + m_2^{\mathcal{D}} (\one, \mathcal{F} (\mathbf{a}))$$
$$=(-1)^{ |\mathbf{a}|'} (-1)^{|\mathbf{a} |} \mathcal{F} (\mathbf{a}) +\mathcal{F} (\mathbf{a}) =0. $$
We see that $||N_{id}||'=-1$ is essential to have cancellation.
We remark that this definition of $N_{id}$ differs from Fukaya's one by $-1$ and so does $M_2$ below. 

Given two pre-natural transformation $N_1:\mathcal{F}_1 \to \mathcal{F}_2$, $N_2:\mathcal{F}_2 \to \mathcal{F}_3$, 
its composition is a pre-natural transformation $M_2(N_1,N_2) :\mathcal{F}_1 \to \mathcal{F}_3$ defined as 

\begin{equation*}
\begin{array}{l}
M_2(N_1,N_2) (a_1 ,\cdots, a_n):= 
 \sum (-1)^* m^{\mathcal{D}}_* \left( \hat{\mathcal{F}_1} (\mathbf{a}^{(1)}), N_1(\mathbf{a}^{(2)}), \hat{\mathcal{F}_2} (\mathbf{a}^{(3)}), N_2 (\mathbf{a}^{(4)}), \hat{\mathcal{F}_3} (\mathbf{a}^{(5)})  \right)
%\\ \sum (-1)^* m^{\mathcal{D}}_*\big(\mathcal{F}_1(\mathbf{a}^{(1)}),\cdots,\mathcal{F}_1(\mathbf{a}^{(i-1)}),N_1(\mathbf{a}^{(i)}),\mathcal{F}_2(\mathbf{a}^{i+1}),
%\cdots,\mathcal{F}_2(\mathbf{a}^{(j-1)}),N_2(\mathbf{a}^{(j)}),\mathcal{F}_3(\mathbf{a}^{(j+1)}),
%\cdots,\mathcal{F}_3(\mathbf{a}^{(n)})\big) \\
\end{array}
\end{equation*}
where the sign is given by
$$ \ast = ||N_1||' \cdot |\mathbf{a}^{(1)}|' + ||N_2||' \cdot \left( |\mathbf{a}^{(1)}|'+ |\mathbf{a}^{(2)}|' + |\mathbf{a}^{(3)}|' \right).$$

Let us first check that $N_{id}$ gives the multiplicative identity with respect to $M_2$ (in $A_\infty$-setting). Let $N_{id} : \mathcal{F} \to \mathcal{F}$ and $N : \mathcal{F} \to \mathcal{G}$. Then
$$ M_2(N_{id}, N)(\mathbf{a})= (-1)^{0} m_\ast^{\mathcal{D}} (\overbrace{N_{id}}^{\rm 0th-comp.},N(\mathbf{a})  ) = m_2^{\mathcal{D}} (id, N (\mathbf{a})) = N (\mathbf{a}).$$
For a natural transformation $N' : \mathcal{G} \to \mathcal{F}$ in the other direction,
$$ M_2 (N', N_{id} ) (\mathbf{a}) = (-1)^{||N_{id}||' \cdot |\mathbf{a}|'} m_2^{\mathcal{D}} ( N' (\mathbf{a}), N_{id}) = (-1)^{|\mathbf{a}|'} m_2^{\mathcal{D}} (N' (\mathbf{a}), id ) = (-1)^{|\mathbf{a}|'} (-1)^{|N' (\mathbf{a}) |} N' (\mathbf{a})$$
since $||N_{id}||' = -1$. Now
$$(-1)^{|\mathbf{a}|' + |N'(a)|} = (-1)^{|N'(a) | - |\mathbf{a}|'  }  = (-1)^{||N' ||}$$
%($|N'(\mathbf{a})(x)| - |x|$ for any $x$ measures the degree of $|N'(\mathbf{a})|$ as a map.

See [FukII] for compatibility between $M_1$ and $M_2$.

If we have two functors  $\mathcal{F}_1,\mathcal{F}_2 : \mathcal{C} \to \mathcal{D}$,
such that they are the same on the level of objects, then we can define the notion of homotopy between these
two functors. Namely, a natural transformation $H$ with one less degree satisfying
 $\mathcal{F}_1 - \mathcal{F}_2  = M_1(H)$. One can check that homotopy of $\AI$-functors is an equivalence relation.

Given homotopy $H_1,H_2$ Homotopies of $\AI$-functors can be composed
$H_2 \circ H_1 = H_1 + H_2 + M_2(H_1,H_2)$.

\bibliographystyle{amsalpha}
\bibliography{geometry}

\end{document}